\documentclass[11pt]{amsart}
\usepackage{amssymb,verbatim}
\usepackage{epic,eepic}
\usepackage{graphicx}
\usepackage{epsfig}
\usepackage{xspace}
\usepackage[all]{xy}
\usepackage[refpage]{nomencl}

\newcommand\cornX{Y}

\newcommand\cornM{W}

\newcommand\man{Q}

\newcommand\Vfeb{\Vf_{\ebo}}
\newcommand\Vfes{\Vf_{\eso}}

\newcommand\cte{\digamma}

\newcommand{\rcal}{{\mathcal{R}}}
\newcommand\cF{\mathcal{F}}

\newcommand\cM{\mathcal{M}}
\newcommand{\GBB}{\textsf{GBB}\xspace}
\newcommand{\GBBs}{\textsf{GBB}s\xspace} 
\newcommand{\EGBB}{\textsf{EGBB}\xspace}
\newcommand{\EGBBs}{\textsf{EGBB}s\xspace}
\newcommand{\ssM}{\mathsf{M}}

\newcommand{\parameter}{\digamma}

\newcommand{\hilbert}{\mathfrak{X}}
\newcommand{\hilbertx}{\mathfrak{X}}
\newcommand{\hilbertp}{\mathfrak{Y}}
\newcommand{\hilbertz}{\mathfrak{Z}}
\newcommand{\hilbertpp}{\widetilde{\mathfrak{Y}}}
\newcommand{\hilberth}{\mathfrak{H}}

\newcommand{\htil}{\tilde h}
\newcommand{\Htil}{\tilde H}

\newcommand{\gammat}{\tilde\gamma}

\newcommand\xib{\xi^{\bo}}
\newcommand\etab{\eta^{\bo}}

\newcommand\taub{\tau^{\bo}}

\newcommand{\bSigma}{{}^{\bo}\Sigma_0}
\newcommand{\sSigma}{{}^{\so}\Sigma_0}
\newcommand{\ebSigma}{{}^{\ebo}\Sigma}
\newcommand{\esSigma}{{}^{\eso}\Sigma}

\newcommand{\bSigmapm}{{}^{\bo}\Sigma_{0,\pm}}
\newcommand{\sSigmapm}{{}^{\so}\Sigma_{0,\pm}}

\newcommand\xies{\underline{\xi}}
\newcommand\etaes{\underline{\eta}}
\newcommand\zetaes{\underline{\zeta}}
\newcommand\taues{\underline{\tau}}

\newcommand\sigmaes{\underline{\sigma}}

\newcommand\xiesh{\hat{\underline{\xi}}}
\newcommand\tauesh{\hat{\underline{\tau}}}
\newcommand\etaesh{\hat{\underline{\eta}}}
\newcommand\zetaesh{\hat{\underline{\zeta}}}

\newcommand\xieb{\xi}
\newcommand\etaeb{\eta}
\newcommand\zetaeb{\zeta}
\newcommand\taueb{\tau}
\newcommand\sigmaeb{\sigma}

\newcommand\xiebh{\hat\xi}
\newcommand\etaebh{\hat\eta}
\newcommand\zetaebh{\hat\zeta}

\newcommand\etah{{\hat\eta}}
\newcommand\tauh{{\hat\tau}}

\newcommand\taubh{\hat\tau^{\bo}}
\newcommand\xibh{\hat\xi^{\bo}}
\newcommand\etabh{\hat\eta^{\bo}}
\newcommand\zetabh{\hat\zeta^{\bo}}

\newcommand\xis{\xi^{\so}}
\newcommand\etas{\eta^{\so}}
\newcommand\zetas{\zeta^{\so}}
\newcommand\taus{\tau^{\so}}

\newcommand\taush{{\hat\tau^{\so}}}
\newcommand\xish{{\hat\xi^{\so}}}
\newcommand\etash{{\hat\eta^{\so}}}

\newcommand{\tLambda}{\tilde{\Lambda}}

\newcommand{\Psieb}[1]{\Psi_{\ebo}^{#1}}

\newcommand{\Teb}{{}^{\ebo}T}
\newcommand{\Tes}{{}^{\eso}T}
\newcommand{\Testar}{{}^{\eo}T^*}
\newcommand{\Tebstar}{{}^{\ebo}T^*}
\newcommand{\Tesstar}{{}^{\eso}T^*}
\newcommand{\Sesstar}{{}^{\eso}S^*}

\newcommand{\Sebstar}[1][\mbox{}]{{}^{\ebo}S^*_{#1}}
\newcommand{\Sestar}[1][\mbox{}]{{}^{\eo}S^*_{#1}}
\newcommand{\Tbstar}{{}^{\bo}T^*}

\newcommand{\Sbstar}{{}^{\bo}S^*}
\newcommand{\abs}[1]{{\left\lvert{#1}\right\rvert}}
\newcommand{\norm}[1]{{\left\lVert{#1}\right\rVert}}
\renewcommand{\Box}{{\square}}

\newcommand{\bH}[1]{{H^{#1}_b}}

\newcommand{\eH}[1]{{H^{#1}_e}}

\DeclareMathOperator{\liptic}{\text{ell}}

\DeclareMathOperator{\Op}{\text{Op}}

\DeclareMathOperator{\eWF}{{{WF}_{\eo}}}

\newcommand{\hcal}{{\mathcal{H}}}
\newcommand{\ecal}{{\mathcal{E}}}
\newcommand{\kcal}{{\mathcal{K}}}
\newcommand{\fcal}{\mathcal{F}^{\ebo}}
\newcommand{\fcalW}{\mathcal{F}^{\bo}}

\newcommand{\ang}[1]{{\left\langle{#1}\right\rangle}}

\newcommand{\ep}{{\epsilon}}

\newcommand{\gcal}{{\mathcal{G}}}

\newcommand{\ePs}[1]{{\Psi^{#1}_{\eo}}}

\newcommand\ff{\operatorname{ff}}
\newcommand\ef{\tilde W}
\newcommand\efspace{\tilde Y}  

\def\noqed{\renewcommand{\qed}{}}

\numberwithin{equation}{section}
\newtheorem{theorem}{Theorem}[section]
\newtheorem{lemma}[theorem]{Lemma}
\newtheorem{proposition}[theorem]{Proposition}
\newtheorem{corollary}[theorem]{Corollary}

\newtheorem{non-theorem}[theorem]{Non-Theorem}

\theoremstyle{remark}
\newtheorem{definition}[theorem]{Definition}
\newtheorem{remark}[theorem]{Remark}

\newcommand\bo{\operatorname{b}}

\newcommand\eo{\operatorname{e}}
\newcommand\so{\operatorname{s}}
\newcommand\ebo{\operatorname{eb}}
\newcommand\eso{\operatorname{es}}

\newcommand{\eb}{\ebo}

\newcommand\Vb{\mathcal{V}_{\bo}}

\newcommand\bV{\mathcal{V}_{\bo}}

\newcommand\Heb[1]{H_{\ebo}^{#1}}

\newcommand\Hes[1]{H_{\eso}^{#1}}
\newcommand\dHes[1]{\dot H_{\eso}^{#1}}
\newcommand\Hesz[1]{H_{\eso,0}^{#1}}

\newcommand\Vf{\mathcal{V}}

\newcommand\bT{{}^{\bo}T}

\newcommand\cFTs{{}^{\Phi}\overline{T}\kern-1pt{}^*}

\newcommand\inside{\operatorname{int}}

\newcommand\WF{\operatorname{WF}}

\newcommand\WFeb{\WF_{\ebo}}
\newcommand\WFebX{\WF_{\ebo,\hilbert}}
\newcommand\WFebXp{\WF_{\ebo,\hilbert'}}
\newcommand\WFebY{\WF_{\ebo,\hilbertp}}

\newcommand\bWF{\WF_{\bo}}
\newcommand\WFb{\WF_{\bo}}

\newcommand\WFbHp{\WF_{\bo,\mathfrak{H}^*}}
\newcommand\WFbX{\WF_{\bo,\hilbert}}
\newcommand\WFbXp{\WF_{\bo,\hilbert^*}}

\newcommand\ebWF{\WF_{\ebo}}

\newcommand\cA{\mathcal{A}}
\newcommand\cB{\mathcal{B}}
\newcommand\tcB{\widetilde{\mathcal{B}}}
\newcommand\cC{\mathcal{C}}
\newcommand\cD{\mathcal{D}}
\newcommand\cH{\mathcal{H}}

\newcommand\cO{\mathcal{O}}
\newcommand\cR{\mathcal{R}}

\newcommand\cQ{\mathcal{Q}}

\newcommand\cG{\mathcal{G}}
\newcommand\cE{\mathcal{E}}

\newcommand{\ssW}{\mathsf{W}}

\newcommand\CC{\mathbb C}

\newcommand\NN{\mathbb N}
\newcommand\bbN{\mathbb N}

\newcommand\RR{\mathbb R}
\newcommand\bbR{\mathbb R}

\newcommand\CIc{{\mathcal{C}}^{\infty}_c}

\newcommand\CI{{\mathcal{C}}^{\infty}}

\newcommand\CmI{{\mathcal{C}}^{-\infty}}

\newcommand\Diag{\operatorname{Diag}}

\newcommand\Diff[1]{\operatorname{Diff}^{#1}}

\newcommand\Diffes[1]{\operatorname{Diff}^{#1}_{\eso}}
\newcommand\Diffesd[1]{\operatorname{Diff}^{#1}_{\eso,\dagger}}
\newcommand\Diffess[1]{\operatorname{Diff}^{#1}_{\eso,\sharp}}
\newcommand\Diffeb[1]{\operatorname{Diff}^{#1}_{\ebo}}

\newcommand\Psib{\Psi_{\bo}}

\newcommand\cFNs{{}^{\Phi}\overline N\kern-1pt{}^*}

\newcommand\reg{\operatorname{reg}}
\newcommand\sing{\operatorname{sing}}

\newcommand\Id{\operatorname{Id}}
\newcommand{\lag}{\mathcal{L}}
\newcommand\Lap{\varDelta}

\newcommand\Span{\operatorname{sp}}
\newcommand\Ran{\operatorname{Ran}}

\newcommand\clos{\operatorname{cl}}

\newcommand\dCI{\dot{\mathcal{C}}^{\infty}}

\newcommand\loc{{\text{loc}}}
\newcommand\comp{{\text{c}}}

\newcommand\pa{\partial}

\newcommand\sgn{\operatorname{sgn}}

\newcommand\supp{\operatorname{supp}}

\newcommand\Mand{\text{ and }}
\newcommand\Mat{\text{ at }}

\newcommand\Mon{\text{ on }}

\newcommand{\F}{{\mathcal{F}}}
\newcommand{\coiso}{{\mathcal{A}}}
\newcommand{\module}{{\mathcal{M}}}

\newcommand{\pieseb}{\pi_{\eso\to\ebo}}
\newcommand{\piesebh}{\widehat{\pi_{\eso\to\ebo}}}

\newcommand{\pismooth}{\pi_{\so\to \bo}}
\newcommand{\dbetaeb}{\varpi_{\ebo}} 
\newcommand{\dbetaes}{\varpi_{\eso}}

\newcommand{\cU}{\mathcal{U}}

\newcommand{\sH}{\mathsf{H}}

\author{Richard Melrose}
\author{Andr\'as Vasy}
\author{Jared Wunsch}

\title[Diffraction on manifolds with corners]%
{Diffraction of singularities for the wave equation on manifolds with
  corners}
\date{\today}

\makenomenclature

\begin{document}

\begin{abstract}
  We consider the fundamental solution to the wave equation on a
  manifold with corners of arbitrary codimension.  If the initial pole
  of the solution is appropriately situated, we show that the
  singularities which are diffracted by the corners (i.e., loosely
  speaking, are not propagated along limits of transversely reflected
  rays) are smoother than the main singularities of the solution.
  More generally, we show that subject to a hypothesis of nonfocusing,
  diffracted wavefronts of any solution to the wave equation are
  smoother than the incident singularities.  These results extend our
  previous work on edge manifolds to a situation where the fibers of
  the boundary fibration, obtained here by blowup of the corner in question,
are themselves manifolds with corners.
\end{abstract}

\maketitle

\newpage

\tableofcontents

\newpage

\section{Introduction}
\subsection{The problem and its history}

Let $X_0$ be a manifold with corners, of dimension $n,$ i.e., a
manifold locally modeled on $(\RR^+)^{f+1} \times \RR^{n-f-1},$
endowed with an incomplete metric, smooth and non-degenerate up to the
boundary. We consider the wave equation 
\begin{equation}
\Box u \equiv D_t^2u-\Lap u=0\Mon M_0=\RR\times X_0,
\label{30.8.2006.37}\end{equation}
where $D_t = \imath^{-1} (\pa/\pa t)$ and $\Lap$ is
the nonnegative Laplace-Beltrami operator; we will impose either
Dirichlet or Neumann conditions at $\pa X_0.$ As is well known
by the classic result of Duistermaat-H\"ormander\footnote{This result,
viewed in the context of hyperbolic equations, built on a
considerable body of work prior to the introduction of the wavefront
set; see especially \cite{Lax}, \cite{Ludwig}.} (see \cite{fio2}),
the wavefront set of a solution $u$ propagates along
null-bicharac\-teristics in the interior. However, the behavior of
singularities striking the boundary and corners of $M_0$ is considerably subtler.

Indeed the propagation of singularities for the wave equation on manifolds
with boundary is already a rather subtle problem owing the the difficulties
posed by ``glancing'' bicharacteristics, those which are tangent to the
boundary. Chazarain \cite{Chazarain} showed that singularities striking the boundary
transversely simply reflect according to the usual law of geometric optics
(conservation of energy and tangential momentum, hence ``angle of incidence
equals angle of reflection'') for the reflection of bicharacteristics. This
result was extended in \cite{Melrose-Sjostrand1} and
\cite{Melrose-Sjostrand2} by showing that, at glancing points,
singularities may only propagate along certain generalized
bicharacteristics. The continuation of these curves may fail to be unique
at (non-analytic) points of infinite-order tangency as shown by Taylor
\cite{Taylor1}. Whether all of these branches of bicharacteristics can carry
singularities is still not known.

As was shown initially in several special examples (namely those amenable
to separation of variables) the interaction of wavefront set with a corner
gives rise to new, \emph{diffractive} phenomena, in which a single
bicharacteristic carrying a singularity into a corner produces 
singularities departing from the corner along a whole family of
bicharacteristics. For instance, a ray carrying a singularity transversely
into a codimension-two corner will in general produce singularities on the
entire cone of rays reflected in such a way as to conserve both energy and
momentum tangent to the corner (see Figure~\ref{figure:diffracted})
\begin{figure}[ht] \includegraphics{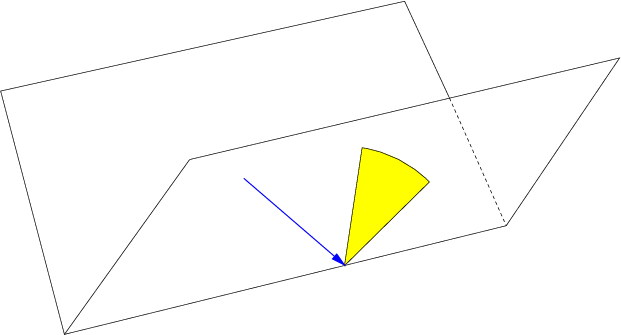} \caption{A ray carrying a
    singularity may strike a corner of codimension two and give rise to a
    whole family of diffracted singularities, conserving both energy and
    momentum along the corner.}  \label{figure:diffracted} \end{figure} The
first diffraction problem to be rigorously treated was that of the exterior
of a wedge,\footnote{This is not in fact a manifold with corners, but is
  quite closely related.} which was analyzed by Sommerfeld
\cite{Sommerfeld1}; subsequently, many related examples were analyzed by
Friedlander \cite{MR20:3703}, and more generally the case of exact cones
was worked out explicitly by Cheeger-Taylor \cite{Cheeger-Taylor2},
\cite{Cheeger-Taylor1} in terms of the functional calculus for the Laplace
operator on the cross section of the cone.  All of these explicit examples
reveal that generically a diffracted wave arises from the interaction of
wavefront set of the solution with singular strata of the boundary of the
manifold; this has long been understood at a heuristic level, with the
geometric theory of diffraction of Keller \cite{Keller1} describing the
classes of trajectories that ought to contribute to the asymptotics of the
solution in various regimes.

Subsequent work has been focused primarily on characterizing the
bi\-char\-ac\-ter\-is\-tics on which singularities can propagate, and on
describing the strength and form of the singularities that arise.  The
propagation of singularities on manifolds with boundary was first
understood in the analytic case by Sj\"ostrand
\cite{Sjostrand3,Sjostrand2,Sjostrand4}, and subsequently generalized
to a very wide class of manifolds, including manifolds with corners,
by Lebeau \cite{Lebeau4,Lebeau5}.  In the $\mathcal{C}^\infty$ setting
employed here, the special case of manifolds with conic singularities
was studied by Melrose-Wunsch \cite{Melrose-Wunsch1} and \emph{edge
  manifolds} (i.e., cone bundles) were considered by
Melrose-Vasy-Wunsch \cite{mvw1}.  Vasy \cite{Vasy5} obtained results
analogous to Lebeau's in the case of manifolds with corners, and it is
the results of this work that directly bear on the situation studied
here.

While the foregoing results characterize which bicharacteristics may
carry singularities for solutions to the wave equation, they ignore
the question of the regularity of the diffracted front.  In general, a
singularity in $\WF^s$ (which is to say, measured with respect to
$H^s$) must propagate \emph{strongly} in the sense that some
bicharacteristics through the point in question must also lie in
$\WF^s.$ The general expectation is that these are certain
``geometric'' bicharacteristics; in simple cases, these are known to
be those which are locally approximable by bicharacteristics that miss
the corners and reflect transversely off boundary hypersurfaces.  More
generally, we can define geometric bicharacteristics as follows: To
begin, we \emph{blow up} the corner, i.e.\ introduce polar coordinates
around it; this serves to replace the corner with its inward-pointing
normal bundle, which fibers over the corner with fiber given by one
orthant of a sphere, $S^f \cap (\RR^+)^{f+1}.$ We will define
\emph{geometric} broken bicharacteristics passing through the corner
as those that lift to the blown-up space to enter and leave the lift
of the corner at points connected by \emph{generalized broken
  geodesics} of length $\pi$ with respect to the naturally defined
metric on $S^f \cap (\RR^+)^{f+1},$ undergoing specular reflection at
its boundaries and corners.\footnote{The actual definition is
  considerably complicated by the existence of glancing rays, and is
  discussed in detail in \S\ref{subsection:bichars}.}
Bicharacteristics that enter and leave the corner at points in a fiber
that are not at distance-$\pi$ in this sense are referred to as
``diffractive.''  See
Figure~\ref{fig:cornerimp}. \begin{figure}[tc] \begin{center}
    \mbox{\epsfig{file=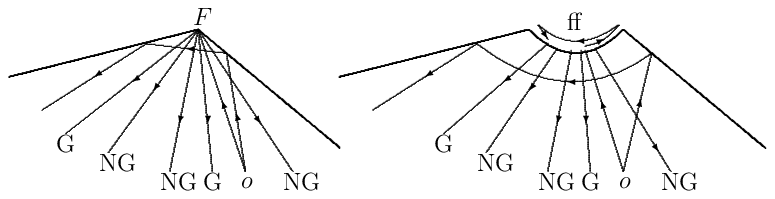, scale=0.9}} \end{center}
  \caption{Geometric optic rays hitting a corner $F$, emanating from a
    point $o$. The rays labelled $G$ are geometric at $F$, while those
    labelled $NG$ are non-geometric at $F$. The leftmost geometric ray is a
    limit of rays like the unlabelled one shown on the figure that just
    miss $F$.  The blown up version of the picture, i.e.\ where
    $(r,\theta)$ are introduced as coordinates at the origin, is shown on
    the right.  The front face (i.e., the lift of the corner) is denoted
    $\ff.$ The reflecting line indicates the broken geodesic of length
    $\pi$ induced on $\ff$ given by $r=0$, $\theta\in[0,\theta_0]$.  The
    total length of the three segments shown on $\ff$ is $\pi$; this can be
    thought of as the sum of three angles on the picture on the left:
    namely the angles between the incident ray and the right boundary
    (corresponding to the first segment), the right and left boundaries,
    finally the left boundary and the emanating reflected ray.}
  \label{fig:cornerimp} \end{figure}

It turns out that subject to certain hypotheses of \emph{nonfocusing,}
the singularities propagating along diffractive bicharacteristics
emanating from the corner will be weaker than those on geometric
bicharacteristics.  In
particular, the fundamental solution satisfies the nonfocusing
condition, hence one consequence of our main theorem is as
follows:

\begin{theorem}\label{intro-theorem} Consider the fundamental solution $u_o$ to the wave
  equation, with pole $o$ sufficiently close to a corner, $Y,$ of
  codimension $k.$
Assume that $o$ is sufficiently far from the
  boundary that every short geodesic from $o$ to $Y$ is transverse to
  all boundary hypersurfaces intersecting at $Y.$

  While $u_o$ lies
  locally in $H^{-n/2+1-0},$ it is less singular
  by  $(k-1)/2$ derivatives along diffractive bicharacteristics emanating
  from $Y,$ that is, it lies microlocally in $H^{(-n+k+1-0)/2}$ there. \footnote{Here and henceforth we emply the notation $s-0$
    to mean $s-\ep$ for all $\ep>0$.}
\end{theorem}

\noindent A more precise version of this result (with ``sufficiently
close\dots'' elucidated) appears in Corollary~\ref{corollary:fundsoln}.

A more general theorem on regularity of the diffracted wave subject
to the nonfocusing condition is the central result of this paper.  See
\S\ref{subsec:hypothesis} for a rough statement of the nonfocusing condition
and \S\ref{section:coisodef} for technical details; the main theorem
is stated and proved in \S\ref{section:geometric}.

There are a few related results known in special cases. G\'erard-Lebeau
\cite{MR93f:35130} explicitly analyzed the problem of an analytic conormal
wave incident on an analytic corner in $\RR^2,$ obtaining a
$1/2$-derivative improvement of the diffracted wavefront over the incident
one. The first and third authors \cite{Melrose-Wunsch1} obtained corresponding results for
manifolds with conic singularities, which the authors subsequently
generalized to the case of edge manifolds \cite{mvw1}.

We remark that our definition of geometric broken
bicharacteristics includes those that interact with the boundaries and
corners of the front face of the blow-up, $S^f \cap (\RR^+)^{f+1},$
according to the simplest laws of reflection as described in
\cite{Vasy5}: we do not distinguish between ``diffractive'' and
``geometric'' interactions within $S^f \cap (\RR^+)^{f+1}.$ We
conjecture that a stronger theorem than ours should hold in which,
instead of simply blowing up the highest-codimension corner, we might
iteratively blow up the corners of lower codimension as well.  This
would enable us to (iteratively) distinguish bicharacteristics that
undergo diffractive or geometric interaction inside the faces of the
blown-up space.  For instance, in the case of a codimension-$3$
corner, such a method would distinguish among rays that are limits
of families of geodesics undergoing simple specular
reflection with boundary hypersurfaces (which we might continue to call
\emph{geometric rays}); limits of families which undergo a single
diffractive interaction with a codimension-$2$ corner (\emph{partially
  diffractive rays}) and all other generalized broken
bicharacteristics entering and leaving the codimension-$3$ corner (the
\emph{completely diffractive rays}).  Our Theorem~\ref{intro-theorem}
only deals with the regularity along the completely diffractive rays,
telling us that the fundamental solution is $(3-1)/2$ derivatives
smoother along them than overall; the conjectural finer result would
also tell us about the partially diffractive rays, yielding an
improvement of $(2-1)/2$ derivatives there.  More generally, such a
result ought to yield a stratification of the rays interacting with a
corner of arbitrary codimension into pieces carrying different levels
of differentiability according to the degree of diffractive
interaction.

\subsection{The hypothesis}\label{subsec:hypothesis}

We now describe the nonfocusing hypothesis in more detail,
in the context of the simplest geometric situation to which our
results apply.

It is easily seen from the explicit form of the fundamental solution that
it is not in general true that diffracted rays are more regular than
incident singularities.  For example, take $\Lap$ to be the Dirichlet Laplacian
in a sectorial domain $\{r \in [0,\infty), \theta \in
[0,\theta_0]\}$ in $\RR^2,$ and consider the solution
\begin{equation}\label{coisoexample}
\frac{\sin t\sqrt \Lap}{\sqrt \Lap} \phi(\theta) \delta(r-r_0),
\end{equation}
where $\phi \in \mathcal{C}_c^\infty((0,\theta_0))$ is supported close to
some value $\theta'.$ This solution is manifestly locally in $H^{1/2-0}$ by
energy conservation.  On the other hand one may see from the explicit form
of the propagator in \cite{Cheeger-Taylor2}, \cite{Cheeger-Taylor1} after
convolution with $\phi(\theta)$ that a spherical wave of singularities
emanates from the corner at time $t=r_0,$ with regularity $H^{1/2-0},$
hence the same as the overall Sobolev regularity of the solution.  The
bicharacteristics along which singularities propagate are, for short time,
just the lifts of the straight lines $r=r_0 \pm t,$ $\theta \in \supp
\phi,$ hence travelling straight into or out of the vertex.  Perturbing
these slightly to make them miss the vertex, we see that in fact there are
two ``geometric'' continuations\footnote{What a geometric continuation of a
  bicharacteristic is in general will be elucidated in
  \S\ref{subsection:bichars}.} for each bicharacteristic, depending on
whether we approximate it by geodesics passing to the left or to the right
of the vertex (see Figure~\ref{fig:cornerimp}).  Thus, the geometric
continuations of the rays on which singularities strike the vertex are
close to the two possible continuations of the single ray $\theta=\theta',$
hence do not include all points on the outgoing spherical wave.  So we have an
example in which there are ``non-geometric'' singularities of full
strength.

The nonfocusing condition serves exactly to rule out such situations.
The above example has the property that applying negative powers of
$(\Id +D_\theta^2)$ does not regularize the short-time solution (or
the initial data) as it is already $\CI$ in the $\theta$ direction.
In this simple setting, the nonfocusing condition says precisely that
the solution is regularized by negative powers of $(\Id +D_\theta^2),$
or, equivalently, that it can be written
$$
(\Id +D_\theta^2)^Nv,\ v\in H^s
$$
for some $s$ exceeding the overall Sobolev regularity.  For instance,
the fundamental solution
$$
u=\frac{\sin t\sqrt \Lap}{\sqrt \Lap} \delta(r-r')\delta(\theta-\theta')
$$
looks, after application of a sufficiently negative power of $(\Id
+D_\theta^2),$ like a distribution of the form
$$
\frac{\sin t\sqrt \Lap}{\sqrt \Lap} \delta(r-r')f(\theta)
$$
with $f \in \mathcal{C}^M,$ $M\gg 0,$ hence we can write 
$$
u \in (\Id +D_\theta^2)^N H^{1/2-0},
$$
for some $N\gg 0,$ at least locally, away from the boundary.  We also
observe that the example \eqref{coisoexample} enjoys a property which
is essential \emph{dual} to the nonfocusing condition, to wit, fixed
regularity under repeated application of $D_\theta.$  We refer to this
property as ``coisotropic regularity'' (the terminology will be
explained in \S\ref{section:coisodef}) and it plays an essential role
in our proof.

The nonfocusing condition and coisotropic regularity in a more general
setting are subtler owing to their irreducibly microlocal nature: the
operator $D_\theta$ has to be replaced by a family of operators
characteristic along the flow-out of the corner under consideration.

\subsection{Structure of the proof}

We now describe the logical structure of the proof, as it is somewhat
involved.  The heart of the argument is a series of results on the
propagation of singularities, obtained by positive
commutator methods; these are sketched in detail in
\S\ref{subsec:propagationsketch} below.  In order to be able to
distinguish between ``geometric'' and ``diffractive'' bicharacteristics at
a corner of $M_0,$ we begin by performing a blow-up of the corner,
i.e.\ introducing polar coordinates centered at it, to obtain a new
manifold with corners $M.$ The commutants that we employ in our commutator
argument \emph{almost} lie in a calculus of pseudodifferential operators,
the edge-b calculus, that behaves like Mazzeo's edge calculus
\cite{MR93d:58152} near the new boundary face introduced by the blow-up
(henceforth, the \emph{edge}) and like Melrose's b calculus
\cite{MR83f:58073, Melrose:Atiyah, Melrose-Piazza:Analytic} at the
remaining boundary faces.  The complication, as in the previous work of
Vasy \cite{Vasy5} on propagation of singularities, is that in order to
control certain error terms we in fact must employ a hybrid
differential-pseudodifferential calculus, in which we keep track of certain
terms involving differential operators normal to the boundary faces other
than $\ef.$

Even this propagation result alone is insufficient to obtain our result, as
it does not allow regularity of greater than a certain degree to propagate
out of the edge, with the limitation in fact not exceeding the a priori
regularity of the solution.  What it does allow for, however, is the
propagation of \emph{coisotropic} regularity of arbitrarily high order,
suitably microlocalized in the edge-b sense.  This allows us to conclude
that given a ray $\gamma$ leaving the edge, if the solution enjoys
coisotropic regularity along all rays incident upon the edge that are
geometrically related to $\gamma$, then we may conclude coisotropic
regularity along $\gamma$ as well.  (If some of these incident rays are
glancing, i.e.\ tangent to the boundary, we require as our hypothesis
actual \emph{differentiability} globally at all incoming glancing rays
rather than coisotropic regularity, which no longer makes sense; the
version of the commutator argument that deals with these rays is the most
technically difficult aspect of the argument.)  In particular, then,
\emph{global} coisotropic regularity together with $\CI$ regularity at
glancing rays implies global coisotropic regularity leaving the edge away
from glancing.  We are then able to dualize this result to show that the
nonfocusing condition propagates as well.

Consequently, we show that subject to the nonfocusing condition, in
the model case of the sector considered above, if $\gamma$ is an outgoing ray
such that the solution is $\CI$ along all incoming rays
geometrically related to it,
$$
u \in (\Id + D_\theta)^N H^s \text{ along } \gamma \text{ for some } N
\in \NN,
$$
where in general $s=(-n+k+1)/2-0$ for the fundamental solution near a
codimension-$k$ corner on an $n$-manifold, hence $s=1/2-0$
for the sector.
On the other hand the microlocal propagation of coisotropic
regularity shows that 
$$
D_\theta^k u \in H^{s_0} \text{ for all } k \text{ along } \gamma
$$
where $s_0$ is the overall regularity of the solution ($-n/2+1-0$ for
the fundamental solution).  An interpolation argument then yields
$$
u \in H^{s-0} \text{ along } \gamma,
$$
proving the theorem.

\subsection{Sketch of the propagation results}\label{subsec:propagationsketch}
We now discuss the propagation results in greater detail, focusing on
the taxonomy of the various spaces of operators that we employ. The
basic  propagation of singularities result
on manifolds with corners $M_0=\RR_t \times X_0$, as on manifolds with
smooth boundaries, is in the setting of b-, or \emph{totally
  characteristic}, operators. Let us choose local coordinates
$(x_1,\ldots,x_{f+1},y_1,\ldots,y_{n-f-1})$ on $X_0$ with $\{x_1\geq
0,\ldots,x_{f+1}\geq 0\};$ thus, $Y=\{x_1=\ldots=x_{f+1}=0\}$
represents a corner of codimension $f+1.$ The b-vector fields
$\Vb(X_0)$ are the linear combinations of $x_j\pa_{x_j}$ and
$\pa_{y_j}$ with $\CI$ coefficients---these are exactly the vector
fields tangent to all boundary hypersurfaces.  We can define an
associated notion of \emph{b-regularity} by iterated regularity under
repeated application of such vector fields.  In particular, for a
distribution $u$, b-regularity relative to a space, such as
$H^1(X_0)$, means that $(x\pa_x)^\alpha \pa_y^\beta u\in H^1(X_0)$ for
all multiindices $\alpha,\beta$ (with
$(x\pa_x)^\alpha=(x_1\pa_{x_1})^{\alpha_1}\ldots(x_{f+1}\pa_{x_{f+1}})^{\alpha_{f+1}}$).
Thus $u$ is b-regular if and only if $u$ is a \emph{conormal}
distribution. Replacing $X_0$ by $M_0$ simply adds $\pa_t$ to the
collection of b-vector fields, i.e.\ $t$ behaves as one of the $y$
variables.  The notion of b-regularity is microlocalized via the
b-pseudodifferential operators, which are roughly speaking operators
of the form $a(x,y,t,xD_x,D_y,D_t)$ where $a$ is a symbol in the last
three sets of variables.  The calculus of these operators gives rise
to a notion of \emph{b-wavefront set,} which is therefore a microlocal
measure of conormality.

The wave operator itself is {\em not} a b-differential operator,
rather a standard differential operator: it is constructed out of
the vector fields $\pa_{x_j}$
rather than $x_j \pa_{x_j}$. Thus, its principal symbol,
hence its bicharacteristics, are curves in the cotangent bundle
$T^*M_0$, which is equipped with canonical coordinates
$(x,y,t,\xis,\etas,\taus)$, corresponding to differential operators
$(x,y,t,D_x,D_y,D_t)$. One cannot work with pseudodifferential
operators based on these standard differential operators, for they
would usually not act on smooth functions in $x\geq 0$, and would not
usually
preserve the boundary conditions. Thus, one works with b-operators,
based on
$(x,y,t,xD_x,D_y,D_t)$, which corresponds
to localizing in conic neighborhoods in the corresponding canonical coordinates
$(x,y,t,\xib,\etab,\taub)$ in the cotangent bundle.  These coordinates
are related to the original ones via
$$
(x,y,t,\xib,\etab,\taub)
 =(x,y,t,x\xis,\etas,\taus).$$
In particular, at $x=0$, passing to the b-coordinates identifies points with
different values of the normal momentum $\xis.$  Continuous propagation in
the b-variables thus allows $\xis$ to jump at the boundary, as occurs
in specular reflection.  It is the phenomenon of propagation along
appropriate generalized bicharacteristics in the b-variables that was
studied in \cite{Vasy5}.

In order to have a more precise result, we need to be able to localize
within the fibers of the blow-up of the corner $Y$, and we also need
to be able to undo the compression of the dynamics implied by working
in the b-picture.  It is only through these refinements that we are
able to distinguish microlocal behavior along different
bicharacteristics hitting the corner $Y$ at the same point and with
the same tangential momenta $\etab,$ e.g.\ between different geodesics
in the conical spray shown leaving the corner in
Figure~\ref{figure:diffracted}.  Therefore we lift the Laplacian on
$X_0$ to the blow-up $X$ of $X_0$ at $Y$, denoted
$X=[X_0;Y]$. For simplicity of notation, assume that $Y$ is a
codimension $2$ corner (cf.\ Figure~\ref{fig:cornerimp} as well as
Figure~\ref{blowupfigure} below). Using polar coordinates $(r,\theta)$
in the $(x_1,x_2)$ we see that under this blow-up smooth vector fields
on $X$ lift to vector fields of the form $r^{-1}V$, where $V$ is
tangent to the fibers of the blow-down map, i.e.\ is a linear
combination of $r\pa_r,\pa_\theta,r\pa_y$ with $\CI$
coefficients.\footnote{In Figure~\ref{blowupfigure} as well as in the
  main exposition, $r$ and $\theta$ are denoted $r,z;$ we preserve the
  more usual radial coordinate notation here for purposes of
  exposition.}  The $\CI$ span of $r\pa_r,\pa_\theta,r\pa_y$ are the
so-called edge-smooth vector fields defined below in
Section~\ref{section:edgebgeom}.  Away from the boundaries,
$\theta=0,\pi/2$ in $[0,\pi/2]_\theta$, these are exactly the edge
vector fields introduced by Mazzeo \cite{MR93d:58152} on manifolds
with smooth boundaries. Here the fibers have boundaries (in our
example, the fibers are just the interval $[0,\pi/2]_\theta$), and
smoothness is required up to these boundaries. A key observation is
that the wave operator lifts to an edge-smooth differential operator
on $M=\RR_t \times X$.

Propagation phenomena in the edge setting (when the fibers have no
boundaries) have been treated in \cite{mvw1}, following
\cite{Melrose-Wunsch1}.  We now recall these results, as they apply in
the setting discussed here, provided we stay away from the fiber
boundaries. We emphasize that in the edge picture both the operator
one studies (the wave operator) and the microlocalizers are edge
pseudodifferential operators, i.e.\ there is no need to use two
different algebras as in the manifolds with corners setting discussed
above. In order to avoid complicating the notation, we simply replace
$[0,\pi/2]_\theta$ by the circle; edge operators are then formally of
the form $a(r,\theta,y,t,rD_r,D_\theta,rD_y,rD_t)$. Writing covectors
as $\xies\,\frac{dr}{r}+\zetaes\,d\theta
+\etaes\,\frac{dy}{r}+\taues\,\frac{dt}{r}$, their symbols are thus
smooth functions of $(r,\theta,y,t,\xies,\zetaes,\etaes,\taues)$; in
the setting of classical operators they are homogeneous in the last
three sets of variables. In particular, the principal symbol of the
wave operator is such a symbol, and its Hamilton vector field is a
smooth homogeneous vector field in these coordinates. Its dynamical system in
the characteristic set governs the analysis of solutions; by
homogeneity, it is convenient to study these dynamics in the
corresponding cosphere bundle. Then there are two (incoming, resp.\
outgoing) sets of critical points over $r=0$, corresponding to
radial points of the Hamilton vector field. These are both saddle
manifolds, with either the stable or the unstable manifold for each of
these contained in the boundary face $r=0$, and the other transversal to it. The Hamilton
flow within $r=0$ connects the incoming and outgoing radial sets, and
fixes the ``slow variables'' $(y,t,\etaes/\taues)$ (with the last
variable projectivized to work on the cosphere bundle); the projection
of its integral curves to the base variables gives the distance $\pi$
propagation of the geometric theorem of \cite{mvw1}. One should thus
picture singularities entering the boundary $r=0$ along (say) the
stable manifold of one of these critical manifolds, propagating through the
critical manifold and out through its unstable manifold; propagating across
the boundary to the stable manifold of the other critical manifold;
and then through it and back out of the boundary along the corresponding
unstable manifold.   As this whole process leaves the slow variables
unaffected, we see that they are preserved under the interaction with
the boundary, leading to the law of specular reflection.

To make sense of the propagation described above, one thus needs to
have a description of propagation at incoming and outgoing radial
points, as well as elsewhere within $r=0$; this was accomplished in
\cite{mvw1}.  It is the radial points that cause the most significant subtleties in the
propagation of singularities: at these points the
relation generated by the flow becomes
multi-valued, as in general a singularity arriving at a critical point
along its stable manifold may produce singularities leaving along the
whole unstable manifold.
An important part of the analysis is to note that at the
radial points, coisotropy corresponding to the stable/unstable
submanifold transversal to $r=0$ implies regularity (absence of edge
wave front set) in the unstable/stable manifold within $r=0$, and
conversely. In particular, an incident wave coisotropic for the
flow-in becomes edge-regular within $r=0$ (away from the radial
points) and then emerges to be coisotropic for the flow-out. A slight
complication is that coisotropy is relative to a function space; there
are losses in the background regularity space due to the radial points.

The added difficulty in our setting relative to that of \cite{mvw1} is that the fibers
have boundaries, and indeed typically corners.   We deal with this by
treating the propagation into and out of these corners inside $r=0$
analogously to the propagation of b-regularity analyzed in \cite{Vasy5}.
We thus compress the edge-smooth cotangent
bundle, essentially by replacing the ``smooth'' vector field
$\pa_\theta$ by one tangent to the boundaries of the fibers, i.e.,
using $f(\theta)\pa_\theta$ instead, where
$f(\theta)=\sin\theta\cos\theta$ has simple zeros at $\theta=0$,
$\theta=\pi/2$ and is non-zero elsewhere in $[0,\pi/2]$.
Note that $r\pa_r$,
$r\pa_y$ and $r\pa_t$ are already tangent to the boundary faces $\theta=0,\pi/2$,
so they do not require any adjustments. The resulting vector fields
are thus those tangent to the fibers of the front face of the blow-up,
as well as to all other boundary faces, and we call these \emph{edge-b}
vector fields. We use a pseudodifferential algebra $\Psieb*(M)$
microlocalizing these vector fields to prove our main results.
In addition to the already discussed results away from the boundaries
of the fibers, we thus need to analyze propagation at incoming and
outgoing radial points at the boundaries of the fibers, as well as the
analogue of hyperbolic and glancing points in the setting of
$M_0$. This is accomplished in Section~\ref{section:edge-propagation}.

Note that conormal regularity in $X_0$ near a point is equivalent,
after blow-up, to conormal regularity near the corresponding fiber of
the front face. Explicitly, in our example of a codimension $2$
corner, regularity with respect to $x_1\pa_{x_1},x_2\pa_{x_2}$ and
$\pa_y$ is equivalent to regularity with respect to $r\pa_r$,
$f(\theta)\pa_\theta$ and $\pa_y$, where
$f(\theta)=\sin\theta\cos\theta$. Thus, away from $\theta=0,\pi/2$,
i.e., in the interior of the front face, one has regularity with
respect to $r\pa_r$, $\pa_\theta$ and $\pa_y$---where that this notion
of regularity ignores the fibration. Edge regularity in the same
region is with respect to $r\pa_r$, $\pa_\theta$ and $r\pa_y$, i.e.,
it is a weaker notion than conormality. However, the ability to microlocalize within
the fibers depends on its use.

\subsection{Organization of the paper}
We start in Section~\ref{section:geometry} by describing the blown-up
space on which our analysis takes place. Then, in
Section~\ref{section:edgebgeom}, we describe in detail the connection
between both the smooth and b-structures on $M_0$, and between the
edge-smooth and edge-b structures on $M$. In the same section, we
study the bicharacteristics in the edge-b setting, i.e.\ that of $M$;
this is in many respects analogous to Lebeau's work \cite{Lebeau5} in
the blown-down setting (e.g.\ on $M_0$), though radial points are an
important new feature.

In Section~\ref{section:edge-b-calc}, with the operator algebra
construction provided by Appendix~\ref{BEPseuda}, we describe edge-b
pseudodifferential operators, and then in
Section~\ref{section:Diff-Psi} the algebra of operators that are both
edge-smooth-differential and edge-b-pseudodifferential; these provide
the link between the wave operator (which is edge-smooth) and the
microlocalizers (which are edge-b). The use of this mixed
differential-pseudodifferential calculus is analogous to the use of
(smooth-)differential, tangential-pseudodifferential operators by
Melrose-Sj\"ostrand 
\cite{Melrose-Sjostrand1,Melrose-Sjostrand2} in the smooth
boundary setting, and (smooth-)differential, b-pseudodifferential
operators in \cite{Vasy5} in the proof of the standard propagation
result on manifolds with corners.  This calculus provides the framework for
the positive commutator estimates proving the edge-b propagation
results.  In Section~\ref{section:coisodef} we discuss coisotropic
distributional, and their dual, non-focusing,
spaces. Section~\ref{section:edge-propagation} proves the edge-b
propagation of singularities. In Section~\ref{section:coisotropic} we
show how coisotropy propagates through the edge. Finally, in
Section~\ref{section:geometric} we prove the main theorem,
Theorem~\ref{theorem:geometric}, and its corollaries, which in
particular imply Theorem~\ref{intro-theorem}.

\emph{To ease the notational burden on the reader, an index of notation
is provided at the end of the paper.}

\subsection{Acknowledgements}
The authors gladly acknowledge the support of the National Science
  Foundation under grants DMS-0408993 (RM), DMS-0733485 and
  DMS-0801226 (AV) and DMS-0700318 (JW).  The second author was also
  supported by a Chambers Fellowship from Stanford University.  All
  three authors are grateful to MSRI for both support and hospitality
  in the fall of 2008, during which the bulk of this manuscript was
  written.  Two anonymous referees contributed many helpful
  suggestions and corrections to the exposition.

\section{Geometry: metric and Laplacian}\label{section:geometry} Let $X_0$
be a connected $n$-dimensional manifold with corners. \nomenclature[X]{$X_0$}{Spatial manifold, not blown up}
 We work locally,
near a given point in the interior of a corner $\cornX$ of codimension $f+1.$\nomenclature[Y]{$Y$}{Corner of spatial manifold}
Thus, we have local coordinates $x_1,\ldots,x_{f+1},y_1,\ldots,y_{n-f-1}$
in which $\cornX$ is given by $x_1=\ldots=x_{f+1}=0$.  Suppose that $g_0$ is a
smooth Riemannian metric on $X_0,$ non-degenerate up to all boundary
faces. We may always choose local coordinates in which it takes the
form
\begin{equation}\label{eq:X_0-metric-form}
g_0=\sum a_{ij}\,dx_i\,dx_j+\sum b_{ij}\,dy_i\,dy_j+2\sum c_{ij}\,dx_i\,dy_j
\end{equation}
with $c_{ij}|_{Y}=0.$ This can be arranged by changing the $y$ variables to
$$
y_j=y'_j+\sum x_k Y_{jk}(y')
$$
while keeping the $x_j$ unchanged. The cross-terms then become
$$
2\sum c_{ij}\,dx_i\,dy'_j+2\sum b_{ij}\,Y_{jk}\,dy'_i\,dx_k,
$$
which can be made to vanish by making the appropriate choice of the $Y_{jk},$
using the invertibility of $\{b_{ij}\}.$

Let $X=[X_0; \cornX],$ be the \emph{real blow-up} of $\cornX$ in $X_0$
\nomenclature[X]{$X$}{Spatial manifold, blown up at corner}
(see
\cite{Melrose:Atiyah, MR95k:58168}) and let $\efspace$ denote the front
face of the blow-up, which we also refer to as the {\em edge face}.
\nomenclature[Y]{$\tilde{Y}$}{Front face of blown-up spatial corner}
 Recall
that the blow-up arises by identifying a neigborhood of $Y$ in $X_0$ with
the inward-pointing normal bundle $N^+\cornX$ of $\cornX$ in $X_0$ and blowing up
the origin in the fibers of the normal bundle (i.e.\ introducing polar
coordinates in the fibers). Since the normal bundle is trivialized by the
defining functions of the boundary faces, a neighborhood of $\efspace$ in
$X$ is globally diffeomorphic to
$$
[0,\infty)\times
\cornX\times Z,\ \text{where}\ Z=S^f \cap[0,\infty)^{f+1}.
$$
We use coordinates
$z_1,\dots, z_f$ in $Z;$ near a corner of $Z$ of codimension $k,$ these are
divided into $z'_1,\dots, z'_k \in [0,\infty)$ and $z''_{k+1},\dots, z''_f \in
\RR.$ There is significant freedom in choosing the identification
of a neighborhood of $\cornX$ and the coordinates on the fibers
of the normal bundle but the naturality of the smooth structure on the
blown up manifold, $[X_0;\cornX],$ corresponds to the fact that these are smoothly
related.

\begin{figure}[ht]
\includegraphics{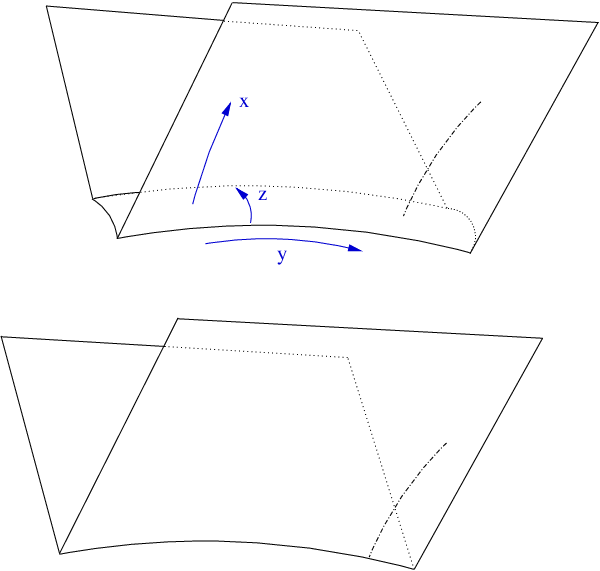}
\caption{A manifold $X_0$ with corners of codimension two (below) and
  its blow-up $X$ (top).  A geodesic hitting the codimension-two corner
is shown, together with its lift to the blown-up space $X,$ which then
strikes the front ``edge'' face of the blow-up.}
\label{blowupfigure}
\end{figure}

The metric $g_0$ identifies $NY$ as a subbundle of $T_{\cornX}X_0.$ This
corresponds to coordinates $(x_i,y_j)$ as above with the $dy_j$ orthogonal
to $dx_i$ at $\cornX.$ In the blow-up polar coordinates are introduced in the
$x_i$ but the $y_j$ are left unchanged. It is convenient to think of these
as polar coordinates induced by $\sum_{ij} a_{ij}\,dx_i\,dx_j.$ In
particular, we choose
$$
x=\Big(\sum_{ij} a_{ij}(0,y)x_i x_j\Big)^{1/2}
$$
as the `polar variable' which is the defining function of the front face.
With this choice, the metric takes the form
\begin{equation}\label{metric}
g = dx^2 + h(y,dy)+ x^2 k(x,y,z,dz)+x k'(x,y,z,dx,dy,x\,dz).
\end{equation}
More generally, one can simply consider the wider class of manifolds with
corners with metrics of the form \eqref{metric}, we refer to these as `edge
metrics' for brevity. Note, however, that there are no results currently
available in this wider setting that limit propagation of singularities to
generalized broken bicharacteristics. Despite this, the results in
\S\ref{section:edge-propagation} remain valid in this more general context.

Now, set
$$
M=\RR_t\times X,\ M_0=\RR\times X_0,\ \ef=\RR\times
\efspace,\ \cornM=\RR\times\cornX,
$$
\nomenclature[M]{$M$}{Spacetime manifold, blown up at
      corner}
\nomenclature[M]{$M_0$}{Spacetime manifold, not blown up}
where $\cornM$, resp.\ $\ef$ now represent the unresolved, resp.\
blown-up,
version of the space-time ``edge.''
\nomenclature[W]{$W$}{Space-time edge}
\nomenclature[W]{$\tilde{W}$}{Front face of blow-up of space-time edge}

Let $\Diffes{*}(X)$ denote the filtered algebra of operators on
$\CI(X)$ 
\nomenclature[D]{$\Diffes{*}(X)$}{Edge-smooth differential operators}
generated by those vector fields that are tangent to the fibers of the
front face $\efspace$ produced by the blow up; $xD_x,$ $xD_{y_j},$ $D_{z_j}$
form a local coordinate basis of these vector fields. See
\S\ref{section:edgebgeom} and \S\ref{section:Diff-Psi} for further explanation
of this algebra and of our terminology. The same definition leads to the
algebra of operators $\Diffes{*}(M)$ on $\CI(M)$ with local generating
basis  $xD_t,$ $xD_x,$ $xD_{y_j},$ $D_{z_j}.$

\begin{lemma}
The Laplace operator $\Lap\in x^{-2}\Diffes{2}(X)$ on $X$ is of the form
$$
\Lap \in D_x^2+\frac{f}{\imath x} D_x + \frac{1}{x^2} \Lap_Z+ \Lap_Y + x^{-1}
\Diffes{2}(X)
$$
where $\Lap_Z$ is the Laplace operator in $Z$ with respect to the metric
$h$ (and hence depends parametrically on $x$ and $y)$ and $\Lap_{\cornX}$ is the
Laplacian on $\cornX$ with respect to the metric $k.$
\end{lemma}

In particular, $\Box=D_t^2-\Delta_X\in x^{-2}\Diffes{2}(M)$.

\section{Bundles and bicharacteristics}\label{section:edgebgeom}

In this section, we discuss several different geometric settings in which
the propagation problem for $\Box$ on $M_0$ may be viewed.  Somewhat
loosely, each of these corresponds to a choice of a Lie algebra of vector
fields with different boundary behavior; these then lead to distinct
bundles of covectors, with the corresponding descriptors used as section
headings here. The first, the ``b''-bundle, can be considered either on
$M_0$ or $M.$ Indeed, the bundle of b-covectors on $M_0$ is the setting for
the propagation results of \cite{Vasy5}: these results are, as will be seen
below, necessarily global in the corner, and do not distinguish between
general diffractive rays and the subset of geometric rays (defined below). In order to
discuss the improvement in regularity which can occur for propagation along
the geometric rays, two more bundles of covectors, lying over the blown-up
space $M,$ are introduced. These, the ``edge-b'' and ``edge-smooth''
bundles, keep track of local information in the fibers of the blow-up
$\ef$ of
$\cornM$ in $M_0,$ and allow us to distinguish the diffractive rays from
geometric ones. The distinction between the edge-b and edge-smooth bundles
comes only at the boundary of $\ef,$ and the relationship between the two
bundles gives rise to reflection of singularities off
boundary faces, uniformly up to the edge $\ef.$

In order to alleviate some of the notational burden on the reader, a table
is included in \S\ref{subsection:summary} in which the various bundles, their
coordinates, their sections, and some of the maps among them are
reviewed. The standard objects for a manifold with corners, $\man,$
correspond to uniform smoothness up to all boundary faces, so $\Vf(\man)$
denotes the Lie algebra of smooth vector fields, $T\man,$ the tangent
bundle, of which $\Vf(\man)$ forms all the smooth sections, $T^*\man$ is
its dual, etc.

\subsection{b-Cotangent bundle} Let $\Vb(\man)\subset\Vf(\man)$
\nomenclature[V]{$\Vb$}{b-vector fields}
 denote the
Lie subalgebra of those smooth vector fields on the general manifold with
corners $\man,$ which are tangent to each boundary face. If we choose
coordinates as in \S\ref{section:edgebgeom},
the local vector fields
$$
x_1\pa_{x_1}, \dots, x_{f+1} \pa_{x_{f+1}}, \pa_{y_1}, \dots, \pa_{y_{n-f-1}};
$$
form a basis over smooth functions. Hence $\Vb(\man)$ is the space of
$\CI$-sections of a vector bundle, denoted
$$
\bT \man.
$$
\nomenclature[T]{$\bT$}{b tangent bundle}
The dual bundle $\Tbstar \man$\nomenclature[T]{$\Tbstar$}{b cotangent bundle}
 therefore has sections spanned by
$$
\frac {dx_1}{x_1}, \dots, \frac {dx_{f+1}}{x_{f+1}}, d{y_1}, \dots, d{y_{n-f-1}}.
$$

The natural map
\begin{equation}\label{eq:pismooth-def}
\pismooth:T^*\man\longrightarrow \Tbstar\man
\end{equation}
\nomenclature[P]{$\pismooth$}{Projection from cotangent bundle to
  b cotangent bundle}
is the adjoint of the bundle map $\iota:\bT\man\to T\man$ corresponding to
the inclusion of $\Vb(\man)$ in $\Vf(\man).$

Canonical local coordinates on $T^*M_0$ correspond to decomposing a
covector in terms of the basis as
$$
\taus\,dt+\sum_j\etas_j\,dy_j+\sum_j \xis_j\,dx_j,
$$
and elements of $\Tbstar M_0$ may be written
$$
\taub\,dt+\sum_j\etab_j\,dy_j+\sum_j \xib_j\,\frac{dx_j}{x_j},
$$
so defining canonical coordinates. The map \eqref{eq:pismooth-def} then
takes the form
$$
\pismooth(x,y,t,\xis,\etas,\taus)=(x,y,t,\xib,\etab,\taub)=
(x,y,t,x\xis,\etas,\taus),
$$
with $x\xis=(x_1\xis_1,\ldots,x_{f+1}\xis_{f+1}).$

The setting for the basic theorem on the propagation of singularities in
\cite{Vasy5} is $\Tbstar M_0.$ In particular, generalized broken
bicharacteristics, or \GBBs, are curves in $\Tbstar M_0.$ In order to
analyze the geometric improvement, spaces that will keep track of finer
singularities are needed. Before introducing these, we first recall the
setup for \GBBs. Note that at $\cornM,$ $\pismooth|_{\cornM}$ maps $N^*\cornM$
onto the zero section over $\cornM,$ and is injective on complementary
subspaces of $T^*_{\cornM}M_0$,
so we may make the identification
$$
\pismooth|_{\cornM}(T^*_{\cornM}M_0) \cong T^*\cornM.
$$

We also recall that it is convenient to work on cosphere bundles.
Since it is linear, $\pismooth$ intertwines the $\RR^+$-actions, but it
does not induce a map on the corresponding cosphere bundles since it maps
part of $T^*M_0\setminus o$ into the zero section of $\Tbstar
M_0.$ However, on the characteristic set of $\Box$ this map is
better behaved. Let
$$
p_0=\sigma_2(\Box)\in\CI(T^*M_0\setminus o)
$$
be the standard principal symbol of $\Box\in\Diff{2}(M_0);$ it is of the form
$$
p_0=(\taus)^2-\Big(\sum A_{jk}\xis_j\xis_k+\sum B_{jk}\etas_j\etas_k+2\sum
C_{jk}\xis_j\etas_k\Big)
$$
with $A_{jk},B_{jk},C_{jk}\in\CI(M_0)$, $A_{jk}=A_{kj}$ and $B_{jk}=B_{kj},$
$C_{jk}|_{x=0}=0.$ Let 
\begin{equation}
\sSigma=p_0^{-1}(\{0\})/\bbR^+\subset S^*M_0
\label{30.8.2006.38}\end{equation}
be the spherical image of the characteristic set of $\Box.$ This has two
connected components, $\sSigmapm,$ corresponding to $\taus\gtrless 0$ since
$\{\taus=0\}\cap\sSigma=\emptyset.$ Now, $N^*\cornM\subset\{\taus=0\},$ so
$N^*\cornM\cap p^{-1}(0)=\emptyset,$ meaning $\cornM$ is non-characteristic for
$\Box.$ Since $N^*\cornM$ is the null space of $\pismooth,$ there is an
induced map on the sphere bundles
$\widehat\pismooth:\sSigma\longrightarrow\Sbstar M_0;$ the range is denoted
\begin{equation}
\bSigma=\widehat\pismooth(p_0^{-1}(0))/\bbR^+\subset\Sbstar M_0.
\label{30.8.2006.39}\end{equation}

Again, $\bSigma$ has two connected components corresponding
to the sign of $\taus$ in $\sSigma$ and hence the sign of $\taub.$ These
will be denoted $\bSigmapm.$

We use $\taus$, resp.\ $\taub$, to obtain functions homogeneous of degree
zero on $T^*M_0\setminus o$ inducing coordinates on $S^*M_0$ near $\sSigma:$
$$
x,y,t,\ \xish=\xis/|\taus|,\ \etash=\etas/|\taus|.
$$
Note also that these coordinates are {\em global in the fibers of $S^*M_0\cap
\sSigmapm\to M_0$} for each choice of sign $\pm.$
$$
\taush=\sgn\taus
$$
lifts to a constant function $\pm1$ on $\sSigmapm.$ There are similar
coordinates on $\Sbstar M_0$ near $\bSigma.$

In these coordinates,
\begin{equation}\label{eq:bSigma-coords}
\bSigma\cap\Sbstar_W M_0=\{(x,y,t,\xibh,\etabh):\ x=0,\ \xibh=0,
\ \sum B_{jk}\etabh_j\etabh_k\leq 1\}\subset S^*W.
\end{equation}
We also remark that with $\sH_{p_0}$ denoting the Hamilton vector
field of $p_0$,
$$
\sH_{\so}=|\taus|^{-1}\sH_{p_0}
$$
is a homogeneous degree zero vector field near $p_0^{-1}(\{0\})$, thus can
be regarded as a vector field on
$S^*M_0$.\nomenclature[H]{$\mathsf{H}_{\so}$}{Hamilton vector
  field on smooth cosphere bundle}

Now we define the b-hyperbolic and b-glancing sets by
\begin{equation}\label{gcalwbdef}
\gcal_{W,\bo} = \{q \in \Sbstar_{W} M_0: \lvert (\pismooth)^{-1}(q)\cap \sSigma\rvert = 1\}
\end{equation}
and
\begin{equation}\label{hcalwbdef}
\hcal_{W,\bo} = \{q \in \Sbstar_{W} M_0: \lvert (\pismooth)^{-1}(q)\cap \sSigma\rvert \geq 2\},
\end{equation}
\nomenclature[G]{$\gcal_{W,\bo},\hcal_{W,\bo}$}{b glancing and
hyperbolic sets}
These are thus also subsets of $S^*W$. In local
coordinates\footnote{The discrete variable $\taubh$ is not, of course,
  part of the coordinate system, but serves to identify which of two
  components of the characteristic set we are in.}
they are given by
\begin{equation}\begin{split}
\gcal_{W,\bo} &=\{(x,y,t,\taubh, \xibh,\etabh):\ x=0,\ \taubh \in
\{\pm 1\},\ \xibh=0,
\ \sum B_{jk}\etabh_j\etabh_k=1\}\\
\hcal_{W,\bo} &=\{(x,y,t,\taubh,\xibh,\etabh):\ x=0,\ \taubh \in
\{\pm 1\},\ \xibh=0,
\ \sum B_{jk}\etabh_j\etabh_k<1\}.
\end{split}\end{equation}
Note that for $q\in\Sbstar_W M_0,$ at the unique point $q_0$ in
$(\pismooth)^{-1}(q)\cap \sSigma$, we have $\xish=0$, and correspondingly
$\sH_{\so}(q_0)$ is tangent to $W$, explaining the ``glancing'' terminology.

Now we discuss bicharacteristics.
\begin{definition}
A \emph{generalized broken bicharacteristic}, or \GBB, is a continuous
map $\gamma:I \to \bSigma$ such that for all $f \in \CI(\Sbstar
M_0),$ real-valued,
\begin{align}\label{eq:liminf}
\liminf_{s \to s_0} &\frac{ (f\circ \gamma)(s)- (f\circ
  \gamma)(s_0)}{s-s_0}\\
& \geq \inf\{ \sH_{\so} (\pismooth^* f)(q): q \in
\pismooth^{-1}(\gamma(s_0))\cap \sSigma \}.
\end{align}
\end{definition}
\nomenclature[G]{\GBB}{Generalized broken bicharacteristic}

\begin{remark}
Replacing $f$ by $-f$, we deduce that the inequality
\begin{align}\label{eq:limsup}
\limsup_{s \to s_0} &\frac{ (f\circ \gamma)(s)- (f\circ
  \gamma)(s_0)}{s-s_0}\\
& \leq \sup\{ \sH_{\so} (\pismooth^* f)(q): q \in
\pismooth^{-1}(\gamma(s_0))\cap \sSigma \}.
\end{align}
also holds.
\end{remark}

We recall an alternative description of \GBBs, which was in fact Lebeau's
definition \cite{Lebeau5}. (One could use this lemma as the defining
property of \GBB; the equivalence of these two possible definitions is
proved in \cite[Lemma~7]{Vasy:Geometric-optics}.)

\begin{lemma}\label{lemma:bich-Lebeau}(See
\cite[Lemma~7]{Vasy:Geometric-optics}.)
Suppose $\gamma$ is a \GBB. Then

\begin{enumerate}
\item
If $\gamma(s_0)\in\gcal_{W,\bo}$, let $q_0$ be the unique
point in the preimage of $\gamma(s_0)$ under
$\widehat{\pismooth}=\pismooth|_{\sSigma}$. Then for
all $f\in\CI(\Sbstar M_0)$ real valued, $f\circ\gamma$ is differentiable at $s_0$,
and
\begin{equation*}
\frac{d(f\circ\gamma)}{ds}|_{s=s_0}=\sH_{\so}\pismooth^*f(q_0).
\end{equation*}
\item
If $\gamma(s_0)\in\hcal_{W,\bo}$,
lying over a corner given in local coordinates
by $x_j=0$, $j=1,\ldots,f+1$, there exists $\ep>0$ such that
$x_j(\gamma(s))=0$ for $s\in(s_0-\ep,s_0+\ep)$ if and only if $s=s_0$.
That is, $\gamma$ does not meet the corner $\{x_1=\ldots=x_{f+1}=0\}$
in a punctured
neighborhood of $s_0$.
\end{enumerate}
\end{lemma}

\begin{remark}
It also follows directly from the definition of \GBB (by combining
\eqref{eq:liminf} and \eqref{eq:limsup}) that, more
generally, if the set
\begin{equation}\label{pullbackset}
\{ \sH_{\so} (\pismooth^* f)(q): q \in
\pismooth^{-1}(\gamma(s_0))\cap \sSigma \}
\end{equation}
consists of a single value (for instance, if 
$\pismooth^{-1}(\gamma(s_0))\cap \sSigma$ is a single point), then
$f\circ \gamma$ must be 
differentiable at $s_0,$ with derivative given by this value. This is
indeed how Lemma~\ref{lemma:bich-Lebeau} is proved. The first part of the
lemma
follows because $\pismooth^{-1}(\gamma(s_0))\cap \sSigma$ is a single point,
giving differentiability.
On the other hand,
the second half follows using $f=\sum \xibh_j$, for which
the single value in \eqref{pullbackset} is $-(1-\sum B_{ij}\etabh_i\etabh_j)<0$, for
$\gamma(s_0)\in\hcal_{W,\bo}$. Thus, $f$ is locally strictly decreasing.
Since $f(q')=0$ if $q'\in\Sbstar_W M_0\cap\bSigma$, in particular at
$\gamma(s_0)$, it is non-zero at $\gamma(s)$ for nearby but distinct
values of $s$---so in particular for such $s$,
$\gamma(s)\notin\Sbstar_W M_0\cap\bSigma$, showing that $\gamma$ leaves $W$
instantaneously. In fact, this argument also demonstrates the following
useful lemma.
\end{remark}

\begin{lemma}\label{lemma:OG-IC-uniform-est-1}
Let $U$ be a coordinate neighborhood around some $p\in W$, $K$ a compact
subset of $U$. Let $\ep_0>0$.
Then there exists an $\delta>0$ with the following property.
Suppose that $\gamma$ is a \GBB and $\gamma(s_0)\in \Sbstar_K M_0$.
If $\sum_{j=1}^{f+1}\xibh_j(\gamma(s_0))>0$ and
$1-h(y(\gamma(s_0)),\etabh(\gamma(s_0)))>\ep_0$
then
$\gamma|_{[s_0,s_0+\delta]}\cap\Sbstar_W M_0=\emptyset$, while if
$\sum_{j=1}^{f+1}\xibh_j(\gamma(s_0))<0$ and
$1-h(y(\gamma(s_0)),\etabh(\gamma(s_0)))>\ep_0$ then
$\gamma|_{[s_0-\delta,s_0]}\cap\Sbstar_W M_0=\emptyset$.
\end{lemma}

\begin{proof}
Let $U_1\subset U$ be open such that $K\subset U_1$, $\overline{U_1}\subset U$.
\GBBs are uniformly Lipschitz, i.e.\ with Lipschitz constant independent of
the \GBB, in compact sets (thus are equicontinuous in compact sets),
so it follows that there is
an $\delta_1>0$ such that $\gamma(s_0)\in \Sbstar_K M_0$
implies that $\gamma(s)\in \Sbstar_{U_1} M_0$
for $s\in[s_0-\delta_1,s_0+\delta_1]$. Now the uniform Lipschitz
nature of the function $1-h(y(\gamma(s)),\etabh(\gamma(s)))$ shows that
there exists $\delta_2\in (0,\delta_1]$ such that for $|s-s_0|\leq \delta_2$,
$1-h(y(\gamma(s)),\etabh(\gamma(s)))>\ep_0/2$.
Now let $f=\sum \xibh_j$. Then
$$
\sH_{\so} (\pismooth^* f)|_{\sSigma}=-\sum A_{ij}\xibh_i\xibh_j+\sum x_jF_{1j}
=-(1-\sum B_{ij}\zetabh_i\zetabh_j)+\sum x_j F_{2j},
$$
with $F_{1j},F_{2j}\in\CI(S^*M_0)$, so there exist $\delta_3>0$ and $c>0$
such that if $x_j<\delta_3$ for $j=1,\ldots,f+1$, then
$\sH_{\so} (\pismooth^* f)|_{\sSigma}\leq -c$. Now if $x_j(\gamma(s_0))>
\delta_3/2$ for some $j$, the uniform Lipschitz character of $x_j\circ\gamma$
shows the existence of $\delta'>0$ (independent of $\gamma$)
such that $x_j(\gamma(s_0))\neq 0$ for $|s-s_0|<\delta'$. On the other hand,
if $x_j(\gamma(s_0)\leq\delta_3/2$ for all $j$, then
the uniform Lipschitz character of $x_j\circ\gamma$
shows the existence of $\delta''\in(0,\delta_2]$
such that $x_j(\gamma(s_0))<\delta_3$ for $|s-s_0|<\delta''$, so
$f(\gamma(s))$ is strictly decreasing on $[s_0-\delta'',s_0+\delta'']$.
In particular, if $f(\gamma(s_0))>0$, then $f(\gamma(s))>0$ for
$s\in[s_0-\delta'',s_0]$, so $\gamma(s_0)\notin\Sbstar_W M_0$, and if
$f(\gamma(s_0))<0$, then $f(\gamma(s))<0$ for
$s\in[s_0,s_0+\delta'']$, so $\gamma(s_0)\notin\Sbstar_W M_0$ again.
This completes the proof of the lemma.
\end{proof}

We now recall the following statement, due to Lebeau.

\begin{lemma}\label{lemma:Lebeau-IC-OG}(Lebeau, \cite[Proposition~1]{Lebeau5})
If $\gamma$ is a generalized broken bicharacteristic, $s_0\in I$,
$q_0=\gamma(s_0)$, then there exist unique
$\tilde q_+, \tilde q_-\in\sSigma$ satisfying
$\pismooth(\tilde q_\pm)=q_0$ and having
the property that if $f\in\CI(\Sbstar M_0)$
then $f\circ\gamma$ is differentiable both from the left and
from the right at $s_0$ and
\begin{equation*}
\left(\frac{d}{ds}\right)(f\circ\gamma)|_{s_0\pm}=\sH_{\so} \pismooth^*f(\tilde q_{\pm}).
\end{equation*}
\end{lemma}

\begin{definition}
A generalized broken bicharacteristic segment $\gamma,$ defined on
$[0, s_0)$ or $(-s_0, 0],$ $\gamma(0)=q \in \hcal_{Y,b}$ is said to
approach $W$ \emph{normally} as $s \to 0$ if for all $j$,
$$
\lim_{s \to 0 \pm} \frac{x_j(\gamma(s))}{s} \neq 0; 
$$
this limit always exists by \cite[Proposition 1]{Lebeau5}.
\end{definition}

\begin{remark}
If $\gamma$ approaches $W$ normally then there is $s_1>0$ such that
$\gamma(s)\in S^*M_0^\circ$ for
$s\in(0,s_1)$ or $s\in(-s_1,0)$ since $x_j(\gamma(0))=0$, and the
one-sided derivative of $x_j\circ\gamma$ is non-zero.

While the actual derivatives depend on the choice of the defining
functions $x_j$ for the boundary hypersurfaces, the
condition of normal incidence is independent of these choices.
\end{remark}

\subsection{Edge-smooth cotangent bundle}

We now discuss another bundle, ultimately in order to discuss the refinement
of \GBBs that allows us to obtain a diffractive improvement.  Let
$\beta:M\to M_0$ be the blow-down map.

Let $\Vfes(M)$ denote the set
of vector fields that are tangent to the fibers of
$\beta|_{\ef}:\ef\to W$ (hence to $\ef$).
\nomenclature[V]{$\Vfes$}{edge-smooth vector fields}
This is a $\CI(M)$-module,
with sections locally spanned by $$x\pa_x,\ x\pa_t,\  x\pa_{y_j},\
\pa_{z'_j},\ \pa_{z''_j}.$$  (In fact, one can always use local coordinate
charts without the $z''$ variables in this setting.)  Under the
blow-down map $\beta:M\to M_0$, elements of $\Vf(M_0)$ lift to certain
vector
fields of the form $x^{-1}V$, $V\in\Vfes(M)$, where $x$ is a defining
function of the front face, $\ef$. Conversely, $x^{-1}\Vfes(M)$ is
spanned by the lift of elements of $\Vf(M_0)$ over $\CI(M)$, i.e.
\begin{equation}\label{eq:Vf-lift-to-M}
x^{-1}\Vfes(M)=\CI(M)\otimes_{\CI(M_0)}\beta^*\Vf(M).
\end{equation}
Let $\Tes M$ denote the ``edge-smooth'' tangent bundle of $M,$ defined as
the bundle whose smooth sections are elements of $\Vfes(M)$; such a bundle
exists by the above description of a local spanning set of sections.
\nomenclature[T]{$\Tes$}{edge-smooth tangent bundle}
Let $\Tesstar M$
  denote the dual bundle.
\nomenclature[T]{$\Tesstar$}{edge-smooth cotangent bundle}
Thus in the coordinates of
  \S\ref{section:geometry}, sections of $\Tesstar M$ are spanned by
\begin{equation}\label{eq:Tesstar-M-coords}
\taues \frac{dt}x+\xies \frac{dx}x + \etaes\cdot\frac{dy}x+\zetaes'\cdot
       {dz'}+\zetaes''\cdot dz''.
\end{equation}

By \eqref{eq:Vf-lift-to-M}, taking into account that $dt^2-g_0$ is a Lorentz
metric on $M_0$, we deduce that its pull-back $g$ to $M$ is a Lorentzian
metric on $x^{-1}\Tesstar M$, i.e.\ that $x^{-2}g$
is a symmetric non-degenerate
bilinear form on $\Tes M$ with signature $(+,-,\ldots,-)$. Correspondingly,
the dual metric $G$ has the property that $x^2 G$ is a Lorentzian metric
on $\Tesstar M$. Note that $G$ is the pull-back of $G_0=\sigma_2(\Box)\in
\CI(T^*M_0\setminus o)$. We thus conclude that
$\sigma_2(\Box)\in\CI(T^*M_0\setminus o)$ lifts to an element of
$x^{-2}\CI(\Tesstar M\setminus o)$;
let
$$
p=\sigma_{\eso,2}(x^2 \Box)\in\CI(\Tesstar M\setminus o)
$$
be such that $x^{-2}p$ is this lift, so
$$
p|_{x=0}=\taues^2-\big(\xies^2+h(y,\etaes)+k(y,z,\zetaes)\big).
$$
Let $\esSigma \subset \Sesstar M$ denote the characteristic set of
$x^2\Box,$ i.e., the set
$$
\esSigma=
p^{-1}(\{0\})/\RR^+=\{\sigma_{\eso,2}(x^2\Box)=0\}/\RR^+.
$$
Thus, using the coordinates
\begin{equation}\label{eq:Tesstar-M-comp-coords}
x,y,t,z,\xiesh=\xies/|\taues|,\ \etaesh=\etaes/|\taues|,
\ \zetaesh=\zetaes/|\taues|,
\ \sigmaes=|\taues|^{-1},
\end{equation}
on $\Tesstar M$, valid where $\taues\neq 0$, hence (outside the
zero section) near where $p=0$,
and dropping $\sigmaes$ to obtain coordinates on $\Sesstar M$,
\begin{equation}\label{eq:esSigma-coords}
\esSigma\cap\Sesstar_{\ef} M
=\{(x=0,t,y,z,\xiesh,\etaesh,\zetaesh):\ \xiesh^2+h(y,\etaesh)+k(y,z,\zetaesh)=1\}.
\end{equation}
The rescaled Hamilton vector field
$$
\sH_{\eso}=|\taues|^{-1}\sH_p
$$
\nomenclature[H]{$\mathsf{H}_{\eso}$}{Hamilton vector field on
  edge-smooth cosphere bundle}
is homogeneous of degree $0$, and thus can be regarded as a vector field on
$\Sesstar M$ which is tangent to $\esSigma$. (Note that while $\sH_{\eso}$
depends on the choice of $x$, and the particular homogeneous degree
$-1$ function, $|\taues|^{-1}$, used to re-homogenize $\sH_p$, these
choices only change $\sH_{\eso}$ by a positive factor, so its direction
is independent of the choices---though our choices are
in any case canonical.)

With the notation of \cite[Section~7]{mvw1} (where it is explained
slightly differently, as the underlying manifold is not a blow-up of another
space),
corresponding to the
edge fibration $$\beta:\ef\to W=Y\times\RR_t,$$
there is a natural map
$$
\dbetaes:\Tesstar_{\ef} M\to T^*W.
$$
\nomenclature[P]{$\dbetaes$}{Map from es cotangent bundle to cotangent
  bundle of blown-down edge}
In fact, in view of \eqref{eq:Vf-lift-to-M}, the bundle $x^{-1}\Tes M$
(whose sections are $x^{-1}$ times smooth sections of $\Tes M$)
can be identified with $\beta^* TM_0$, so one has a natural map
$x^{-1}\Tes M\to TM_0$. Dually, $x\Tesstar M$ can be identified with
$\beta^*T^* M_0$, so one has a natural map $x\Tesstar M\to T^*M_0$.
Multiplication by $x$ maps $\Tesstar M$ to $x\Tesstar M$, and
$\pismooth:T^*M_0\to \Tbstar M_0$ restricts to the quotient map
$T^*_WM_0\to T^*W=T^*M_0/N^*W$ over $W$, so $\dbetaes$ is given by
the composite map
\begin{equation*}\begin{split}
\Tesstar_{\ef} M\ni \alpha\mapsto x\alpha\in x\Tesstar_{\ef} M
&\mapsto \beta_*(x\alpha)\in T^*_WM_0\\
&\qquad \mapsto[x\alpha]\in
T^*W\subset \Tbstar M_0,
\end{split}\end{equation*}
which in local
coordinates \eqref{eq:Tesstar-M-coords} is given by
$$
\dbetaes(x=0,y,t,z,\xies,\etaes,\taues,\zetaes)
=(y,t,\etaes,\taues).
$$
The fibers
can be identified with $\RR_{\xi}\times T^* Z$. In view of
the $\RR^+$-action on $\esSigma$, this
gives rise to a map $\dbetaes:\esSigma\to S^*W$, which is a fibration
over $\hcal_{W,\bo}$ (where $1-h(y,\etabh)>0$)
with fiber
\begin{equation*}\begin{split}
\dbetaes^{-1}(y,t,\etabh)=\{(x=0,y,t,z,\xiesh,\etaesh,\zetaesh):
&\ \etaesh=\etabh,\\
&\ |\xiesh|^2+k(y,z,\zetaesh)=1-h(y,\etabh)\};
\end{split}\end{equation*}
the fibers degenerate at $\gcal_{W,\bo}$.
Then $\sH_{\eso}$ is tangent to the fibers of
$\dbetaes$.
In fact, as computed in \cite[Proof of Lemma~2.3]{mvw1}
(which is directly valid in our setting),
using coordinates \eqref{eq:Tesstar-M-comp-coords}
on $\Tesstar M$,
\begin{equation}\begin{split}\label{eq:sH-at-ef-mod-bvf}
&-\frac12\sH_{\eso}=-\frac12\sigmaes \sH_p\\
&=\xiesh\big(x\pa_x-\sigmaes\pa_{\sigmaes}-\zetaesh\cdot \pa_{\zetaesh}\big)
+K^{ij}\zetaesh_i\pa_{z_j}+K^{ij}\zetaesh_i\zetaesh_j\pa_{\xiesh}
-\frac{1}{2}\,\frac{\pa K^{ij}}{\pa z_k}\zetaesh_i\zetaesh_j\pa_{\zetaesh_k}
+xH',
\end{split}\end{equation}
with $H'$ tangent to the boundary, hence as a vector field on $\Sesstar M$,
restricted to $\Sesstar_{\ef}M$,
$\sH_{\eso}$ is given by
\begin{equation}\label{eq:sH-at-ef}
-\frac12\sH_{\eso}=-\xiesh\zetaesh\pa_{\zetaesh}
+K^{ij}\zetaesh_i\pa_{z_j}+K^{ij}\zetaesh_i\zetaesh_j\pa_{\xiesh}
-\frac{1}{2}\,\frac{\pa K^{ij}}{\pa z_k}\zetaesh_i\zetaesh_j\pa_{\zetaesh_k}.
\end{equation}
It is thus tangent to the fibers given by the
constancy of $y,t,\etaesh$. Notice also that $\sH_{\eso}$ is indeed tangent
to the characteristic set, given by \eqref{eq:esSigma-coords},
and in $\Sesstar_{\ef} M$, it vanishes exactly at $\zetaesh=0$.
We let
$$
\rcal_{\eso}=\{q\in\esSigma\cap\Sesstar_{\ef} M:\ \sH_{\eso}(q)=0\}=\{(t,y,z,
\xiesh,\etaesh,\zetaesh)\in\esSigma:\ \zetaesh=0\}
$$
be the $\eso$-radial set.
\nomenclature[R]{$\rcal_{\eso}$}{edge-smooth radial set}

\subsection{Edge-b cotangent bundle}\label{subsec:edgebbundle}

Finally, we construct a bundle $\Tebstar M$ over $M$ that behaves like
$\Tbstar M$ away from $\ef$, and behaves like $\Testar M$ near the
interior of $\ef$. Before doing so, we remark that the pullback
of $\Tbstar M_0$ to $M$ is $\Tbstar M$, so $\beta:M\to M_0$ induces a map
$$
\beta_\sharp:\Tbstar M\to \Tbstar M_0,
$$
such that
$$
\beta_\sharp|_{\Tbstar_w M}\to\Tbstar_{\beta(w)}M_0,\ w\in M,
$$
is an isomorphism.
It commutes with the $\RR^+$-action, hence induces a map
$$
\hat\beta_\sharp:\Sbstar M\to \Sbstar M_0,
$$
such that
$$
\hat\beta_\sharp|_{\Sbstar_w M}\to\Sbstar_{\beta(w)}M_0,\ w\in M,
$$
is an isomorphism.

More precisely, $\Tebstar M$ arises from the lift of vector fields on $M_0$
which are tangent to all faces of $M_0$ and vanish at $W$.  (The set
$\tilde\Vf$ of such vector fields is a $\CI(M_0)$-module, but is {\em
  not} all sections of a vector bundle over $M_0$---unlike its
analogue, $\Vf(M_0),$ in the construction of $\Vfes(M)$; locally
$x_j\pa_{x_j},$ $j=1,\ldots,f+1$,  $x_i \pa_t,$ and $x_i\pa_{y_j},$
$i=1,\ldots,f+1$, $j=1,\ldots,n-f-1,$ give a spanning list.)

\begin{definition}\label{def:Vfeb}
Let $\Vfeb(M)$ consist of vector fields 
tangent to all of $\pa M$ and to the fibers of
$\ef.$
\nomenclature[V]{$\Vfeb$}{edge-b vector fields}
\end{definition}

This is again a $\CI(M)$-module, and
locally $x\pa_x$, $x\pa_t,$ $x\pa_{y_j},$ $z'_j\pa_{z'_j},$ and $\pa_{z''_j}$ give
a spanning set; in fact
$$
\Vfeb(M)=\CI(M)\otimes_{\CI(M_0)}\beta^*\tilde\Vf.
$$
Thus, there is a vector bundle, called
the ``edge-b'' tangent bundle of $M,$ denoted
$\Teb M$, whose sections are exactly elements of $\Vfeb(M)$.
\nomenclature[T]{$\Teb$}{Edge-b tangent bundle}
Let $\Tebstar M$
  denote the dual bundle.
\nomenclature[T]{$\Tebstar$}{edge-b cotangent bundle}
Thus in the coordinates of
  \S\ref{section:geometry}, sections of $\Tebstar M$ are spanned by
$$
\tau \frac{dt}x+\xi \frac{dx}x + \eta\cdot\frac{dy}x+\sum\zeta_i'
       \frac{dz_i'}{z'_i}+ \zeta''\cdot dz''
$$

In particular, we point out that
the lift of $\sum x_j D_{x_j}$ from $M_0$ to $M$ by $\beta$ is $xD_x$,
up to $x\Vfeb(M)$, hence considering their principal symbols gives
$$
\sum_j\beta^*\xib_j=\xieb\ \text{at}\ x=0.
$$
Dividing by $\beta^*|\taub|=x^{-1}|\taueb|$ yields
\begin{equation}\label{eq:xib-xiebh}
\sum_j\beta^*\xibh_j=x \xiebh+O(x^2),\ \xiebh=\xieb/|\taueb|.
\end{equation}

There exists a natural bundle map
$$
\pieseb: \Tesstar M \to \Tebstar M,
$$
analogous to the bundle map $\pismooth:T^*M_0\to\Tbstar M_0$ of
\eqref{eq:pismooth-def}.
\nomenclature[P]{$\pieseb$}{Projection from es cotangent bundle to
  eb cotangent bundle}
In canonical coordinates, this maps
$$
(\taues,\xies,\etaes,\zetaes',\zetaes'')\mapsto (\tau=\taues,\ \xi=\xies,\
\eta=\etaes,\ \zeta_i'=\zetaes_i'z_i',\ \zeta''=\zetaes'').
$$
This map commutes with the $\RR^+$-action of dilations in the fibers,
and maps $p^{-1}(\{0\})\subset\Tesstar M\setminus o$ into the complement
of the zero section of $\Tebstar M$, so it gives rise to a map
$$
\piesebh:\esSigma\to \Sebstar M.
$$
Let $$\ebSigma=\piesebh (\esSigma) \subset \Sebstar M.$$
In coordinates
$$
x,y,t,z,\xiebh=\xieb/|\taueb|,\ \etaebh=\etaeb/|\taueb|,
\ \zetaebh=\zetaeb/|\taueb|,
$$
on $\Sebstar M$, and analogously defined coordinates on $\Sesstar M$,
$$
\piesebh(x,y,t,z,\xiesh,\etaesh,\zetaesh',\zetaesh'')= (x,y,t,z,
\xiebh=\xiesh,\
\etaebh=\etaesh,\ \zetaebh_i'=\zetaesh_i'z_i',\ \zetaebh''=\zetaesh''),
$$
so
for $w\in \ef$, $z_1'(w)=\ldots=z'_p(w)=0,\ z'_{p+1}(w)\neq 0,\ldots,\ z'_k(w)\neq0,$ with $p\geq 1$,
\begin{align*}
\ebSigma\cap \Sebstar[w] M=\{(\xiebh,&\etaebh,\zetaebh)\in\Sebstar[w] M:
\ \zetaebh'_1=\ldots=\zetaebh'_p=0,\\
& 1\geq \xiebh^2+h(y,\etaebh)
+k(y,z,(0,\ldots,0,\frac{\zetaebh'_{p+1}}{z'_{p+1}},\dots,\frac{\zetaebh'_k}{z'_k}),\zetaebh'')\}.
\end{align*}
We again also obtain a map
$\dbetaeb:\ebSigma\cap\Sebstar[\ef]M
\to S^*W$\nomenclature[P]{$\dbetaeb$}{Map from eb characteristic set to
  cotangent bundle of blown-down edge}
 analogously to $\dbetaes$ which is a fibration over
$\hcal_{W,\bo}$; in local coordinates
(on $S^*W$ near the projection of
$\ebSigma$, $(y,t,\etah)$ are local coordinates, $\etah=\eta/|\tau|$)
\begin{equation}\label{eq:dbetaeb-loc-coords}
\dbetaeb(0,y,t,z,\xiebh,\etaebh,\zetaebh)= (y,t,\etaebh=\etaesh).
\end{equation}

More invariantly we can see this as follows.
As discussed in \cite[Section~7]{mvw1}
in the setting where the fibers on $\ef$ have no boundaries,
one considers the map
$$
x\cdot:\Tebstar M\to\Tbstar M
$$
given by multiplication of
the covectors by $x$ away from $\ef$, which extends to a $\CI$ map
as indicated, namely
\begin{equation*}\begin{split}
&x\cdot:\tau \frac{dt}x+\xi \frac{dx}x + \eta\cdot\frac{dy}x+\sum\zeta_i'
       \frac{dz_i'}{z'_i}+ \zeta''\cdot dz''\\
&\qquad\mapsto \tau \,dt+x\xi \frac{dx}x + \eta\,dy+\sum x\zeta_i'
       \frac{dz_i'}{z'_i}+ x\zeta''\cdot dz''.
\end{split}\end{equation*}
Note that at $x=0$, this gives
\begin{equation}\begin{split}\label{eq:x-cdot-x=0}
&x\cdot(\alpha)
=\tau \,dt+ \eta\,dy,\\
&\alpha=\tau \frac{dt}x+\xi \frac{dx}x + \eta\cdot\frac{dy}x+\sum\zeta_i'
       \frac{dz_i'}{z'_i}+ \zeta''\cdot dz''\in\Tebstar_w M,\ w\in\ef.
\end{split}\end{equation}
In particular, as the image under $(x\cdot)\circ\pieseb$ of
$p^{-1}(\{0\})\subset\Tesstar M\setminus o$ is disjoint from the zero
section, and
since multiplication by $x$ commutes with the $\RR^+$-action in the
fibers, $\hat\beta_\sharp\circ(x\cdot)$ descends to a map
$$
\dbetaeb:\ebSigma\to\bSigma,
$$
and away from $\ef$ it
is given by the restriction of the
natural identificantion of $\Sebstar[M\setminus\ef] M$
with $\Sbstar_{M_0\setminus W} M_0$, while at $\ef$, as \eqref{eq:x-cdot-x=0}
shows, is given by \eqref{eq:dbetaeb-loc-coords}, where we consider
$S^*W\subset\Sbstar M_0$, cf.\ \eqref{eq:bSigma-coords}.

We now introduce sets of covectors that are respectively elliptic,
glancing, and hyperbolic with respect to the boundary faces of $M_0$
meeting at the corner $W;$ these sets are thus of covectors over the
boundary of $\ef:$
\begin{align*}
\ecal &= \Sebstar[\pa\ef] M\setminus
\ebSigma=\{q \in \Sebstar[\pa\ef] M: (\piesebh)^{-1}(q) =
\emptyset\},\\
\gcal &= \{q \in \Sebstar[\pa\ef] M: \lvert \piesebh^{-1}(q)\rvert =
1\},\\
\hcal &= \{q \in \Sebstar[\pa\ef]: \lvert \piesebh^{-1}(q)\rvert \geq 2\},
\end{align*}
so $\ebSigma\cap \Sebstar[\pa\ef] M=\gcal\cup\hcal$.
\nomenclature[E]{$\ecal,\gcal,\hcal$}{edge-b elliptic, glancing and
hyperbolic sets}

In coordinates, note that, for instance, for
$$
w\in \ef,\ z_1'(w)=\ldots=z'_p(w)=0,\ z'_{p+1}(w)\neq 0,\ldots,\ z'_k(w)\neq0,
$$
with $p\geq 1$,
\begin{equation}\label{egh-coordinates}\begin{aligned}
\ecal\cap \Sebstar[w] &M \\
=&\{\exists j,\ 1\leq j\leq p,\ \zetaebh'_j\neq 0\}\\
&\cup
\Bigg\{1<\xiebh^2+h(y,\etaebh)
+k\Bigg(y,z,\Big(0,\ldots,0,\frac{\zetaebh'_{p+1}}{z'_{p+1}},\dots,\frac{\zetaebh'_k}{z'_k}\Big),\zetaebh''\Bigg)\Bigg\},\\
\gcal\cap \Sebstar[w] M &= \Bigg\{\zetaebh'_1=\ldots=\zetaebh'_p=0,\\
&\qquad\  1 =
\xiebh^2+h(y,\etaebh)
+k\Bigg(y,z,\Big(0,\ldots,0,\frac{\zetaebh'_{p+1}}{z'_{p+1}},\dots,\frac{\zetaebh'_k}{z'_k}\Big),\zetaebh''\Bigg)\Bigg\},\\
\hcal\cap \Sebstar[w] M &= \Bigg\{\zetaebh'_1=\ldots=\zetaebh'_p=0,\\
&\qquad\  1 >
\xiebh^2+h(y,\etaebh) +k\Bigg(y,z,\Big(0,\ldots,0,\frac{\zetaebh'_{p+1}}{z'_{p+1}},\dots,\frac{\zetaebh'_k}{z'_k}\Big),\zetaebh''\Bigg)\Bigg\}.
\end{aligned}\end{equation}

\begin{remark}
  The set $\gcal_{W,\bo}$ defined in \eqref{gcalwbdef} represents rays
  that are glancing \emph{with respect to the corner $W$,} i.e., are
  tangent to \emph{all} boundary faces meeting at $W,$ while $\gcal$
  defined above describes the rays that are glancing with respect to
  \emph{one or more} of the boundary faces meeting at $W$ (see Figure~\ref{glancingfigure}). The sets
  $\gcal_{W,\bo}$ and $\hcal_{W,\bo}$ live in $S^*W\subset \Sbstar_W
  M_0$. This can be lifted to $\Sbstar M$ by $\beta$ (since
  $\Tbstar M=\beta^*\Tbstar M_0$), but in this picture $\gcal_{W,\bo}$
  and $\hcal_{W,\bo}$ are {\em global} in the fibers of $\beta$, i.e.,
  live over all of $\ef$, not merely over its boundary.
\end{remark}

\begin{figure}[ht]
\includegraphics{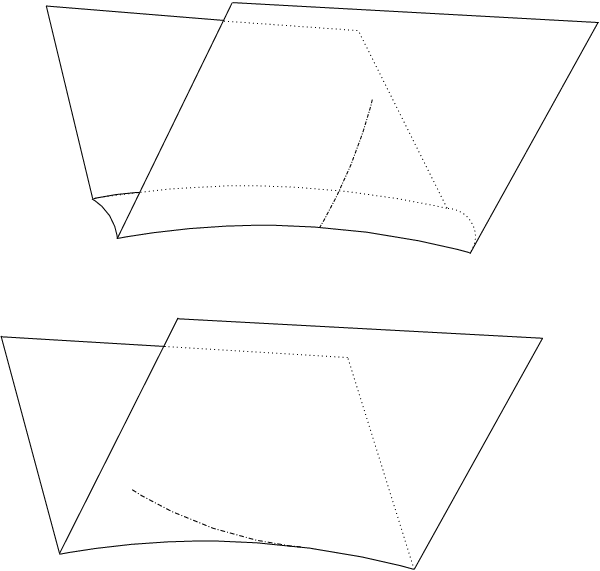}
\caption{Glancing rays.  The ray depicted at top, in $M$ (projected
  down to $X$), terminates at
  a point in $\gcal.$  The ray depicted at bottom, in $M_0$ (projected
  down to $X_0$), terminates at a point in $\gcal_{W,\bo}.$}
\label{glancingfigure}
\end{figure}

\subsection{Bicharacteristics}\label{subsection:bichars}
We now turn to bicharacteristics in $\ebSigma$, which will be the
dynamical locus of the geometric improvement for the propagation result.
Taking into account that $\sH_{\eso}$ is tangent to the
fibers of $\dbetaes$, one expects that over $\ef$,
these bicharacteristics will lie
in a single fiber of the related map $\dbetaeb$, i.e.\ $y,t,\etaebh$ will
be constant along these. The fibers of $\dbetaes$ and $\dbetaeb$
have a rather different character depending on whether they are
over a point in $\gcal_{W,\bo}$ or in $\hcal_{W,\bo}$. Namely,
over $\gcal_{W,\bo}$ the fibers of $\dbetaes$ resp.\ $\dbetaeb$ are
$\xibh=0,\zetabh=0$ resp.\ $\xieb=0, \zetaeb=0$ i.e.\ they are the 
zero section.  By contrast over $\alpha=(t,y,\etah)\in\hcal_{W,\bo}$, the
fiber of $\dbetaes$ is
$$
\hcal_{\eso\to\alpha,\bo}=\Big\{(t,y,z,\xiesh,\etaesh=\etah,\zetaesh)\in\Sesstar M:
\xiesh^2+k(y,z,\zetaesh)= 1-h(y,\etaesh)\Big\}
$$
while that of
$\dbetaeb$ is
\begin{align*}
\hcal_{\ebo\to\alpha,\bo}=\Bigg\{(&t,y,z,\xiebh,\etah,\zetaebh)\in\Sebstar M:
\ \zetaebh'_1=\ldots=\zetaebh'_p=0,\\
&\ 
\xiebh^2
+k\Bigg(y,z,\Big(0,\ldots,0,\frac{\zetaebh'_{p+1}}{z'_{p+1}},\dots,\frac{\zetaebh'_k}{z'_k}\Big),\zetaebh''\Bigg)
\leq
1-h(y,\etaebh)\Bigg\}.
\end{align*}

The geometric improvement will take place over $\hcal_{W,\bo}$, so from
now on we concentrate on this set.
Now, for $\alpha=(t,y,\etah)\in\hcal_{W,\bo}$
$$
\rcal_{\eso}\cap \hcal_{\eso\to\alpha,\bo}=
\big\{(t,y,z,\xiesh,\etaesh=\etah,\zetaesh=0)\in\Sesstar M:
\xiesh^2=1-h(y,\etaesh)\big\},
$$
hence has two connected components which we denote by
$$
\rcal_{\eso,\alpha,I/O}=
\Big\{(t,y,z,\xiesh,\etaesh=\etah,\zetaesh=0)\in\Sesstar M:
\xiesh=\pm\sgn(\taues)\sqrt{1-h(y,\etaesh)}\Big\},
$$
with $\sgn(\taues)$ being the constant function $\pm 1$ on the two
connected components of $\esSigma$.

Here the labels ``I/O'',
stand for ``incoming/outgoing.'' This is explained by
$$
-\frac{1}{2}\sH_{\eso}t=-\tauesh ,\ -\frac{1}{2}\sH_{\eso} x=\xiesh x,
$$
so in a neighborhood of $\rcal_{\eso,\alpha,I}$, $\sH_{\eso} t$ and
$\sH_{\eso}x$ have the opposite signs, i.e.\ if $t$ is increasing, $x$
is decreasing along $\sH_{\eso}$, just as one would expect an `incoming
ray' to do; at outgoing points the reverse is the case.

We also let
\begin{align*}
\rcal_{\ebo,\alpha,I/O}&=\piesebh(\rcal_{\eso,\alpha,I/O})\\
&=
\Big\{(t,y,z,\xiebh,\etah,\zetaebh=0)\in\Sebstar M:
\xiebh=\pm\sgn(\taueb)\sqrt{1-h(y,\etaebh)}\Big\},
\end{align*}
and
\begin{equation*}
\rcal_{\ebo,S,I/O}=\cup_{\alpha\in S}\rcal_{\ebo,\alpha,I/O}
\end{equation*}
for $S\subset\hcal_{W,\bo}$.
\nomenclature[R]{$\rcal_{\ebo}$}{edge-b radial set}

\begin{definition}
An \emph{edge generalized broken bicharacteristic}, or \EGBB, is a continuous
map $\gamma:I \to \ebSigma$ such that for all $f \in \CI(\Sebstar
M),$ real-valued,
\begin{equation}\label{egbb}
\begin{aligned}
&\liminf_{s \to s_0} \frac{ (f\circ \gamma)(s)- (f\circ
  \gamma)(s_0)}{s-s_0}\\
&\qquad \geq \inf\Big\{ \sH_{\eso} (\piesebh^* f)(q): q \in
\piesebh^{-1}(\gamma(s_0))\cap \esSigma \Big\}.
\end{aligned}
\end{equation}
\end{definition}
\nomenclature[E]{\EGBB}{Edge generalized broken bicharacteristic}

\begin{lemma}\label{lemma:EGBB-GBB}
\begin{enumerate}
\item
An \EGBB outside $\Sebstar[\ef]M$ is a reparameterized \GBB (under the
natural identification of $\Sbstar_{M_0\setminus W} M_0$ with
$\Sebstar[M\setminus\ef]M$), and conversely.
\item
If a point $q$ on an \EGBB lies in $\Sebstar[\ef]M$, then the whole
\EGBB lies in $\Sebstar[\ef]M$, in $\dbetaeb^{-1}(\dbetaeb(q))$, i.e.\ in
the fiber of $\dbetaeb$ through $q$.
\item
The only \EGBB through a point in $\rcal_{\eb,\alpha,I/O}$ is the constant
curve.
\item
For $\alpha\in\hcal_{W,\bo}$, an \EGBB in $\hcal_{\ebo\to\alpha,\bo}
\setminus\rcal_{\eb,\alpha,I/O}$
projects to a reparameterized \GBB in $\Tbstar Z$, hence to 
geodesic of length $\pi$ in $Z.$
\end{enumerate}
\end{lemma}

\begin{proof}\mbox{}
\begin{enumerate}
\item As $\sH_{\so}$ and $\sH_{\eso}$ differ by an overall factor
under the natural identification $\iota:S^*_{M_0\setminus W}M_0\to
\Sesstar_{M\setminus \ef}M$, namely
$$
\iota_*\sH_{\so}=|\taus|^{-1}\,|\taues|\,x^{-2}\sH_{\eso}=x^{-1}\sH_{\eso},
$$
we obtain this immediately.
\item 
The tangency of
  $\sH_{\eso}$ to the fibers of $\dbetaes$ means that
if we set $f$ equal to any of $\pm y_j,\pm t,\pm\etabh_j,$
$\sH_{\eso}f=0.$ By \eqref{egbb}, then
 $(f\circ\gamma)'(s_0)=0$ for all $s_0,$ and for each of these
 choices.  This ensures that $\gamma$ remains in the fiber.
\item $\sH_{\eso}$ vanishes at the unique
$q \in
\piesebh^{-1}(\gamma(s_0))\cap \esSigma$ if
$\gamma(s_0)\in\rcal_{\ebo,\alpha,I/O}.$  Moreover, the function
$\xiebh\circ\gamma$ is in $C^1,$ as
$$
\sH_{\eso}\xiesh=2K(y,t,z,\zetaesh)=2\big(1-h(y,\etaesh)-\xiesh^2\big)
$$
on $\esSigma.$  Thus,
\eqref{egbb} entails that if
$\xiesh=\pm\sqrt{1-h(y,\etaesh)}$ at some point on an $\EGBB$, then it
is constant.
\item This
follows from a reparameterization argument, as in \cite{mvw1}, taking
into account that $\sH_{\eso}$ is tangent to the fibers of
$\Tesstar_{\ef}M$, hence can be considered as a vector field
on $\RR_{\xiesh}\times T^*Z$. (In fact, a completely analogous argument
takes place in \cite[Section~6]{Vasy:Propagation-Many} in the setting
of $N$-body scattering.)\qed
\end{enumerate}
\noqed
\end{proof}

Suppose now that $\gamma:[0,\delta_0]\to\bSigma$ is a \GBB with
$\gamma(0)=\alpha\in\hcal_{W,\bo}$. Thus, assuming $\delta_0>0$ is
sufficiently small,
by Lemma~\ref{lemma:OG-IC-uniform-est-1},
$\gamma|_{(0,\delta_0]}\cap\Sbstar_W M_0=\emptyset$.
Since $\Sbstar_{M_0\setminus W} M_0$ is naturally diffeomorphic to
$\Sebstar[M\setminus\ef] M$, we can lift $\gamma|_{(0,\delta_0]}$
to a curve $\gammat:(0,\delta_0]\to\Sebstar M$ in a unique fashion.
It is natural to ask whether this lifted curve extends continuously
to $0$, which is a question we now address.

The following is easily deduced
from Lebeau, \cite[Proposition~1]{Lebeau5} (stated here in
Lemma~\ref{lemma:Lebeau-IC-OG}) and its proof:

\begin{lemma}\label{lemma:GBB-eb-lift}
Suppose that $\alpha\in\hcal_{W,\bo}$. There exists $\delta_0>0$ with the
following property.

Suppose $\gamma:[0,\delta_0]\to\Sbstar M_0$
is a \GBB with $\gamma(0)=\alpha$. Let $\tilde\gamma:(0,\delta_0]\to\Sebstar M$
be the unique lift of $\gamma|_{(0,\delta_0]}$ to $\Sebstar M$. Then
$\tilde\gamma$ (uniquely) extends to a continuous map
$\tilde\gamma:[0,\delta_0]\to\Sebstar M$, with $\tilde\gamma(0)
\in\rcal_{\ebo,\alpha,O}$.

In addition, $\gamma$ approaches $W$ normally if and only if
$$
\tilde\gamma(0)\notin\Sebstar[\pa\ef]M\cap\rcal_{\ebo,\alpha,O}
=\cG\cap\rcal_{\ebo,\alpha,O}.
$$

The analogous results hold if $[0,\delta_0]$ is replaced by $[-\delta_0,0]$
and $\rcal_{\ebo,\alpha,O}$ is replaced by $\rcal_{\ebo,\alpha,I}$.
\end{lemma}

\begin{remark}
The proof in fact shows that $\delta_0$ can be chosen independent of
$\alpha$ as long as we fix some $K\subset\hcal_{W,\bo}\subset\Sbstar_W M_0$
compact and require $\alpha\in K$.
\end{remark}

\begin{remark}\label{remark:normal-GBB-lift}
  The special case of a \emph{normal} \GBB segment $\gamma$, which
  lifts to a curve $\gammat:[0,\delta_0]\to\Sebstar M$ starting at
  $\ef^\circ,$ follows directly by the description of geodesic in edge
  metrics from \cite{mvw1}, since normality implies that for
  sufficiently small $\delta_0>0$, $\gamma|_{(0,\delta_0]}$ has image
  disjoint from $x_j=0$ for all $j$, i.e.\ the boundaries can be
  ignored, and one is simply in the setting of \cite{mvw1}. This
  argument also shows that given $\alpha\in\hcal_{W,\bo}$ and
  $p\in\rcal_{\ebo,\alpha,O} \setminus\cG$, for sufficiently small
  $\delta_0>0$, there is a unique \GBB $\gamma:[0,\delta_0]\to\Sbstar
  M_0$ with $\gamma(0)=\alpha$ such that the lift $\gammat$ of
  $\gamma$ satisfies $\gamma(0)=p$.
\end{remark}

\begin{proof}
Let $\alpha=(y_0,t_0,\etabh_0)$.
First, by Lemma~\ref{lemma:OG-IC-uniform-est-1},
$\gamma|_{(0,\delta_0]}\cap\Sbstar_W M_0=\emptyset$ for $\delta_0>0$
sufficiently small, hence the lift $\gammat|_{(0,\delta_0]}$
exists and is unique.
Lebeau proves in \cite[Proof of Proposition~1]{Lebeau5} (with our
notation)
that
$$
\lim_{s\to 0}\xiebh(\gammat(s))=\sqrt{1-h(y_0,\etabh_0)}
\ \Mand
\ \frac{dx(\gamma(s))}{ds}|_{s=0}=2\sqrt{1-h(y_0,\etabh_0)}>0.
$$
This
implies that
$$
\sup\{|\zetaesh(q)|:\ q\in(\widehat{\pismooth})^{-1}(\gamma(s))\}\to 0
\ \text{as}\ s\to 0+,
$$
since
$$
\sum K_{ij}(y,z)\zetaesh_i\zetaesh_j=1-h(y,\etaesh)-\xiesh^2+xG
\leq 1-h(y,\etaesh)-\xiesh^2+Cx
$$
on $\esSigma$, and
$1-h(y(\gammat(s)),\etaesh(\gammat(s)))-\xiesh(\gammat(s))^2
+Cx(\gammat(s))\to 0$.
It remains to show that the coordinates $z_j$ have a limit as $s\to 0$.
But by Lemma~\ref{lemma:Lebeau-IC-OG}, $(dx_j\circ\gamma/ds)|_{s=0}
=2\xish_j(0)$ exists, and
$\sum A_{ij}\xish_i(0)\xish_j(0)=1-h(y_0,\etabh_0)>0$. Thus, considering
$z_j(\gamma(s))=x_j(\gamma(s))/x(\gamma(s))$, L'H\^opital's rule shows
that $\lim_{s\to 0+} z_j(\gamma(s))=\xish_j(0)/\sqrt{1-h(y_0,\etabh_0)}$
exists, finishing the proof of the first claim. The second claim follows
at once from the last observation regarding $\lim_{s\to 0+} z_j(\gamma(s))$.
\end{proof}

We also need the following result, which is a refinement of
Lemma~\ref{lemma:GBB-eb-lift}, insofar as Lebeau's result
only deals with a single \GBB emanating from the corner $W$ of $M_0:$
the following lemma extends Lemma~\ref{lemma:GBB-eb-lift} uniformly to \GBBs
starting close to but not at the corner.
For simplicity of notation, we only state the results for the outgoing
direction.

\begin{lemma}\label{lemma:equicontinuty-lifted-GBB}
Suppose that $\alpha\in\hcal_{W,\bo}$, $p\in\rcal_{\ebo,O,
\alpha}$,
$p_n\in\Sebstar[M\setminus\ef] M$, and $p_n\to p$ in $\Sebstar M$.
Suppose $\delta_0>0$ is sufficiently small (see following remark).
Let $\gamma_n:[0,\delta_0]\to\Sbstar M_0$ be \GBB such that $\gamma_n(0)=p_n$.
For $n$ sufficiently large,
let $\tilde\gamma_n:[0,\delta_0]\to\Sebstar M$ be the unique lift of
$\gamma_n$ to a map $[0,\delta_0]\to\Sebstar M$.
Then for $N$ sufficiently large,
$\{\tilde\gamma_n\}_{n\geq N}$ is equicontinuous.
\end{lemma}

\begin{remark}\label{remark:unif-nbhd}
  As $p_n\to p$, $\xiebh(p_n)\to\xiebh(p)>0$, so there exists $N>0$
  such that
$$
\sum_{j=1}^{f+1}\xibh_j(p_n)=x(p_n)\xiebh(p_n)+O\big(x(p_n)^2\big)>0
$$
for $n\geq N$;
cf.\ \eqref{eq:xib-xiebh}.
  Thus, by Lemma~\ref{lemma:OG-IC-uniform-est-1}, there exists
  $\delta_0>0$ such that $\gamma_n|_{[0,\delta_0]}\cap\Sbstar_W
  M_0=\emptyset$ for $n\geq N$---this is the $\delta_0$ in the
  statement of the lemma.  Hence, for $n$ sufficiently large,
  $\gamma_n$ has a unique lift $\tilde\gamma_n$ to $\Sebstar M$, since
  $\Sebstar M$ and $\Sbstar M_0$ are naturally diffeomorphic away from
  $\ef$, resp.\ $W$ as previously noted.
\end{remark}

\begin{proof}
Note first that $\{\gamma_n\}_{n\in\NN}$ is equicontinuous by
Lebeau's result \cite[Corollaire~2]{Lebeau5} (see also the proof of
\cite[Proposition~6]{Lebeau5})---indeed, this follows directly from
our definition of \GBB. This implies that
$\{\tilde\gamma_n\}_{n\in\NN}$ is equicontinuous at all $s_0\in(0,\delta_0]$,
for given such a $s_0$, there exists $K_0\subset M_0$ compact disjoint from
$W$ such that $\gamma_n|_{[s_0,\delta_0]}$ has image in $\Sbstar_{K_0}M_0$,
which is canonically diffeomorphic to $\Sebstar[\beta^{-1}(K_0)]M$.
Thus, it remains to consider equicontinuity at $0$.

For sufficiently large $n$, all $\gamma_n$ have image in $\Sbstar_K M_0$
where $K$ is compact and $K\subset O$ for a coordinate chart $O$ on $M_0$.
Thus, by the equicontinuity of $\gamma_n$, the coordinate functions
$$
x_j\circ\gamma_n,\ t_j\circ\gamma_n,\ y_j\circ\gamma_n,\ \xibh_j\circ\gamma_n,\ \etabh_j\circ\gamma_n
$$
are equicontinuous. We need to show that for the lifted curves,
$\tilde\gamma_n$, the coordinate functions
$$
x\circ\tilde\gamma_n,\ t_j\circ\tilde\gamma_n,\ y_j\circ\tilde\gamma_n,\ z_j\circ\tilde\gamma_n,
\ \xiebh\circ\tilde\gamma_n,\ \etaebh\circ\tilde\gamma_n,\ \zetaebh_j\circ\tilde\gamma_n
$$
are equicontinuous at $0$.
By the above description, and $y_j\circ\gammat_n=y_j\circ\gamma_n$,
$t_j\circ\gammat_n=t_j\circ\gamma_n$ and
$\etaebh_j\circ\gammat_n=\etabh_j\circ\gamma_n$ are equicontinuous,
as is $x\circ\gammat_n$ in view of $x=(\sum a_{ij} x_i x_j)^{1/2}$.
Thus, it remains to consider $\xiebh\circ\tilde\gamma_n$,
$z_j\circ\tilde\gamma_n$ and $\zetaebh_j\circ\tilde\gamma_n$.

Let $p=(0,y_0,z_0,\xiebh_0,\etaebh_0,0)$, and write $\mu=\xiebh_0>0$.
Thus,
$$
\mu=\sqrt{1-h(y_0,\etaebh_0)}.
$$
Let $\ep_1>0$. One can show easily, as
in the proof of Lebeau's \cite[Proposition~1]{Lebeau5}, that for all $n$
sufficiently large (so that $p_n$ sufficiently close to $p$) and
$s_0>0$ sufficiently small,
\begin{equation}\label{eq:xiebh-outgoing}
s\in [0,s_0]\Longrightarrow\xiebh\circ\gammat_n(s)\in [\mu-\ep_1,\mu+\ep_1].
\end{equation}
Indeed,
$\sH_{\so}\xiebh=2x^{-1}\sum K_{ij}\zetaesh_i\zetaesh_j+F$ with $F$ smooth,
so $\sH_{\so}\xiebh\geq -C_0$ over the compact set $K$, hence
\begin{equation}\label{eq:xiebh-lower-bd}
\xiebh\circ\gammat_n(s)\geq \xiebh(p_n)-C_0s.
\end{equation}
On the other hand,
on $\esSigma$,
$$
\xiesh^2=1-h(y,\etaesh)-\sum K_{ij}(y,z)\zetaesh_i\zetaesh_j+xG
\leq 1-h(y,\etaesh)+C_1 x,
$$
hence on $\ebSigma$,
$$
\xiebh^2\leq 1-h(y,\etaebh)+C_1 x.
$$
Let
$$
\Phi(x,y,\etabh)=\sqrt{1-h(y,\etabh)+C_1 x};
$$
this is thus a Lipschitz function on a neighborhood of $\alpha$ in
$\Sbstar M_0$, hence there is $s_0'>0$ such that
$\Phi\circ\gamma_n|_{[0,s_0']}$ is uniformly Lipschitz
for $n$ sufficiently large.
Thus,
\begin{equation*}\begin{split}
|\xiebh(\gammat_n(s))|&\leq\Phi(\gamma_n(s))\leq\Phi(\alpha)
+|\Phi(p_n)-\Phi(\alpha)|+|\Phi(\gamma_n(s))-\Phi(p_n)|\\
&\leq \sqrt{1-h(y_0,\etaebh_0)}+|\Phi(p_n)-\Phi(\alpha)|+C's.
\end{split}\end{equation*}
Thus,
for sufficiently large $n$ (so that $p_n$ is close to $p$),
\begin{equation}\label{eq:xiebh-upper-bd}
|\xiebh(\gammat_n(s))|\leq \sqrt{1-h(y_0,\etaebh_0)}+\ep_1/2+C's.
\end{equation}
Combining \eqref{eq:xiebh-lower-bd} and \eqref{eq:xiebh-upper-bd}
gives \eqref{eq:xiebh-outgoing}.

Now consider the function
$$
\Theta=1-h(y,\etaebh)-\xiebh^2,
$$
so $\pi_{\so\to\bo}^*\Theta|_{\sSigma\cap S^*M_0}=
\sum K_{ij}\zetaesh_i\zetaesh_j$.
This satisfies
$$
\sH_{\so}\Theta=-2\xiebh\sH_{\so}\xiebh+F_1
=-4x^{-1}\xiebh\sum K_{ij}\zetaesh_i\zetaesh_j+F_2=-4x^{-1}\xiebh\Theta+F_3
$$
with $F_j$ smooth.
Now,
$$
x(p_n)+(\mu-\ep_1)s\leq x\circ\gamma_n(s)\leq x(p_n)+(\mu+\ep_1)s,
$$
so
$$
\frac{d}{ds}\Theta\circ\gamma_n+4x^{-1}\xiebh\Theta\leq C
$$
implies that
$$
\frac{d}{ds}\Theta_n+\frac{4(\mu-\ep_1)}{x(p_n)+(\mu+\ep_1)s}
\Theta_n\leq C,
$$
where we write $\Theta_n=\Theta\circ\gamma_n$.
Multiplying through by
$$
\big(x(p_n)+(\mu+\ep_1)s\big)^{4(\mu-\ep_1)/(\mu+\ep_1)}
$$
gives
\begin{equation}\begin{split}
&\frac{d}{ds}
\Big(\big(x(p_n)+(\mu+\ep_1)s\big)^{4(\mu-\ep_1)/(\mu+\ep_1)}\Theta_n\Big)\\
&\qquad\qquad\leq C\big(x(p_n)+(\mu+\ep_1)s\big)^{4(\mu-\ep_1)/(\mu+\ep_1)}.
\end{split}\end{equation}
Integration gives
\begin{equation}\begin{split}
&\big(x(p_n)+(\mu+\ep_1)s\big)^{4(\mu-\ep_1)/(\mu+\ep_1)}\Theta_n(s)
-x(p_n)^{4(\mu-\ep_1)/(\mu+\ep_1)}\Theta_n(0)\\
&\qquad\leq C'\Big(\big(x(p_n)+(\mu+\ep_1)s\big)^{1+4(\mu-\ep_1)/(\mu+\ep_1)}
- x(p_n)^{1+4(\mu-\ep_1)/(\mu+\ep_1)}\Big).
\end{split}\end{equation}
Thus,
\begin{multline}\label{eq:stronger-Theta-estimate}
\Theta_n(s) \leq \big(1+(\mu+\ep_1)s/x(p_n)\big)^{-4(\mu-\ep_1)/(\mu+\ep_1)}
\Theta_n(0)\\
+C'\Big(\big(x(p_n)+(\mu+\ep_1)s\big)
-x(p_n)(1+(\mu+\ep_1)s/x(p_n)\big)^{-4(\mu-\ep_1)/(\mu+\ep_1)}\Big).
\end{multline}
Since
$$
\big(1+(\mu+\ep_1)s/x(p_n)\big)^{-4(\mu-\ep_1)/(\mu+\ep_1)}<1,
$$
this yields
\begin{equation}
\Theta_n(s)\leq
\Theta_n(0)+C'\big(x(p_n)+(\mu+\ep_1)s\big)
\end{equation}
On the other hand, as on $\sSigma$,
$\Theta=\sum K_{ij}\zetaesh_i\zetaesh_j+xF$,
with $F$ smooth, so $\Theta\geq -Cx$, we deduce that
$$
\Theta_n(s)\geq -C\big(x(p_n)+(\mu+\ep_1)s\big).
$$
Thus,
$$
-C\big(x(p_n)+(\mu+\ep_1)s\big)\leq\Theta_n(s)\leq \Theta_n(0)+C'\big(x(p_n)+(\mu+\ep_1)s\big).
$$
Suppose now that $\ep>0$ is given.
As $p_n\to p$, there is an $N$ such that for $n\geq N$,
$Cx(p_n)+\Theta_n(0),C'x(p_n)\leq \ep/2$. Moreover, let $s_0>0$ such that
$C(\mu+\ep_1)s_0,C'(\mu+\ep_1)s_0<\ep/2$. Then for $n\geq N$, $s\in[0,s_0]$,
$-\ep\leq\Theta_n(s)-\Theta_n(0)\leq \ep$, giving the equicontinuity
of $\Theta_n$ at $0$ for $n\geq N$. In view of the definition of $\Theta_n$
and the already known equicontinuity of $y\circ\gammat_n$ and
$\etaebh\circ\gammat_n$, it follows that $(\xi\circ\gammat_n)^2$, hence
$\xi\circ\gammat_n$ are
equicontinuous. As on $\esSigma$, $|\zetaesh|^2\leq C|\Theta|+C'x$, we also
have $|\zetaebh|^2\leq C|\Theta|+C'x$ there,
so
$$
|\zetaebh\circ\gammat_n(s)-\zetaebh(p_n)|\leq
|\zetaebh(p_n)|+|\zetaebh\circ\gammat_n(s)|
\leq |\zetaebh(p_n)|+C|\Theta_n(s)|+C'x(\gammat_n(s)).
$$
Given $\ep>0$, by the equicontinuity of $\Theta_n$ and $x\circ\gammat_n$,
there is $s_0$ such that for $s\in[0,s_0]$,
$C|\Theta_n(s)|+C'x(\gammat_n(s))<\ep/2$. As $\zetaebh(p_n)\to 0$ due to
$p_n\to p$, for $n$ sufficiently large, $|\zetaebh(p_n)|<\ep/2$, so
for $n$ sufficiently large and $s\in[0,s_0]$,
$|\zetaebh\circ\gammat_n(s)-\zetaebh(p_n)|\leq\ep$, giving the
equicontinuity of $\zetaebh\circ\gammat_n$ at $0$.

It remains to check the equicontinuity of $Z_n=z\circ\gammat_n$.
But
$$
\abs{\frac{dZ_n}{ds}}\leq C 
\sup\Big\{x(q)^{-1}|\zetaesh(q)|:\ q\in\sSigma,\ \pismooth(q)=\gamma_n(s)\Big\},
$$
and for such $q$, by \eqref{eq:stronger-Theta-estimate},
\begin{equation*}\begin{split}
x^{-2}|\zetaesh|^2&\leq Cx^{-2}(|\Theta|+x)\\
&\leq C\big(x(p_n)+(\mu-\ep_1)s\big)^{-2}
\big(1+(\mu+\ep_1)s/x(p_n)\big)^{-4(\mu-\ep_1)/(\mu+\ep_1)}
\Theta_n(0)\\
&\qquad+C\big(x(p_n)+(\mu+\ep_1)s\big),
\end{split}\end{equation*}
so
\begin{equation*}\begin{split}
&x^{-1}|\zetaesh|\\
&\leq C\big(x(p_n)+(\mu-\ep_1)s\big)^{-1}
\big(1+(\mu+\ep_1)s/x(p_n)\big)^{-2(\mu-\ep_1)/(\mu+\ep_1)}
\Theta_n(0)^{1/2}\\
&\qquad+C\sqrt{x(p_n)+(\mu+\ep_1)s}\\
&\leq C x(p_n)^{-1}
\big(1+(\mu-\ep_1)s/x(p_n)\big)^{-1-2(\mu-\ep_1)/(\mu+\ep_1)}
\Theta_n(0)^{1/2}\\
&\qquad+C\sqrt{x(p_n)+(\mu+\ep_1)s}.
\end{split}\end{equation*}
Thus, integrating the right hand side shows that
\begin{equation*}\begin{split}
|Z_n(s)-Z_n(0)|&\leq C'\Theta_n(0)^{1/2}
\Big(\big(1+(\mu-\ep_1)s/x(p_n)\big)^{-2(\mu-\ep_1)/(\mu+\ep_1)}-1\Big)\\
&\qquad+C's\sqrt{x(p_n)+(\mu+\ep_1)s}\\
&\leq C'\Theta_n(0)^{1/2}
+C's\sqrt{x(p_n)+(\mu+\ep_1)s}.
\end{split}\end{equation*}
An argument as above gives the desired equicontinuity for $n$
sufficiently large, completing the proof of the lemma.
\end{proof}

\begin{corollary}\label{cor:GBB-lifted-conv}
Suppose that $\alpha\in\hcal_{W,\bo}$, $p\in\rcal_{\ebo,O,\alpha}$,
$p_n\in\Sebstar[M\setminus\ef] M$, and $p_n\to p$ in $\Sebstar M$.
Let $\gamma_n:[0,\delta_0]\to\Sbstar M_0$ be \GBB such that $\gamma_n(0)=p_n$.
Then there is a \GBB $\gamma:[0,\delta_0]\to\Sbstar M_0$
and $\gamma_n$ has a subsequence, $\{\gamma_{n_k}\}$, such that
$\gamma_{n_k}\to\gamma$ uniformly, the lift $\tilde\gamma:[0,\delta_0]
\to\Sebstar M$ of $\gamma$ satisfies $\tilde\gamma(0)=p$, and
the lift $\tilde\gamma_{n_k}$ of
$\gamma_{n_k}$ converges to $\tilde\gamma$ uniformly.
\end{corollary}

\begin{proof}
As $p_n\to p$, it follows that there is a compact set $K_0\subset M_0$
such that $\gamma_n(s)\in\Sbstar_{K_0} M_0$ for all $n$ and all
$s\in[0,\delta_0]$. Then by the compactness of the set of \GBBs with image
in $\Sbstar_{K}M_0$ in the topology of uniform convergence,
\cite[Proposition~6]{Lebeau5}, $\gamma_n$ has a subsequence,
$\gamma_{n_k}$, uniformly converging to a
\GBB $\gamma:[0,\delta_0]\to\Sbstar M_0$. In particular,
$\gamma(0)=\lim_k\gamma_{n_k}(0)=\lim_k\dbetaeb(p_{n_k})
=\dbetaeb(p)=\alpha$. By Lemma~\ref{lemma:GBB-eb-lift}, $\gamma$
lifts to a curve $\tilde\gamma:[0,\delta_0]\to \Sebstar M$. We claim
that $\tilde\gamma(0)=p$---once we show this, the corollary is proved.

Let $\tilde\gamma_n:[0,\delta_0]\to\Sebstar M$ be the lift of $\gamma_n$.
By Lemma~\ref{lemma:equicontinuty-lifted-GBB},
$\{\tilde\gamma_{n_k}\}_{k\in\NN}$ is equicontinuous. Since for $\delta>0$
$\gamma_{n_k}|_{[\delta,\delta_0]}\to\gamma$ uniformly, and these curves
all have images in $\Sbstar_{K_1}M_0$ for some $K_1$ compact, disjoint
from $W$, where $\Sbstar_{K_1}M_0$ and $\Sebstar[\beta^{-1}(K_1)] M$
are canonically diffeomorphic, we deduce that
$\tilde\gamma_{n_k}|_{[\delta,\delta_0]}\to\tilde\gamma|_{[\delta,\delta_0]}$
uniformly; in particular $\{\tilde\gamma_{n_k}|_{[\delta,\delta_0]}\}$
is a Cauchy sequence in the uniform
topology.

Let $d$ be a metric on $\Sebstar M$ giving rise to its topology.
Given $\ep>0$ let $\delta>0$ be such that
for $0\leq s\leq \delta$ and for all $n$,
one has $d(\tilde\gamma_n(s),\tilde\gamma_n(0))=d(\tilde\gamma_n(s),p_n)
<\ep/3$---this $\delta$
exists by equicontinuity. Next, let $N$ be such that for $k,m\geq N$,
$d(p_{n_k},p_{n_m})<\ep/3$ and for $k,m\geq N$, $\delta\leq s\leq\delta_0$,
$d(\tilde\gamma_{n_k}(s),\tilde\gamma_{n_m}(s))\leq\ep/3;$ such a
choice of $N$ exists 
by the uniform Cauchy statement above, and the convergence
of $\{p_n\}$. Thus, for $k,m\geq N$ and $0\leq s
\leq\delta$,
$$
d(\tilde\gamma_{n_k}(s),\tilde\gamma_{n_m}(s))\leq
d(\tilde\gamma_{n_k}(s),p_{n_k})
+d(p_{n_k},p_{n_m})+d(p_{n_m},\tilde\gamma_{n_m}(s))\leq \ep.
$$
Since we already know the analogous claim for $\delta\leq s\leq\delta_0$,
it follows that $\{\tilde\gamma_{n_k}\}$ is uniformly Cauchy, hence
converges uniformly to a continuous map $\hat\gamma:[0,\delta_0]
\to\Sebstar M$. In particular, $\hat\gamma(0)=\lim_k\tilde\gamma_{n_k}(0)
=\lim_k p_{n_k}=p$. But $\tilde\gamma_{n_k}|_{[\delta,\delta_0]}
\to\tilde\gamma|_{[\delta,\delta_0]}$ uniformly for $\delta>0$, so
$\tilde\gamma|_{[\delta,\delta_0]}=\hat\gamma|_{[\delta,\delta_0]}$.
The continuity of both $\tilde\gamma$ and $\hat\gamma$ now shows that
$\tilde\gamma=\hat\gamma$, and in particular $\tilde\gamma(0)=p$ as
claimed.
\end{proof}

Now we are ready to introduce the bicharacteristics that turn out
in general to carry full-strength, rather than weaker, diffracted,
singularities.
\begin{definition}
A \emph{geometric} \GBB is a \GBB $\gamma:(-s_0,s_0) \to \bSigma$ with
$q =\gamma(0) \in \hcal_{W,\bo}$ such that there is an \EGBB
$\rho:\RR \to \Tebstar_{\ef} M$ with
\begin{align*}
\lim_{s\to -\infty} \rho(s)&=\lim_{t\to 0^-} \gammat_-(t),\\
\lim_{s\to +\infty} \rho(s)&=\lim_{t\to 0^+} \gammat_+(t),
\end{align*}
with $\gammat_+$, resp.\ $\gammat_-$,
denoting the lifts $\gamma|_{[0,\delta_0]}$, resp.\ $\gamma|_{[-\delta_0,0]}$,
$\delta_0>0$ sufficiently small, to $\Sebstar M$.

We say that two points $w,w' \in \bSigma$ are \emph{geometrically
  related} if they lie along a single geometric \GBB.
\end{definition}

Let $T$ be a large parameter, fixed for the duration of this
paper.
\begin{definition}

  For $p \in \hcal_{W,\bo}$ the flow-out of $p,$
  denoted $\fcalW_{O,p},$ is the union of images $\gamma((0,T])$
  of \GBBs $\gamma: [0, T] \to \bSigma$ with
  $\gamma(0)=p.$
\nomenclature[F]{$\fcalW_{I/O}$}{b flow-in/flow-out}

  For $p \in \hcal_{W,\bo},$ the \emph{regular
    part} of the flow-out of $p,$ denoted $\fcalW_{O,p,\reg},$ is the
  union of images $\gamma((0,s_0))$ of normally approaching (or
  \emph{regular}) \GBBs
  $\gamma: [0, s_0) \to \bSigma$ with $\gamma(0)=p$ and $\gamma(s) \in
  T^*M^\circ$ for $s \in (0, s_0).$

The regular part of the flow-out of a subset of $\hcal_{W,\bo}$ is the
union of the regular parts of the flow-outs from the points in the set.

We let
$$
\fcalW_{O,p,\sing}
$$
denote the union of images $\gamma(0,T]$ of
\emph{non-}normally-approaching \GBBs $\gamma,$ i.e.\ those
\GBBs $\gamma$ with $\gamma(0) \in \gcal \cap \rcal_{\ebo}.$

The flow-in and its regular part are defined correspondingly and
denoted
$$
\fcalW_{I,p}, \fcalW_{I,p,\reg}.
$$
We let $\fcalW_{I/O}$ denote the union of the flow-ins/flow-outs of all $p \in \hcal_{W,\bo}.$
\end{definition}

We also need to define the flow-in/flow-out of a single hyperbolic
point $q \in \rcal_{\ebo,\alpha,I/O}\backslash \Sebstar[\pa \ef] M$
(i.e.\ for $p\in\hcal_{W,\bo}$ as above, we
will consider the flow in/out to a single point in a fiber
$q\in\rcal_{\ebo,p,I/O}$). By Remark~\ref{remark:normal-GBB-lift},
given such a $q,$
there is a unique
\GBB $\gamma(s)$, defined on $[0,T]$ (or $[-T,0]$,
in case of $I$), with lift $\tilde\gamma$ satisfying
$\lim_{s\to 0} \tilde\gamma(s) =q.$

\begin{definition}
For $q \in \rcal_{\ebo,I/O}\backslash \Sebstar[\pa \ef] M,$ let
$\fcal_{I/O,q}$ denote the image $\tilde\gamma((0,T])$ (or  $\tilde\gamma([-T,0))$
in case of $I$)
where $\gamma$ is
the unique \GBB with lift $\tilde\gamma$ satisfying
$\lim_{s\to 0} \tilde\gamma(s) =q.$
\nomenclature[F]{$\fcal_{I/O}$}{edge-b flow-in/flow-out}
Let
$\fcal_{I/O,q,\reg}$ be defined as the union of $\tilde\gamma((0,s_0))$ with
$\tilde\gamma(s) \in T^*M^\circ$ for all $s \in (0,s_0)$.  Additionally, let $\fcal_{I/O}$
denote the union of all flow-ins/flow-outs of $q \in \rcal_{\ebo,I/O}\backslash
\Sebstar[\pa \ef] M,$ and let $\fcal=\fcal_I \cup \fcal_O.$
\end{definition}

{\em For brevity, we often use the word `flow-out' to refer to both
  the flow-in and the flow-out.}

One needs some control over the intervals on which normally approaching
\GBB do not hit the boundary of $M$:

\begin{lemma}\label{lemma:uniform-normal}
Suppose $K\subset\ef^\circ$ is compact, $\kcal\subset\hcal_{W,\bo}$
is compact, $\ep_0>0$.
Then there is $\delta_0>0$
such that if
$\gamma:[0,\ep_0]\to\Sbstar M_0$ a
\GBB with lift $\tilde\gamma$,
$\tilde\gamma(0)\in\rcal_{\ebo,\alpha,O}\cap\Sebstar[K]M$ for
some $\alpha\in\kcal$, then
$\tilde\gamma((0,\delta_0))\cap\Sebstar[\pa M]M=\emptyset$.
\end{lemma}

\begin{proof}
First, by Lemma~\ref{lemma:OG-IC-uniform-est-1}
there is a $\delta'_0>0$ such that any \GBB $\gamma$ with
$\gamma(0)\in\kcal$ satisfies $\gamma|_{(0,\delta'_0]}$ disjoint from
$\Sbstar_{W}M_0$.

Suppose now that there is no $\delta_0>0$ as claimed. Then
there exist \GBBs $\gamma_j:[0,\ep_0]\to\Sbstar M_0$,
$p'_j\in \Sebstar[K]M\cap\rcal_{\ebo,\alpha_j,O}$, $\alpha_j\in\kcal$,
and $\delta_j>0$, $\delta_j\to 0$,
such that $\gamma_j(\delta_j)\in\Sbstar_{\pa M_0}M_0$, and the
lift $\tilde\gamma_j$ of $\gamma_j$ satisfies
$\tilde\gamma_j(0)=p'_j$. We may assume that $\delta_j<\ep_0/2$
and $\delta_j<\delta'_0$ for all $j$,
hence $\gamma_j(\delta_j)\notin\Sbstar_{W}M_0$. By passing to a subsequence,
using the compactness of $\kcal$ and of $K$, hence of
$\Sebstar[K]M\cap\rcal_{\ebo,\kcal,O}$,
we may assume that $\{\alpha_j\}$ converges to some $\alpha\in\kcal$, and
$\{p'_j\}$ converges
to some $p\in \Sebstar[K]M\cap\rcal_{\ebo,\alpha,O}$.
Using the continuity of $\tilde\gamma_j$ for each $j$,
we may then choose some $0<\ep_j<\delta_j$ such that
$p_j=\tilde\gamma_j(\ep_j)\to p$ as well; note that $p_j\notin\Sebstar[\ef] M$.
(We introduce $\ep_j$ to shift the argument of $\gamma_j$
by $\ep_j$, namely to ensure that $\gamma_j(.+\ep_j)$ at $s=0$ is
outside $\Sbstar_WM$, so Corollary~\ref{cor:GBB-lifted-conv} is applicable.)
Thus, we can apply
Corollary~\ref{cor:GBB-lifted-conv} to conclude that $\gamma_j(.+\ep_j):
[0,\ep_0/2]\to\Sbstar M_0$ has a subsequence $\gamma_{n_j}$ such
that $\gamma_{n_j}(.+\ep_{n_j})$ converges uniformly to a \GBB $\gamma$, the
lifts $\tilde\gamma_{n_j}(.+\ep_{n_j})$
also converge uniformly to the lift $\tilde\gamma$,
and $\tilde\gamma(0)=p$.
Thus, $\tilde\gamma_{n_j}((\delta_{n_j}-\ep_{n_j})+\ep_{n_j})
\to\tilde\gamma(0)=p$ since $\delta_{n_j}-\ep_{n_j}\to 0$.
As $\tilde\gamma_{n_j}((\delta_{n_j}-\ep_{n_j})+\ep_{n_j})
\in\Sebstar[\pa M\setminus\ef] M$ and $\Sebstar[\pa M\setminus\ef^\circ]M$
is closed, it follows that $p\in\Sebstar[\pa M\setminus\ef^\circ]M$,
contradicting $p\in\Sebstar[K]M$. This proves the lemma.
\end{proof}

\begin{remark}
Another proof could be given that uses the description of the edge
bicharacteristics in \cite{mvw1}, since the \GBB covered are normally
incident.
\end{remark}

\begin{corollary}\label{cor:coisotropic}
Suppose $U\subset\ef^\circ$ is open
with $\bar U\subset\ef^\circ$ compact, $\cU\subset\hcal_{W,\bo}$ is open
with $\bar \cU\subset\hcal_{W,\bo}$ compact.
Then there is $\delta_0>0$ such that
the set $O$ of points $p\in\Sebstar M$ for which there is a \GBB
$\gamma$ with lift $\tilde\gamma$ such that $\tilde\gamma(0)\in
\Sebstar[U]M\cap\rcal_{\ebo,\cU,O}$ and $\gamma(s)=p$ for some
$s\in[0,\delta_0)$ is a $\CI$ coisotropic submanifold
of $\Sebstar M$ transversal to $\Sebstar[\ef]M.$
\end{corollary}

\begin{proof}
By Lemma~\ref{lemma:uniform-normal}, with $K=\bar U$, $\kcal=\bar\cU$,
there is a $\delta_0>0$ as in the lemma, hence the set $O$ consists
of points $p$ for which the \GBB $\gamma$ only meet $\pa M$ at $s=0$,
so (taking into account part (2) of Lemma~\ref{lemma:EGBB-GBB} as well)
$O$ is a subset of the edge flow-out studied in \cite{mvw1} (e.g.\ by
extending the edge metric $g$ smoothly across the boundary hypersurfaces
other than $\ef$). In particular, the
properties of the flow-out of such an open subset being $\CI$,
coisotropic\footnote{\label{footnote:coisotropic}In \cite{mvw1}, being coisotropic is considered as
a property of submanifolds of a symplectic manifold, $\Testar M\setminus o$,
$M$ being
an edge manifold. Conic submanifolds of $\Testar M\setminus o$
can be identified with submanifolds of $\Sestar M$, and conversely,
thus one can talk about submanifolds of $\Sestar M$ being coisotropic.
Alternatively, this notion could be defined using the {\em contact structure}
of $\Sestar M$, but for the sake of simplicity, and due to the role of
symplectic structures in classical microlocal analysis, we did not
follow this route in \cite{mvw1}, necessitating making the connection
via homogeneity here.}
and transversal to $\Sebstar[\ef]M$ follow from
Theorem~4.1 of \cite{mvw1}.
\end{proof}

We now turn to properties of the singular flow-out.

\begin{lemma}\label{lemma:singularflow-out}
The singular flow-out, $\fcal_{\sing},$ is closed in
$\Sebstar[M\backslash \ef] M.$
\end{lemma}
\begin{proof}
Suppose $p_n\in\fcal_{\sing}$, and let $\gamma_n$ be such that the
lift $\tilde\gamma_n$ of $\gamma_n$ satisfies
$\tilde\gamma_n(0)\in\gcal\cap\rcal_{\ebo}$, and $\gamma_n(s_n)=p_n$,
$s_n\in(0,T]$. Suppose that $p_n\to p\in \Sebstar[M\backslash \ef]
M.$ Then there exists a compact subset $K$ of $M$ such that
$\gamma_n(s)\in\Sbstar_K M$ for all $n$ and all $s\in[0,T]$.
By passing to a subsequence we may assume that $s_n\to s$; as
$p\notin\Sebstar[\ef]M$, $s\neq 0$. By passing to yet another
subsequence we may also assume that $\gamma_n(0)\to
q\in\gcal\cap\rcal_{\ebo}$. Let $\ep_n>0$, $\ep_n\to 0$,
so $\gamma_n(\ep_n)\notin\Sebstar[\ef]M$ and $\gamma_n(\ep_n)\to q$. By
Corollary~\ref{cor:GBB-lifted-conv} we conclude that $\gamma_n(.+\ep_n):
[0,T]\to\Sbstar M_0$ has a subsequence $\gamma_{n_j}$ such
that $\gamma_{n_j}(.+\ep_{n_j})$ converges uniformly to a \GBB $\gamma$, the
lifts $\tilde\gamma_{n_j}(.+\ep_{n_j})$
also converge uniformly to the lift $\tilde\gamma$,
and $\tilde\gamma(0)=q$. In particular, as
$\gamma_{n_j}((s_{n_j}-\ep_{n_j})+\ep_{n_j})=\gamma_{n_j}(s_{n_j})=p_{n_j}\to
p$, and $s_{n_j}-\ep_{n_j}\to s$, $\gamma(s)=p$, so
$p\in\fcal_{\sing}$ as claimed.
\end{proof}

\begin{lemma}\label{lemma:uniform-glancing}
Suppose $K\subset\ef^\circ$ is compact, $\kcal\subset\hcal_{W,\bo}$
is compact. Then $K$ has a neighborhood $U$ in $M$ and there is $\ep_0>0$
such that if $\gamma:[0,\ep_0]\to\Sbstar M_0$ is a \GBB with lift
$\tilde\gamma$, $\tilde\gamma(0)\in\rcal_{\ebo,O}\cap\gcal$, $\gamma(0)
\in\kcal$ then
$\tilde\gamma(s)\notin\Sebstar[U]M$ for $s\in(0,\ep_0]$.
\end{lemma}

\begin{proof}
Let $\ep_0>0$ be such that
any \GBB $\gamma$ with
$\gamma(0)\in\kcal$ satisfies $\gamma|_{(0,\ep_0]}$ disjoint from
$\Sbstar_{W}M_0$; such $\ep_0$ exists by
Lemma~\ref{lemma:OG-IC-uniform-est-1}.

Now suppose that no $U$ exists as stated. Then there exist \GBB
$\gamma_n$ and $s_n\in(0,\ep_0]$ such that the lifts $\tilde\gamma_n$
of $\gamma_n$ satisfies
$\tilde\gamma_n(0)\in\rcal_{\ebo,O}\cap\gcal$, $\gamma_n(0)\in\kcal$
and $\pi(\tilde\gamma_n(s_n))\to q$, $q\in K$, where $\pi:\Sebstar M\to M$
is the bundle projection.

By the compactness of $\kcal$ and the compactness
of $\cup_{\alpha\in\kcal}\rcal_{\ebo,O,\alpha}\cap\gcal$
we may pass to a subsequence (which we do not indicate in notation) such that
$\gamma_n(0)$ converges to some $\alpha\in\kcal$ and $\tilde\gamma_n(0)$
converges to some $p\in\rcal_{\ebo,O}\cap\gcal$. We may further
pass to a subsequence such that $s_n\to s_0\in[0,\ep_0]$, and still
further (taking into account the compactness of the fibers of
$\Sebstar M\to M$) that $\tilde\gamma_n(s_n)\to \tilde p\in\Sebstar[K]M$.
Choose\footnote{Again,
we do this so that Corollary~\ref{cor:GBB-lifted-conv} is applicable;
cf.\ the proof of Lemma~\ref{lemma:uniform-normal}.} $\ep_n\in (0,s_n)$
sufficiently small such that $\ep_n\to 0$ and $\gamma_n(\ep_n)\to p$.
By Corollary~\ref{cor:GBB-lifted-conv} $\gamma_n(.+\ep_n)$ has a convergent
subsequence $\gamma_{n_k}$ such that $\gamma_{n_k}(.+\ep_{n_k})$ converge
uniformly to a \GBB $\gamma$ and the lifts
$\tilde\gamma_{n_k}(.+\ep_{n_k})$ converge uniformly to the lift $\tilde\gamma$
and $\tilde\gamma(0)=p$. Thus, $\tilde\gamma_{n_k}(s_{n_k}+\ep_{n_k})
\to\tilde\gamma(s_0)$, so $\tilde\gamma(s_0)=\tilde p\in\Sebstar[K]M$.
But by the definition of $\ep_0$, $\gamma(s_0)\notin \Sbstar_{W}M_0$
if $s_0>0$, while $s_0=0$ is impossible as
$\tilde\gamma(0)=p\in\Sebstar[\pa\ef] M$, while $K\subset \ef^\circ$.
This contradiction shows that the claimed $U$ exists, proving the
lemma.
\end{proof}

\begin{corollary}\label{corollary:o-no-glancing}
Suppose $K\subset\ef^\circ$ is compact, $\kcal\subset\hcal_{W,\bo}$
is compact. Then $K$ has a neighborhood $U$ in $M$ and there is $\ep_0>0$
such that if $o\in U\setminus\ef$ and $\gamma$ is a \GBB with
$\gamma(0)\in\Sbstar_o M_0$ then for $s\in[-\ep_0,0]$, $\gamma(s)\in\kcal$
implies $\gamma$ is normally incident.

In particular, if $q\in\ef^\circ$, $\alpha\in\hcal_{W,\bo}$
and $\gamma_0$ is a \GBB with $\gamma_0(0)=\alpha$ and lift
$\tilde\gamma_0(0)\in\Sebstar[q] M$ then there is $\delta_0>0$ such that
$s\in(0,\delta_0]$,
$\gamma_0(s)\in\Sbstar_o M_0$ implies that
every \GBB $\gamma$ with $\gamma(0)\in\Sbstar_o M_0$,
$\gamma(s)=\alpha$, $s\in[-\ep_0,0]$, is normally
incident.
\end{corollary}

\begin{proof}
Let $U$ and $\ep_0$ be as in Lemma~\ref{lemma:uniform-glancing}.
If $o\in U$, $\gamma$ is a \GBB with
$\gamma(0)\in\Sbstar_o M_0$, $s_0\in[-\ep_0,0]$, $\gamma(s_0)\in\kcal$
and $\gamma$ is not normally incident, then the lift $\tilde\gamma$
of $\gamma$ satisfies $\tilde\gamma(s_0)\in\rcal_{\ebo,O}\cap\gcal$ by
Lemma~\ref{lemma:GBB-eb-lift}. Thus, with $\gamma_0(s)=\gamma(s-s_0)$,
so $\tilde\gamma_0(0)\in \rcal_{\ebo,O}\cap\gcal$, $\gamma_0(0)\in\kcal$,
Lemma~\ref{lemma:uniform-glancing} shows that
$\tilde\gamma_0(s)\notin\Sebstar[U]M$ for $s\in(0,\ep_0]$, contradicting
$\tilde\gamma_0(-s_0)\in\Sbstar_o M_0$.

The second half follows by taking $\kcal=\{\alpha\}$, $K=\{q\}$.
\end{proof}

\subsection{A summary}\label{subsection:summary}

The following table summarizes a number of the most useful facts about
the bundles that we have introduced above.

\

\begin{tabular}{|l||c|c|c|c|}\hline
Manifold & $M_0$ & $M_0$ & $M$ & $M$ \\ \hline\hline
Bundle & b & s & eb & es\\ \hline
Vector fields & $x_i \pa_{ x_i}, \pa_{ y_k}$ & $\pa_{ x_j}, \pa_{ y_k}$ & $x \pa_x, x
\pa_y, z_i'\pa_{z_i'}, \pa_{z''_j}$ & $x \pa_x, x
\pa_y, \pa_{z_i}$
\\ \hline Dual coords & $\xis_j,\etas_j$ & $\xib_j,\etab_j$
 & $\xi, \eta, \zeta', \zeta''$ &
$\xies, \etaes, \zetaes$
 \\
\hline
Char.~ set & $\bSigma$ & $\sSigma$ & $\ebSigma$ & $\esSigma$ \\ \hline
\end{tabular}

\

(We have omitted time coordinates and their duals, as they behave just
like $y$ variables, and the notation follows suit.)

We also employ a number of maps among these structures, the most
common being:
\begin{align*}
\pismooth: T^* M_0 &\to \Tbstar M_0,\\
\pieseb: \Tesstar M &\to \Tebstar M,\\
\dbetaes: \Tesstar_{\ef} M &\to T^*W,\\
\dbetaeb: \Tebstar_{\ef} M &\to T^*W.
\end{align*}
Recall that hats over maps indicate their restrictions to the relevant
characteristic set.

\section{Edge-b calculus}\label{section:edge-b-calc}

Recall from Definition~\ref{def:Vfeb} that
$\Vfeb(M)$ is the space of smooth vector fields that are tangent to all
of $\pa M$ and
tangent to the fibration of $\ef \subset \pa M$ given
by blowdown. Thus, in local coordinates, $\Vfeb(M)$ is spanned over
$\CI(M)$ by the vector fields
\begin{equation}\label{eq:eb-Vfs}
x\pa_x,\ x\pa_t,\ x\pa_y,\ z'_i \pa_{z'_i}, \pa_{z''}
\end{equation}

\begin{definition}
The space $\Diffeb{*}(M)$ is the filtered algebra of operators over $\CI(M)$
generated by $\Vfeb(M).$
\end{definition}
\nomenclature[D]{$\Diffeb{*}$}{Edge-b differential operators}

Recall also that $\Vfeb(M)=\CI(M;\Teb M)$, and $\Tebstar M$ is the
dual bundle of $\Teb M.$ In Appendix~\ref{BEPseuda} the corresponding
pseudodifferential operators are constructed.

\begin{theorem}\label{BEPseud}
There exists a pseudodifferential calculus $\Psieb{*} (M) $ microlocalizing
$\Diffeb{*}(M).$
\end{theorem}
\nomenclature[P]{$\Psieb{*}$}{edge-b pseudodifferential calculus}

The double space $M^2_{\eb}$ on which the kernels are defined is such that
the quotient $x/x'$ of the same boundary defining function on the left or
right factor, lifts to be smooth except near the `old' boundaries at which
the kernels are required to vanish to infinite order. It follows that
$x/x'$ is a multiplier (and divider) on the space of kernels. This
corresponds to the action by conjugation of these defining functions, so it is
possible to define a weighted version of the calculus. Set
$$
\Psieb{m,l} (M) =  x^{-l} \Psieb{m}(M).
$$
\begin{proposition}
$\Psieb{*,*}(M)$ is a bi-filtered calculus.
\end{proposition}

Now, $\Psieb{*}(M)$ has all the properties (I--VII) of \cite[Section~3]{mvw1},
where $\Vf$ in \cite[Section~3]{mvw1} is replaced by $\ebo.$ Since the
multiplier $x/x'$ is identically equal to one on the lifted diagonal, the
symbol is unaffected by this conjugation and hence the principal symbol map
extends to
$$
\sigma_{\ebo,m,l}:\Psieb{m,l}(M)\to x^{-l} S^m_{\hom}(\Tebstar M),
$$
with the standard short exact sequence---see properties (III--IV). There
are edge-b-Sobolev spaces, $\Heb{s}(M),$ defined via the elliptic elements
of $\Psieb{s}(M),$ and on which the elements define bounded maps
$$
A\in\Psieb{m}(M)\Longrightarrow A:\Heb{s}(M)\to\Heb{s-m}(M)
$$
(see property (VII)).

The symbol of the commutator of $A \in \Psieb{m,l}(M)$ and $B
\in\Psieb{m',l'}(M)$ is given by
$$
\sigma_{\ebo, m+m'-1,l+l'}(\imath[A,B] )=
\sH_{\ebo,\sigma_{\ebo,m,l}(A)} (\sigma_{\ebo,m',l'}(B)).
$$
In local coordinates the edge-b Hamilton vector field becomes
\begin{multline}\label{eq:sH-ebo}
\sH_{\ebo,f} = \frac{\pa f}{\pa \xi} x \pa_x - \Big (x \frac{\pa f}{\pa
  x}+\eta \cdot \frac{\pa f}{\pa \eta}\Big) \pa_\xi + x \frac{\pa f}
{\pa \eta} \pa_y + \Big( \frac{\pa f}{\pa \xi} \eta - x \frac{\pa
  f}{\pa y}\Big)\cdot \pa_\eta \\
+ \sum\Big(\frac{\pa
  f}{\pa \zeta_j'} z_j' \pa_{z_j'} - z_j'\frac{\pa f}{\pa z_j'}
\pa_{\zeta_j'} \Big)
+\sum\Big(\frac{\pa
  f}{\pa \zeta_j''} \pa_{z_j''} - \frac{\pa f}{\pa z_j''}
\pa_{\zeta_j''} \Big) .
\end{multline}
In particular,
\begin{equation}\label{eq:weighted-formula}
x^{-k}\sH_{\ebo,x^k a}=k a\pa_\xi+\sH_{\ebo,a}.
\end{equation}
In the space-time setting, where one of the $y$ variables, $t$, is
distinguished (and we still write $y$ for the rest of the base variables),
it is useful to rewrite this using the re-homogenized dual
variables $\etaebh=\eta/|\tau|$, $\xiebh=\xi/|\tau|$, $\zetaebh=\zeta/|\tau|$,
$\sigma=|\tau|^{-1}$, valid near $\ebSigma$, this becomes
\begin{multline}\label{eq:sH-ebo-hom}
\sigma^{-1}\sH_{\ebo,f} = \frac{\pa f}{\pa \xiebh} x \pa_x
- \Big (x \frac{\pa f}{\pa
  x}-\sigma\frac{\pa f}{\pa\sigma}
-\zetaebh \cdot \frac{\pa f}{\pa \zetaebh}\Big) \pa_{\xiebh}
+ x \frac{\pa f}{\pa \etaebh} \pa_y\\
-x\Big(\sigma\frac{\pa f}{\pa\sigma}
+\etaebh\cdot\frac{\pa f}{\pa\etaebh}+\xiebh\frac{\pa f}{\pa\xiebh}
+\zetaebh\cdot\frac{\pa f}{\pa\zetaebh}\Big)\pa_t
-\frac{\pa f}{\pa \xiebh}
\big(\sigma\pa_\sigma+\zetaebh\cdot\pa_{\zetaebh}\big)\\
+x\frac{\pa f}{\pa t}\cdot\big(\sigma\pa_\sigma+\etaebh\cdot\pa_{\etaebh}
+\xiebh\pa_{\xiebh}+\zetaebh\cdot\pa_{\zetaebh}\big)
-x\frac{\pa f}{\pa y}\cdot\pa_{\etaebh}\\
+ \sum\Big(\frac{\pa
  f}{\pa \zetaebh_j'} z_j' \pa_{z_j'} - z_j'\frac{\pa f}{\pa z_j'}
\pa_{\zetaebh_j'} \Big)
+\sum\Big(\frac{\pa
  f}{\pa \zetaebh_j''} \pa_{z_j''} - \frac{\pa f}{\pa z_j''}
\pa_{\zetaebh_j''} \Big) .
\end{multline}
This is tangent to the fibers of
$\dbetaeb:\ebSigma\cap\Sebstar[\ef]M\to S^*W$, in fact to its natural
extension to a neighborhood of $\ebSigma\cap\Sebstar[\ef]M$ in
$\Sebstar[\ef]M$, so if $b\in\CI(\Sebstar[\ef]M)$ with $b|_{\Sebstar[\ef]M}$
constant along the fibers of this extension, then
$\sigma^{\mu-1}\sH_{\ebo,f}b\in x\CI(\Sebstar M)$ for $f$ homogeneous
degree $\mu$.

The fact that the operators are defined by kernels which are conormal means
that there is an operator wave front set $\ebWF'$ for the $\ebo$-calculus,
i.e.\ for $A\in\Psieb{*}(M)$, $\ebWF'(A)\subset\Sebstar M,$ with the
properties (A)--(F) of \cite[Section~3]{mvw1}, so in particular algebraic
operations are microlocal, see properties (A)--(B), and there are microlocal
parametrices at points at which the principal symbol is elliptic (see
property (E)). These parametrices have error terms with which are smooth on
the double space, but they are not compact.  We will abuse notation by
writing 
$$
\WF' = \ebWF'
$$
when there is no possibility of confusion (i.e., usually).

As is the case for the b-calculus, for each boundary face $\{z'_j=0\}$ we may define a
\emph{normal operator} $N_j;$ in the special case of a differential
operator in $\Diffeb{*}(M),$ written in the form
$$
P=\sum P_k (z'_j D_{z'_j})^k
$$
where $P_k \in \Diffeb{*}(M)$ have no factors of $(z'_j D_{z'_j})^k$
in terms of the local basis \eqref{eq:eb-Vfs},
$N_j(P)$ is the family of operators on the face $z'=0$ given by
$$
N_j(P)(\zeta'_j) = \sum (P_k)|_{z'_j=0} (\zeta'_j)^k.
$$
This map extends to a homomorphism on $\Psieb{*}(M),$ and its
vanishing is the obstruction to an operator lying in $z'_j
\Psieb{*}(M),$ i.e., enjoying extra vanishing at the boundary face in question.
 (See \cite[Section~3]{mvw1} for a brief discussion of
    normal operators and \cite{MR93d:58152} for further
    details.)

    As a consequence of the normal operator homomorphisms,
    $\Psieb{*}(M)$ has the additional property that the {\em radial
      vector fields} $V_j$ for all boundary hypersurfaces
    $\{z'_j=0\},$ i.e., all boundary hypersurfaces other than $\ef$, $[A,V_j]\in z'_j\Psieb{m}(M)$
    if $A\in\Psieb{m}(M)$, i.e., there is a gain of $z'_j$ over the
    \emph{a priori} order.  In local coordinates a radial vector field
    for $z'_j=0$ is given by $z'_j\,\pa_{z'_j};$ $V_j$ being a radial
    vector field for $z'_j=0$ means that $V_j-z'_j\,\pa_{z'_j} \in
    z'_j\Vfeb(M)$. This latter requirement can easily be seen to be
    defined independently of choices of coordinate systems.  The fact
    that the normal operator of $z'_j \pa_{z'_j}$ is a scalar then
    proves the assertion.

\section{Differential-pseudodifferential operators}\label{section:Diff-Psi}

\subsection{The calculus}

We start by defining an algebra of operators which includes $\Box.$ First,
recall that $\Vfes(M)$ is the Lie algebra of vector fields that are tangent
to the front face and to the fibers of the blow down map restricted to the
front face, $\beta|_{\ef}:\ef\to W$ (but are not required to be tangent to
other boundary faces). Thus, elements $V$ of $\Vfes(M)$ define operators
$V:\dCI(M)\to\dCI(M)$ and also $V:\CI(M)\to\CI(M).$

\begin{definition} Let $\Diffes{}(M)$ be the filtered algebra of operators
  (acting either on $\dCI(M)$ or $\CI(M)$) over $\CI(M)$ generated by
  $\Vfes(M)$.

We also let $\Diffes{k,l}(M)=x^{-l}\Diffes{k}(M)$; this is an algebra
of operators acting on $\dCI(M)$, and also on the space of functions
classical conormal to $\ef$, $\cup_{s\in\RR} x^{-s}\CI(M)$.
\end{definition}

\begin{remark}\label{remark:Diffes-def}
Note that the possibility of the appearance of boundary terms requires care
to be exercised with adjoints, as opposed to formal adjoints. See for instance
Lemma~\ref{lemma:Diffes-adjoints}.

We also remark that $\Diffes{k}(M)$, hence $x^{-l}\Diffes{k}(M)$, is closed
under conjugation by $x^{-r}$ where $x$ is a defining function for $W.$
This follows from the fact that $\Diffes{1}(M)$ is so closed; the key
property is that
$$
x^r (x\pa_x)x^{-r}=(x\pa_x)-r\in\Diffes{1}(M).
$$
\end{remark}

We will require, for commutator arguments that involve interaction of
singularities with $\pa M \backslash \ef,$ a calculus of mixed
differential-pseudodifferential operators, mixing edge-b-pseudodifferential
operators with these (more singular) edge-smooth differential operators.

\begin{definition}
Let
$$\Diffes{k} \Psieb{m} (M) = \Big\{\sum A_j B_j: A_j \in
\Diffes{k}(M),\ B_j \in \Psieb{m}(M)\Big\}$$
\end{definition}
\nomenclature[D]{$\Diffes{k} \Psieb{m}$}{edge-b pseudodifferential,
  edge-smooth differential calculus}

\begin{proposition}\label{prop:DiffPsi-alg}
$\bigcup_{k,m}\Diffes{k} \Psieb{m} (M)$ is a filtered $\CI(M)$-module,
and an algebra under composition; it is commutative to top $\ebo$-order, i.e.\ for
$P\in\Diffes{k} \Psieb{m} (M),$ $Q\in \Diffes{k'} \Psieb{m'} (M),$
$$
[P,Q]\in\Diffes{k+k'}\Psieb{m+m'-1}(M).
$$
\end{proposition}

The key is the following lemma.

\begin{lemma}\label{lemma:Psieb-Vfes-comm}
If $A\in\Psieb{m}(M)$ and $Q\in\Vfes(M)$, then 
\begin{equation}
[A,Q]=\sum Q_j A_j+B,\ [A,Q]=\sum A'_j Q'_j +B'
\label{30.8.2006.59}\end{equation}
where $B,$ $B'\in\Psieb{m}(M),$ $A_j,$ $A_j'\in\Psieb{m-1}(M)$ and $Q_j,$
$Q'_j\in\Vfes(M).$ 
\end{lemma}

\begin{proof} As both $\Vfes(M)$ and $\Psieb{m}(M)$ are $\CI(M)$-modules,
  we can use a partition of unity, and it suffices to work locally and with
  a spanning set of vector fields. Since $xD_x,$ $xD_{y_j},$
  $D_{z''_j}\in\Vfeb(M),$ the conclusion is automatic for $Q$ chosen from
  among these vector fields since then $B=[A,Q]\in\Psieb{m}(M).$ Thus it
  only remains to consider the $Q=D_{z'_j}$ where $z'_j$ is a defining
  function for one of the other boundary faces. Then for $\tilde Q=z'_jQ=z'_j
  D_{z'_j}\in\Diffeb{1}(M),$ $[A,\tilde Q]\in\Psieb{m}(M)$). The normal
  operator at $z'_j=0$ satisfies $N_j([A,\tilde Q])=[N_j(A),N_j(\tilde
    Q)],$ and $N_j(\tilde Q)$ is scalar, and hence commutes with
    $N_j(A).$
  Thus $N_j([A,\tilde Q])=0,$ so $[A,\tilde Q]\in z'_j\Psieb{m}(M).$
  Consequently,
\begin{equation}\label{eq:Dzp-edge-b-commutator}
-[A,Q]=[Q,A]=(z'_j)^{-1}[\tilde Q,A]+([(z'_j)^{-1},A]z'_j)
(z'_j)^{-1}\tilde Q,
\end{equation}
with the first term on the right hand side in $\Psieb{m}(M)$,
the second of the form $\tilde A Q$, $\tilde A\in\Psieb{m-1}(M)$. This
proves the first half of the lemma. The other part is similar.
\end{proof}

\begin{proof}[Proof of Proposition~\ref{prop:DiffPsi-alg}]
The algebra properties follow immediately from the lemma. It only remains 
to verify the leading order commutativity.

As the bracket is a derivation in each argument, it suffices to
consider $P,$ $Q$ lying in either $\Vfes(M)$ or $\Psieb{*}(M).$  If both
operators are in $\Psieb{*}(M),$ the result follows from the symbol
calculus. If $P,$ $Q \in \Vfes(M),$ we have $[P,Q] =R\in \Vfes (M).$  We
need to write $R$ as a sum of elements of $\Diffes{2} (M)$ times elements of
$\Psieb{-1}(M).$  To this end, let $\Lambda$ be an elliptic element of
$\Psieb{2}(M)$ given by a sum of square of vector fields in
$\Vfeb(M),$ e.g.\ in local coordinates
$$
\Lambda = (x\pa_x)^2+(x\pa_t)^2+\sum (x\pa_y)^2+\sum(z'_i \pa_{z'_i})^2+\sum \pa_{z''_j}^2.
$$
We write $\Lambda=\sum V_j^2$ for brevity.
Let $\Upsilon \in \Psieb{-2}(M)$ be an elliptic parametrix for
$\Lambda.$  Then we may write
$$
\Id = \sum V_j (V_j\Upsilon)+E,
$$
with $E \in \Psieb{-\infty}(M).$
Now since $\Vfeb(M)\subset\Vfes(M),$ we certainly have $V_j \in \Vfes
(M)$ for each $j,$ hence $R V_j \in \Diffes{2}(M).$  Moreover $V_j
\Upsilon \in \Psieb{-1}(M).$  Thus, 
$$
R=\sum (R V_j) (V_j\Upsilon)+ RE,
$$
and we have shown that $R \in \Diffes{2}\Psieb{-1}(M).$

Finally, if $P \in \Vfes(M)$ and $Q \in \Psieb{m}(M)$ (or vice-versa) then
using the lemma (and its notation) we may write
$$
[P,Q] =\sum Q_j A_j +B.
$$
Using the same method as above to write $B =\sum R V_j (V_j \Upsilon )
+ BE$ we find that $[P,Q] \in \Diffes{1} \Psieb{m-1}(M).$
\end{proof}

The above proof also yields the following useful consequence.
\begin{lemma}
For all $m,$ $l \in \RR,$ and $k \in \NN,$
$$
\Diffes{m} \Psieb{l}(M) \subset \Diffes{m+k}\Psieb{l-k}(M).
$$
\end{lemma}

We note the following consequence of \eqref{eq:Dzp-edge-b-commutator}:

\begin{lemma}\label{lemma:dzcomm}
Let $A \in \Psieb{m}(M),$ $a=\sigma_{\ebo,m}(A)$.  Then
$$
\imath[x^{-1}D_{z'_j}, A] = A_1 x^{-1} D_{z'_j} + x^{-1} A_0
$$
where $A_0\in\Psieb{m}(M)$, $A_1\in\Psieb{m-1}(M)$,
$$
\sigma_{\ebo,m}(A_0)=\frac{\pa a}{\pa z'_j},
\quad \sigma_{\ebo,m-1}(A_1)=\frac{\pa a}{\pa
  \zeta'_j} + \frac{\pa a}{\pa \xi}.
$$
\end{lemma}
\noindent
Note that this is exactly what one would expect from computation
at the level of edge-b symbols: the Hamilton vector field of $\zeta'_i/(x
z'_i)$ is
$$
(\zeta'_i/(x z'_i)) (\pa_{\zeta'_i}+\pa_\xi) + x^{-1} \pa_{z'_i}.
$$

\begin{proof}
This follows immediately from writing
$$
[x^{-1}D_{z'_j},A]=[x^{-1},A]D_{z'_j}+x^{-1}[D_{z'_j},A].
$$
We then use
\eqref{eq:Dzp-edge-b-commutator} together with the
following principal symbol calculations in $\Psieb{*}(M)$, see
\eqref{eq:sH-ebo}:
\begin{align*}
&\imath\sigma_{\ebo,m}([Q',A])=z'_j\pa_{z'_j} a,\\ &\imath\sigma_{\ebo,m-1}([(z'_j)^{-1},A]z'_j)=\pa_{\zeta'_j} a,\\
&\imath\sigma_{\ebo,m-1}([x^{-1},A])=x^{-1}\pa_{\xi} a,
\end{align*}
as well as $[x^l\Psieb{k}(M),x^{l'}\Psieb{k'}(M)]\subset x^{l+l'}
\Psieb{k+k'-1}(M)$, which allows one to exchange factors after the previous
steps without affecting the computed principal symbols.
\end{proof}

We now define the edge-smooth Sobolev spaces.   It is with respect to
these base spaces that we will measure regularity in proving propagation of
edge-b wavefront set.

\begin{definition}
  For $s\geq 0$ integer,
$$
\Hes{s,l-(f+1)/2}(M)=\{u\in x^l L^2_g(M):\ A\in\Diffes{s}(M)
\Rightarrow Au\in x^l L^2_g(M)\}.
$$
\nomenclature[H]{$\Hes{s,l}$}{edge-smooth Sobolev space}
The norm in $\Hes{s,l-(f+1)/2}(M)$, up to
  equivalence, is defined using any finite number of generators $A_j$
  for the finitely generated $\CI(M)$-module $\Diffes{s}(M)$ by
$$
\|u\|_{\Hes{s,l-(f+1)/2}(M)}
=\left(\sum_j \|x^{-l}A_j u\|_{L^2_g(M)}^2\right)^{1/2}.
$$

The space $\Hesz{s,l-(f+1)/2}(M)$ is the closure of $\dCI(M)$ in
$\Hes{s,l-(f+1)/2}(M)$.
\nomenclature[H]{$\Hesz{s,l}$}{For $s\geq 0$,
closure of $\dCI(M)$ in edge-smooth Sobolev space}
\end{definition}

\begin{remark}\label{remark:Hes}
The orders above are chosen so that setting $s=0$, $l=0$, we obtain $L^2_g(M)=
\Hes{0,-(f+1)/2}(M).$ Thus $x^{(f+1)/2}L^2_g(M)=L^2(M,x^{-(f+1)}\,dg)$
is the $L^2$-space corresponding to densities that are smooth up
to all boundary hypersurfaces of $M$ except $\ef$, and that are
b-densities at the interior of $\ef$, meaning that $x(x^{-(f+1)}\,dg)$
is actually a smooth non-degenerate density on $M$. This convention
keeps the weights consistent with \cite{mvw1}.

Note also that the subspace $\cC$ of $\CI(M)$ given by
\begin{equation}\label{cC}
\cC = x^\infty \CI(M)
\end{equation}
\nomenclature[C]{$\cC$}{Smooth functions on $M$ vanishing to infinite order at
  blown-up edge}
is dense in $\Hes{s,l-(f+1)/2}(M)$ for
all $s$ and $l$; one could even require supports disjoint from
$\ef$. Thus, the difference between $\Hesz{s,l-(f+1)/2}(M)$
and $\Hes{s,l-(f+1)/2}(M)$ corresponds to the behavior at the boundary
hypersurfaces
of $M$ other than $\ef$, i.e.\ those arising from the boundary hypersurfaces
of $M_0$, where the boundary conditions are imposed. Thus, this difference
is similar to the difference between $H^s(\Omega)$ and $H^s_0(\Omega)$ for
domains $\Omega$ with smooth boundary in a manifold.
\end{remark}

The boundedness of $\Psieb{0}(M)$ on $\Hes{1,1-(f+1)/2}(M)$ is an
immediate consequence of the commutation property in
Lemma~\ref{lemma:Psieb-Vfes-comm}.

\begin{theorem}\label{thm:Psieb-0-bd}
$\Psieb{0}(M)$ is bounded on both $\Hes{1,1-(f+1)/2}(M)$ and
on the closed subspace $\Hesz{1,1-(f+1)/2}(M)$.
\end{theorem}

\begin{remark}
The more general case of $\Hes{1,l-(f+1)/2}(M)$ with
arbitrary $l$ follows from the case of $l=1$ using
$x^{-l}A x^l\in \Psieb{0}(M)$ for $A\in\Psieb{0}(M)$.

In fact, reduction to $l=0$ would make the proof below even more
transparent.

The case of $\Hes{s,l-(f+1)/2}(M)$ can be proved similarly, but we do not
need this here.
\end{remark}

\begin{proof}
As $\Psieb{0}(M):\dCI(M)\to\dCI(M)$, the second statement follows
from the first and the definition of
$\Hesz{1,1-(f+1)/2}(M)$.

As above, let $\cC$ be the subspace of $\CI(M)$
consisting of functions vanishing
to infinite order at $\ef,$ which is thus dense in $\Hes{1,1-(f+1)/2}(M)$.
Let $A\in\Psieb{0}(M)$.
As $\Psieb{0}(M):\cC\to\cC$, and $A$ is bounded on $L^2_g(M)$,
one merely needs to check
that for $Q\in\Diffes{1}(M)$ there exists $C>0$ such that for $u\in\cC$,
\begin{equation*}
\|x^{-1}Q Au\|_{L^2_g}\leq C\|u\|_{\Hes{1,1-(f+1)/2}(M)}.
\end{equation*}
But
$$
x^{-1}QA u=([x^{-1},A]x) (x^{-1}Qu)+x^{-1}[Q,A]u+A(x^{-1}Qu).
$$
By Lemma~\ref{lemma:Psieb-Vfes-comm}, $[Q,A]=\sum A_jQ_j+B$,
$B\in\Psieb{0}(M)$, $A_j\in\Psieb{-1}(M)$, $Q_j\in \Diffes{1}(M)$,
hence $x^{-1}[Q,A]=\sum (x^{-1}A_j x) (x^{-1} Q_j) +(x^{-1}Bx) x^{-1}$,
\begin{align*}
x^{-1}QA u=([x^{-1},A]x) (x^{-1}Qu)&+\sum (x^{-1}A_j x) (x^{-1} Q_ju)\\
&+(x^{-1}Bx) (x^{-1}u)+A(x^{-1}Qu),
\end{align*}
so the desired conclusion follows from
$$
\|x^{-1}Qu\|_{L^2_g(M)},\|x^{-1}Q_ju\|_{L^2_g(M)},\|x^{-1}u\|_{L^2_g(M)}
\leq C\|u\|_{\Hes{1,1-(f+1)/2}(M)},\ u\in\cC,
$$
and additionally 
$[x^{-1},A]x, x^{-1}A_j x\in\Psieb{-1}(M)\subset\Psieb{0}(M)$
(which are thus
bounded on $L^2_g(M)$, just as $A,x^{-1}Bx\in\Psieb{0}(M)$ are).
\end{proof}

We can now define the $\ebo$-wave front set relative to a given
Hilbert (or even Banach) space, which in practice will be either the
Dirichlet form domain or a weighted edge-smooth Sobolev space serving
as a stand-in for the Neumann form domain.  We also define the
relevant Sobolev spaces with respect to which these wavefronts sets
measure regularity.  For future reference, we also include the
analogous definitions with respect to the b-calculus.

\begin{definition}\label{def:rel-eb-Sob}
 Let $\hilbert\subset\CmI(M)$ denote a Hilbert space on which, for
 each $K\subset M$ compact,
operators in $\Psieb{0}(M)$ with Schwartz kernel supported in $K\times
K$ are bounded, with the operator norm of
$\Op (a)$ depending on $K$ and a fixed seminorm of $a.$ Let $\hilbert_\loc$
consist of distributions $u$ such that $\phi u\in\hilbert$ for all $\phi\in\CI_c(M)$.

For $m\geq 0,$ $r \leq 0,$ let
$$
H_{\ebo,\hilbert,\loc}^{m,r}(M)=\big\{u \in \hilbert_\loc: Au \in \hilbert_\loc \text{ for
  all } A \in \Psieb{m,r}(M)\big\}.
$$ 
\nomenclature[H]{$H_{\ebo,\hilbert}$}{edge-b Sobolev space relative to
a Hilbert space $\hilbert$}

Let $q\in\Sebstar M$, $u \in \hilbert_\loc .$
For $m\geq 0,r\leq 0$, we say that $q\notin\WFebX^{m,r}(u)$ if there exists
$A\in\Psieb{m,r}(M)$ elliptic at $q$ such that $Au\in\hilbert_\loc .$  We define
$q \notin \WFebX^{\infty,r}(u)$ if there exists
$A\in\Psieb{0,r}(M)$ elliptic at $q$ such that $Au \in
H_{\ebo,\hilbert,\loc}^{\infty,0}(M).$
\nomenclature[WF]{$\WFebX^{m,r}(u)$}{edge-b wave front set relative to a Hilbert
space $\hilbert$}

\end{definition}
There is an inclusion 
$$
\WFebX^{m,r}u \subset \WFebX^{m',r'} u
$$
if
$$
m\leq m',\ r\leq r'.
$$

\begin{remark}
We could alter this definition to allow $u$ a priori to lie in the
larger space
$$
\sum A_j (\hilbert)
$$
with $A_j \in \Psieb{\infty, 0}(M);$ this would allow us to give a
non-trivial definition of $\WFebX^{m,r} u$ even for $m<0.$

The restriction to $r \leq 0$ is more serious: operators in
$\Psieb{*,0}(M)$ would in general fail to be microlocal with respect to a
putative $\WFebX^{m,r}(M)$ with $r>0,$ simply because such operators
would fail to be bounded on $\hilbert.$

Note also that if $\hilbert'$ is a closed subspace of $\hilbert$,
with the induced norm, and if elements of $\Psieb{0}(M)$ restrict to
(necessarily bounded) maps $\hilbert'\to\hilbert'$, then for $u\in\hilbert'$,
\begin{equation}\label{eq:subspace-WF}
\WFebX^{m,r}(u)=\WFebXp^{m,r}(u).
\end{equation}
In particular, this holds with $\hilbert=\Hes{k,l}(M)$ and $\hilbert'
=\Hesz{k,l}(M)$.
\end{remark}

The $\ebo$-wave front set captures $\ebo$-regularity:

\begin{lemma}\label{lemma:WF-to-reg}
If $u\in\hilbert$, $r\leq 0$, $m\geq 0$ and
$\WFebX^{m,r}(u)=\emptyset,$ then
$u\in H_{\ebo,\hilbert,\loc}^{m,r}(M)$, i.e.\ for all $A\in\Psieb{m,r}(M)$
with compactly supported kernel,
$Au\in\hilbert$.
\end{lemma}

\begin{proof}
  This is a standard argument (see e.g.\ \cite[Lemma~3.10]{Vasy5}):
  For each $q\in \Sebstar M$ there is $B_q\in\Psieb{m,r}(M)$ elliptic
  at $q$ such that $B_q u\in\hilbert$.  By compactness, $\Sebstar M$
  can be covered by $\bigcup_j \liptic (B_{q_j})$ for finitely many
  points $q_j.$ Now choose $Q\in\Psieb{-m,-r}(M)$ elliptic,
  and set $B=\sum QB^*_{q_j}B_{q_j}.$  Then $B$ is elliptic and
 $Bu\in\hilbert$. As
  $B$ has a parametrix $G\in\Psieb{-m,-r}(M)$ with
  $GB-\Id\in\Psieb{-\infty,0}(M)$,
$$
Au=AG(Bu)+(A(\Id-GB))u,\text{ and } A(\Id-GB)\in\Psieb{-\infty,r}(M)\subset\Psieb{0,0}(M),
$$
shows the claim.
\end{proof}

Pseudodifferential operators are microlocal, as follows by a standard
argument:

\begin{lemma}\label{lemma:microlocal}(Microlocality)
If $B\in\Psieb{s,l}(M)$ then for $r,r-l\leq 0$,
$u\in\hilbert$,
$$
\WFebX^{m-s,r-l}(Bu)\subset\WF'(B)\cap\WFebX^{m,r}(u).
$$
In particular, if $\WF'(B)\cap\WFebX^{m,r}(u)=\emptyset$ then
$Bu\in H_{\ebo,\hilbert,\loc}^{m-s,r-l}(M)$.
\end{lemma}

\begin{proof}
We assume $m\geq s$ and $m\geq 0$ in accordance with the definition above;
but the general case is treated easily by the preceeding remarks.

If $q\in\Sebstar M$, $q\notin\WF'(B)$, let $A\in\Psieb{m-s,r-l}(M)$
be elliptic at $q$ such that $\WF'(A)\cap\WF'(B)=\emptyset.$ Thus
$AB\in\Psieb{-\infty,r}(M)\subset\Psieb{0,0}(M)$, hence $AB u\in\hilbert$,
so $q\notin\WFebX^{m-s,r-l}(Bu)$.  (Note that we used $r\leq 0$
here.)

On the other hand, if $q\in\Sebstar M$, $q\notin\WFebX^{m,r}(u)$, then
there is $C\in\Psieb{m,r}(M)$ elliptic at $q$ such that $Cu\in\hilbert$.
Let $G$ be a microlocal parametrix for $C$, so $G\in\Psieb{-m,-r}(M)$,
and $q\notin\WF'(GC-\Id)$. Let $A\in\Psieb{m-s,r-l}$ be elliptic at $q$
and such that $\WF'(A)\cap\WF'(GC-\Id)=\emptyset$. Then
$$
ABu=ABGC u+AB(\Id-GC)u,
$$
and
$AB(\Id-GC)\in\Psieb{-\infty,r}(M)\subset\Psieb{0,0}(M)$
since $\WF'(A)\cap\WF'(\Id-GC)=\emptyset$,
so the second term on the right hand side is in $\hilbert$.
On the other hand, $Cu\in\hilbert$ and $ABG\in\Psieb{0,0}(M)$, so
$ABG(Cu)\in\hilbert$ as well, proving the wave front set containment.

The final claim follows immediately from this and Lemma~\ref{lemma:WF-to-reg}.
\end{proof}

There is a quantitative version of the lemma as well. Since the proof
is similar, cf.\ \cite[Lemma~3.13]{Vasy5}, we omit it.

\begin{lemma}\label{lemma:uniformly-microlocal}
Suppose that $K\subset\Sebstar M$ is compact, $U$ is a neighborhood of $K$,
$\tilde K\subset M$ compact.

Let $Q\in\Psieb{s,l}(M)$ elliptic on $K$ with $\WF'(Q)\subset U$ and
the Schwartz kernel of $Q$ supported in $\tilde K\times\tilde K$.

If $\cB$ is a bounded family in $\Psieb{s,l}(M)$ with
Schwartz kernel supported in $\tilde K\times\tilde K$ and
with $\WF'(\cB)\subset K$
then for $r,r-l\leq 0$,
there is $C>0$ such that for all
$u\in\hilbert$ with $\WFebX^{m,r}(u)\cap U=\emptyset$,
$$
\|Bu\|_{\hilbert}\leq C(\|u\|_{\hilbert}+\|Qu\|_{\hilbert}) \text{ for
  all } B \in \cB.
$$
\end{lemma}

\subsection{Dual spaces and adjoints}
We now discuss the dual spaces. {\em For simplicity of notation
we suppress the $\loc$ and $\comp$ subscripts for
the local spaces and compact supports.} In principle this should only
be done if $M$ is compact, but, as
this aspect of the material is standard, we feel that this would only
distract from the new aspects. See for instance
\cite[Section~3]{Vasy5}
for a treatment where all the compact supports and local spaces are spelled
out in full detail.

Recall now from Appendix~\ref{appendix:functional}
that if $\hilbert$ is a dense subspace of $L^2_g$, equipped
with an inner product $\langle.,.\rangle_{\hilbert}$
in which it is a Hilbert space and the inclusion map $\iota$ into $L^2_g$ is
continuous, then there is a linear injective inclusion map
$L^2_g\to\hilbert^*$ with dense range,
namely
$$
\iota^*=\iota^\dagger\circ j\circ c:L^2_g\to\hilbert^*
$$
where $\iota^\dagger:(L^2_g(M))^*\to\hilbert^*$ is the standard adjoint map,
$j:L^2_g(M)\to L^2_g(M)^*$ the standard conjugate-linear identification of
a Hilbert space with its dual, and $c$ is pointwise
complex conjugation of functions. In particular, one has the chain
of inclusions $\hilbert\subset L^2_g(M)\subset\hilbert^*$, and one considers
$\hilbert^*$, together with these inclusions, as the dual space of $\hilbert$
with respect to $L^2_g(M)$.

\begin{definition}
For $s\geq 0$,
the dual space of $\Hes{s,l}(M)$ with respect to the $L^2_g(M)$ inner product
is denoted $\dHes{-s,-l-(f+1)}(M)$.

For $s\geq 0$,
the dual space of the closed subspace $$\Hesz{s,l}(M)\equiv\dHes{s,l}(M)$$
is denoted $\Hes{-s,-l-(f+1)}(M);$ this is a quotient space
of $\dHes{-s,-l-(f+1)}(M)$. We denote the quotient map by
$$
\rho:\dHes{-s,-l-(f+1)}(M)\to\Hes{-s,-l-(f+1)}(M).
$$
\end{definition}
\nomenclature[H]{$\dHes{s,l}$}{Edge-smooth Sobolev space of supported distributions}

The standard characterization of these distribution spaces, by doubling
across all boundary faces of $M$ except $\ef$, is still valid---see
\cite[Appendix~B.2]{Hormander3} and \cite[\S3]{Vasy5}.
Note that for all $s,l$, elements of $\dHes{-s,-l-(f+1)}(M)$ are
in particular
continuous linear functionals on $\cC$, which in turn is a dense subspace
of $\Hes{s,l}(M)$. In particular, they can be identified as elements
of the dual $\cC'$ of $\cC$. Thus, were it not for the infinite order vanishing
imposed at $\ef$ for elements of $\cC$, these would be ``supported
distributions''---hence the notation with the dot.
On the other hand, elements of $\Hes{-s,-l-(f+1)}(M)$ are only
continuous linear functionals on $\dCI(M)$ (rather than on $\cC$),
though by the Hahn-Banach
theorem can be extended to continuous linear functionals on $\cC$ in
a non-unique fashion.

If $P\in x^{-r}\Diffes{k}(M)$, then it defines a continuous linear
map
$$
P:\Hes{k,l}(M)\to \Hes{0,l-r}(M).
$$
Thus, its Banach space adjoint
(with respect to the {\em sesquilinear} dual pairing) is
a map
\begin{equation}\begin{split}\label{eq:Banach-adjoint}
&P^*:(\Hes{0,l-r}(M))^*=\dHes{0,r-l-(f+1)}(M)\to (\Hes{k,l}(M))^*
=\dHes{-k,-l-(f+1)}(M),\\
&\qquad\qquad
\langle P^* u,v\rangle=\langle u,Pv\rangle,\ u\in\dHes{0,r-l-(f+1)}(M),
\ v\in \Hes{k,l}(M).
\end{split}\end{equation}
In principle, $P^*$ depends on $l$ and $r$. However, the density of $\cC$
in these spaces shows that in fact it does not.

There is an important distinction here between considering $P^*$ as stated,
or as composed with the quotient map, $\rho\circ P^*$.

\begin{lemma}\label{lemma:Diffes-adjoints}
Suppose that $P\in x^{-r}\Diffes{k}(M)$. Then there exists a unique
$Q\in x^{-r}\Diffes{k}(M)$ such that $\rho\circ P^*=Q$.
However, in general, acting on $\cC$, $P^*\neq Q$.

If, on the other hand, $P\in x^{-r}\Diffeb{k}(M)$, then there exists a unique
$Q\in x^{-r}\Diffeb{k}(M)$ such that $P^*=Q$.
\end{lemma}

\begin{proof}
For the first part we integrate by parts in $\langle u,Pv\rangle$
using $u,v\in\dCI(M)$ (noting that
$\dCI(M)$ is dense in $\Hesz{k,l}(M)$).
Thus, one can localize. In local coordinates the density
is $dg=J x^f\,dx\,dy\,dz$, with $J\in\CI(M)$, so for a vector field
$V\in\Vfes(M)$, noting the lack of boundary terms due to the infinite
order vanishing of $u$ and $v$, one has (with the first equality being
the definition of $V^*$)
\begin{equation*}\begin{split}
\langle V^*u,v\rangle&=\langle u,Vv\rangle=\int u \,\overline{Vv} \,J x^f
\,dx\,dy\,dz\\
&=\int \big(J^{-1} x^{-f}V^\dagger (Jx^f u)\big) \,\overline{v}
\,J x^f
\,dx\,dy\,dz,
\end{split}\end{equation*}
where for $V=a(xD_x)+\sum b_j (xD_{y_j})+\sum c_j D_{z_j}$, with
$a,b_j,c_j\in\CI(M)$,
$$
V^\dagger=D_x x \overline{a}+\sum x D_{y_j} \overline{b_j}
+\sum D_{z_j} \overline{c_j}\in\Diffes{1}(M).
$$
Conjugation of $V^\dagger$ by $Jx^f$ still yields an operator
in $\Diffes{1}(M)$.
This shows the existence (and uniqueness!) of the
desired $Q$, namely
$$
Q=J^{-1} x^{-f}V^\dagger Jx^f.
$$
The density of
$\dCI(M)$ in $\Hes{0,r-l-(f+1)}(M)$ now finishes the proof of the first claim
when $P=V\in\Vfes(M)$,
since this means that $\langle P^* u,v\rangle=\langle Qu,v\rangle$ for
all $u\in\Hes{0,r-l-(f+1)}(M)$, $v\in\Hesz{k,l}(M)$. The general case follows
by induction and adding weight factors
(recalling Remark~\ref{remark:Diffes-def}).

The same calculation works even if $u,v\in\cC$ provided that $V\in\Vfeb(M)$:
in this case $D_{z_j}$ is replaced by vector fields tangent to all boundary
faces, i.e.\ $D_{z''_j}$ and $z'_j D_{z'_j}$, for which there are no
boundary terms---in the second case due to the vanishing
factor $z'_j$. This proves the claim if $P\in x^{-r}\Diffeb{k}(M)$.

Note, however,
that this calculation breaks down if $u,v\in\cC$ and $V\in\Vfes(M)$:
the $D_{z'_j}$ terms
gives rise to non-vanishing boundary terms in general, namely
$$
\sum_j\int_{H_j}(-\imath)\overline{c_j}\,u\,\overline{v}\,Jx^f\,dx\,dy\,d\hat z_j
=\sum_j\langle -\imath \overline{\sigma_{\eso,1}(V)(dx_j)}u,v\rangle_{H_j},
$$
where $H_j$ is the boundary hypersurface $z'_j=0$, $d\hat z_j$ shows that
$dz'_j$ is dropped from the density, and on $H_j$ one uses the density
induced by the Riemannian density and $dz'_j$.
This completes the proof of the lemma.
\end{proof}

We now define an extension of $\Diffes{}(M)$ as follows.

\begin{definition}
Let $x^{-r}\Diffesd{k}(M)$ denote the set of Banach space adjoints of elements
of $x^{-r}\Diffes{k}(M)$ in the sense of \eqref{eq:Banach-adjoint}.
\nomenclature[D]{$\Diffesd{k}(M)$}{Adjoints of edge-smooth differential operators}

Also let $x^{-2r}\Diffess{2k}(M)$ denote
operators of the form
$$
\sum_{j=1}^N Q_j P_j,\ P_j\in x^{-r}\Diffes{k}(M),\ Q_j\in x^{-r}\Diffesd{k}(M).
$$
\nomenclature[D]{$\Diffesd{*}$}{Adjoints of edge-smooth
  differential operators}
\nomenclature[D]{$\Diffess{*}$}{Compositions of edge-smooth differential
  operators with adjoints}
For $M$ non-compact, the sum is taken to be locally finite.
\end{definition}

Thus, if $P\in x^{-2r}\Diffess{2k}(M)$, $P_j$, $Q_j$ as above, and $Q_j
=R_j^*$, $R_j\in x^{-r}\Diffes{k}(M)$, then
$$
\langle Pu,v\rangle=\sum_{j=1}^N\langle P_j u,R_j v\rangle.
$$

We are now ready to discuss Dirichlet and Neumann boundary conditions
for $P\in x^{-2r}\Diffess{2k}(M)$.

\begin{definition}\label{def:Dirichlet-Neumann}
Suppose $P\in x^{-2r}\Diffess{2k}(M)$. By the {\em Dirichlet operator}
associated
to $P$ we mean the map
$$
\rho\circ P:\Hesz{k,l}(M)\to\Hes{-k,l-2r}(M),
$$
where $\rho:\dHes{-k,l-2r}(M)\to\Hes{-k,l-2r}(M)$ is the quotient map.
For $f\in\Hes{-k,l-2r}(M)$ we say that $u\in\Hesz{k,l}(M)$ solves
{\em the Dirichlet problem} for $Pu=f$  if $\rho\circ Pu=f$.
We also say in this
case that $Pu=f$ with {\em Dirichlet boundary conditions}.

Similarly, for $f\in\dHes{-k,l-2r}(M)$ we say that $u\in\Hes{k,l}(M)$
solves the {\em Neumann problem} for $Pu=f$  if $Pu=f$. We also say in this
case that $Pu=f$ with {\em Neumann boundary conditions}. Correspondingly,
for the sake of completeness,
by the Neumann operator associated to $P$ we mean $P$ itself.
\end{definition}

\begin{remark}
For the Lorentzian metric $\tilde{g}=dt^2-g$ on $M_0$ lifted to $M$,
and with $P=d^*d$, the equation
$$
Pu=f,\ f\in\dHes{-1,l-1-(f+1)/2}(M),\ u\in\Hes{1,l+1-(f+1)/2}(M)
$$
with the Neumann boundary condition means $\ang{du,dv}_{\tilde{g}}=\ang{f,v}_{\tilde{g}}$
for all $v\in\Hes{1,-l+1-(f+1)/2}(M)$, or equivalently for all $v\in\cC$.
Away from $\ef$, this is the standard
formulation of the Neumann problem on a manifold with corners (or
indeed on a Lipschitz domain): pairing with $v$ vanishing at the
boundary and integrating by
parts yields $Pu=f$ in the
interior; pairing with $v$ nonvanishing at boundary faces other than
$\ef$ then yields vanishing of normal derivatives at those faces.

Thus near $\ef$, we impose the Neumann condition in the sense
described above on all other
boundary hypersurfaces, uniformly up to $\ef$, but there is no
condition associated to $\ef$. In particular, a Neumann solution $u$ (just
like a Dirichlet solution) on $M$ need not solve the corresponding
problem on $M_0$, where a condition is enforced even at $\cornM$: $u$
may blow up arbitrarily fast at $\ef$.
\end{remark}

\begin{remark}
As noted in Lemma~\ref{lemma:Diffes-adjoints}, when considering
the action of $\Diffes{}(M)$ on $\dCI(M)$, $\Diffes{}(M)$ is closed
under adjoints (which thus map
to $\CmI(M)$, i.e.\ extendible distributions), so one can suppress the
subscript $\sharp$ on $\Diffess{}(M)$. Thus, the subscript's main role
is to keep the treatment of the Neumann problem clear---without
such care, one would need to use quadratic forms throughout, as was done
in \cite{Vasy5}.
\end{remark}

We now turn to the action of $\Psieb{m,l}(M)$ on the dual spaces.
Note that any $A\in\Psieb{m,l}(M)$ maps $\cC$ to itself, and that
$\Psieb{m,l}(M)$ is closed under formal adjoints, i.e.\ if
$A\in\Psieb{m,l}(M)$ then there is a unique $A^*\in\Psieb{m,l}(M)$ such
that $\langle Au,v\rangle=\langle u,A^*v\rangle$ for all $u,v\in\cC$
---cf.\ $\Diffeb{}(M)$ in Lemma~\ref{lemma:Diffes-adjoints}.
We thus define
$A:\cC'\to\cC'$ by $\langle Au,v\rangle=\langle u,A^*v\rangle$, $u\in\cC'$,
$v\in\cC$. Since $\cC$ is (even sequentially) dense in $\cC'$ endowed
with the weak-* topology, this definition is in fact
the only reasonable one, and if $u\in\cC,$ the element of $\cC'$
given by this is the linear functional induced by $Au$ on $\cC$.

Next, for subspaces of $\cC'$ we have improved statements. In particular,
most relevant here, dually to Theorem~\ref{thm:Psieb-0-bd}, any $A\in
\Psieb{0}(M)$ is bounded on $\dHes{-1,l}(M)$ and on $\Hes{-1,l}(M)$.

We now turn to an extension of $\Diffes{}\Psieb{}(M)$. First,
taking adjoints in Lemma~\ref{lemma:Psieb-Vfes-comm}, we deduce:

\begin{lemma}
If $A\in\Psieb{m}(M)$ and $Q\in\Diffesd{1}(M)$, then $[A,Q]=\sum Q_j A_j+B$,
$B\in\Psieb{m}(M)$, $A_j\in\Psieb{m-1}(M)$, $Q_j\in\Diffesd{1}(M)$.

Similarly, $[A,Q]=\sum A'_j Q'_j+B'_j$, $B'\in\Psieb{m}(M)$,
$A'_j\in\Psieb{m-1}(M)$, $Q'_j\in\Diffesd{1}(M)$.
\end{lemma}

\begin{proof}
The proof is an exercise in duality; we only spell it out to emphasize our
definitions.
We have for $u\in\cC'$, $v\in\cC$,
$$
\langle [A,Q]u,v\rangle=\langle (AQ-QA)u,v\rangle
=\langle u,(Q^*A^*-A^*Q^*)v\rangle=\langle u,[Q^*,A^*]v\rangle
$$
where $Q^*\in\Diffes{1}(M)$, $A^*\in\Psieb{M}$. Thus, by
Lemma~\ref{lemma:Psieb-Vfes-comm} (applied with $\Vfes(M)$ replaced
by $\Diffes{1}(M)$), there exist $\tilde A_j\in\Psieb{m-1}(M)$,
$\tilde B\in\Psieb{m}(M)$, $\tilde Q_j\in\Diffes{1}(M)$ such that
$[Q^*,A^*]=-[A^*,Q^*]=\sum \tilde Q_j\tilde A_j+\tilde B$.
Thus,
$$
\langle [A,Q]u,v\rangle=\Big\langle u,\Big(\sum_j \tilde Q_j\tilde A_j+\tilde B\Big)v\Big\rangle
=\Big\langle \Big(\sum_j\tilde A_j^* \tilde Q_j^*+\tilde B^*\Big)u,v\Big\rangle,
$$
with $\tilde A_j^*\in\Psieb{m-1}(M)$, $\tilde B^*\in\Psieb{m}(M)$,
$\tilde Q_j^*\in\Diffesd{1}(M)$. This proves the second half of the lemma.
The first half is proved similarly, using the second half of the
statement of Lemma~\ref{lemma:Psieb-Vfes-comm} rather than its first half.
\end{proof}

In fact, the analogue of Lemma~\ref{lemma:dzcomm} also holds with $D_{z'_j}$
replaced by $D_{z'_j}^*\in\Diffesd{1}(M)$:

\begin{lemma}\label{lemma:dzcomm-2}
Let $A \in \Psieb{m,l}(M),$ $a=\sigma_{\ebo,m}(A)$.  Then
$$
\imath[x^{-1}D_{z'_j}^*, A] = A_1 x^{-1} D_{z'_j}^* + x^{-1} A_0
$$
where $A_0\in\Psieb{m,l}(M)$, $A_1\in\Psieb{m-1,l}(M)$,
$$
\sigma_{\ebo,m}(A_0)=\frac{\pa a}{\pa z'_j},
\quad \sigma_{\ebo,m-1}(A_1)=\frac{\pa a}{\pa
  \zeta'_j} + \frac{\pa a}{\pa \xi}.
$$
\end{lemma}

We thus make the following definition:

\begin{definition}
Let
$$\Diffess{k} \Psieb{m} (M) = \Big\{\sum A_\alpha B_\beta: A_\alpha \in
\Diffess{k}(M),\ B_\beta \in \Psieb{m}(M)\Big\}.$$
\end{definition}

Using Proposition~\ref{prop:DiffPsi-alg} and duality, as in the previous
lemma, we deduce
the following:

\begin{proposition}
$\Diffess{k}\Psieb{m,l}(M)$ is a $\Psieb{0}(M)$-bimodule, and
\begin{equation}\begin{split}
&P\in\Diffess{k}\Psieb{m,l}(M),\ A\in\Psieb{s,r}(M)\Rightarrow\\
&PA,AP\in\Diffess{k}\Psieb{m+s,l+r}(M),\ [P,A]\in\Diffess{k}\Psieb{m+s-1,l+r}(M).
\end{split}\end{equation}
\end{proposition}

\subsection{Domains}

In this section, we discuss the relationship between Dirichlet and
Neumann form domains of $\Lap$ and the scales of weighted Sobolev
spaces that
we have introduced.  First, we identify the Dirichlet quadratic form
domain in terms of the edge-smooth Sobolev spaces.

The Friedrichs form domain of $\Lap$ with Dirichlet
boundary conditions on $X_0$ is
$$
H^1_0(X_0),
$$
also denoted by $\dot H^1(X_0)$ (see \cite[Appendix B.2]{Hormander3});
we may also view this space as the completion of $\CI_c(X_0)$ in the $H^1(X_0)$-norm,
$$
\|u\|_{H^1(X_0)}=\|u\|_{L^2_{g_0}(X_0)}+\|du\|_{L^2_{g_0}(X_0;T^* X_0)}.
$$
Equivalently in terms of ``doubling'' $X_0$ across all boundary
hypersurfaces, $H^1_0(X_0)$ consists
of $H^1$-functions on the ``double''
supported in $X_0$.

\begin{lemma}
On $\dCI(X)=\beta^*\dCI(X_0)$, the norms
$$
\|u\|_{H^1(X_0)}=\big(\|u\|_{L^2_{g_0}}^2+\|du\|_{L^2_{g_0}}^2\big)^{1/2}
$$
and 
$$
\|u\|_{\Hes{1,1-(f+1)/2}(X)}
$$
are equivalent.
\end{lemma}

\begin{proof}
  Multiplication by elements of $\CI(X_0)$ is bounded with respect to
  both norms (with respect to $\Hes{1,1-(f+1)/2}(X)$ even $\CI(X)$ is
  bounded), so one can localize in $X_0$, or equivalently in $X$ near
  a fiber $\beta^{-1}(p)$, $p\in W$, of $\ef$, and assume that $u$ is
  supported in such a region.

Elements of $\Vf(X_0)$ lift under $\beta$ to span $x^{-1}\Vfes(X)$ as
a $\CI(X)$-module by \eqref{eq:Vf-lift-to-M}. In particular, merely
since $\beta^* \Vf (X_0) \subset x^{-1}\Vfes(X),$ we
obtain\footnote{We use the notation that $a \lesssim b$ if there
exists $C>0$ such that $a\leq Cb$. Usually $a$ and $b$ depend on various
quantities, e.g.\ on $u$ here, and $C$ is understood to be independent
of these quantities.}
\begin{equation}\label{eq:H10-Hes}
\|u\|_{H^1_0(X_0)}\lesssim \|u\|_{\Hes{1,1-(f+1)/2}(X)},\ u\in\dCI(X).
\end{equation}
We now prove the reverse inequality. By the spanning property, we have
\begin{equation}\label{eq:Vfes-H10}
\|x^{-1}A u\|_{L^2_g}\lesssim \|u\|_{H^1(X_0)}
\end{equation}
for any $A\in\Vfes(X)$ as $\CI(X)$ is bounded acting by multiplication
on $L^2_g(X)=L^2_{g_0}(X_0)$.
As $\Vfes(X)$ together with the identity operator generates
$\Diffes{1}(X)$, we only need to prove
$$
\|x^{-1}u\|_{L^2_g(X)}\lesssim \|u\|_{H^1(X_0)},
$$
for $u\in\dCI(X)$ supported near a fiber $\beta^{-1}(p)$, $p\in Y$, of
$\tilde Y$.
However, this follows easily from identifying a neighborhood of
$\beta^{-1}(p)$ with $[0,\ep)_x\times O_y\times Z_z$,
where $O\subset\RR^{n-f-1},$
and using the Poincar\'e inequality in $Z$, namely that
$$
\|u(x,y,.)\|_{L^2(Z)}\leq C \|(d_Z u)(x,y,.)\|_{L^2(Z)},\quad u \in \dCI(X_0).
$$
Multiplying the square of both sides by $x^{-2+f}$ and integrating
in $x,y$, yields
$$
\|x^{-1}u(x,y,.)\|_{L^2_g(X)}\leq C \|(x^{-1}d_Z u)(x,y,.)\|_{L^2_g(X)}
\leq C'\|u\|_{H^1_0(X)}
$$
by \eqref{eq:Vfes-H10}.
\end{proof}

In view of the definition of $\Hesz{1,1-(f+1)/2}(X)$ as the closure
of $\dCI(X)$ in $\Hes{1,1-(f+1)/2}(X)$, we immediately deduce:

\begin{proposition}\label{proposition:dirichletdomain}
The Dirichlet form domain of $\Lap$ is given by
\begin{equation}\label{eq:domain-identification}
H^1_0(X_0)=\Hesz{1,1-(f+1)/2}(X)
\end{equation}
in the strong sense that the natural (up to equivalence)
Hilbert space norms on the two
sides are equivalent. In particular, for $u\in H^1_0(X_0)$, we have
\begin{equation*}
\|x^{-1}Q u\|_{L^2_g}\leq C\|u\|_{H^1(X_0)}
\end{equation*}
for all $Q\in\Diffes{1}(X)$.
\end{proposition}

For Neumann boundary conditions the quadratic form
domain is $H^1(X_0)$, whose lift is
not quite so simple in terms of the edge-smooth spaces. However, we have
the following lemma, which suffices for the edge-b propagation results below
(with a slight loss).

\begin{lemma}\label{lemma:domaininclusion}
We have $\Hes{1,1-(f+1)/2}(X)\subset H^1(X_0)\subset \Hes{1,-(f+1)/2}(X)$,
with all inclusions being continuous.
\end{lemma}

\begin{proof}
The first inclusion is an immediate consequence of \eqref{eq:H10-Hes}
holding for $u\in\cC$, $\cC$ as in Remark~\ref{remark:Hes} (thus
dense in $\Hes{1,l-(f+1)/2}(X)$ for all $l$),
using again that elements of $\Vf(X_0)$ lift under $\beta$ to span
(and in particular lie in) $x^{-1}\Vfes(X)$ as
a $\CI(X)$-module by \eqref{eq:Vf-lift-to-M}.

For the second inclusion, we need to prove that
$\|Au\|_{L^2_g(X)}\leq C\|u\|_{H^1(X_0)}$ for $A\in\Diffes{1}(X)$. As
this is automatic for $A\in\CI(X)$, we are reduced to considering
$A\in\Vfes(X)$. But \eqref{eq:Vfes-H10} still holds for $u\in\cC$,
so $\|Au\|_{L^2_g(X)}\leq C'\|x^{-1}A u\|_{L^2_g}\leq C\|u\|_{H^1(X_0)}$
for $A\in\Vfes(X)$. This finishes the proof of the lemma.
\end{proof}

\subsection{The wave operator as an element of $x^{-2}\Diffess{2}(M)$}

For $f\in\cC$, in local coordinates,
$$
df=(x\pa_x f)\,\frac{dx}{x}+\sum_j (x\pa_{y_j}f)\,\frac{dy_j}{x}
+\sum_j (\pa_{z_j}f)
\,dz_j.
$$
Thus, the exterior derivative satisfies
$$
d\in \Diffes{1}(M;\underline{\CC},\Tesstar M),
$$
with $\underline{\CC}$
denoting the trivial bundle. As the dual Riemannian metric is
of the form $x^{-2}G$, where $G$ is a smooth fiber metric on
$\Tesstar X$, and $\Delta=d^* d$, we deduce that
$\Box\in x^{-2}\Diffess{2}(M)$. However,
we need a more precise description of $\Box$ for our commutator
calculations.

So suppose now that $U$ is a coordinate chart near a point $q$ at $\pa\ef$ with
coordinates $(x,y,z',z'')$ centered at $q$, and
recall from \eqref{metric}
that the Riemannian metric has the form
\begin{equation}
g = dx^2 + h(y,dy)+ x^2 k(x,y,z,dz)+x k'(x,y,z,dx,dy,x\,dz).
\end{equation}
By changing $z''$ if necessary
(while keeping $x,y,z'$ fixed---cf.\ the argument of \S\ref{section:geometry}
leading to \eqref{eq:X_0-metric-form},
we can arrange that the dual metric $K$ of $k$ have the form
\begin{equation}\label{metricform}\begin{split}
K(0,y,z)=\sum_{i,j=1}^k k_{1,ij}(0,y,z)\zetaes'_i\zetaes'_j&+\sum_{i=1}^k
\sum_{j=k+1}^f k_{3,ij}(0,y,z) \zetaes'_i\zetaes''_j\\
&+\sum_{i,j=k+1}^f k_{2,ij}(0,y,z)\zetaes''_i\zetaes''_j,\qquad k_3|_C=0,
\end{split}\end{equation}
where $$C=\{x=0,\ z'=0\}.$$

We deduce the following lemma:

\begin{lemma}\label{lemma:es-eb-wave-op}
Let $U$ be a coordinate chart near a point with $x=0$ and $z'=0$, and
suppose that we have arranged that at $$C=\{x=0,\ z'=0\},$$
the vector spaces
$$
\Span\{dz'_i,\ i=1,\ldots,k\}\Mand\Span\{dz''_j,\ j=k+1,\ldots,f\}
$$
are orthogonal with respect to $K.$
With $Q_i=x^{-1}D_{z'_i}$, the wave operator satisfies
\begin{equation}
\Box=\sum_{i,j}Q_i^* \kappa_{ij} Q_j+\sum_i (x^{-1}M_i Q_i
+Q_i^* x^{-1}M'_i)+x^{-2}\Htil\ \text{on } U
\end{equation}
with
\begin{equation}\begin{split}
&\kappa_{ij}\in\CI(M),\ M_i,M'_i\in \Diffeb{1}(M),\ \Htil\in\Diffeb{2}(M)\\
&\sigma_{\ebo,1}(M_i)=m_i=\sigma_{\ebo,1}(M'_i),
\ \htil=\sigma_{\ebo,2}(\Htil),\\
&\kappa_{ij}|_{\ef}=-k_{1,ij}(y,z),
\ m_i|_C=0,\ m_i|_{\ef}=-\frac{1}{2} \sum_{j=k+1}^f k_{3,ij}\zetaeb''_j,\\
&\htil|_{\ef}
=\taueb^2-\xi^2-h(y,\etaeb)
-\sum_{i,j=k+1}^{f}k_{2,ij}(y,z)\zetaeb''_i\zetaeb''_j.
\end{split}\end{equation}
\end{lemma}

We next note microlocal elliptic regularity.

\begin{proposition}\label{prop:elliptic}
Let $u\in\hilbert\equiv\Hes{1,l}(M)$,
and suppose that
$$
\Box u\in\dHes{-1,l-2}(M)=\hilbertp
$$
with Dirichlet
or Neumann boundary conditions. Then
$$
\WFebX^{m,0}(u)\subset\ebSigma\cup\WFebY^{m,0}(\Box u).
$$

In particular, if $\Box u=0$, then $\WFebX^{\infty,0}(u)\subset\ebSigma.$
\end{proposition}

\begin{proof}
The proof goes along the same lines as Proposition~4.6 of \cite{Vasy5} and
Theorem~8.11 of \cite{mvw1}; we thus provide a
sketch.  An essential ingredient is the top-order commutativity of
$x^l\Diffess{2}\Psieb{m}(M),$ which allows us to treat all commutators
as error terms. The key estimate is stated in Lemma~\ref{lemma:Dirichlet-form}
below.

Given the lemma, one proceeds by an inductive argument, showing
that if $\WFebX^{s-1/2,0}(u)\subset\ebSigma\cup\WFebY^{s-1/2,0}(\Box u)$
(which is a priori known for $s=1/2$, starting our inductive argument)
then $\WFebX^{s,0}(u)\subset\ebSigma\cup\WFebY^{s,0}(\Box u)$.
In order to show this, one
takes $A\in\Psieb{s,l+(f-1)/2}(M)$, with
$\WF'(A)\cap\ebSigma=\emptyset$, $\WF'(A)\cap\WFebY^{s,0}(\Box
u)=\emptyset$.
Let
$\Lambda_\gamma$ be uniformly bounded in $\Psieb{0,0}(M)$, $\gamma\in(0,1]$,
with $\Lambda_\gamma\in\Psieb{-1,0}(M)$ for all $\gamma$,
$$
\sigma_{\ebo,0,0}(\Lambda_\gamma)=\big(1+\gamma(|\xieb|^2+|\taueb|^2+
|\etaeb|^2+|\zetaeb|^2)\big)^{-1},
$$
so $A_\gamma=\Lambda_\gamma A$
is uniformly bounded in $\Psieb{m,l+(f-1)/2}(M)$ and $A_\gamma\to A$
in $\Psieb{m+\delta,l+(f-1)/2}(M)$ ($\delta>0$ fixed) as $\gamma\to 0$.
One then concludes by  Lemma~\ref{lemma:Dirichlet-form}
that for all $\ep\in(0,1]$,
\begin{equation}\label{4.5.2011A}
|\langle \pa_t A_\gamma u,\pa_t A_\gamma u\rangle
-\langle d_X A_\gamma u,d_X A_\gamma u\rangle|
\leq
\ep \|A_\gamma u\|_{\Hes{1,1-(f+1)/2}(M)}^2+C\ep^{-1},
\end{equation}
with $C$ uniformly bounded, independent of $\gamma$. 

We now note that the
Dirichlet form is microlocally
elliptic for Dirichlet boundary conditions, i.e.
\begin{equation*}\begin{split}
&\|A_\gamma u\|_{\Hes{1,1-(f+1)/2}(M)}^2\\
&\leq C'|\langle \pa_t A_\gamma u,\pa_t A_\gamma u\rangle
-\langle d_X A_\gamma u,d_X A_\gamma u\rangle|+C'\|u\|_{\Hes{1,1-(f+1)/2}(M)}^2.
\end{split}\end{equation*}
For details, see
the proof of Proposition~4.6 of \cite{Vasy5}, which can be followed
essentially verbatim\footnote{To give a rough idea, one distinguishes
  between the two components of the elliptic set in terms of
  \eqref{egh-coordinates} and uses a square root construction in the
  edge-b algebra;
in the first component noting in addition that the Dirichlet form involves
$D_{z'_j}u$, so in $z'_j<\delta$, one can estimate $\delta^{-1}z'_j
D_{z'_j}u$ using this.}   since the non-trivial aspect is the b-behavior in
the fibers of the edge; the $(x,y)$ variables here, as well as the
$z''$ variables, play the role of the $y$ variables in \cite{Vasy5},
the $z'$ variables here play the role of the $x$ variables in
\cite{Vasy5}, and $\Psieb{}(M)$ plays the role of $\Psib(X)$ in \cite{Vasy5}
(where $X$ is spacetime).   (Likewise is simple to modify
the inductive arguments for the Neumann condition as the $\Hes{0,1-(f+1)/2}(M)$
norm of an additional eb-derivative, which one would need to bound, can be
bounded in terms of the $\Hes{1,1-(f+1)/2}(M)$ norm; this is the same process
as in \cite{Vasy5}.)

Thus, for sufficently small $\ep\in(0,1]$,
$\ep \|A_\gamma u\|_{\Hes{1,1-(f+1)/2}(M)}^2$ in \eqref{4.5.2011A} can be absorbed in
$(C')^{-1}\|A_\gamma u\|_{\Hes{1,1-(f+1)/2}(M)}^2$, and then one concludes that
$\|A_\gamma u\|_{\Hes{1,1-(f+1)/2}(M)}$
is uniformly bounded independent of $\gamma$.
As $A_\gamma\to A$ strongly, one concludes by a standard argument
$Au\in\Hes{1,1-(f+1)/2}(M)$. Thus, $x^{l+(f-1)/2}Au\in\Hes{1,l}(M)$, hence
(as $\hilbert=\Hes{1,l}(M)$)
$$
\liptic(A)\cap \WFebX^{s,0}(u)=\emptyset,
$$
completing the iterative step.
\end{proof}

As mentioned above, the key ingredient in proving microlocal elliptic
regularity is the following lemma.

\begin{lemma}\label{lemma:Dirichlet-form}
For Neumann boundary conditions, let $\hilbert=\Hes{1,l+1}(M)$,
$\hilbertp=\dHes{-1,l-1}(M)$; for Dirichlet boundary conditions let
$\hilbert=\Hesz{1,l+1}(M)$,
$\hilbertp=\Hes{-1,l-1}(M)$.
Let $K\subset \Sebstar M$ be compact, $U\subset\Sebstar M$ open,
$K\subset U$. Suppose that $\cA$ is a bounded
family of ps.d.o's in $\Psieb{m,l+(f+1)/2}(M)$ with $\WF'(\cA)\subset K$, such that
for $A\in\cA$, $A\in\Psieb{m-1,l+(f+1)/2}(M)$ (but the bounds for $A$ in
$\Psieb{m-1,l+(f+1)/2}(M)$ are not necessarily uniform in $A$!).
Then there exist
$G \in \Psieb{m-1/2,0}(M)$, $\tilde G\in\Psieb{m,0}(M)$
with $\WF' G,\WF'\tilde G \subset U$ and $C_0>0$ such that for $\ep>0$,
$A\in\cA$,
\begin{equation}\begin{split}\label{eq:Hes-Dirichlet-est}
u\in \hilbert,&\ \WFebX^{m-1/2,0}(u)\cap U=\emptyset,
\ \WFebY^{m,0}(\Box u)\cap U=\emptyset
\Rightarrow\\
|\langle \pa_t A u,&\pa_t A u\rangle
-\langle d_X A u,d_X A u\rangle|\\
&\leq \ep \|Au\|_{\Hes{1,1-(f+1)/2}(M)}^2+C_0\big(\|u\|_{\hilbert}^2+\|Gu\|_{\hilbert}^2\\
&\qquad\qquad\qquad\qquad\qquad\qquad
+\ep^{-1}\|\Box u\|_{\hilbertp}^2+\ep^{-1}\|\tilde G\Box u\|^2_{\hilbertp}
\big).
\end{split}\end{equation}
\end{lemma}

\begin{remark}
Recall that $u\in\Hes{1,1-(f+1)/2}(M)$ is equivalent to $d_X u\in L^2_g(M)$,
$\pa_t u\in L^2_g(M)$
and $x^{-1}u\in L^2_g(M)$,
so $\ep \|Au\|_{\Hes{1,1-(f+1)/2}(M)}^2$
on the right hand side of \eqref{eq:Hes-Dirichlet-est} is comparable to
the terms $\langle \pa_t A u,\pa_t A u\rangle$ and
$\langle d_X A u,d_X A u\rangle$. However, if $A$ is supported
away from $\ebSigma$, the Dirichlet form is microlocally elliptic, by
the same arguments
as those in the proof of Proposition~4.6 of \cite{Vasy5} and
Theorem~8.11 of \cite{mvw1},
so this term can be absorbed into the left hand side, as was done
in Proposition~\ref{prop:elliptic}.

The hypotheses in \eqref{eq:Hes-Dirichlet-est} assure that the
other terms on the right hand side are finite, independent of $A\in\cA$.
\end{remark}

\begin{proof}
Again, this follows the argument as Lemma~4.2 and 4.4 of \cite{Vasy5} and
Lemma~8.8 and 8.9 of \cite{mvw1}, so we only sketch the proof.
We sketch the Neumann argument; the Dirichlet case needs only simple changes.
We have
$$
\langle \pa_t u,\pa_t A^*Au\rangle-\langle d_X u, d_X A^*Au\rangle
=\langle\Box u,A^*A u\rangle
$$
for all $u\in\hilbert$ and $A\in\Psieb{m-1,l+(f-1)/2}(M)$
since
$A^*Au\in\Hes{1,-l-(f-1)}(M)$, which is mapped by $\Box$ into
$\dHes{-1,-l-(f+1)}(M)=(\Hes{1,l}(M))^*$.
Modulo commutator terms,
one can rewrite the left hand side as
$$
\langle \pa_t Au,\pa_t Au\rangle-\langle d_X Au, d_X Au\rangle,
$$
which is the left hand side of
\eqref{eq:Hes-Dirichlet-est}.
The commutator terms can be estimated by the second and third terms
(which do not depend on $\ep$) on the right hand side of
\eqref{eq:Hes-Dirichlet-est}.
The other terms on the right hand side arise by estimating (using that the
dual of $\Hes{1,l}(M)$ is $\dHes{-1,-l-(f+1)}(M)$)
\begin{equation*}\begin{split}
&|\langle \Box u,A^*A u\rangle|
\leq\|A\Box u\|_{\dHes{-1,-1-(f+1)/2}(M)}
\,\|Au\|_{\Hes{1,1-(f+1)/2}(M)}\\
&\qquad\leq \ep^{-1}\|A\Box u\|^2_{\dHes{-1,-1-(f+1)/2}(M)}
+\ep\|A u\|^2_{\dHes{-1,-1-(f+1)/2}(M)}\\
&\qquad= \ep^{-1}\|x^{-l-(f+1)/2}A\Box u\|^2_{\dHes{-1,l-1}(M)}
+\ep\|A u\|^2_{\dHes{-1,-1-(f+1)/2}(M)},
\end{split}\end{equation*}
and as $x^{-l-(f+1)/2}A$ is uniformly bounded in $\Psieb{m,0}(M)$, with
wave front set in $K$, $\|x^{-l-(f+1)/2} A\Box u\|^2_{\hilbertp}$
can be estimated by a multiple of
$\|\Box u\|_{\hilbertp}^2+\|\tilde G\Box u\|^2_{\hilbertp}$ in view
of Lemma~\ref{lemma:uniformly-microlocal}. This completes the proof.
\end{proof}

The following is analogous to Lemma 7.1 of \cite{Vasy5} and Lemma~9.8
of \cite{mvw1} and states that near $\gcal$ the fiber derivatives
$x^{-1}D_{z_i'}$ of microlocalized solutions $Au$ to the wave equation
can be controlled by a small multiple of the time derivative, modulo
error terms (note that $G$ is lower order than $A$ by $1/2$).  The
theorem mentions a $\delta$-neighborhood of a compact set
$K\subset\gcal$ (for $\delta<1$); by this we mean the set of points of
distance $<\delta$ from $K$ with respect to the distance induced by
some Riemannian metric on $\Sebstar M$.  Note that the choice of the
Riemannian metric is not important, and in particular, $\gcal$ is
defined by $x=0$, $z'=0$,
$1-h(y,\etaebh)-\xiebh^2-k(y,z,\zetaebh'=0,\zetaebh'')=0$, so the set
given by
$$
x<C'\delta,\ |z'|<C'\delta,
\ \Big|1-h(y,\etaebh)-\xiebh^2-k(y,z,\zetaebh'=0,\zetaebh'')\Big|<C'\delta,
$$
is contained in
a $C''\delta$-neighborhood of $\cG$ for some $C''>0$, with $C''$
independent of $\delta$ (as long as $C'$ is bounded).

\begin{lemma}\label{lemma:glancingestimate}
For Dirichlet or Neumann boundary conditions let $\hilbert$ and
$\hilbertp$ be as in Lemma~\ref{lemma:Dirichlet-form}.

Let $K\Subset \gcal.$ There exists $\delta_0\in (0,1)$ and $C_0>0$ with the
following property.

Let $0<\delta<\delta_0$, and $\delta>0,$ and let
$U$ be a $\delta$-neighborhood of $K$ in $\Sebstar M$.
Suppose $\cA$ is a bounded
family of ps.d.o's in $\Psieb{m,l+(f+1)/2}(M)$ with $\WF'(\cA)\subset U$, such that
for $A\in\cA$, $A\in\Psieb{m-1,l+(f+1)/2}(M)$.
Then there exist
$G \in \Psieb{m-1/2,0}(M)$, $\tilde G\in\Psieb{m,0}(M)$
with $\WF' G,\WF'\tilde G \subset U$ and $\tilde C=\tilde C(\delta)>0$
such that for
$A\in\cA,$
$$
u\in \hilbert,\ \WFebX^{m-1/2,0}(u)\cap U=\emptyset,
\ \WFebY^{m,0}(\Box u)\cap U=\emptyset
$$
implies
\begin{equation*}\begin{split}
&\sum \norm{ x^{-1}D_{z_i'} A u}^2\\
&\qquad \leq C_0\delta\norm{D_t A u}^2 +
\tilde C
\big(\|u\|_{\hilbert}^2+\|Gu\|_{\hilbert}^2
+\|\Box u\|_{\hilbertp}^2+\|\tilde G\Box u\|^2_{\hilbertp}
\big).
\end{split}\end{equation*}
\end{lemma}

\begin{proof}
This is an analogue of Lemma~7.1 of \cite{Vasy5} and Lemma~9.8 of \cite{mvw1},
so we only indicate the main idea. By Lemma~\ref{lemma:Dirichlet-form}
one has control of the Dirichlet form in terms of the second through fifth
terms on the right hand side, so it suffices to check that
$\sum \norm{ x^{-1}D_{z_i'} A u}^2$ can be controlled by the Dirichlet
form and $\delta\norm{D_t A u}^2$. This uses that $K\Subset\gcal$, $D_t$ is
elliptic on $\ebSigma$, and $\langle \tilde HAu,Au\rangle$ is small as
$\WF'(A)\subset U$; see
the aforementioned Lemma~7.1 of \cite{Vasy5} and Lemma~9.8 of \cite{mvw1}
for details.
\end{proof}

\begin{corollary}\label{cor:glancingestimate}
For Dirichlet or Neumann boundary conditions let $\hilbert$ and
$\hilbertp$ be as in Lemma~\ref{lemma:Dirichlet-form}.
Let $K\Subset \gcal,$ $\delta>0$.

Then there exists a neighborhood $U$ of $K$ in $\Sebstar M$ with the following
property.
Suppose that $\cA$ is a bounded
family of ps.d.o's in $\Psieb{m,l+(f+1)/2}(M)$ with
$\WF'(\cA)\subset U$, such that
for $A\in\cA$, $A\in\Psieb{m-1,l+(f+1)/2}(M)$.
Then there exist
$G \in \Psieb{m-1/2,0}(M)$, $\tilde G\in\Psieb{m,0}(M)$
with $\WF' G,\WF'\tilde G \subset U$ and $\tilde C=\tilde C(\delta)>0$
such that for
$A\in\cA$,
$$
u\in \hilbert,\ \WFebX^{m-1/2,0}(u)\cap U=\emptyset,
\ \WFebY^{m,0}(\Box u)\cap U=\emptyset
$$
implies
\begin{equation*}\begin{split}
&\sum \norm{ x^{-1}D_{z_i'} A u}^2\\
&\qquad \leq \delta\norm{D_t A u}^2 +
C
\big(\|u\|_{\hilbert}^2+\|Gu\|_{\hilbert}^2
+\|\Box u\|_{\hilbertp}^2+\|\tilde G\Box u\|^2_{\hilbertp}
\big).
\end{split}\end{equation*}
\end{corollary}

\begin{proof}
Fix a Riemannian metric on $\Sebstar M$.
Let $\delta_0,C_0$ be as in Lemma~\ref{lemma:glancingestimate}, and let
$\delta'=\min(\delta_0/2,\delta/C_0)$. Applying
Lemma~\ref{lemma:glancingestimate} with $\delta'$ in place of $\delta$
gives the desired conclusion, if we let $U$ be a $\delta'$-neighborhood
of $K$.
\end{proof}

Recall now that $C=\{x=0,\ z'=0\}$ denotes one
boundary face of $\ef$ in local coordinates, and that
as a vector field on $\Tesstar M$ tangent to $\ef$ (but not necessarily the
other boundary faces),
restricted to $\Tesstar_{\ef}M$,
$\sH_{\eso}$ is given by
\begin{equation*}
-\frac12\sH_{\eso}=\xiesh x\pa_x-\xiesh\sigmaes\pa_{\sigmaes}
-\xiesh\zetaesh\pa_{\zetaesh}
+K^{ij}\zetaesh_i\pa_{z_j}+K^{ij}\zetaesh_i\zetaesh_j\pa_{\xiesh}
-\frac{1}{2}\,\frac{\pa K^{ij}}{\pa z_k}\zetaesh_i\zetaesh_j\pa_{\zetaesh_k};
\end{equation*}
see \eqref{eq:sH-at-ef-mod-bvf}-\eqref{eq:sH-at-ef}.
We can expand the $K^{ij}$ terms by breaking
them up into $z'$ and $z''$ components at $C$, using \eqref{metricform}.
This becomes particularly interesting at a point $q\in\esSigma$ which is
the unique point in the preimage of $p\in\Sebstar[C]M\cap\gcal$ under
$\piesebh$.
At such points $\zetaesh'=0$, so many terms vanish. One thus obtains
\begin{equation*}
-\frac12\sH_{\eso}(q)=\xiesh x\pa_x-\xiesh\sigmaes\pa_{\sigmaes}
-\xiesh\zetaesh''\pa_{\zetaesh''}
+k_{2,ij}\zetaesh''_i\pa_{z''_j}+k_{2,ij}\zetaesh''_i\zetaesh''_j\pa_{\xiesh}
-\frac{1}{2}\,\frac{\pa k_{2,ij}}{\pa z''_k}
\zetaesh''_i\zetaesh''_j\pa_{\zetaesh''_k}.
\end{equation*}
Pushing forward under $\piesebh$, we obtain
\begin{equation*}\begin{split}
(\piesebh_* \sH_{\eso})(p)=&-2\xiebh x\pa_x+2\xiebh\sigmaeb\pa_{\sigmaeb}
+2\xiebh\zetaebh''\pa_{\zetaebh''}\\
&\qquad-2k_{2,ij}\zetaebh''_i\pa_{z''_j}-2k_{2,ij}\zetaebh''_i\zetaesh''_j\pa_{\xiebh}
+\frac{\pa k_{2,ij}}{\pa z''_k}
\zetaebh''_i\zetaebh''_j\pa_{\zetaebh''_k}.
\end{split}\end{equation*}
Below, this appears as the vector field $|\tau| V_0$, and will give the
direction of propagation at glancing points in
Theorem~\ref{theorem:edgeglancing}.

\begin{lemma}\label{lemma:Box-eb-comm}
Let $Q_i=x^{-1}D_{z'_i},$ $\kappa_{ij}$, $m_i$, $h$ be as in
Lemma~\ref{lemma:es-eb-wave-op}.
For $A\in\Psieb{m,l}(M)$,
\begin{equation}
\imath[\Box,A^*A]=\sum Q_i^* L_{ij} Q_j+\sum (x^{-1}
L_i Q_i+Q_i^* x^{-1}L'_i)
+x^{-2}L_0,
\end{equation}
with
\begin{equation}\label{whatahorror}\begin{split}
&L_{ij}\in\Psieb{2m-1,2l}(M),\ L_i,L'_i\in\Psieb{2m,2l}(M),\ L_0\in\Psieb{2m+1,2l}(M),\\
&\sigma_{\ebo,2m-1}(L_{ij})=2aV_{ij}a,\ \text{where } V_{ij}
=\kappa_{ij}(\pa_{\zetaeb'_i}+\pa_{\zetaeb'_j}+2\pa_{\xieb})
+\sH_{\ebo,\kappa_{ij}},\\
&\sigma_{\ebo,2m}(L_i)=\sigma_{\ebo,2m}(L'_i)=2a V_i a,\ \text{where}\\
&V_i=\sum_j\kappa_{ij}\pa_{z'_j}
+\frac{1}{2}(m_i\pa_\xi+\sH_{\ebo,m_i})+\frac{1}{2}m_i
(\pa_{\xi}+\pa_{\zeta'_i}),\\
&\sigma_{\ebo,2m+1}(L_0)=2a V_0a,\ V_0=2\htil\pa_\xi+\sH_{\ebo,\htil}
+\sum_i m_i\pa_{z'_i},\\
&\ebWF'(L_{ij}),\ebWF'(L_i),\ebWF'(L'_i),\ebWF'(L_0)\subset\ebWF'(A).
\end{split}\end{equation}
In particular, for $f\in\CI(\Sebstar M)$ with $f|_{\ef}=\dbetaeb^* \phi$
for some $\phi\in \CI(S^* W)$,
\begin{equation}\label{eq:basic-comm}
V_{ij}f|_{\ef}=0,\ V_i f|_{\ef}=0,\ V_0 f|_{\ef}=0.
\end{equation}
Moreover, as smooth vector fields tangent to $\Tebstar_{\ef}M$ (but
not necessarily tangent to the other boundaries),
\begin{equation}\begin{split}\label{eq:sH-eb-h}
&V_0|_{C}=-2\xi\,x\pa_x-2\Big(\xi^2+\sum_{ij}k_{2,ij}\zeta''_i\zeta''_j\Big)\pa_\xi
-2\xi\big(\tau\,\pa_\tau+\eta\,\pa_\eta\big)\\
&\qquad\qquad\qquad-2\sum_{i,j} k_{2,ij}\zeta''_i\pa_{z''_j}
+\sum_{\ell,i,j}(\pa_{z''_\ell}k_{2,ij})\zeta''_i\zeta''_j\pa_{\zeta''_\ell} ,\\
&V_{ij}|_C
=-k_{1,ij}\big(\pa_{\zeta'_i}+\pa_{\zeta'_j}+2\pa_\xi\big)
+\sum_\ell (\pa_{z''_\ell}k_{1,ij})\pa_{\zeta''_\ell},
\ V_i|_C=-\sum_j k_{1,ij}\pa_{z'_j},
\end{split}\end{equation}
and
\begin{equation}\begin{split}\label{eq:xiebh-weights}
&\big(|\tau|V_0\xiebh\big)|_{\ef}=-2\sum_{ij}k_{2,ij}(0,y,z)\zeta''_i\zeta''_j\\
&\big(|\tau|V_i\xiebh\big)|_{\ef}=-\sum_j k_{3,ij}(0,y,z)\zeta''_j,
\ \big(|\tau|V_{ij}\xiebh\big)|_{\ef}=-2 k_{1,ij}(0,y,z),\\
&\big(|\tau|^{-s-1}x^{-r}V_0(|\tau|^s x^r)\big)|_{\ef}=-2(r+s)\xiebh,
\ \big(|\tau|^{-s}x^{-r}V_i(|\tau|^s x^r)\big)|_{\ef}=0,\\
&\big(|\tau|^{-s+1}x^{-r}V_{ij}(|\tau|^s x^r)\big)|_{\ef}=0,
\end{split}\end{equation}
\end{lemma}

\begin{remark}
This is the main commutator computation that we use in the next
section. We stated explicitly the results we need. First, equation
\eqref{eq:basic-comm} shows that functions of the ``slow variables''
do not affect the commutator to leading order at $\ef$, hence they
are negligible for all of our subsequent calculations.

Next, \eqref{eq:sH-eb-h} gives the form of the commutator explicitly
at $C$; this is what we need for hyperbolic or glancing propagation
within $\ef$, i.e.\ at points of $\cH$, resp.\ $\cG$ away
from radial points. These are sufficiently local that we only need
the explicit calculation {\em at} $C$, rather than at all of $\ef$.

Finally,
\eqref{eq:xiebh-weights} contains the results we need at radial points
in $\cG$: there the construction is rather global in $\ef$, so it
would be insufficient to state these results at $C$ only. On the other
hand, localization in $\zetaebh$ is accomplished by localizing in $\xiebh$,
the ``slow variables'' and the characteristic set, so fewer features
of $V_{ij}$, etc., are relevant.
\end{remark}

\begin{proof}
By Lemma~\ref{lemma:es-eb-wave-op},
\begin{equation*}\begin{split}
[\Box,A^*A]=&\sum_{i,j}\big([Q_i^*,A^*A] \kappa_{ij} Q_j+Q_i^*\kappa_{ij}
[Q_j,A^*A]+Q_i^*[\kappa_{ij},A^*A]Q_j\big)\\
&+\sum_i \Big(x^{-1}M_i [Q_i,A^*A]+[x^{-1}M_i,A^*A]Q_i\\
&\qquad\qquad+[Q_i^*,A^*A]x^{-1}M'_i+Q_i^*[x^{-1}M'_i,A^*A]\Big)\\
&+[x^{-2}\Htil,A^*A]\quad  \text{on } U.
\end{split}\end{equation*}
The three terms on the first line of the right hand side are the only ones
contributing to $L_{ij}$; in the case of the third term, via
$$
\imath\sigma_{\ebo,2m-1}([\kappa_{ij},A^*A])=\sH_{\ebo,\kappa_{ij}}a^2
=2a\sH_{\ebo,\kappa_{ij}}a,
$$
while in the case of the first two terms by evaluating the commutators
using Lemma~\ref{lemma:dzcomm} and taking only the $A_1$-terms, with the
notation of the lemma.
The $A_0$-terms of the first two commutators on the first line of the
right hand side (with the notation of Lemma~\ref{lemma:dzcomm})
contribute to $L_i$ or $L'_i$, as do the second
and fourth terms on the second line
and the $A_1$-term of the first and third terms on the second line. Finally,
the expression on the third line, as well as the $A_0$-term of the first
and third terms on the second line contribute to $L_0$. We also use
\eqref{eq:weighted-formula} to remove the weight from $\sH_{\ebo,x^{-1}m_i}$
and $\sH_{\ebo,x^{-2}\htil}$, e.g.\ $x^2\sH_{\ebo,x^{-2}\htil}=-2\htil\pa_\xi
+\sH_{\ebo,\htil}$.

The computation of the Hamilton vector fields at $C$ then follows from
Lemma~\ref{lemma:es-eb-wave-op} and
\eqref{eq:sH-ebo} (recalling that $t$ is one of the $y$-variables).
\end{proof}

\section{Coisotropic regularity and non-focusing}\label{section:coisodef}

In this section we recall from \cite{mvw1} the notion of coisotropic
regularity and, dually, that of nonfocusing.  We will be working
microlocally near $\fcal_{\reg}$ and in particular, away from the
difficulties of the glancing rays in $\fcal_{\sing}.$  Consequently
all the results in this section have proofs identical to those in
\cite[Section 4]{mvw1}, where the fiber $Z$ is without boundary.

Let $\mathcal{K}$ be a compact set in $\rcal^\circ_{\ebo}.$ By
Lemma~\ref{lemma:singularflow-out}, there exists an open set $U\subset
\Sebstar M$ such that $\mathcal{K} \subset U$ and $\overline{U} \cap \fcal \subset \fcal_{\reg}.$
Recall from Corollary~\ref{cor:coisotropic} that in this case $\fcal\cap U$
is a \emph{coisotropic} submanifold of $\Sebstar M$---recall from
Footnote~\ref{footnote:coisotropic} that a submanifold of $\Sebstar M$
is defined to be coisotropic if the corresponding conic submanifold of
$\Tebstar M\setminus o$ is coisotropic.

In what follows, we let $U$ be an arbitrary open subset of
$\Sebstar M$ satisfying $\overline{U} \cap \fcal \subset \fcal_{\reg},$ thus
$U\cap\fcal$ is a $\CI$ embedded coisotropic submanifold of $\Sebstar M$;
the foregoing remarks establish that such subsets are plentiful.

\begin{definition}
Given $U$ as above, let $\module$ denote the module (over $\Psieb{0,0}(M))$ of operators
$A \in \Psieb{1,0}(M)$ such that
\begin{itemize}
\item $\displaystyle \WF' A \subset U,$
\item $\displaystyle \sigma_{\ebo, 1}(A) |_{\fcal_{\reg}}=0.$
\end{itemize}
Let $\coiso$ be the algebra generated by $\module$ with $\coiso^k=
\coiso \cap \Psieb{k,0}(M).$
\end{definition}
\nomenclature[M]{$\module$}{Module of first-order operators
  characteristic on $\fcal_{\reg}$}
\nomenclature[A]{$\mathcal{A}$}{Algebra generated by operators characteristic on $\fcal_{\reg}$}

As a consequence of coisotropy of $\fcal_{\reg},$ we have:
\begin{lemma}\label{lemma:moduleclosed}
The module $\module$ is closed under commutators, and is finitely
generated, i.e., there exist finitely many $A_i \in \Psieb{1}$ with
$\sigma_{\ebo, 1}(A_i) |_{\fcal_{\reg}}=0$ such that
$$
\module=\big\{A\in\Psieb{1}(U):\ \exists Q_i\in\ePs{0}(U),
\ A=\sum_{i=0}^N Q_iA_i\big\}.
$$
Moreover we may take $A_N$ to have symbol
$\abs{\tau}^{-1}\sigma_{\ebo, 2,0} (x^2\Box)$ and $A_0=\Id$.

We thus also obtain
\begin{equation}
\coiso^{k}=\left\{\sum\limits_{|\alpha|\le k} Q_\alpha
\prod\limits_{i=1}^{N}A_i^{\alpha _i},\ Q_\alpha \in\Psieb{0}(U)\right\}
\label{HMV.102}\end{equation}
where $\alpha$ runs over multiindices $\alpha
:\{1,\dots,N\}\to\bbN_0$ and $|\alpha |=\alpha _1+\dots+\alpha _N.$ 
\end{lemma}

\begin{definition}\label{def:coiso-nonfocus}
Let $\hilbert$ be a Hilbert space on
which $\Psieb{0,0}(M)$ acts, and let $K \subset U.$ 
We say that\footnote{Note that our choice
  of $U$ containing $K$ does not matter in the definition.}
\emph{$u$ has coisotropic regularity of order $k$ relative to $\hilbert$
  in $K$ if there exists $Q \in \Psieb{0,0}(M),$ elliptic on $K,$ such
  that $$\coiso^k Q u \in \hilbert.$$}

We say that \emph{u satisfies the nonfocusing condition of order $k$
  relative to $\hilbert$ on $K$ if there exists $Q \in
  \Psieb{0,0}(M),$ elliptic on $K,$ such that $$Q u \in \coiso^k(\hilbert).$$}

We say that $u$ is nonfocusing resp.\ coisotropic of order $k$
relative to $\hilbert$ on an arbitrary open subset $S$ of $\fcal$ if
for every open $O \subset S$ with closure disjoint from
$\fcal_{\sing},$ it is nonfocusing resp.\ coisotropic on $O$ of order
$k$ with respect to $\hilbert.$

We say that \emph{u satisfies the nonfocusing condition
  relative to $\hilbert$ on $K$} (without specifying an order) if $u$
satisfies the nonfocusing condition of \emph{some} order $k \in \NN.$
\end{definition}

\begin{remark}\begin{enumerate}\
\item
$u$ is coisotropic on $K$ if and only if $u$ is coisotropic at every $p\in K$,
i.e.\ on $\{p\}$ for every $p\in K$. This can be seen by a partition
of unity and a microlocal elliptic parametrix construction, as usual.
\item The conditions of coisotropic regularity and nonfocusing should be,
loosely speaking, considered to be dual to one another; a precise
statement to this effect appears in the proof of
Theorem~\ref{theorem:geometric} below.
\item Coisotropy and nonfocusing are only of interest on
  $\fcal_{\reg}$ itself: away from this set, to be coisotropic of order
  $k$ with respect to $\hilbert$ means merely to be microlocally in
  $H^k_{\ebo, \hilbert}$ while to be nonfocusing means to be
  microlocally in $H^{-k}_{\ebo, \hilbert}.$
\item
Certainly, away from $\ef,$ $\sigma(\Box)$ vanishes on $\fcal_{\reg},$
as the latter lies in the characteristic set $\Sigma$ by definition.
Splitting $\Sigma$ into component according to the sign of $\tau,$ and
letting $\Pi_\pm$ be pseudodifferential operators over $M^\circ$
microlocalizing near each of these components, we
thus have
$$
\Box \Pi_{\pm} = Q_{\pm} A_{\pm}+R
$$
with $A_\pm$ in $\module,$ $Q_{\pm}$ elliptic of order $1,$ and $R$
smoothing.
\end{enumerate}
\end{remark}

>From Lemma~\ref{lemma:moduleclosed}, we obtain the following.
\begin{corollary}
If $u$ is coisotropic of order $k$ on $K$ relative to $\hilbert$ then
there exists $U$ open, $K\subset U$ such that for $\tilde Q\in\Psieb{0,0}(M)$,
$\WF'(\tilde Q)\subset
U$ implies $\tilde Q A^\alpha u\in\hilbert$ for $|\alpha|\leq k$.

Conversely, suppose $U$ is open and
for $\tilde Q\in\Psieb{0,0}(M)$,
$\WF'(\tilde Q)\subset
U$ implies $\tilde Q A^\alpha u\in\hilbert$ for $|\alpha|\leq k$.
Then for $K\subset U$, $u$ is coisotropic of order $k$ on $K$ relative
to $\hilbert$.
\end{corollary}

\begin{proof}
Suppose first that
$u$ is coisotropic of order $k$ on $K$ relative to $\hilbert$.
By definition, there exists $Q$ elliptic on $K$ such that $\coiso^k Qu
\subset\hilbert$.
Let $U$ be such that $Q$ is elliptic on $\overline{U}$, $K\subset U$,
and let $S\in\Psieb{0,0}(M)$
be a microlocal parametrix for $Q$, so $\WF'(R)\cap\overline{U}
=\emptyset$ where $R=SQ-\Id$.

We prove the corollary by induction, with the case $k=0$ being immediate
as one can write $\tilde Qu=\tilde QSQu+\tilde Q Ru$,
$\tilde QS\in\Psieb{0,0}(M)$ is bounded on $\hilbert$, $Qu\in\hilbert$,
$\tilde Q R\in\Psieb{-\infty,0}(M)$ (for they have disjoint $\WF'$),
so $\tilde QR u\in\hilbert$.

Suppose now that $k\geq 1$, and the claim has been proved for $k-1$. By
Lemma~\ref{lemma:Psi-rewrite}, applied with $Q_n=Q$ (i.e.\ there is no
need for the subscript $n$, or for uniformity),
$$
A^\alpha Q=Q A^\alpha +\sum_{|\beta|\leq|\alpha|-1} C_\beta A^\beta.
$$
Thus, for $|\alpha|=k$,
$$
\tilde Q A^\alpha u=\tilde Q SQ A^\alpha+\tilde QR A^\alpha u
$$
and
$$
\tilde Q SQ A^\alpha=\tilde Q S A^\alpha Q
-\sum_{|\beta|\leq|\alpha|-1} \tilde Q S C_\beta A^\beta
$$
together with the induction hypothesis (due to which and to
$\tilde Q S C_\beta\in\Psieb{0,0}(M)$ with $\WF'(\tilde Q S C_\beta)
\subset U$, $\tilde Q S C_\beta A^\beta u\in\hilbert$) and
$\tilde Q R\in\Psieb{-\infty,0}(M)$ imply
$\tilde Q A^\alpha u\in\hilbert$, providing the inductive step.

The proof of the converse statement is similar.
\end{proof}

We now set
$$
\hilberth=L^2_g(I\times X_0)
$$
where $I$ is a compact interval.
We additionally introduce another Hilbert space $\hilbertx\subset
\hilberth,$ given by $H^1_0(I\times X_0)$ or $H^1(I \times X_0)$ with $I$ an interval
and the $0$ denoting vanishing at $I\times \pa X_0.$  Note that
$\Id +\Lap: \hilbertx \to \hilbertx^*$ is an isometry.

Suppose $K$ is compact. For $N\geq k+r$ we let $\hilbertp_K$ denote the
subspace of $\hilberth$
\begin{equation*}\begin{split}
\hilbertp_K=\{u\in\hilbertx^*:\ \WFbHp^N(u)\subset K,
\ u\ \text{is coisotropic of order}&\ k \text{ w.r.t. }\\
&H^r_{\bo,\hilbertx^*}
\text{ on } K\}.
\end{split}\end{equation*}
Let
\begin{equation*}\begin{split}
\hilbertz_K=\{\phi\in \hilbertx:
u\ \text{is coisotropic of order}&\ k \text{ w.r.t.}\ H^r_{\bo,\hilbertx}
\text{ on } K\}.
\end{split}\end{equation*}
Also, for $\nu\in\RR$, we choose a family of operators
for adjusting orders; we let
$$
T_{\nu}\in\Psieb{\nu,0}(M)
$$
be (globally) elliptic of order
$\nu$. Thus, $T_\nu$ are simply weights. Later, in \eqref{eq:refine-weight},
we make a slightly more specific choice.

\begin{lemma}\label{lemma:coisoinclusion}
Suppose that $K\subset O$, $K$ compact, $O$ open with compact closure, and $Q\in\Psi^0(M)$
such that $\WF'(\Id-Q)\cap K=\emptyset$, $\WF'(Q)\subset O$.
Let
\begin{equation}\label{eq:K-nbhd-coiso}
\hilbertp=\{u\in\hilbertx^*:\ T_N(\Id-Q)u\in\hilbertx^*,\ |\alpha|\leq k\Rightarrow
T_r A^\alpha Qu\in
\hilbertx^*\}
\end{equation}
and
\begin{equation}\label{second:eq:K-nbhd-coiso}
\hilbertz=\{u\in\hilbertx:\ |\alpha|\leq k\Rightarrow
T_r A^\alpha Qu\in
\hilbertx\}
\end{equation}
Then
$$
\hilbertp_K\subset\hilbertp\subset\hilbertp_{\overline{O}}.
$$
and
$$
\hilbertz_K\subset\hilbertz\subset\hilbertz_{\overline{O}}.
$$
\end{lemma}
\begin{proof}
If $u\in \hilbertp_K,$ then $\WFbHp^N(u)\subset K$ implies that $T_N
(\Id-Q) u \in \hilbertx^*.$ Moreover, since $u$ is coisotropic on $K,$ it
is coisotropic on a neighborhood $O'$ of $K;$ we construct $Q' \in
\Psi^0(M)$ with $\WF Q' \subset O',$ $\WF (\Id-Q') \cap K=\emptyset.$
Then
$$
T_r A^\alpha Q u = T_r A^\alpha Q'Qu + T_r A^\alpha (\Id-Q') Qu,
$$
and the first term is in $\hilbertx^*$ by coisotropy of $u$ on $O'$ while
the latter is in $\hilbertx^*$ by the wavefront condition on $u.$

On the other hand, if $u \in \hilbertp,$ we have $T_N
(\Id-Q)u\in\hilbertx^*$ hence $\WFbHp^N (u) \cap \liptic
(\Id-Q)=\emptyset,$ so in particular, $\WFbHp^N(u) \subset O^c,$ since
$\Id-Q$ must be elliptic on $O^c.$ It remains, given $p \in
\overline{O},$ to check coisotropic regularity at $p.$ If $p \in
\liptic (\Id-Q),$ it again follows from the wavefront set condition,
hence it suffices to consider $p \in \liptic Q\supset [\liptic
(\Id-Q)]^c;$ at such points coisotropic regularity follows from
$T_r A^\alpha Q u \in \hilbertx^*.$

The proof for $\hilbertz$ works analogously.
 \end{proof}

\begin{corollary}
Suppose $K=\cap_j O_j$, $O_j$ open with compact closure, $\overline{O_{j+1}}
\subset O_j$. Let $\hilbertp_j,$ $\hilbertz_j$
be given by \eqref{eq:K-nbhd-coiso}, \eqref{second:eq:K-nbhd-coiso} where $Q_j$ satisfies
$\WF'(\Id-Q_j)\cap K=\emptyset$, $\WF'(Q_j)\subset O$. Then
$\hilbertp_K=\cap_j\hilbertp_j,$ $\hilbertz_K =\cap_j \hilbertz_K.$

In particular, $\hilbertp_K$ and $\hilbertz_K$ become Fr\'echet spaces when equipped
with the $\hilbertp_j,$ $\hilbertz_j$ norms.
\end{corollary}

\begin{remark}
It is easy to see that the Fr\'echet topology is independent of the choice
of the particular $O_j$.
\end{remark}

\begin{proof}
The fact that $\hilbertp_K\subset\cap_j\hilbertp_j$ follows from
Lemma~\ref{lemma:coisoinclusion}. For the reverse inequality,
note that $u\in\cap\hilbertp_j\subset\cap\hilbertp_{\overline{O_j}}$
has $\WFbHp(u)\subset\cap\overline{O_j}=K$. On the other hand,
as $u\in\hilbertp_1$, $|\alpha|\leq k\Rightarrow
A^\alpha Q_1u\in
\hilbert$ and $Q_1$ is elliptic on $K$. Thus, $u\in\hilbertp_K$.

The same holds for $\hilbertz_K.$
\end{proof}

We now note the following functional-analytic facts:
\begin{lemma}\label{lemma:duals}
Let $Q$ be as above, and again let
\begin{equation}\label{eq:dual-coiso}
\hilbertp=\{u\in\hilbertx^*:\ T_N(\Id-Q)u\in\hilbertx^*,\ |\alpha|\leq k\Rightarrow
T_r A^\alpha Qu\in
\hilbertx^*\},
\end{equation}
and
\begin{equation*}
\hilbertz=\{u\in\hilbertx:\ |\alpha|\leq k\Rightarrow
T_r A^\alpha Qu\in
\hilbertx\}
\end{equation*}
Then the dual of $\hilbertp$ with respect to the space $\hilberth$
(see Appendix~\ref{appendix:functional}) is
$$
\hilbertp^*=\{u:\ u=v_0+ T_N(\Id-Q)v_1 +
\sum_{\abs{\alpha}\leq k} T_r A^\alpha Qv_\alpha,\ v_0,v_1, v_\alpha \in
\hilbertx\},
$$
and the dual of $\hilbertz$ with respect to $\hilberth$ is
$$
\hilbertz^*=\{u:\ u=v_0+
\sum_{\abs{\alpha}\leq k} T_r A^\alpha Qv_\alpha,\ v_0, v_\alpha \in
\hilbertx^*\},
$$
\end{lemma}

\begin{proof}
First consider the dual of $\hilbertz$ with respect to $\hilberth$.
We apply the discussion of Appendix~\ref{appendix:functional} leading
to \eqref{eq:hilbertp-dual}. More precisely, with the notation of
the Appendix, we take
$\hilberth=L^2_g(I\times X_0)$, and $\hilbertx=H^1(I\times X_0)$, resp.
\ $\hilbertx=H^1_0(I\times X_0)$, as set out earlier. We also let
$\cD=\CI(I\times X_0)$, resp.\ $\cD=\dCI(I\times X_0)$ (with the dot
indicating infinite order vanishing at $I\times\pa X_0$).
We define the
operators $B_k$ in \eqref{eq:hilbert-hilberth-ip} as follows:
we take $B_k$, $k=1,\ldots,N$,
to be a collection of $\CI$ vector fields on $X_0$ which span $\Vf(X)$
over $\CI(X)$, $B_0=\Id$, and define the
$\hilbertx$-norm on $\cD$ by
\begin{equation}\label{eq:hilbertx-norm}
\|u\|_{\hilbertx}^2=\|B_0 u\|^2_{\hilberth}+\sum_{k=1}^N\|B_k u\|^2_{\hilberth},
\end{equation}
cf.\ \eqref{eq:hilbertpp-norm}; then $\hilbertx$ is the completion of
$\cD$.
Then we take the collection of $A_j$ in defining the space $\hilbertp$
in Appendix~\ref{appendix:functional}, with the norm \eqref{eq:hilbertp-norm},
to be $T_r A_\alpha Q$, $|\alpha|\leq k$. Then our claim about
$\hilbertz^*$ follows from \eqref{eq:hilbertp-dual} and
\eqref{eq:A-formal-adjoint}, taking into account that the principal
symbol of the conjugate of a pseudo-differential operator
by complex conjugation is the complex conjugate of the principal symbol
of the original operator, so its vanishing on $\fcal_{\reg}$ is unaffected.

We now consider the dual of $\hilbertp$, with $\hilbertp$ given by
\eqref{eq:dual-coiso}.
As $\Id+\Delta:\hilbertx\to\hilbertx^*$ is an isomorphism,
the norm on $\hilbertx^*$ is given by
$$
\|u\|_{\hilbertx^*}=\|(\Id+\Delta)^{-1}u\|_{\hilbertx}=\sum_{k=0}^N
\|B_k(\Id+\Delta)^{-1}u\|^2_{\hilberth},
$$
with $B_k$ as in \eqref{eq:hilbertx-norm},
we are again in the setting leading to \eqref{eq:hilbertp-dual} with
$\hilbertx$ in the Appendix given by our $\hilbertx^*$, the $B_k$ in
the Appendix given by $B_k(\Id+\Delta)^{-1}$, the space $\hilbertp$
in the Appendix being our space $\hilbertp$ in \eqref{eq:dual-coiso},
and the $A_j$ in the appendix given by $T_N(\Id-Q)$ and $T_r A^\alpha Q$,
$|\alpha|\leq k$. Then our claim about $\hilbertp^*$ follows from
\eqref{eq:hilbertp-dual} and \eqref{eq:A-formal-adjoint}.
\end{proof}

Now let
$$
\widetilde\hilbertp_K=\{u\in T_N(\hilbertx):\ u\ \text{is nonfocusing
  of order}\ k\ \text{ w.r.t. } H^{-r}_{\bo, \hilbertx} \text{ on } K\},
$$
$$
\widetilde\hilbertz_K=\{u\in \hilbertx^*:\ u\ \text{is nonfocusing
  of order}\ k\ \text{ w.r.t. } H^{-r}_{\bo,\hilbertx^*} \text{ on } K\}.
$$
\begin{lemma}\label{lemma:nfinclusions}
Define $\hilbertp,$ $\hilbertz$ as above.  Then
$$
\widetilde\hilbertp_K\subset\hilbertp^*\subset\widetilde\hilbertp_{\overline{O}}.
$$
and
$$
\widetilde\hilbertz_K\subset\hilbertz^*\subset\widetilde\hilbertz_{\overline{O}}.
$$
\end{lemma}
The proof follows that of Lemma~\ref{lemma:coisoinclusion} closely,
using the characterization of $\hilbertp^*$ and $\hilbertz^*$ from
Lemma~\ref{lemma:duals}.

\

We remark that away from $\ef,$ we may always (locally) conjugate by
an FIO to a convenient normal form: being coisotropic, locally $\fcal$
can be put in a model form $\zeta=0$ by a symplectomorphism $\Phi$ in
some canonical coordinates $(y,z,\eta,\zeta)$, see
\cite[Theorem~21.2.4]{Hormander3} (for coisotropic submanifolds one
has $k=n-l$, $\dim S=2n$, in the theorem).  We may moreover arrange
the $(z,\zeta)$ coordinates (i.e.\ apply a further symplectomorphism) so that 
$\sigma(\Box) \circ \Phi=q\zeta_1$ for some symbol $q$ elliptic in a
small open set.  We now quantize $\Phi$ to
a FIO $T,$ elliptic on some small neighborhood of a
$w\in \fcal_{\reg},$ which can be arranged to have the following
properties:
\begin{itemize}
\item
$\displaystyle T\Box = Q D_{z_1}T+R$ where $Q \in \Psi^1(M^\circ)$
is elliptic near $\Phi(w)$ and $R$ is a smoothing operator.
\item
$u$ has coisotropic regularity of order $k$ (near $w$) with respect
to $H^s$ if and only if $D_z^\alpha Tu \in H^s$ whenever $\abs{\alpha}<k.$
\item
$u$ is nonfocusing of order $k$ (near $w$) with respect
to $H^s$ if and only if $Tu \in \sum_{\abs{\alpha}\leq k} D_z^\alpha H^s.$
\end{itemize}
Let $G\in\Psi^{-1}(M^\circ)$ be a parametrix for $Q$.
As a consequence of the above observation,
$\Box u =f$ implies that $D_{z_1} Tu -GTf\in \CI$
microlocally near $\Phi(w),$ and if $f$ is coisotropic of order $k$
relative to $H^{s-1}$, then
$D_z^\alpha GTf\in H^{s}$ for $\abs{\alpha}\leq k$ (with an analogous
statement for non-focusing)
hence we have now sketched the proof of
the following:

\begin{proposition}\label{prop:coisotropic-prop}
Suppose $u$ is a distribution on $M^\circ$, $\Box u=f$.
If $f$ is coisotropic of
order $k$, resp.\ nonfocusing of order $k$,
with respect to $H^{s-1}$ then the coisotropic regularity of order
$k$, resp.\ non-focusing regularity of order $k$, with respect to $H^s$,
is invariant under the Hamilton flow over $M^\circ$.

In particular, for a solution to the wave equation, coisotropic regularity
of order $k$ with respect to $H^s$ and nonfocusing of order $k$ with respect
to $H^s$ are invariant under the Hamilton flow over $M^\circ.$
\end{proposition}

(We remark that one could certainly give an alternative proof of this
proposition by positive commutator arguments similar to, but much
easier than, those used for propagation of edge regularity in the
following section.)

\begin{corollary}\label{cor:inhomog-coiso}
Suppose that $f$ is coisotropic, resp.\ non-focusing, of order $k$ relative
to $H^{m-1}$, supported in $t>T$. Let $u$ be the unique solution of $\Box u=f$
with Dirichlet or Neumann boundary conditions, supported in $t>T$.
Then $u$ is coisotropic, resp.\ non-focusing, of order $k$ relative
to $H^m$ at $p\in S^*M^\circ$ provided every\footnote{The restriction of
this \GBB to $[0,s_0]$, if $s_0\geq 0$, or $[s_0,0]$ if $s_0<0$,
is unique under
these assumptions.} \GBB $\gamma$
with $\gamma(0)=p$ has the property that there exists $s_0$ such that
$t(\gamma(s_0))<T$, and for $s\in [0,s_0]$ (or $s\in[s_0,0]$, if $s_0<0$),
$\gamma(s)\in S^*M^\circ$.

The analogous statements hold if $f$ is supported in $t<T$, and $u$
is the unique solution of $\Box u=f$ supported in $t<T$, provided
we replace $t(\gamma(s_0))<T$ by $t(\gamma(s_0))>T$.
\end{corollary}

\begin{proof}
This is an immediate consequence of Proposition~\ref{prop:coisotropic-prop},
taking into account that $u$ is coisotropic, resp. non-focusing, in $t<T$,
by virtue of vanishing there.
\end{proof}

If $K\subset M^\circ$ is compact, then there is $\delta>0$ such that
if $p\in S^*_K M^\circ$ and $\gamma$ is a \GBB with $\gamma(0)=p$,
then for $s\in (-\delta,\delta)$, $\gamma(s)\in M^\circ$. As $s$ is equivalent
to $t$ as a parameter along \GBB, we deduce the following result.

\begin{corollary}\label{corollary:coisocauchy}
Suppose $K\subset M^\circ$ is compact.
Suppose that $f$ is co\-isotropic, resp.\ non-focusing, of order $k$ relative
to $H^{m-1}$, supported in $t>T$. Let $u$ be the unique solution of $\Box u=f$
with Dirichlet or Neumann boundary conditions, supported in $t>T$.
Then there exists $\delta_0>0$ such that
$u$ is coisotropic, resp.\ non-focusing, of order $k$ relative
to $H^m$ at $p\in S^*_K M^\circ$ if $t(p)<T+\delta_0$.
\end{corollary}

Of course, what happens to coisotropic regularity and nonfocusing when
bicharacteristics reach $\ef$ is of considerable interest, and will be
discussed below.

\section{Edge propagation}\label{section:edge-propagation}

This section contains a series of theorems that will enable us to
track propagation of regularity into and back out of the edge.  They
are as follows:
\begin{itemize}
\item Theorem~\ref{theorem:edge}, which governs propagation of
  regularity into and out of the interior or $\ef$ as well as the
  microlocal propagation of coisotropic regularity there (i.e.\
  iterated regularity under application of operators in $\coiso$).
\item Theorem~\ref{thm:bdy-radial}, which governs propagation of
  regularity into $\ef$ along glancing rays, tangent to one or more of
  the boundary faces meeting at $\RR\times Y$ (in the blown-down
  picture).
\item Theorem~\ref{theorem:edgehyperbolic}, which governs the
  propagation of edge regularity at non-radial
  hyperbolic points at the boundary of the edge $\ef.$
\item Theorem~\ref{theorem:edgeglancing}, which governs the
  propagation of edge regularity at glancing points at the boundary of the edge $\ef.$
\end{itemize}
These theorems will then be assembled (together with the propagation
over the interior of the edge, which we may simply quote from \cite{mvw1})
to yield the propagation of
coisotropic regularity into and out of the edge in
Theorem~\ref{theorem:microlocal-coiso}, and this result is the key
ingredient in proving the ``geometric'' improvement in regularity on
the diffracted wave.

\subsection{Radial points in the interior of the edge}
The following theorem enables us to track edge wavefront set entering
and leaving the edge at radial points over its interior. Since we are
working locally (even within the fibers!)
over the interior of the edge, i.e.\ over $\ef^\circ$, we can use
edge, edge-b and edge-smooth objects interchangably, for the only
boundary in this region is the edge itself.

\begin{theorem}[Propagation at radial points in the interior of the
  edge](See \cite[Theorem~11.1]{mvw1})\label{theorem:edge}
Let $u \in \Hes{1,l}(M)$ solve
$\Box u=0$ with Dirichlet or Neumann boundary conditions.

\begin{enumerate}
\item
Let $m>l+f/2.$ Given $\alpha\in\cH_{W,\bo}$, and
$p \in \cR_{\ebo,\alpha,I}^\circ,$ 
if $(\fcal_{I,p}\backslash{\pa M})
\cap \WF^m Au=\emptyset,$ for all $A \in \coiso^k$ then $p \notin
\eWF^{m,l'} B u$ for all $l'<l$ and all $B \in \coiso^k.$
\item
Let $m<l+f/2.$ Given $\alpha\in\cH_{W,\bo}$, $p \in \cR_{\ebo,\alpha,O}^\circ,$
if a neighborhood $U$ of $p$ in
$\Sestar|_{\pa M}M$ is such that $\eWF^{m,l}(Au)\cap U\subset\pa\fcal_O$
for all $A \in \coiso^k$ then
$p \notin \eWF^{m,l}(Bu)$ for all $B \in \coiso^k.$
\end{enumerate}
\end{theorem}
This theorem is literally the same theorem as
\cite[Theorem~11.1]{mvw1} as we are restricting our attention to the
interior of the fibers, hence the presence of a fiber boundary in our
setting is irrelevant.  We thus refer the reader to \cite{mvw1} for
the proof.
\begin{remark}
In fact, we could take $u \in \Hes{-\infty, l}(M)$ here, but the
restriction on regularity will be necessary in later theorems to
maintain the boundary condition at the side faces $z'_i=0,$ and we
prefer to keep a uniform hypothesis. The boundary conditions
are irrelevant here; again, they are stated for the sake of uniformization.
\end{remark}

\subsection{Propagation into radial points over the boundary of the edge}\label{subsection:radialboundary}

We now turn to the question of propagation into the edge at
glancing points, i.e.\ at points over the boundary of the fibers of
$M.$  Note that the hypotheses of this theorem are \emph{global} in
the boundary of the fiber: we do not attempt to distinguish different
points in the fiber boundary.

\begin{theorem}[Propagation into radial points over the boundary of the
  edge]\label{thm:bdy-radial}
Let $u \in \hilbert \equiv \Hes{1, l}(M)$ solve
$\Box u=0$ with Dirichlet or Neumann boundary conditions
(see Definition~\ref{def:Dirichlet-Neumann}).

Let $m>l+f/2-1$
with $m\geq 0.$
Suppose that $q \in \cH_{W,\bo}$ and
there exists a neighborhood $U$ of
$$
\rcal_{\eb,q,I}\cap\Sebstar[\pa\ef] M=\rcal_{\eb,q,I}\cap\cG
$$
in $\Sebstar M$
such that $\{x>0\} \cap U \cap \WFebX^{m,*} u
=\emptyset.$ Then
$$
\rcal_{\eb,q,I}\cap\Sebstar[\pa \ef] M\cap \WFebX^{m,l'} u =
\emptyset
$$
for all $l'<0.$
\end{theorem}

\begin{proof}
Choose local coordinates on $W$, and let $q=(y_0,t_0,\tauh_0 \in \{\pm
1\}, \etah_0)\in\cH_{W,\bo}.$
Choose $\xi_0$ such that $\xiebh_0^2=1-h(y,\etah_0)$ with
$\sgn\xiebh_0=\sgn\taueb_0$ (this is the incoming point).

One of the central issues in proving the theorem is to construct a
symbol that is localized in the hypothesis region that is sufficiently
close to being flow-invariant.  To begin, we will need a localizer
in the fiber variables.  Fix any $K \Subset \ef^\circ$ and fix a small
number $\ep_K>0.$
Let
$$
\Upsilon: \Sestar[K] (M)\cap \{\abs{\zetaes}/\abs{\xies}<\ep_K\} \to Z
$$
be locally defined by
$$
\Upsilon (q')=z(\exp_{z,\zetaesh} s_\infty \sH_{\eso}), \quad s_\infty =
\frac{\sgn \xiesh}{\abs{\zetaesh}_{K_q}} \arctan\frac{\abs{\zetaesh}}{\xiesh}.
$$
where $q' \in \Sesstar_{K} (M)$ has ``edge-smooth'' coordinates
$(t,y,z,\tau,\xi,\eta,\zeta)$ (we are using the canonical
identification of $\Sestar M$ with $\Sesstar M$ away from $\pa M\backslash
\ef$). This map is well-defined provided $\ep_K$ is chosen sufficiently
small (so that the flow stays away from $\pa \ef$).  The map simply takes a
point over the boundary to its limit point in the fiber variables along the
forward bicharacteristic flow, hence on $\ef$, we certainly have
$\Upsilon_* (\sH_{\eso}) = 0.$

We now employ $\Upsilon$ to create a localizer away from $\pa \ef.$ Fix
$$
K''\subset U'
\subset K' \subset U
\subset K \subset \ef^\circ$$
with $K'',K', K$ compact and $U',U$ open such
that
\begin{enumerate}
\item
$\alpha\in\Sesstar_{K\setminus U}M$ and
$\abs{\zetaesh(\alpha)/\xiesh(\alpha)}<\ep_K$ imply $\Upsilon_K(\alpha)\in \ef
\setminus K'$,
\item
$\alpha\in\Sesstar_{K''}M$ and
$\abs{\zetaesh(\alpha)/\xiesh(\alpha)}<\ep_K$ imply $\Upsilon_K(\alpha)\in U'$.
\end{enumerate}
Now let $\chi\in\CI(\ef)$ be equal to $0$ on
$\ef\backslash K'$ and $1$ on $U'$. For $\ep_K$ sufficiently small,
$\chi\circ\Upsilon_K$ vanishes on $\Sesstar_{K\setminus U}M$, hence can
be extended as $0$ to $\Sesstar_{\ef\setminus U}M$ to define a $\CI$
function. Thus, this extension of $\chi\circ\Upsilon_K$ 
is well defined and smooth on
$$
\Big\{\abs{\zetaesh/\xiesh}<\ep_K,\ x<\ep_K\Big\}\cup\Sesstar_{\ef\setminus U}M
\subset \Sesstar M;
$$
it equals $0$
on the fibers over $\ef\backslash U$ and $1$ on those over $K''.$
But
$$
\Big\{\abs{\zetaebh/\xiebh}<\ep'_K,\ x<\ep'_K\Big\}\cup\Sesstar_{\ef\setminus U}M
\subset
\Big\{\abs{\zetaesh/\xiesh}<\ep_K,\ x<\ep_K\Big\}\cup\Sesstar_{\ef\setminus U}M
$$
for $\ep'_K>0$ sufficiently small as on $U$, $\abs{\zetaesh/\xiesh}
\lesssim \abs{\zetaebh/\xiebh},$ since $\abs{z'_i}$ are all bounded
away from $0$ there.  Due to the vanishing near $\pa\ef$, we can
equivalently regard this extension of $\chi\circ\Upsilon_K$ as a $\CI$
function $\rho$ on the following subset of the edge-b cosphere bundle:
$$
\Big\{\abs{\zetaebh/\xiebh}<\ep'_K,\ x<\ep'_K\Big\}\cup\Sebstar[\ef\setminus U]M.
$$
The resulting
function satisfies
\begin{equation}\label{flow-invariance}
\sH_{\eso} (\rho) = O(x)\quad \text{ on }
\Big\{\abs{\zetaebh/\xiebh}<\ep'_K,\ x<\ep'_K\Big\}\cup\Sebstar[\ef\setminus U]M.
\end{equation}
and
\begin{equation}\label{support}
\rho = 0 \text{ on } \Sebstar[\ef\setminus U]M.
\end{equation}
It is convenient to extend $\rho$ to all of $\Sebstar M$ by defining it to
be an arbitrary fixed positive constant, say $1$, where it is not previously
defined. Note that by \eqref{support}, when we need to calculate derivatives
of $\rho$ in a commutator calculation, we may always assume that we are
away from $\pa\ef$, hence use the edge-calculus Hamilton vector field result.

Now consider the function
$$ \omega = 
\abs{\etaebh-
\etaebh_0}^2+ \abs{y-y_0}^2+ \rho^2+ \abs{t-t_0}^2.
$$
(Note that keeping $\omega,$ $|\xi-\xi_0|$ and $x$ sufficiently small on
$\Sigma_{\ebo}$ automatically means that $\hat\zeta$ is small as well.)

We now identify some appropriate neighborhoods in which to localize.
First, choose $\ep_0,\ep_1<1$ such that
$$
\lvert \hat\zeta \rvert^2 <\ep_0,\ x<\ep_1,\ \omega <\ep_1,
\ |\xi-\xi_0|^2<\ep_1 \Longrightarrow 
\lvert \zetaebh/\xiebh \rvert<\ep'_K/2.
$$
Second, choose $\ep_2<\ep_1$ such that
$$
\Sigma_{\ebo} \cap\{x<\ep_2\}\cap \{\omega<\ep_2\}\cap
\big\{|\xiebh-\xiebh_0|^2<\ep_2\big\}
\subset \Big\{ \lvert \hat\zeta \rvert^2<\frac{\ep_0}2\Big\}.
$$
Let
$$
\tilde K=\{x\leq \ep_2\}\cap \{\omega\leq \ep_2\}\cap
\big\{|\xiebh-\xiebh_0|^2\leq \ep_2\big\}\cap
\Big\{ \lvert \hat\zeta \rvert^2\leq \frac{\ep_0}2\Big\}.
$$
Next, given $\delta>0$, which will depend on $\tilde K$, let
$U=U_\delta$ be as in Corollary~\ref{cor:glancingestimate}.
Finally, given any $\beta>0$ (to be specified below) we will choose
$\ep=\ep(\beta,\delta)$ so that
\begin{equation}\label{makesure}
\ep, \ep(1+\beta) <\ep_2,
\end{equation}
and so that
\begin{equation*}
K_\ep= \big\{x\leq\ep,\ \omega\leq \ep(1+\beta),\ |\xiebh-\xiebh_0|^2\leq\ep,
\ |\zetaebh|^2\leq\ep_0\big\}\subset U=U_\delta.
\end{equation*}
(Note that $K_\ep\subset\tilde K$ by \eqref{makesure}.)

Let $\phi\in\CI_c([0,\ep)),$ $\psi_0\in\CI_c([0,\ep_0))$,
identically $1$ on $[0, \ep_0/2],$ $\psi_1\in\CI_c([0,\ep))$,
identically $1$ on $[0,\ep/2]$, $\psi\in\CI_c(
(-\infty,\ep))$, all non-increasing,
$$
a=a_\ep=
\abs{\taueb}^{s} x^{-r} \psi(\omega-\beta x)
\psi_1(|\xiebh-\xiebh_0|^2)
\phi(x)\psi_0(|\zetaebh|^2)
$$
Thus,
$$
x\leq\ep,\ \omega\leq \ep(1+\beta),\ |\xiebh-\xiebh_0|^2\leq\ep,
\ |\zetaebh|^2\leq\ep_0\ \text{on}\ \supp a_\ep.
$$
We usually suppress the $\ep$-dependence of $a$ below in our notation.
Equation~\eqref{makesure} ensures that $\ep(1+\beta)<1$ on $\supp a,$ so
$\rho<1$, and thus \eqref{flow-invariance} holds.  We have also
arranged that
$$
\abs{\zetaebh/\xiebh}<\ep'_K/2 \text{ on } \supp a
$$
and that $\psi_0(\lvert \hat\zeta \rvert^2)=1$ on $\supp (\psi(\omega-\beta x)
\phi(x)\psi_1(|\xiebh-\xiebh_0|^2))\cap \Sigma_{\ebo}.$  This latter observation means that we
need never consider derivatives falling on the $\psi_0$ term when
computing the action of the Hamilton vector field on $a.$
(The cutoff $\psi_0(\hat\zeta)$ is therefore not necessary for correct
localization of $a,$ as that is achieved by the cutoffs in $\omega,$ $\xiebh$
and $x$ if we restrict our attention to $\Sigma_{\ebo};$ rather this
is necessary to make $a$ a symbol, which it would not be if
independent of $\zeta.$)

We quantize $a$ to $A \in\Psieb{s,r}(M),$ i.e.\ take any $A$ with
$\sigma_{\ebo,s}(A)=a$.
By Lemma~\ref{lemma:Box-eb-comm},
\begin{equation}\label{eq:Box-comm-rad}
\imath[\Box,A^*A]=\sum Q_i^* L_{ij} Q_j+\sum (x^{-1}
L_i Q_i+Q_i^* x^{-1}L'_i)
+x^{-2}L_0,
\end{equation}
with
\begin{equation}\begin{split}\label{eq:Box-comm-rad-terms}
&L_{ij}\in\Psieb{2s-1,2r}(M),\ L_i\in\Psieb{2s,2r}(M),
\ L_0\in\Psieb{2s+1,2r}(M),\\
&\sigma_{\ebo,2s-1}(L_{ij})=2aV_{ij}a,
\ \sigma_{\ebo,2s}(L_i)=2a V_i a,
\ \sigma_{\ebo,2s+1}(L_0)=2a V_0a,\\
&\ebWF'(L_{ij}),\ebWF'(L_i),\ebWF'(L'_i),\ebWF'(L_0)\subset\ebWF'(A),
\end{split}\end{equation}
with $V_{ij}$, $V_i$ and $V_0$ smooth vector fields tangent on $\Tebstar M$
tangent to $\ef$ and such that
for $f\in\CI(\Sebstar M)$ with $f|_{\ef}=\dbetaeb^* \phi$
for some $\phi\in S^* W$,
\begin{equation}\label{eq:basic-comm-2}
V_{ij}f|_{\ef}=0,\ V_i f|_{\ef}=0,\ V_0 f|_{\ef}=0.
\end{equation}
In view of Corollary~\ref{cor:glancingestimate},
we are led to regard the $L_{ij}$ and
$L_i$ terms as negligible, {\em provided that} their principal
symbol is bounded by a constant multiple of $\sigma_{\ebo,s+1}(L_0)$
times the appropriate power of $|\tau|$ (to arrange homogeneity of
the same degree).
Also by Lemma~\ref{lemma:Box-eb-comm},
\begin{equation}\begin{split}\label{eq:xiebh-weights2}
&\big(|\tau|V_0\xiebh\big)|_{\ef}=-2\sum_{ij}k_{2,ij}(0,y,z)\zeta''_i\zeta''_j\\
&\big(|\tau|V_i\xiebh\big)|_{\ef}=-\sum_j k_{3,ij}(0,y,z)\zeta''_j,
\ \big(|\tau|V_{ij}\xiebh\big)|_{\ef}=-2 k_{1,ij}(0,y,z),\\
&\big(|\tau|^{-s'-1}x^{-r'}V_0(|\tau|^{s'} x^{r'})\big)|_{\ef}=-2(r'+s')\xiebh,
\ \big(|\tau|^{-s'}x^{-r'}V_i(|\tau|^{s'} x^{r'})\big)|_{\ef}=0,\\
&\big(|\tau|^{-s'+1}x^{-r'}V_{ij}(|\tau|^{s'} x^{r'})\big)|_{\ef}=0,
\end{split}\end{equation}
In particular, with $s'=0$, $r'=1$, $|\tau|^{-1}V_0 x=-2\xiebh x$, while
$V_i x, V_{ij}x$ are $O(x^2)$.

In computing $Va$ for various homogeneous degree $\mu-1$
vector fields $V$ on $\Tebstar M$,
we will employ the following arrangement of terms:
\begin{equation*}\begin{split}
Va=&\psi(\omega-\beta x)\phi(x)\psi_0(|\zetaebh|^2)
\psi_1(|\xiebh-\xiebh_0|^2)
V(\abs{\taueb}^{s} x^{-r})\\
&+\abs{\taueb}^{s} x^{-r}
\psi_0(|\zetaebh|^2)\phi(x)\psi_1(|\xiebh-\xiebh_0|^2)\psi'(\omega-\beta x)
(V\omega-\beta Vx)\\
&+\abs{\taueb}^{s} x^{-r}\psi(\omega-\beta x)\psi_0(|\zetaebh|^2)
\psi_1(|\xiebh-\xiebh_0|^2)V(\phi(x))\\
&+\abs{\taueb}^{s} x^{-r}\psi(\omega-\beta x)\psi_0(|\zetaebh|^2)
\phi(x)V(\psi_1(|\xiebh-\xiebh_0|^2))\\
&+\abs{\taueb}^{s} x^{-r}\psi(\omega-\beta x)\phi(x)\psi_1(|\xiebh-\xiebh_0|^2)
V(\psi_0(|\zetaebh|^2)).
\end{split}\end{equation*}
As $|\tau|^{-1} V_0 x=-2\xiebh x$ while
$|\tau|^{-1} V_0\omega=xf$ for
some $f\in\CI(\Sebstar X)$, and $|\xiebh|$ is bounded below on
$\tilde K$ (which is a compact subset of $\Sebstar M$), it follows
that
there exists $\beta>0$ such that $-(\sgn\tau_0)|\tau|^{-1}
(V_0\omega-\beta V_0 x)\geq x$
on $\tilde K$, and thus
\begin{equation}\label{eq:slow-var-comm}
-(\sgn\tau_0)|\tau|^{-1}(V_0\omega-\beta V_0 x)=xc_2^2
\end{equation}
for some smooth
positive function $c_2$ defined on $\tilde K$, hence on
a neighborhood of $\supp a$
in $\Sebstar M$.
Moreover,
$$
V_0\psi_1(|\xiebh-\xiebh_0|^2)=-4(\xiebh-\xiebh_0)
\psi_1'(|\xiebh-\xiebh_0|^2)\sum_{ij}k_{2,ij}(0,y,z)\zeta''_i\zeta''_j.
$$
Similar computations hold
for the $V_i$ and $V_{ij}$ terms, with result shown below in
\eqref{eq:xiebh-comm}.

We start by discussing the terms in
\eqref{eq:Box-comm-rad}-\eqref{eq:Box-comm-rad-terms} in which the vector
fields $V_0$, $V_i$, $V_{ij}$ differentiate $\psi_1(|\xiebh-\xiebh_0|^2)$.
These terms altogether have the form
\begin{multline}\label{eq:xiebh-comm}
\sum Q_i^* L_{ij,1} Q_j+\sum (x^{-1}L_{i,1} Q_i+Q_i^*x^{-1}L'_{i,1})
+x^{-2}L_{0,1},\ \text{where}\\
\sigma_{\ebo,2s-1}(L_{ij,1})|_{\ef}=-(\sgn\tau_0)a_1 k_{1,ij}(0,y,z),\\
\sigma_{\ebo,2s}(L_{i,1})|_{\ef}=\sigma_{\ebo,2s}(L'_{i,1})|_{\ef}
=-\frac{1}{2}(\sgn\tau_0)a_1\sum_j k_{3,ij}(0,y,z)\zeta''_j,\\
\sigma_{\ebo,2s+1}(L_{0,1})|_{\ef}=-(\sgn\tau_0)a_1 \sum_{ij}k_{2,ij}(0,y,z)\zeta''_i\zeta''_j,\\
\text{with } a_1=8a(\sgn\tau_0)(\xiebh-\xiebh_0)
\abs{\taueb}^{s} x^{-r}\psi(\omega-\beta x)\psi_0(|\zetaebh|^2)
\phi(x)\psi_1'(|\xiebh-\xiebh_0|^2).
\end{multline}
On $\tilde K\cap\supp \psi_1'(|\xiebh-\xiebh_0|^2)\cap\ebSigma$,
$\xiebh-\xiebh_0$ has sign $-\sgn\xiebh_0=-\sgn\tau_0$,
so $(\sgn\tau_0)(\xiebh-\xiebh_0)<0$ there.
Thus,
noting that the right hand side on the last line
is a square for $x$ sufficiently small in view of $\psi_1'\leq 0$
when $(\sgn\tau_0)(\xiebh_0-\xiebh)>0$, it has
the form
\begin{multline}
(-\sgn\tau_0)
d_Z^* x^{-1}C_0^*C_0 x^{-1} d_Z+E_0+E_0'+F_0,\\
C_0\in\Psieb{s-1/2,r}(M),\ E_0,E'_0\in\Diffess{2}\Psieb{2s-1,2r+2}(M)\\
\sigma_{\ebo,s-1/2}(C_0)=\left(H((\sgn\tau_0)(\xi_0-\xi))\tilde\phi(x)
a_1\right)^{1/2},\\
\WFeb'(E'_0)\cap\ebSigma=\emptyset,\ \WFeb'(E_0)\subset\{x>0\}\cap\supp a,\\
F_0\in \Diffess{2}\Psieb{2s-2,2r+2}(M),\ \WFeb'(E'_0),
\ \WFeb'(F_0)\subset\supp a,
\end{multline}
where $H$ is the Heaviside step function (recall that $\psi'_1=0$ near the
origin, and $\psi'\leq 0$) and $\tilde\phi\in\CI_c([0,\ep_2))$ is identically
$1$ near $0$ and has sufficiently small support.

Next, the terms in \eqref{eq:Box-comm-rad}-\eqref{eq:Box-comm-rad-terms}
in which the vector
fields $V_0$, $V_i$, $V_{ij}$ differentiate $\psi(\omega-\beta x)$ have the
form
\begin{equation*}\begin{split}
&\sum Q_i^* L_{ij,2} Q_j+\sum (x^{-1}L_{i,2} Q_i+Q_i^* x^{-1}L'_{i,2})
+x^{-2}L_{0,2},\ \text{where}\\
&\sigma_{\ebo,s-1}(L_{ij,2})=a_2^2 f_{ij,2},
\ \sigma_{\ebo,s}(L_{i,2})=\sigma_{\ebo,s}(L'_{i,2})= a_2^2 f_{i,2},\\
&\sigma_{\ebo,s+1}(L_{0,2})= -(\sgn\tau_0)a_2^2 c_2^2,\ c_2
\ \text{as in \eqref{eq:slow-var-comm}}\\
&a_2^2=-2xa
\abs{\taueb}^{s+1} x^{-r}\psi'(\omega-\beta x)\psi_0(|\zetaebh|^2)
\phi(x)\psi_1(|\xiebh-\xiebh_0|^2),
\end{split}\end{equation*}
with $f_{ij,2}$, $f_{i,2}$ smooth.
Moreover, terms in \eqref{eq:Box-comm-rad}--\eqref{eq:Box-comm-rad-terms}
in which the vector
fields $V_0$, $V_i$, $V_{ij}$ differentiate $\psi_0(|\zetaebh|^2)$ have
wave front set disjoint from $\ebSigma$ as already discussed, while the
terms in which these vector fields differentiate $\phi(x)$ are
supported in $\supp a\cap \{x>0\}$, where we will assume the absence of
$\WFebX^{s-1,*} u$ (the weight is indicated by an asterisk as we are
away from $x=0$, so it is irrelevant).

Finally, the terms in \eqref{eq:Box-comm-rad}--\eqref{eq:Box-comm-rad-terms}
in which the vector
fields $V_0$, $V_i$, $V_{ij}$ differentiate $\abs{\taueb}^{s} x^{-r}$
have the form
\begin{equation*}\begin{split}
&\sum Q_i^* L_{ij,3} Q_j+\sum (x^{-1}L_{i,3} Q_i+Q_i^* x^{-1}L'_{i,3})
+x^{-2}L_{0,2},\ \text{where}\\
&\sigma_{\ebo,s-1}(L_{ij,3})=a^2 xf_{ij,3},
\ \sigma_{\ebo,s}(L_{i,3})=\sigma_{\ebo,s}(L'_{i,3})=a^2 xf_{i,3},\\
&\sigma_{\ebo,s+1}(L_{0,3})=-a^2 (\sgn\tau_0)|\tau|2(s-r)c_3^2,
\end{split}\end{equation*}
where $c_3^2|_{\ef}=4(\sgn\tau_0)\xiebh >0$.

Finally, recall that terms with $\psi_0$ derivatives are supported in
the elliptic set of $\Box.$

We are now ready to piece together the above information to compute
the commutator $[\Box, A^* A].$  First we choose a family of operators
convenient for adjusting orders: pick
\begin{equation}\label{eq:refine-weight}
T_{\nu}\in\Psieb{\nu,0}(M),\ \sigma_{\ebo,\nu}(T_{\nu})
=|\tau|^{\nu}\ \text{near}\ \tilde K.
\end{equation}
\nomenclature[T]{$T_{\nu}$}{Family of elliptic operators of order $\nu$}
Thus, $T_\nu$ are simply weights, for $|\tau|^\nu$ is elliptic of order
$\nu$ on a neighborhood of $\tilde K$.

Adding all the terms computed above, and rearranging them as needed, noting the top
order commutativity in $\ebo$-order of $\Diffess{}\Psieb{}(M)$,
we finally deduce that
\begin{equation}\begin{split}\label{eq:Box-comm-rad-long}
-\imath &(\sgn\tau_0)[\Box,A^*A]\\
=&A_2^*\Big(C_2^* x^{-2}C_2+\sum_i (x^{-1} F_{i,2}Q_i+Q_i^* x^{-1}F'_{i,2})
+\sum_{ij}Q_i^* F_{ij,2}Q_j \Big) A_2\\
&+A^*T_{1/2}^*
\Big(C_3^*2(s-r)x^{-2}C_3+\sum_i (x^{-1}F_{i,3} Q_i+Q_i^* x^{-1}F'_{i,3})\\
&\qquad\qquad\qquad+\sum_{ij} Q_i^* F_{ij,3}Q_j \Big)T_{1/2} A\\
&+d_Z^* x^{-1}C_0^*C_0 x^{-1} d_Z+E+E'+R''
\end{split}\end{equation}
with
\begin{enumerate}
\item
$A_2\in\Psieb{s+1/2,r-1/2}(M)$, $\sigma_{\ebo,s+1/2}(A_2)=a_2$,
$\WFeb'(A_2)\subset\supp a$
\item $C_2,C_3 \in \Psieb{0,0}(M);$ $F_{i,2},F'_{i,2},F_{i,3},F'_{i,3}
\in \Psieb{-1,0}(M)$ and $F_{ij,2},F_{ij,3} \in
  \Psieb{-2,0}(M);$
\item  On
$\tilde K$, $\sigma_{\ebo,0}(C_2)\neq 0$ and
$\sigma_{\ebo,0}(C_3)=(\sgn\tau_0)\xiebh\neq 0,$
\item $C_0\in\Psieb{s-1/2,r}(M)$, $\WFeb'(C_0)\subset\supp a$,
\item $E,E' \in \Diffess{2}\Psieb{2s-1,2r+2}(M)$,
\item $R'' \in \Diffess{2}\Psieb{2s-2,2r+2}(M)$ (i.e.\ is lower order),
$\WFeb'(R'')\subset\supp a$,
\item $\ebWF' E\subset \{x>0\} \cap \supp a$ (our hypothesis region),
\item $\ebWF' E' \cap \ebSigma=\emptyset,$ $\WFeb'(E')\subset\supp a$.
\end{enumerate}
When we pair both sides of this equation (suitably regularized) with a
solution to the wave equation the terms $E$, $E'$ and $R''$ will be
controlled respectively by the hypothesis on $u$ in $x>0$, microlocal
elliptic regularity, and an inductive hypothesis in the iterative
argument in which we improve the order by $1/2$ (or less) in each step.  The
remaining terms on the right hand side are either positive, or involve
$Q_i$, and the latter terms are controlled by the former, by
Corollary~\ref{cor:glancingestimate}.  Thus, save for the need to
mollify to make sure that we can actually apply this commutator to $u$
and pair it with $u$, and also be able to rewrite the commutator as
the difference of products, this would give our positive commutator
result, controlling $\|x^{-1}T_{1/2}C_3A u\|.$

We do, however, need to mollify.
Let
$\sigma>0$ (typically we take $\sigma=1/2$, always $\sigma\in(0,1/2]$)
$\Lambda_\gamma\in\Psieb{-\sigma}(M)$ for $\gamma>0$, such that
$\{\Lambda_\gamma: \
\gamma\in(0,1]\}$ is a bounded family in $\Psieb{0}(M)$, and
$\Lambda_\gamma\to\Id$ as
$\gamma\downarrow 0$ in $\Psieb{\tilde\ep}(M)$, for all $\tilde\ep>0.$
Let the principal symbol of
$\Lambda_\gamma,$ considered a
bounded family in $\Psieb{0}(M),$ be
$(1+\gamma|\tau|^2)^{-\sigma/2}$ on a neighborhood of $\tilde K$. Let
$A_\gamma=\Lambda_\gamma A$.
We now have
$A_\gamma\in \Psieb{s-\sigma,r}(M)$ for $\gamma>0$, and $A_\gamma$ is uniformly bounded in
$\Psieb{s,r}(M)$, $A_\gamma\to A$ in $\Psieb{s+\tilde\ep,r}(M)$.
Moreover,
\begin{equation}
\imath[\Box,A_\gamma^*A_\gamma]=\Lambda_\gamma^* \imath[\Box,A^*A]\Lambda_\gamma
+A^*\imath[\Box,\Lambda_\gamma^*\Lambda_\gamma] A+\tilde R,
\end{equation}
with $\tilde R$ uniformly bounded in $\Diffess{2}\Psieb{2s-2,2r+2}(M)$
(hence lower order).
Now, for a vector field $V$ on $\Tebstar M$,
$$
V(1+\gamma|\tau|^2)^{-\sigma/2}
=-(\sigma/2)\gamma (1+\gamma|\tau|^2)^{-\sigma/2-1}V|\tau|^2.
$$
Applying this, the general formula
\eqref{eq:Box-comm-rad}--\eqref{eq:Box-comm-rad-terms} with $\Lambda_\gamma$
in place of $A$ and
\eqref{eq:xiebh-weights2} with $r'=0$, $s'=2$, we deduce that
\begin{equation}\begin{split}\label{eq:Box-regularizer-comm}
-&(\sgn\tau_0)A^*\imath[\Box,\Lambda_\gamma^*\Lambda_\gamma] A\\
&=A^*\Lambda_\gamma^* T_{1/2}^*\tilde\Lambda_{\gamma}^*
\bigg(-2\sigma C_3^* x^{-2}C_3+\sum_i (x^{-1}F_{i,4} Q_i+Q_i^* x^{-1}F'_{i,4})\\
&\qquad\qquad\qquad+\sum_{ij} Q_i^* F_{ij,4}Q_j \bigg)T_{1/2}\tilde\Lambda_{\gamma}\Lambda_\gamma A+R'_\gamma,
\end{split}\end{equation}
with $F_{ij,4}\in\Psieb{-2,0}(M)$, $F_{i,4},F'_{i,4}\in\Psieb{-1,0}(M)$,
$\tilde\Lambda_\gamma$ uniformly bounded in $\Psieb{0,0}(M)$ with
principal symbol
$$
\sigma_{\ebo}(\tilde\Lambda_\gamma)=
\Big(\gamma|\tau|^2(1+\gamma|\tau|^2)^{-1}\Big)^{1/2}\leq 1,
$$
$C_3 \in \Psieb{0,0}(M)$ and $T_{1/2}\in\Psieb{1/2,0}(M)$
as in \eqref{eq:Box-comm-rad-long} and $R'_\gamma$ uniformly bounded
in $\Diffess{2}\Psieb{2s-2,2r+2}(M)$, hence lower order.
Note that this commutator has the {\em opposite} sign from
\eqref{eq:Box-comm-rad-long}, which limits our ability to regularize.
However, as long as $\sigma'-\sigma>0$, we can write
$$
2\sigma'\Id-2\sigma\tilde\Lambda_{\gamma}^*\tilde\Lambda_\gamma
=B_\gamma^* B_\gamma
$$
with $B_\gamma$ uniformly bounded in $\Psieb{0,0}(M)$. Thus, if $s-r>\sigma$,
taking $\sigma'$ such that $\sigma<\sigma'<s-r$, we deduce that
\begin{equation}\begin{split}\label{regularizedcommutator}
-\imath &(\sgn\tau_0)[\Box,A_\gamma^*A_\gamma]\\
=&A_{2,\gamma}^*\bigg(C_2^* x^{-2}C_2+\sum_i (x^{-1} F_{i,2}T_1 Q_i T_{-1}
+T_{-1}Q_i^* T_1 x^{-1}F'_{i,2})\\
&\qquad\qquad\qquad+\sum_{ij}T_{-1}Q_i^* Q_j F_{ij,2}T_1 \bigg) A_{2,\gamma}\\
&+A_\gamma^*T_{1/2}^*
\bigg(C_3^*2(s-r-\sigma') x^{-2}C_3\\
&\qquad\qquad\qquad+\sum_i (x^{-1}F_{i,5}T_1 Q_i T_{-1}
+T^{-1}Q_i^* T_1 x^{-1}F'_{i,5})\\
&\qquad\qquad\qquad+\sum_{ij} T_{-1} Q_i^* Q_j F_{ij,5} T_1 \bigg)T_{1/2} A_\gamma\\
&+A_\gamma^* T_{1/2}^* C_3^* B_\gamma^* B_\gamma C_3 T_{1/2}A_\gamma\\
&+d_Z^* x^{-1}\Lambda_\gamma^*C_0^*C_0 \Lambda_\gamma x^{-1} d_Z
+E_\gamma+E'_\gamma+R''_\gamma,
\end{split}\end{equation}
with the terms as in \eqref{eq:Box-comm-rad-long}, in particular $F_{ij,5},
F'_{ij,5}$
as $F_{ij,3}$, etc., there, and $A_{2,\gamma}=A_2\Lambda_\gamma$, etc.
Here we rewrote the terms in \eqref{eq:Box-comm-rad-long} somewhat, inserting
$T_1$ and $T_{-1}$ in places (recall that $T_1 T_{-1}$ differs from $\Id$
by an element of $\Psieb{-1,0}(M)$ on $\tilde K$, and this difference
can be absorbed in $R''_\gamma$) in order to be able to use
Corollary~\ref{cor:glancingestimate} directly below.
Applying both sides of \eqref{regularizedcommutator} to $u$ and pairing
with $u,$ we claim we may
integrate by parts for any $\gamma>0$ on the right hand side of the
resulting expression to obtain
\begin{equation}\begin{split}\label{eq:regularized-commutator-pairing}
-\imath &(\sgn\tau_0)\langle [\Box,A_\gamma^*A_\gamma]u,u\rangle\\
=&\|x^{-1}C_2A_{2,\gamma} u\|^2+2(s-r-\sigma')\| x^{-1}C_3T_{1/2} A_\gamma u\|^2\\
&+\sum_{ij}\langle Q_j F_{ij,2} T_1 A_{2,\gamma}u,Q_i T_{-1}^* A_{2,\gamma}u
\rangle\\
&+\sum_i \bigg(
\langle Q_i T_{-1}A_{2,\gamma}u, x^{-1} T_1^* F_{i,2}^*A_{2,\gamma} u
\rangle
+\langle x^{-1}T_1F'_{i,2}A_{2,\gamma}u,Q_i T_{-1}A_{2,\gamma}u\rangle\bigg)
\\
&+\sum_i \bigg(\langle  Q_i T_{-1}T_{1/2} A_\gamma u,
x^{-1}T_1^*F_{i,5}^* T_{1/2} A_\gamma u
\rangle\\
&\qquad\qquad+\langle x^{-1}T_1F'_{i,5}T_{1/2}A_{\gamma}u,Q_i T_{-1}T_{1/2}A_{\gamma}u\rangle\bigg)
\\
&+\sum_{ij} \langle Q_j T_1 F_{ij,5}T_{1/2} A_\gamma u,
Q_i T_{-1}^*T_{1/2} A_\gamma u
\rangle \\
&+\| B_\gamma C_3 T_{1/2}A u\|^2
+\|C_0 \Lambda_\gamma x^{-1} d_Z u\|^2
+\langle (E_\gamma+E'_\gamma+R''_\gamma)u,u\rangle,
\end{split}\end{equation}
and that we may similarly expand the left side
by using
\begin{equation}\label{eq:comm-exp}
\langle[\Box,A_\gamma^*A_\gamma]u,u\rangle
=\langle A_\gamma^*A_\gamma u,\Box u\rangle
-\langle \Box u, A_\gamma^*A_\gamma u\rangle,
\end{equation}
so that pairing with a solution to the wave equation yields
identically zero.

We begin by justifying these two integrations by parts, after which we
will read off the consequences. We start with the Dirichlet case.
Note that the $L^2_g$-dual of $\hilbertp=\Hesz{1,l}(M)$ is
$\Hes{-1,-l-(f+1)}(M)$ (where as usual the $x^{f+1}$ factor derives
from the difference between the metric density used in the pairing and the ``edge-density''
used to define the norm on $\Hesz{\cdot,\cdot}(M)$). We have
$$
x^{-2} \Diffess{2}(M)\ni\Box: \hilbertp \to \Hes{-1,l-2}(M) =
x^{2l+(f+1)-2} (\hilbertp)^*.
$$
Here we suppressed the quotient map $\rho:\dHes{-1,l-2}(M)\to\Hes{-1,l-2}(M)$,
i.e.\ the stated mapping property is, strictly speaking, for $\rho\circ\Box$.
Furthermore, the dual of 
$H_{\ebo,\hilbertp}^{s',r'}(M)$ is
$$
\big(H_{\ebo,\hilbertp}^{s',r'}(M)\big)^* 
=H_{\ebo,(\hilbertp)^*}^{-s',-r'}(M).
$$
Equation \eqref{eq:comm-exp}
makes sense directly and naively for $\gamma>0$ if the products of $\Box$ with
$A_\gamma^* A_\gamma\in \Psieb{2s-2\sigma, 2r}(M)$ map $H_{\ebo,\hilbertp}^{s',r'}(M)$ to 
its dual, $H_{\ebo,(\hilbertp)^*}^{-s',-r'}(M).$  We thus require
$$
A_\gamma^*
A_\gamma:
H_{\ebo,\hilbertp}^{s',r'}(M)\to H_{\ebo,\hilbertp}^{-s',-r'-2l-(f+1)+2}(M)
$$
which holds if
\begin{equation}\label{crudenumerology}
\begin{aligned}
s-\sigma &\leq s',\\
r&\leq r'+l+(f+1)/2-1.
\end{aligned}
\end{equation}
Following the same line of reasoning shows that if we are willing to settle for just
\eqref{eq:regularized-commutator-pairing}, by contrast, we only
require the milder hypotheses
\begin{equation}\label{betternumerology}
\begin{aligned}
s-\sigma&\leq s'+1/2,\\ r&\leq r'+l+(f+1)/2-1.
\end{aligned}
\end{equation}
In fact, we claim that \eqref{betternumerology} suffices for
\emph{both} \eqref{eq:regularized-commutator-pairing} and
\eqref{eq:comm-exp}, with the latter being obtained via the following
subtler regularization.

This is best done by replacing
$u$ in the second slot of the pairing by a separate factor
$\tLambda_{\gamma'} u,$  where $\tLambda_\gamma$ is constructed just
as $\Lambda_\gamma,$ but with the greater degree of regularization
$\sigma=1.$  Thus we have a replaced the lost half of an edge derivative (on each
factor) which obtains from assuming \eqref{betternumerology} instead
of \eqref{crudenumerology} and may again
integrate by parts to obtain, for $\gamma,\gamma'>0,$
\begin{equation}\begin{split}\label{eq:extra-reg}
&\langle[\Box,A_\gamma^*A_\gamma]u,\Lambda_{\gamma'} u\rangle
=\langle A_\gamma^*A_\gamma u,\Box \Lambda_{\gamma'} u\rangle
-\langle \Box u, A_\gamma^*A_\gamma \Lambda_{\gamma'} u\rangle\\
&\qquad=\langle A_\gamma^*A_\gamma u,\Lambda_{\gamma'}\Box u\rangle
+\langle A_\gamma^*A_\gamma u,[\Box,\Lambda_{\gamma'}] u\rangle
-\langle \Box u, A_\gamma^*A_\gamma \Lambda_{\gamma'} u\rangle
\end{split}\end{equation}
Now, $\Lambda_{\gamma'}\to \Id$ {\em strongly} (but not in norm)
on $H_{\ebo,\hilbertp}^{s',r'}(M)$ and on $H_{\ebo,\hilbertp^*}^{s',r'}$
for all $s',r';$ this takes care of the first and third terms.  Furthermore,
$[\Box,\Lambda_{\gamma'}]\to 0$ {\em strongly} (but not in norm)
as a map from $H_{\ebo,\hilbertp}^{s',r'}(M)$ to
$H_{\ebo,\Hes{-1,l}(M)}^{s'+1,r'}(M)$. Thus, letting $\gamma'\to 0$
shows \eqref{eq:comm-exp} just under the assumption
$s-\sigma\leq s'+1/2$, $r\leq r'-1+l+(f+1)/2$.

The Neumann case is completely analogous, except that then
$L^2_g$-dual of $\hilbert=\Hes{1,l}(M)$ is
$\dHes{-1,-l-(f+1)}(M).$ We have
$$
x^{-2} \Diffess{2}(M)\ni\Box: \hilbert \to \dHes{-1,l-2}(M) =
x^{2l+(f+1)-2} \hilbert^*.
$$
Furthermore, the dual of 
$H_{\ebo,\hilbert}^{s',r'}(M)$ is
$$
\big(H_{\ebo,\hilbert}^{s',r'}(M)\big)^* 
=H_{\ebo,\hilbert^*}^{-s',-r'}(M).
$$
The rest of the argument proceeds unchanged.

Having justified our integrations by parts, we now
show that we can absorb the $Q_i$-terms in
\eqref{eq:regularized-commutator-pairing} in the positive terms
(uniformly as $\gamma\downarrow 0$) by using
Corollary~\ref{cor:glancingestimate}. Thus, given $\delta>0$, let $U$ be
as in Corollary~\ref{cor:glancingestimate}; for sufficiently small
$\ep>0$, $\supp a\subset U$. For instance, by Cauchy-Schwarz,
\begin{equation*}\begin{split}
&|\langle Q_j F_{ij,2} T_1 A_{2,\gamma}u,Q_i T_{-1}^* A_{2,\gamma}u
\rangle|
\leq \|Q_i T_{-1}^* A_{2,\gamma}u\|^2+\|Q_j F_{ij,2} T_1 A_{2,\gamma} u\|^2\\
&\qquad\leq \delta \big(\|D_t T_{-1}^* A_{2,\gamma}u\|^2
+\|D_t F_{ij,2} T_1 A_{2,\gamma} u\|^2\big)\\
&\qquad\qquad+\cte(\|u\|^2_{H^{1,r+1/2-(f+1)/2}_{\eso}(M)}
+\|Gu\|^2_{H^{1,r+1/2-(f+1)/2}_{\eso}(M)}),
\end{split}\end{equation*}
where $G\in\Psieb{s-1,0}(M)$. The
the $\|Gu\|^2_{H^{1,r+1/2-(f+1)/2}_{\eso}(M)}$
term can be estimated as $\langle R''_\gamma u,u\rangle$ since
\begin{equation}\label{eq:G-weighted-est}
\|Gu\|^2_{H^{1,r+1/2-(f+1)/2}_{\eso}(M)}=\|x^{-r+1/2}d_M Gu\|^2_{L^2_g(M;\Lambda M)},
\end{equation}
and $(x^{-r+1/2}d_M G)^*(x^{-r+1/2}d_M G)\in\Diffes{2}\Psieb{2s-2,2r+1}(M)$,
hence in fact a little better than $R''_\gamma$, which has weight $2r+2$.
Now, for $\cte_0>0$ sufficiently large, depending on $\tilde K$ but {\em not}
on $\ep>0$ (as long as $\ep$ satisfies \eqref{makesure}
and $\supp a\subset U$, i.e.\ $\ep>0$ is sufficiently small),
we have
$$
\sigma_{\ebo,0,1}(D_t F_{ij,2}T_1)\leq \cte_0\sigma_{\ebo,0,1}(x^{-1}C_2)
$$
on a neighborhood of $\tilde K$. Thus,
\begin{equation*}\begin{split}
&\|D_t F_{ij,2} T_1 A_{2,\gamma} u\|^2\\
&\leq 2\cte_0 \|x^{-1}C_2A_{2,\gamma} u\|^2\\
&\qquad+\cte'(\|x^{-r-1}u\|^2_{H^{1,r+1/2-(f+1)/2}_{\eso}(M)}
+\|G' u\|^2_{H^{1,r+1/2-(f+1)/2}_{\eso}(M)}),
\end{split}\end{equation*}
with $G'\in\Psieb{s-1,0}(M)$ (so the last term behaves like
\eqref{eq:G-weighted-est}).
Thus, if we choose $\delta>0$ such that
$8\cte_0n^2\delta<1$, the first term (for all $i,j$) can be absorbed in 
$\|x^{-1}C_2A_{2,\gamma} u\|^2$, while the last two terms are estimated
as $\langle R''_\gamma u,u\rangle$. Essentially identical arguments deal
with all the other terms with $Q_i$ and $Q_j$. In the case where $Q_i$
is present on one side of the pairing only, we write, for instance,
\begin{equation*}\begin{split}
&|\langle Q_i T_{-1}A_{2,\gamma}u, x^{-1} T_1^* F_{i,2}^*A_{2,\gamma} u
\rangle|\\
&
\leq \delta^{-1/2}\| Q_i T_{-1}A_{2,\gamma}u\|^2
+\delta^{1/2}\|x^{-1} T_1^* F_{i,2}^*A_{2,\gamma} u\|^2.
\end{split}\end{equation*}
Using Corollary~\ref{cor:glancingestimate} on the first term, we have
an estimate as above after possibly reducing $\delta>0$.

Recall that uniform finiteness of
$\norm{x^{-1}C_3T_{1/2}A_\gamma u}$ as $\gamma \downarrow 0$ will give
absence of $\WFebX^{s-1/2,r-l-(f+1)/2} u\cap \liptic A$ (as always
the contribution to the weight
of $(f+1)/2$ comes from the metric weight while $l$ comes from the
weight in the definition of the base space, $\hilbert$).  Similarly evaluating
the other terms in the pairing, we take the extreme values of $s',r'$
allowed by \eqref{betternumerology} to
obtain
\begin{equation}
\begin{aligned}
&\WFebX^{s-1/2-\sigma,r+1-l-(f+1)/2} u\cap \WF' A=\emptyset,\\
&\Mand\WFebX^{s-1/2,
  r+1-l-(f+1)/2} u \cap  \WF' A \cap \{x>0\}=\emptyset,\\ 
&\Mand s>r+\sigma,\ \sigma\in(0,1/2]\\
&\Longrightarrow \WFebX^{s-1/2,r+1-l-(f+1)/2} u\cap \liptic A=\emptyset,
\end{aligned}
\end{equation}
or, relabeling,
\begin{equation}\label{commutatorimplication}
\begin{aligned}
  \WFebX^{s,r} u\cap \WF' A=\emptyset,\ &\WFebX^{s+\sigma,r} u
  \cap  \WF' A \cap \{x>0\}=\emptyset,\\ &s>r+l+f/2-1,\\ 
  &\Longrightarrow \WFebX^{s+\sigma,r} u\cap \liptic A=\emptyset.
\end{aligned}
\end{equation}
Recall here that $a=a_\ep$, and
\begin{equation}\label{eq:WF-elliptic}
0<\ep<\ep'\Rightarrow
\WF'(A_\ep)\cap\ebSigma\subset\liptic A_{\ep'}\cap\ebSigma.
\end{equation}

Finally, we show how to use \eqref{commutatorimplication} iteratively,
together with an interpolation argument, to finish the proof of the theorem.
A priori we have $u \in \hilbert,$ i.e.\
$$
\WFebX^{0,0} u =\emptyset,
$$ If $0> l+f/2-1,$ we may iteratively apply
\eqref{commutatorimplication} (shrinking $\ep>0$ by an arbitrarily small
amount, using \eqref{eq:WF-elliptic} to estimate the lower order
error terms $R''_\gamma$) starting with $s=0$ and always keeping
$r=0,$ to obtain the conclusion of the theorem.  (We choose
$\sigma=1/2$ at every stage in this process, until we are applying
\eqref{commutatorimplication} with $s$ such that $s+1/2>m,$ at which
point we finish the iteration by choosing $\sigma=m-s$ so as to retain our estimates on
the wavefront set in the hypothesis region.)

However if $0\leq l+f/2-1,$ we may not apply
\eqref{commutatorimplication} directly owing to the lack of positivity
of the commutator, and we must employ an interpolation argument as
follows.  Applying \eqref{commutatorimplication} iteratively,
this time with $r=r_0<0$ chosen sufficiently negative that we recover $0>r_0+l+f/2-1,$
shows that we obtain
\begin{equation}\label{eq:rough-estimate}
\WFebX^{m,r_0} u \cap S=\emptyset,
\end{equation}
with $S=\WF' A_\ep$ for some $\ep>0$, $A_\ep$ constructed as above.
Let
$$
L = \sup\{r': \WFebX^{m,r'} u \cap S=\emptyset,\ r'\leq 0\}.
$$
Note that the set on the right hand side is non-empty by
\eqref{eq:rough-estimate}.
We aim to show
that $L=0.$
To this end, note that if
$L<0,$ then
for any $r'<L$
$$
\WFebX^{m,r'} u \cap S= \WFebX^{0,0}\cap S=\emptyset.
$$
An interpolation then yields, for $\delta\in (0,1)$,
$$
\WFebX^{m\delta,r'\delta} u \cap S= \emptyset.
$$
Note that for any $\delta\in (0,1)$ fixed,
the compactness of $S$ implies that for some $\ep'>\ep$,
\begin{equation*}
\WFebX^{m\delta,r'\delta}
u \cap \WF' A_{\ep'}=\emptyset
\end{equation*}
still holds.
If $\delta\in(0,1)$ in addition satisfies
\begin{equation}\label{eq:delta-upper}
m\delta>r'\delta+l+f/2-1
\end{equation}
then by iterating \eqref{commutatorimplication}, shrinking $\ep'$ in
each step (but keeping it larger than $\ep$),
we conclude that
$$
\WFebX^{m,r'\delta} u \cap S=\emptyset,
$$
providing a contradiction with the definition of $L$ if
\begin{equation}\label{eq:delta-lower}
r'\delta>L.
\end{equation}
It remains to check whether $\delta\in (0,1)$ satisfying both
\eqref{eq:delta-upper} and \eqref{eq:delta-lower} exists.  This
is
evident from Figure~\ref{interpolationfigure},
\begin{figure}[ht]
\includegraphics[bb = 281 500 433 677,scale=.8]{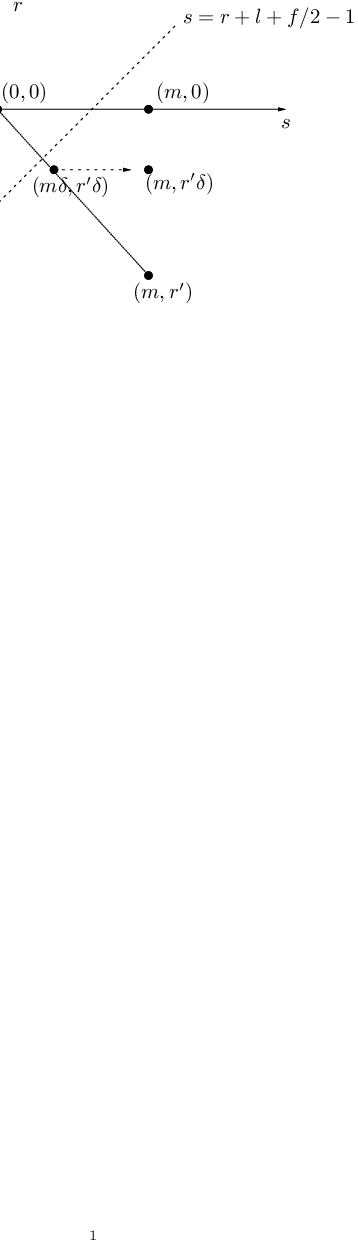}
\caption{The interpolation argument.  The figure shows the $(s,r)$
  plane, where we plot
the values for which there is no $\WFebX^{m,r}(u)$ (i.e.\
microlocal regularity of this order holds).
We have a priori regularity of order
$(0,0)$ and wish to conclude regularity of order $(m,l')$ with $l'<0.$
By \eqref{commutatorimplication} we may take a step to the right of
length $\sigma$ for any $\sigma \in (0, 1/2]$ provided that our
starting regularity is below the line $s=r+l+f/2-1$ and that our
endpoint is on $s\leq m.$   If we know
$(m,r')$ regularity, we know regularity by interpolation on the whole
line connecting this point to the origin; then starting on the
interpolation line just below its intersection with $s=r+l+f/2-1$
allows us to achieve $(m,r'\delta)$ regularity by moving to the right,
thus improving the optimal weight for which we have our estimate.}
\label{interpolationfigure}
\end{figure}
but a proof is as
follows: we have $l+f/2-1\geq 0$ by assumption (otherwise we are in the
preceding case); moreover $m>0$ (so that the theorem is
not vacuous), and $l+f/2-1<m$ by hypothesis.  Thus, for any $r'<0,$
$$
0\leq \frac{l+f/2-1}{m-r'}<\frac{m}{m-r'}<1.
$$
Setting
$$
\delta_0(r')=\frac{l+f/2-1}{m-r'}\in [0,1)
$$
we see
that \eqref{eq:delta-upper} is an equality with $\delta =\delta_0(r)$
and that taking $\delta\in
(\delta_0(r'),1)$ yields \eqref{eq:delta-upper}. In particular, \eqref{eq:delta-upper}
is satisfied by $\delta=\delta(r',\lambda)=\delta_0(r')(1+\lambda)$ for any
$\lambda>0$ sufficiently small.  If $L<0,$ we have $\delta_0(L)<1,$ hence the
function of $r'$ and $\lambda$ given by
$$
r'\delta(r',\lambda) = r'\delta_0(r')(1+\lambda)
$$
is strictly greater than $L$ for $r'=L,\lambda=0.$  Hence increasing 
$\lambda$ slightly and decreasing $r'$ slightly preserves this
relationship by continuity, and these choice of $r'$ and $\delta$
yield $r'\delta>L$ as desired.
\end{proof}

In order to verify the hypotheses of Theorem~\ref{thm:bdy-radial}, which are stated at
points over the edge, we will employ the following geometric result.

First note that if $q \in \cH_{W,\bo}$ then $\rcal_{\ebo,q,I}$ has a
neighborhood $O_1$ in $\Sebstar M$ and there is a $\delta_0>0$
such that any \GBB $\gamma$
with $\gamma(0)\in O_1\cap\{x>0\}$
satisfies $\gamma|_{[-\delta_0,0]}\cap\Sbstar_W M_0=\emptyset$. Indeed,
we simply need to take a coordinate neighborhood
\begin{equation*}\begin{split}
O_1=\{p\in\Sebstar M:
&\ \xiebh(p)<-\sqrt{1-h(q)}/2,\ x(p)<\ep_1,\ |y(p)-y(q)|<\ep_1,\\
&\ |t(p)-t(q)|<\ep_1,
\ |\etaebh(p)-\etabh(q)|<\ep_1\},
\end{split}\end{equation*}
$\ep_1>0$ sufficiently small, since
on its intersection with $\{x>0\}$ (where $\Sebstar M$
is naturally identified with $\Sbstar M_0$), $\sum\xibh_j<0$, hence
Lemma~\ref{lemma:OG-IC-uniform-est-1} gives the desired $\delta_0$
(cf.\ the argument of Remark~\ref{remark:unif-nbhd}). Thus,
such \GBB $\gamma$ can be uniquely lifted to curves $\gammat$ in $\Sebstar M$.

\begin{lemma}\label{lemma:hypothesis-region}
Suppose that $q \in \cH_{W,\bo}.$ There exists $\ep_0>0$ with the following
property.

Suppose that $0<\ep_1\leq \ep_0$, and $U$ is a neighborhood
of $\fcalW_{I,q,\sing}
\cap \{t=t(q)-\ep_1\}.$
Then there is a neighborhood $O$ of $\rcal_{\ebo,q,I}\cap\cG$
in $\Sebstar M$ such
that for every maximally backward extended \GBB $\gamma$ with
$\gamma(0)\in O\cap\{x>0\}$ there is an $s_0<0$ such that $\gamma(s_0)\in U$,
$\gamma(s)\notin \Sbstar_W M_0$ for $s\in [s_0,0]$.
\end{lemma}

\begin{proof}
It follows from the discussion preceeding the statement of the lemma
that there is a neighborhood $O_1$ of $\rcal_{\ebo,q,I}$ and $\delta_0>0$
such that every \GBB $\gamma$ defined on $[-\delta_0,0]$, with
$\gamma(0)\in O_1\cap\{x>0\}$ satisfies
$\gamma(s)\notin \Sbstar_W M_0$ for $s\in [-\delta_0,0]$.
As $t(\gamma(s))-t(\gamma(0))=2\taubh(\gamma(0))s$, this implies that there
is an $\ep_0>0$ such that for $t(\gamma(s))\in [t(q)-\ep_0,t(q)]$,
$\gamma(s)\notin \Sbstar_W M_0$.

Suppose now for the sake of contradiction that there is no neighborhood $O$
of $\rcal_{\ebo,q,I}\cap\cG$
in $\Sebstar M$ such
that for every (maximally extended) backward \GBB $\gamma$ with
$\gamma(0)\in O\cap\{x>0\}$, there exists $s_0<0$ with
$\gamma(s_0)\in U$.
As $\rcal_{\ebo,q,I}\cap\cG$ is compact, we conclude that
there is a sequence of points $p_n\in O_1\subset\Sebstar M$ with $x(p_n)>0$
(so $p_n$ can be regarded as a point in $\Sbstar M_0$)
and \GBB $\gamma_n$ such that
\begin{itemize}
\item $\gamma_n(0)=p_n$,
\item the image of $\gamma_n$
disjoint from $U$,
\item $p_n\to p\in
\rcal_{\ebo,q,I}\cap\cG$.
\end{itemize}
By Corollary~\ref{cor:GBB-lifted-conv}, $\{\gamma_n\}$ has a subsequence
$\{\gamma_{n_k}\}$ converging
uniformly to a \GBB $\gamma$ such that the lift $\tilde\gamma$ of
$\gamma$ to $\Sebstar M$ satisfies $\tilde\gamma(0)=p$. Thus,
by Lemma~\ref{lemma:GBB-eb-lift}, $\gamma$ is not normally incident, so
the image of $\gamma$ is in $\fcalW_{I,q,\sing}$, and
thus intersects $\fcalW_{I,q,\sing}
\cap \{t=\ep_1\}.$
As $\gamma_{n_k}\to\gamma$ uniformly, for large enough $k$,
$\gamma_{n_k}$ intersects $U$, providing a contradiction. Thus, there
exists $O$ such that for every (maximally extended) backward \GBB $\gamma$ with
$\gamma(0)\in O\cap\{x>0\}$, there exists $s_0<0$ with
$\gamma(s_0)\in U$. We may assume that $O\subset O_1$ by replacing
$O$ by $O\cap O_1$ if needed.

To finish the proof, we note that, provided $\ep_0>0$ is sufficiently small,
if $\gamma(0)\in O\subset O_1$, $t(\gamma(s))\in [t(q)-\ep_0,t(q)]$
implies $\gamma(s)\notin \Sbstar_W M_0$.
\end{proof}

Theorem~\ref{thm:bdy-radial} and this lemma immediately give the
following Corollary.

\begin{corollary}
Let $u \in \hilbert \equiv \Hes{1, l}(M)$ solve
$\Box u=0$ with Dirichlet or Neumann boundary conditions.

Let $m>l+f/2-1$
with $m\geq 0.$  Suppose that $q \in \cH_{W,\bo}$ and
$\fcalW_{I,q,\sing}\cap\WFebX^{m,*} u
=\emptyset.$
Then $\Sebstar[\pa \ef] M\cap \WFebX^{m,l'} u =
\emptyset$ for all $l'<0.$
\end{corollary}

\begin{proof}
Let $\ep_0>0$ be as in Lemma~\ref{lemma:hypothesis-region}.
As $\WFebX^{m,*} u$ is closed, $\fcalW_{I,q,\sing}\cap\{t=t(q)-\ep_0/2\}$
has a neighborhood
$U$ disjoint from $\WFebX^{m,*} u$. By Lemma~\ref{lemma:hypothesis-region},
$\rcal_{\ebo,q,I}\cap\cG$ has a neighborhood $O$ such that every
backward \GBB $\gamma$ with $\gamma(0)\in O\cap\{x>0\}$ intersects $U$
and is disjoint from $\Sbstar_W M_0$.
By the propagation of singularities, \cite{Vasy5}, $\WFebX^{m,*}(u)\cap
O\cap\{x>0\}=\WF_{\bo,\hilbert}^{m,*}(u)\cap O\cap \{x>0\}=\emptyset$.
Note that this uses the fact that every
backward \GBB $\gamma$ with $\gamma(0)\in O\cap\{x>0\}$ intersects $U$
and is disjoint from $\Sbstar_W M_0$, for we do not assume that $u$
lies in a b-derivative of $\hilbert$ as we allow arbitrary weights at $\ef$.
Thus, by Theorem~\ref{thm:bdy-radial},
$\Sebstar[\pa \ef] M\cap \WFebX^{m,l'} u =
\emptyset$ for all $l'<0.$
\end{proof}

\subsection{Propagation at hyperbolic points within the edge}

Now we consider propagation within $\Sebstar[\ef] M,$ away from the radial
points.  The propagation away from $\pa \ef$ is given by the results in
\cite{mvw1}: on the edge cosphere bundle over $\ef^\circ,$ we find that
$\ebWF u=\WFebX u$ (with, say, $\hilbert=\Hes{1,l}(M)$,
for Dirichlet boundary conditions,
$\hilbert=\Hesz{1,l}(M)$ for Neumann boundary conditions---though this
is irrelevant since we are working away from $\pa\ef$)
given by is a union of integral curves of
$\sH_{\eso}|_{\Sesstar_{\ef}M}$, given by \eqref{eq:sH-at-ef}, i.e.
\begin{equation*}
\frac12\sH_{\eso}=-\xiesh\zetaesh\pa_{\zetaesh}
+K^{ij}\zetaesh_i\pa_{z_j}+K^{ij}\zetaesh_i\zetaesh_j\pa_{\xiesh}
-\frac{1}{2}\,\frac{\pa K^{ij}}{\pa z_k}\zetaesh_i\zetaesh_j\pa_{\zetaesh_k},
\end{equation*}
where, as before,
hats denote variables divided by $\abs{\taues},$ hence coordinates
in the edge-smooth cosphere bundle (which over $\ef^\circ$ is canonically
identified with the edge cotangent bundle).  This leaves open only the
question of how bicharacteristics reaching $z'=0$ interact with those
leaving $z'=0,$ i.e.\ the problem of reflection/diffraction from the
boundary faces and corners of $Z.$
Since the propagation over the interior of $\ef$ can be considered as
a special case of propagation at $\cG\setminus\rcal_{\ebo}$ (see
Theorem~\ref{theorem:edgeglancing}, with no $z'$ variables, i.e.\ with
$k=0$ in the notation of the theorem), we do not state the interior
propagation result of \cite{mvw1} here explicitly.

Let us thus begin by considering a
\emph{hyperbolic} point $q \in
\cH$ given by
\begin{equation}\label{hyperbolicincoordinates}
x=0,\ t=t_0,\ y=y_0,\ z'=0,\ z''=z''_0,\ \xiebh=\xiebh_0,\ \etaebh=\etaebh_0,
\ \zetaebh'=0,\ \zetaebh_0'',
\end{equation} in edge-b canonical coordinates.  Thus, in
addition to $\zeta'=0,$ we have $$1 >\xiebh^2 +h(y_0,\etaebh_0)+
k(y_0,z'=0,z''_0, \zetaebh'=0,\zetaebh_0'').$$

In the special case that $z'$ is a variable in $\RR^1,$
i.e.\ if $q$ lies on a codimension-one boundary face of $\eb,$ then
two points in $\Sebstar[\ef](M)$ lie above $q$ and two edge
bicharacteristics in $\Sebstar[\ef^\circ](M)$ contain $q$ in their
closures; we denote them $\gamma_{\pm}$ with the $\pm$ given by $\sgn
({\zetas}'\cdot z');$ we will take $\gamma_{\pm}$ to be only the
segments of these bicharacteristics in $\abs{z'}<\ep\ll 1$ in order
not to enter into global considerations.  Our sign convention is such
that $\gamma_\pm$ tends toward $q$ under the forward resp.\ backward
bicharacteristic flow.  What we will show in this case is that if $u
\in \hilbert\equiv \Hes{1,l}(M)$ and $\Box u=0$ with Dirichlet
or Neumann boundary conditions then
$$
\gamma_- \cap\WFebX^{m,0}  u =\emptyset \Longrightarrow \gamma_+\cap \WFebX^{m,0} u
=\emptyset\quad \text{for any } m.
$$
More generally, we have the following result, which via standard
geometric arguments (see \cite{Melrose-Sjostrand2}) implies the propagation
along \EGBBs through $p:$

\begin{theorem}\label{theorem:edgehyperbolic}
For Neumann boundary conditions, let $\hilbert=\Hes{1,l}(M)$,
$\hilbertp=\dHes{-1,l-2}(M)$; for Dirichlet boundary conditions let
$\hilbert=\Hesz{1,l}(M)$,
$\hilbertp=\Hes{-1,l-2}(M)$.

Let $u \in \hilbert$ solve
$\Box u=f$, $f\in\hilbertp$.
Let $p \in \hcal$ be given by \eqref{hyperbolicincoordinates}.
Let $U$ be an open neighborhood of $p$ in $\Sebstar[\ef](M)$,
let $m\in\RR$, $l'\leq 0$, and suppose that
$\WFebY^{m+1,l'}(f)\cap U=\emptyset$.
Then
$$
U \cap \Big\{\sum \zetaebh'_i>0\Big\}
\cap \WFebX^{m,l'}(u) =\emptyset \Longrightarrow p \notin\WFebX^{m,l'}(u).
$$
\end{theorem}

Thus, the hypothesis region of the theorem, in which we make a
wavefront assumption lies within the points with at least one $z_i'$ non-zero,
i.e.\ away from $\Sebstar[C]M$, where $C=\{x=0,\ z'=0\}$,
and with momenta directed
\emph{toward} the boundary $z'=0.$

\begin{proof}
As usual, one needs to prove that if in addition to the hypotheses above
$p \notin\WFebX^{m-1/2,l'}(u)$ then $p \notin\WFebX^{m,l'}(u)$, with
a slightly more controlled (but standard) version if $m=\infty$.
So we assume $p \notin\WFebX^{m-1/2,l'}(u)$ from now on.

For a constant $\beta$ to be determined later, let
\begin{equation}\label{phidef}
\phi = \sum \hat{\zeta'_i} + \beta\omega
\end{equation}
where
$$
\omega = 
\abs{z'}^2 + \abs{z''-z''_0}^2+
\abs{\hat{\zeta''}-\hat{\zeta''_0}}^2 + \abs{y-y_0}^2+ \abs{t-t_0}^2 + x^2 +
  \abs{\hat{\eta}-\hat{\eta_0}}+ \abs{\hat{\xi}-\hat{\xi_0}}^2
$$
Then for $\beta$ sufficiently small, we have
$$
\abs{\tau}^{-1} \sH_p\phi >0.
$$ Now let $\chi_0\in\CI(\RR)$ with support in $[0,\infty)$ and
$\chi_0(s)=\exp(-1/s)$ for $s>0.$ Thus, $\chi_0'(s)=s^{-2}\chi_0(s).$ Take
$\chi_1\in\CI(\RR)$ to have support in $[0,\infty),$ to be equal to $1$ on
$[1,\infty)$ and to have $\chi_1'\geq0$ with $\chi_1'\in\CIc((0,1)).$
Finally, let $\chi_2\in\CIc(\RR)$ be supported in $[-2c_1,2c_1]$ and be
identically equal to $1$ on $[-c_1,c_1].$ Pick $\delta<1.$  Set
$$
a = \abs{\tau}^s x^{-r}\chi_0(\ssM(1-\phi/\delta)) \chi_1\Big(\sum\hat\zeta'_j/\delta+1\Big) \chi_2
\big(\lvert\hat\zeta'\rvert^2\big).
$$
Note that on the support of $a,$ we have
\begin{equation}\label{supportest1}\sum\zeta'_j>-\delta,\end{equation}
hence we also obtain
\begin{equation}\label{supportest2}
0\leq \omega < 2 \frac{\delta}{\beta}.
\end{equation}
Thus, by keeping $\delta$ and $\delta/\beta$ both small, we can keep the
support of $a$ within any desired neighborhood of $\zeta'=0,$ $\omega=0.$

We now quantize $a$ to $A \in \Psieb{s,r}(M).$ 
We claim that
\begin{multline}\label{hyperboliccommutator}
\imath[\Box, A^*A] \\ =  B^*\Big(\sum
D_{z'_i}^* C_{ij} D_{z'_j}+R_0+\sum
 (R_iD_{z'_i}+D_{z'_i}^*R'_i) + \sum D_{z'_i} R_{ij}D_{z'_j} \Big)B\\ +
A^* W A+
R''+ E+ E'
\end{multline}
where
\begin{enumerate}
\item\label{zeroitem2} $B \in \Psieb{s+1/2,r+1}(M)$ has symbol
$$
\ssM^{1/2}\tau^{s+1/2} x^{-r-1} \delta^{-1/2}\chi_1\chi_2\sqrt{\chi_0\chi_0'}
$$
\item\label{firstitem2} $C_{ij} \in \Psieb{-2,0}(M),$ and the
  symbol-valued quadratic form $\sigma(C_{ij})$ is strictly positive
  definite on a neighborhood of $\ebWF'B,$ \item $R,R_i,R'_i,R_{ij}$ are in
  $\Psieb{0,0}(M),$ $\Psieb{-1,0}(M),$
$\Psieb{-1,0}(M),$ and $\Psieb{-2,0}(M)$
  respectively and have (unweighted) symbols bounded by multiples of
  $\sqrt\delta(\sqrt\beta+1/\sqrt{\beta}).$
\item $R'' \in \Diffess{2}\Psieb{-2,0}(M),$ $E,E' \in
\Diffess{2}\Psieb{-1,0}(M),$
\item $E$ is microsupported where we have assumed regularity,
\item  $W \in \Diffess{2}\Psieb{-1,2}(M),$
\item\label{lastitem2} $E'$ is supported off the characteristic set.
\end{enumerate}
These terms arise as follows.  Applying Lemma~\ref{lemma:Box-eb-comm}, we
have (with $Q_i = x^{-1} D_{z'_i}$)
\begin{equation}
\imath [\Box,A^*A]=\sum Q_i^* L_{ij} Q_j+\sum (x^{-1}
L_i Q_i+Q_i^* x^{-1}L'_i)
+x^{-2}L_0,
\end{equation}
with
\begin{equation}\begin{split}
&L_{ij}\in\Psieb{2s-1,2r}(M),\ L_i,L'_i\in\Psieb{2s,2r}(M),\ L_0\in\Psieb{2s+1,2r}(M),\\
&\sigma_{\ebo,2m-1}(L_{ij})=2aV_{ij}a,\ V_{ij}
=\kappa_{ij}(\pa_{\zetaeb'_i}+\pa_{\zetaeb'_j}+2\pa_{\xieb})
+\sH_{\ebo,\kappa_{ij}},\\
&\sigma_{\ebo,2m}(L_i)=\sigma_{\ebo,2m}(L'_i)=2a V_i a,\\
&V_i=\sum_j\kappa_{ij}\pa_{z'_j}
+\frac{1}{2}(m_i\pa_\xi+\sH_{\ebo,m_i})+\frac{1}{2}m_i
(\pa_{\xi}+\pa_{\zeta'_i}),\\
&\sigma_{\ebo,2m+1}(L_0)=2a V_0a,\ V_0=2\htil\pa_\xi+\sH_{\ebo,\htil}
+\sum_i m_i\pa_{z'_i},\\
&\ebWF'(L_{ij}),\ebWF'(L_i),\ebWF'(L'_i),\ebWF'(L_0)\subset\ebWF'(A).
\end{split}\end{equation}
with
\begin{equation}\begin{split}\label{eq:sH-eb-h-2}
&V_0|_{C}=-2\xi\,x\pa_x-2\Big(\xi^2+\sum_{ij}k_{2,ij}\zeta''_i\zeta''_j\Big)\pa_\xi
-2\xi(\tau\,\pa_\tau+\eta\,\pa_\eta)\\
&\qquad\qquad\qquad-2\sum_{i,j} k_{2,ij}\zeta''_i\pa_{z''_j}
+\sum_{\ell,i,j}(\pa_{z''_\ell}k_{2,ij})\zeta''_i\zeta''_j\pa_{\zeta''_\ell} ,\\
&V_{ij}|_C
=-k_{1,ij}(\pa_{\zeta'_i}+\pa_{\zeta'_j}+2\pa_\xi)
+\sum_\ell (\pa_{z''_\ell}k_{1,ij})\pa_{\zeta''_\ell},
\ V_i|_C=-\sum_j k_{1,ij}\pa_{z'_j},
\end{split}\end{equation}
First we evaluate the terms in $L_{ij}$ coming from terms in which
$V_{ij}$ hits $\chi_0(\ssM(1-\phi/\delta)).$  The main contribution will be
from the derivatives falling on $\zetaebh'$, with the rest controlled
by shrinking $\beta;$ in particular,
$$
V_{ij} (\phi/\delta) = -2 k_{1,ij} + r_{ij}
$$
with 
$$
\abs{r_{ij}}\leq \text{const} (\beta \sqrt\omega+ \sqrt{\omega});
$$
on the support of $a,$ this is in turn controlled by a multiple of
$$
\beta\sqrt{\delta/\beta} + \sqrt{\delta/\beta}.
$$
Thus, from these two terms, we obtain corresponding terms in $\imath[\Box, A^*A]$ of the
forms
$$
B^*\big(\sum
D_{z'_i}^*C_{ij} D_{z'_j}\big)B
$$
and
$$
B^*\big(\sum
D_{z'_i}^* R_{ij} D_{z'_j}\big)B
$$
respectively.

Similarly, terms with $V_i$ and $V_0$ hitting $\chi_0$ go into the
$R_i$ and $R_0$ terms in \eqref{hyperboliccommutator} respectively.

The terms arising from
$$
V_{\bullet}\Big(\chi_1\big(\sum\hat\zeta'_j/\delta+1\big)\Big)
$$
are supported on the hypothesis region, $\{\sum \hat\zeta'_j<0\},$
hence give commutator terms of the form $x^{-2} E$ above.

The terms arising from
$$
V_{\bullet}\Big(\chi_2 \big(\lvert\hat\zeta'\rvert^2\big)\Big)
$$
lie off of the characteristic set, hence give commutator terms of the
form $x^{-2}E'$ above.

The term arising from differentiating $\abs{\tau}^s x^{-r}$ gives the
commutator term $A^*W A.$

As we are interested in edge-b wavefront set, the term $D_{z'_i}^*C_{ij}
D_{z'_j}$ is slightly inconvenient, but we note that owing to strict
positivity of $C_{ij}$ we may replace it by a multiple of
$\Lap_{z'}/x^2$ plus another positive term.  Rewriting $\Lap_{z'}/x^2=
(\Lap_{z'}/x^2+\Box)-\Box,$ and noting that the first of these terms
is in $x^{-2}\Diffeb{2}(M)$ and elliptic on the hyperbolic set, we see
that we in fact have
\begin{multline}\label{commfinalform}
\imath [A^*A,\Box] =  R'''\Box + B^*\Big(C^*C + \sum
D_{z'_i}^* \tilde{C}_{ij}D_{z'_j}+R_0\\
+\sum
D_{z'_i}^* R'_i+R_i D_{z'_i} +A^*W A+ \sum D_{z'_i}^* R_{ij}D_{z'_j}\Big)B +
R''+E+E'
\end{multline}
where $\tilde{C}_{ij} \in \Psieb{0,2}(M)$ is a positive matrix of operators
(this is a priori true only at the symbolic level, but we may absorb
lower-order terms in $R_{ij}$).

Following \cite{Vasy5}, we find that for any $\parameter>0,$
\begin{equation}\label{rest} \abs{\ang{R_0 w,w}} \leq
C(\sqrt\delta)\big(\sqrt\beta+1/\sqrt\beta\big)\norm{w}^2+ \parameter^{-1}
\norm{ R_0' w}^2 + \parameter\norm{w}^2,
\end{equation} where $R_0'\in \Psieb{-1}$ has the same microsupport as
$R_0,$ and is \emph{one order lower}.  Here we have employed $L^2$
boundedness of $\Psieb{0,0}(M),$ or more specifically, the square-root
argument used to prove it (cf.\ \cite{Vasy5} for details, specifically the
treatment following (6.18)).  By the same token, we can estimate
\begin{multline}\label{rqest} \ang{R_i D_{z'_i} w,w} \leq C(\sqrt\delta)
\big(\sqrt\beta+1/\sqrt\beta\big)\bigg(\norm{T_{-1}D_{z'_i} w}^2+\norm{w}^2\bigg)\\
+ 2\parameter \norm{w}^2 + \parameter^{-1} \norm{\tilde{R}_i D_{z'_i}
w}^2,
\end{multline}
where $\tilde{R}_i\in \Psieb{-2,0}(M)$ has the same microsupport as $R_i,$
and is one order lower. We also compute
\begin{multline}\label{rqqest} \abs{\ang{R_{ij} D_{z'_i} w,D_{z'_j} w}} \leq
C(\sqrt\delta) \big(\sqrt\beta+1/\sqrt\beta\big)\bigg(\norm{T_{-1}D_{z'_i}
w}^2+\norm{T_{-1}D_{z'_j} w}^2\bigg)\\+ 2 \parameter \bigg(\norm{T_{-1} D_{z'_j}
w}^2 + \norm{ T_{-1}D_{z'_j} w}^2\bigg) + 2\parameter \norm{ w}^2\\
+\parameter^{-1} \norm{ \tilde{R}_{ij} D_{z'_i} w}^2+ \parameter^{-1}
\norm{ \tilde{R}'_{ij} D_{z'_j} w}^2,
\end{multline} where $\tilde{R}_{ij},\tilde{R}'_{ij}\in \Psieb{-3}(M)$ have the same microsupport as
$R_{ij},$ and are one order lower.  Although the argument is identical to
that in \cite{Vasy5}, \cite{mvw1}, we reproduce the derivation of
\eqref{rqqest} for the convenience of the reader; \eqref{rest} and
\eqref{rqest} follow by similar (easier) arguments.  To begin, we note that
$T_{1}^* R_{ij} \in \Psieb{-1}(M)$ has symbol bounded by $C(\sqrt\delta)
(\sqrt\beta+1/\sqrt\beta),$ hence by the H\"ormander square-root argument
\begin{equation}\label{foo}
\norm{T_{1}^* R_{ij} u}^2 \leq  C(\sqrt\delta)
\big(\sqrt\beta+1/\sqrt\beta\big) \norm{u}^2 + \norm{\tilde{R}_{ij} u}^2
\end{equation}
with $\tilde{R}_{ij}$ as described above.  Now write $D_{z'_j} w =
T_{1} T_{-1} D_{z'_j} w - F D_{z'_j} w;$ this permits us to expand 
$$ \ang{ R_{ij} D_{z'_i} w, D_{z'_j} w} = \ang{T_{1}^* R_{ij}
 D_{z'_i} w, T_{-1} D_{z'_j} w} - \ang{ R_{ij} D_{z'_i} w,  F
D_{z'_j}w}.
$$ The first term on the right may be controlled, using \eqref{foo} and
Cauchy-Schwarz, by the RHS of \eqref{rqqest}; the second term may also be
so estimated by again applying Cauchy-Schwarz and absorbing $\norm{F 
D_{z'_j} w}^2$ into a term $\norm{ \tilde{R}'_{ij} D_{z'_j} w}^2$
with appropriately enlarged $\tilde{R}'_{ij}.$

Now we turn to making our commutator argument.  Let $u$ be a solution to
$$
\Box u =f
$$
with Dirichlet or Neumann boundary conditions.
Choose $\Lambda_{\gamma}\in \Psieb{-1}(M)$ converging to the identity
as $\gamma\downarrow 0$ as in \S\ref{subsection:radialboundary}. Note
that by making $\Lambda_\gamma\in\Psieb{-1}(M)$, we are combining
the roles of the regularizer $\Lambda_\gamma\in\Psieb{-\sigma}(M)$ in
\S\ref{subsection:radialboundary}, required for obtaining an improvement
over the a priori assumptions, and the regularizer $\tLambda_{\gamma'}$
used to justify the pairing argument, see \eqref{eq:extra-reg}.
Let $A_\gamma = \Lambda_\gamma A$ with $A$ constructed as above.
As before, we have
\begin{equation}
\imath [\Box,A_\gamma^*A_\gamma]=\Lambda_\gamma^* \imath[\Box,A^*A]\Lambda_\gamma
+A^*\imath [\Box,\Lambda_\gamma^*\Lambda_\gamma] A+\tilde R,
\end{equation}
with $\tilde R$ uniformly bounded in $\Diffess{2}\Psieb{2s-2,2r+2}(M)$
(hence lower order), and where
$$
[\Box, \Lambda_\gamma^*\Lambda_\gamma]\in \Diffess{2}\Psieb{-1,2}(M)
$$
is uniformly bounded, and in fact
$$
[\Box, \Lambda_\gamma^*\Lambda_\gamma]
=\Lambda_\gamma^* \tilde W_\gamma\Lambda_\gamma,
$$
with $\tilde W_\gamma$ uniformly bounded in $\Diffess{2}\Psieb{-1,2}(M)$,
cf.\ \eqref{eq:Box-regularizer-comm}.

Now we pair $A_\gamma^*A_\gamma$ with $u.$  Letting $B_\gamma= B
\Lambda_\gamma,$ provided integrations by parts can be justified, we have
\begin{multline}\label{commutatorargument}
\langle A_\gamma\Box u,A_\gamma u\rangle
-\langle A_\gamma,A_\gamma\Box u\rangle\\
=\ang{\imath [A_\gamma^* A_\gamma, \Box] u, u} = \norm{C B_\gamma u}^2 + \sum \ang{ \tilde{C}_{ij}
 D_{z'_j}B_\gamma u, D_{z'_i} B_\gamma u}\\  +\ang{R_0 B_\gamma u,
B_\gamma u}+\sum \ang{D_{z'_i}B_\gamma u, R_i B_\gamma u}
+\ang{R'_i B_\gamma u, D_{z'_i} B_\gamma u} \\+ \sum \ang{D_{z'_j}
R_{ij} B_\gamma u, D_{z'_i} B_\gamma u}+ \ang{W_\gamma  A_\gamma u, A_\gamma
u}  + \ang{(R_\gamma''+E_\gamma+E_\gamma')
  u, u}
\end{multline}
where $W_\gamma$ is uniformly bounded in $\Diffess{2}\Psieb{-1,2}(M)$
and comprises both the $W$ term from above and the term containing
$[\Box, \Lambda_\gamma^*\Lambda_\gamma].$  The integrations by parts 
may by justified, for any $\gamma>0,$ if
\begin{equation}\label{eq:A-u-pair}
\WFebX^{s',l'} u \cap \WF'A=\emptyset
\end{equation}
whenever
$$
s-1\leq s',\quad r\leq l'+l+\frac{(f+1)}2 -1
$$
since then the products of $\Box$ with $A_\gamma^*A_\gamma$ map
$H_{\ebo,\hilbert}^{s',l'}(M)$ to its dual (as required
in the Neumann setting),
as well as mapping $H_{\ebo,\hilbertp}^{s',l'}(M)$,
$\hilbertp
=\Hesz{1,l}(M)$,  to its dual (as required in the Dirichlet setting).
We take $s'=m-1/2$, hence $s=m+1/2$, and $r=l'+l+\frac{(f+1)}2 -1$ here,
and note that it suffices to have the microlocal assumptions
\eqref{eq:A-u-pair} rather than global assumptions in view
the microlocality of $\Psieb(M)$, see
Lemma~\ref{lemma:microlocal} and Lemma~\ref{lemma:uniformly-microlocal}.

We now examine the terms on the RHS.  The first two are positive.
To the third, we apply \eqref{rest}, with $w=B_\gamma u$: if $\delta,$
$\delta/\beta,$ and $\parameter$ are sufficiently small, we may absorb the
first and third terms on the RHS of \eqref{rest} in $\norm{C B_\gamma
u},$ while the lower-order second is uniformly bounded by our wavefront
assumptions.  Likewise, applying \eqref{rqest} and \eqref{rqqest}, we may
choose $\parameter,\delta,\beta$ so as to absorb terms involving $\parameter$ and
$(\sqrt\delta) (\sqrt\beta+1/\sqrt\beta)$ in $\norm{x^{-1} C B_\gamma u}^2;$ the
$\parameter^{-1}$ terms, as they are lower order, remain bounded. Moreover,
as $\chi_0'(s)=s^{-2}e^{-1/s}$ for $s>0$ and vanishes for $s\leq 0$,
$$
\ssM^2(1-\phi/\delta)^2\chi_0'(\ssM(1-\phi/\delta))
=\chi_0(\ssM(1-\phi/\delta)).
$$
Thus,
\begin{equation*}\begin{split}
&\abs{\tau}^{1/2} x^{-1}a\\
&= \abs{\tau}^{s+1/2} x^{-r-1}\ssM(1-\phi/\delta)
\sqrt{\chi_0(\ssM(1-\phi/\delta))\chi_0'(\ssM(1-\phi/\delta))}\\
&\qquad\qquad\qquad\qquad\qquad
\chi_1\Big(\sum\hat\zeta'_j/\delta+1\Big) \chi_2
\big(\lvert\hat\zeta'\rvert^2\big)\\
&= \ssM^{1/2}\delta^{1/2}(1-\phi/\delta) b
\leq 2\ssM^{1/2}\delta^{1/2} b
\end{split}\end{equation*}
as $\phi\geq\sum\zetaebh'_i\geq -\delta$ on $\supp a$.
We deduce that
$\|T_{1/2}x^{-1}A_\gamma u\|$ can be estimated by
$4\ssM^{1/2}\delta^{1/2} \|CB_\gamma u\|$ plus lower order terms,
and hence, for
$\ssM$ chosen sufficently small,
we may absorb the $W_\gamma$ term in
$\norm{C B_\gamma u}^2.$

Finally we consider the LHS of \eqref{commutatorargument}.
We have
\begin{multline}
|\langle A_\gamma\Box u,A_\gamma u\rangle|
\leq \|(T_{-1/2})^*x A_\gamma\Box u\|
\,\|T_{1/2} x^{-1}A_\gamma u\|+|\langle xA_\gamma\Box u,x^{-1}FA_\gamma u\|\\
\leq \cte_1^{-1}\|(T_{-1/2})^*x A_\gamma\Box u\|^2
+\cte_1\|T_{1/2} x^{-1}A_\gamma u\|^2
+|\langle xA_\gamma\Box u,x^{-1}FA_\gamma u\|
\end{multline}
with $F\in\Psieb{-1,0}(X)$, hence $x^{-1}FA_\gamma$ uniformly bounded
in $\Psieb{s-1,r+1}(M)$, $xA_\gamma$ uniformly bounded in $\Psieb{s,r-1}(M)$,
so as $s=m+1/2$, and $r=l'+l+\frac{(f+1)}2 -1$, the last term is uniformly
bounded by the a priori assumptions. Similarly,
$\|(T_{-1/2})^*x A_\gamma\Box u\|$ is uniformly bounded, as
$(T^{-1/2})^* xA_\gamma$ is uniformly bounded in $\Psieb{s-1/2,r-1}(M)$,
while $\|T_{1/2} x^{-1}A_\gamma u\|^2$ can be absorbed in
$\norm{C B_\gamma u}^2$ (for $\cte_1$ sufficiently small) as discussed above.

The net result is that
$$
\norm{x^{-1} C Bu_\gamma}^2
$$
remains uniformly bounded as $\gamma\downarrow 0.$
Noting
that $CB_\gamma\to CB$ strongly (cf.\ the proof of Lemma~\ref{lemma:density}),
$CB\in\Psieb{s+1/2,r+1}(M)$ is elliptic at $q$,
$s=m+1/2$, and $r=l'+l+\frac{(f+1)}2 -1$, $\hilbert =\Hes{1, l}(M)$, we
can complete the proof
in the standard manner.
\end{proof}

\subsection{Propagation at glancing points within the edge}

Let $q \in \gcal\backslash \rcal_{\ebo}$
be given by
$$
x=0,\ t=t_0,\ y=y_0,\ z'=0,\ z''=z''_0,\ \xiebh=\xiebh_0,
\ \etaebh=\etaebh_0, \ \zetaebh'=0,\ \zetaebh''=\zetaebh_0''.
$$
As $q\in\gcal$,
$$
\xiebh^2_0+h(y_0,\etaebh_0)+k(y,z,\zetaebh'=0,\zetaebh_0'')=1.
$$
As $q\notin\rcal_{\ebo}$,
$$
\xiebh_0^2+h(y_0,\etaebh_0)<1,
$$
so $\zetaebh_0''\neq 0$, and $h(y_0,\etaebh_0)<1$, so $\varpi_{\ebo}(q)\in
\hcal_{W,\bo}$.
We will let $\Pi$ locally denote the coordinate projection onto the
variables
$$
(\xiebh, z'',\zetaebh'');
$$
Let $\ssW$ be a homogeneous vector field equal to $V_0$ (from
\eqref{eq:sH-eb-h}) at $q,$ and extended in local coordinates to a
constant vector field in $(\xiebh,z'',\zetaebh''),$ i.e., to be a
vector field in these variables only, with constant
coefficients.

\begin{theorem}\label{theorem:edgeglancing}
For Neumann boundary conditions, let $\hilbert=\Hes{1,l}(M)$,
$\hilbertp=\dHes{-1,l-2}(M)$; for Dirichlet boundary conditions let
$\hilbert=\Hesz{1,l}(M)$,
$\hilbertp=\Hes{-1,l-2}(M)$.

Let $u \in \hilbert$ solve
$\Box u=f$, $f\in\hilbertp$.
Let $q \in \gcal \backslash \rcal_{\ebo}$ be as above, and suppose
that $m\in\RR$, $l'\leq 0$, and
$q\notin\WFebY^{m+1,l'}(f)$.

There exists $\delta_0>0$ and $C_0>0$ such that for all $\delta\in (0,
\delta_0)$ and $\beta \in (C_0 \delta, 1)$,
\begin{multline}\label{eq:edgeglancing-implication}
\Sigma_{\ebo} \cap\big\{q': |\Pi(q')-\Pi(q)-\delta
\ssW|<\delta\beta ,\ 
 |z'(q')|<\delta\beta \big\}\cap
 \WFebX^{m,l'} u=\emptyset\\
\Longrightarrow  q \notin \WFebX^{m,l'}u
\end{multline}
\end{theorem}

\begin{remark}
Here the interesting case is taking $\beta$ as small as possible,
i.e.\ $\cO(\delta)$, to localize in a $\cO(\delta^2)$-ball around
$\Pi(q)+\delta \ssW$, which is what makes the proof of propagation of
singularities
result possible (by eventually letting $\delta\to 0$).
The statement of the theorem may be vacuous for $\beta$ large.
\end{remark}

\begin{proof}
Below we will choose $\delta_0>0$ sufficiently small so that
$\WFebY^{m+1,l'}(f)$ is disjoint from a $\delta_0$-neighborhood of $q$
(see the discussion before Lemma~\ref{lemma:glancingestimate}).

Let  $k$ be the codimension of
the face over which $q$ lies.
Let $\rho_{2n-2k}$ be the degree-zero homogeneous function with
$$
\rho_{2n-2k}|_{\ef}=\taueb^{-2}\tilde h|_{\ef},
$$
with $\tilde h$ as in Lemma~\ref{lemma:es-eb-wave-op};
note that $\rho_{2n-2k}(q)=0$ by \eqref{egh-coordinates} and
$d\rho_{2n-2k}(q)\neq 0$ since at least one of the $d\zetaebh''_{k+1}(q),
\ldots,d\zetaebh''_f(q)$ components of $d\rho_{2n-2k}(q)$ is non-zero,
in view of the quadratic nature of $\taueb^2(1-\rho_{2n-2k})$
in the fibers of the cotangent
bundle and \eqref{egh-coordinates} and $\zetaebh''_0\neq 0$ as observed above.
We remark that, with $V_0$ as
in \eqref{eq:sH-eb-h},
$$
V_0\rho_{2n-2k}|_{\ef}=0.
$$
Note that $\Sebstar M$ has dimension $2(n+1)-1=2n+1$, thus, with $C=\{x=0,
z'=0\}$,
$\cG\cap\Sebstar[C] M$
has dimension $2n+1-2k-2=2n-2k-1$ in view of \eqref{egh-coordinates}.
We proceed by remarking that
$$
(\ssW x)|_{\ef}=0,\ (\ssW y_j)|_{\ef}=0,\ (\ssW t)|_{\ef}=0,
\ (\ssW\etaebh_j)|_{\ef}=0,
$$
so $t$, $y_j$, $\etaebh_j$ give $1+2(n-f-1)=2(n-f)-1$ homogeneous degree
zero functions on $\Tebstar M$ (or equivalently
$\CI$ functions on $\Sebstar M$) whose restrictions to
$\cG\cap\Sebstar[C] M$ have linearly independent differentials at $q$.
We let $\rho_2,\ldots,\rho_{2n-2f}$ be given by these functions, and let
$\rho_1=x$.
We next remark that, in the notation of \eqref{metricform},
$$
\ssW \xiebh(q)=-2k_2(x=0,y_0,z'=0,z_0''),\zetaebh_0'')<0
$$
as $\zetaebh_0\neq 0$, hence $\ssW(q)\neq 0$.
Further, we let $\rho_j,$ $j=2n-2f+1,\ldots,2n-2k-1$,
be degree-zero homogeneous functions on $\Tebstar M$ (or equivalently
$\CI$ functions on $\Sebstar M$)
such that
$d(\rho_2|_{\cG}),\ldots,d(\rho_{2n-2k-1}|_{\cG})$ have linearly
independent differentials at $q,$ and such that
$$
\ssW \rho_j(q)=0.
$$
Such functions $\rho_j$ exist as
$\cG\cap\Sebstar[C] M$ has dimension $2n+1-2k-2=2n-2k-1$, so the $2n-2f-1$
functions $\rho_2,\ldots,\rho_{2n-2f}$ can be complemented by
some functions $\rho_{2n-2f+1},\ldots,\rho_{2n-2k-1}$ to obtain $2n-2k-1$
functions whose pullbacks to $\cG\cap\Sebstar[C] M$
have linearly independent differentials and which are annihilated by $\ssW$
at $q$, for the space of such one-forms is $2n-2k-2$ dimensional.
Thus,
by dimensional considerations (using $\ssW(q)\neq 0$),
$\{d\rho_j|_{\cG}(q):\ j=2,\ldots,2n-2k-1\}$ spans the space of
one-forms on $\cG$ annihilated by $\ssW(q)$, and $d\rho_2,\ldots,d\rho_{2n-2k-1}$
together with $d(\xiebh|_{\cG}-\xiebh_0)$ span $T^*\cG$.
Let
$$
\omega_0 = \sum_{j=1}^{2n-2k-1} \rho_j^2;
$$
then keeping in mind
that $|x|\leq\omega_0^{1/2}$, and with $V_0,V_i,V_{ij}$ as in
\eqref{eq:sH-eb-h},
\begin{equation*}\begin{split}
&|\tau|^{-1} |V_0 \omega_0| \lesssim \sqrt{\omega_0}(\sqrt{\omega_0}+
|\xiebh-\xiebh_0|+|z'|),\\
&|V_i\omega_0|\lesssim \sqrt{\omega_0}\,x\leq\omega_0,
\qquad |\tau|\,|V_{ij}\omega_0|\lesssim \sqrt{\omega_0},
\end{split}\end{equation*}
by \eqref{eq:sH-eb-h} and \eqref{whatahorror}.
Note also that
\begin{equation*}\begin{split}
&\big\lvert\tau|^{-1}|V_0 |z'|^2\big\rvert\lesssim |z'|(|z'|+x)\\
&\big\lvert V_i |z'|^2\big\rvert \lesssim |z'|,\qquad \abs{\tau}\big\lvert V_{ij} |z'|^2\big\rvert\lesssim |z'|^2.
\end{split}\end{equation*}
Let $\omega=\omega_0+|z'|^2.$  Then
\begin{equation}\begin{split}\label{eq:omega-tgt-est}
&|\tau|^{-1} |V_0 \omega| \lesssim \sqrt{\omega}
(\sqrt{\omega}+ |\xiebh-\xiebh_0|),
\qquad |V_i\omega|\lesssim \sqrt{\omega},
\qquad |\tau|\,|V_{ij}\omega|\lesssim \sqrt{\omega}.
\end{split}\end{equation}
Let
$$
\phi = \xiebh-\xiebh_0+ \frac{1}{\beta^2\delta} \omega
$$
By \eqref{eq:xiebh-weights},
\begin{equation}\begin{split}\label{eq:xiebh-tgt-est}
&\big\lvert|\tau|^{-1}\,V_0\xiebh+2\sum_{ij}
k_{2,ij}\zetaebh''_i\zetaebh''_j\big\rvert \lesssim x
\leq\omega^{1/2},\\
&|V_i\xiebh|\lesssim (|z'|+x)\lesssim\omega^{1/2},
\qquad |\tau|\,|V_{ij}\xiebh|\lesssim 1.
\end{split}\end{equation}
In particular, as $\zetaebh''_0\neq 0$,
$$
|\tau|^{-1}\,V_0\xiebh\leq -c_0+C_1'\omega^{1/2},
$$
for some $c_0>0$, $C_1'>0$.

Set
$$
a= |\tau|^s x^{-r}\chi_0(\ssM(2-\phi/\delta))
\chi_1\big((\xiebh-\xiebh_0+\delta)/\beta\delta +1\big) \chi_2\big(|\zetaebh'|^2\big).
$$
We always assume for this argument that
$\beta<1$, so on $\supp a$ we have
\begin{equation*}
\phi\leq 2\delta\Mand \xiebh-\xiebh_0\geq-\beta\delta-\delta\geq-2\delta.
\end{equation*}
Since $\omega\geq 0$, the first of these inequalities implies that
$\xiebh-\xiebh_0\leq 2\delta$, so on $\supp a$
\begin{equation}\label{xilimits}
|\xiebh-\xiebh_0|\leq 2\delta.
\end{equation}
Hence,
\begin{equation}\label{eq:omega-delta-est-t}
\omega\leq \beta^2\delta(2\delta-(\xiebh-\xiebh_0))\leq4\delta^2\beta^2.
\end{equation}
Moreover, on $\supp d\chi_1$,
\begin{equation}\label{eq:dchi_1-supp}
\xiebh-\xiebh_0\in[-\delta-\beta\delta,-\delta],\ \omega^{1/2}\leq 2\beta\delta,
\end{equation}
so this region lies in the hypothesis region of \eqref{eq:edgeglancing-implication}
after $\beta$ and $\delta$ are both replaced by appropriate constant
multiples.

We now quantize $a$ to $A \in \Psieb{s,r}(M).$ 
By Lemma~\ref{lemma:Box-eb-comm},
\begin{equation}\label{eq:Box-comm-rad-3}
\imath[\Box,A^*A]=\sum Q_i^* L_{ij} Q_j+\sum (x^{-1}
L_i Q_i+Q_i^* x^{-1}L'_i)
+x^{-2}L_0,
\end{equation}
with
\begin{equation}\begin{split}\label{eq:Box-comm-rad-terms-3}
&L_{ij}\in\Psieb{2s-1,2r}(M),\ L_i,L'_i\in\Psieb{2s,2r}(M),
\ L_0\in\Psieb{2s+1,2r}(M),\\
&\sigma_{\ebo,2s-1}(L_{ij})=2aV_{ij}a,\\
&\ \sigma_{\ebo,2s}(L_i)=\sigma_{\ebo,2s}(L'_i)=2a V_i a,
\ \sigma_{\ebo,2s+1}(L_0)=2a V_0a,\\
&\ebWF'(L_{ij}),\ebWF'(L_i),\ebWF'(L'_i),\ebWF'(L_0)\subset\ebWF'(A),
\end{split}\end{equation}
with $V_{ij}$, $V_i$ and $V_0$ as above,
given by \eqref{eq:sH-eb-h}.
Thus, we obtain
\begin{multline}\label{glancingcommutator}
\imath[\Box, A^*A] \\ =  B^*\Big( C^* C +\sum D_{z'_i}^*C_{ij} D_{z'_j}
+ \sum
 \big(R_i D_{z'_i}+D_{z'_i}^* R'_i\big) + \sum D_{z'_i}^* R_{ij}D_{z'_j} \Big)B\\ +
A^* WA+R''+ E+ E'
\end{multline}
where
\begin{enumerate}
\item\label{zeroitem3} $B \in \Psieb{s+1/2,r+1}(M)$ has symbol
$$
\ssM^{1/2}|\tau|^{s+1/2} x^{-r-1} \delta^{-1/2}\chi_1\chi_2\sqrt{\chi_0\chi_0'}
$$
\item\label{firstitem3} $C \in \Psieb{0,0}(M),$ has strictly positive
symbol on a neighborhood of $\ebWF'B,$ given by $(-V_0\phi)^{1/2}$ near
$\ebWF'B$,
\item $C_{ij}\in\Psieb{-2,0}(M)$, $(C_{ij})$ positive semidefinite,
\item $R_i,R'_i,R_{ij}$ are in
  $\Psieb{-1,0}(M),$ $\Psieb{-1,0}(M),$ and $\Psieb{-2,0}(M)$
  respectively and have (unweighted) symbols $r_i$, $r'_i$, $r_{ij}$
with
\begin{equation}\label{eq:r-tgt-est}
|\taueb|\, |r_i|,|\taueb|\, |r'_i|,|\taueb|^2\, |r_{ij}|\lesssim
1/\beta,
\end{equation}
\item $W \in \Diffess{2}\Psieb{-1,-2}(M),$
\item $R'' \in \Diffess{2}\Psieb{-2,0}(M),$ $E,E' \in
\Diffess{2}\Psieb{-1,0}(M),$
\item $E$ is microsupported where we have assumed regularity,
\item\label{lastitem3} $E'$ is supported off the characteristic set.
\end{enumerate}

These terms arise as follows.  By \eqref{eq:omega-tgt-est},
\eqref{eq:xiebh-tgt-est}, \eqref{xilimits}, and \eqref{eq:dchi_1-supp},
\begin{equation*}\begin{split}
V_0\phi&=V_0\big(\xiebh-\xiebh_0\big)+\frac{1}{\beta^2\delta}V_0\omega\\
&\leq -c_0+C_1'\omega^{1/2}
+\frac{1}{\beta^2\delta}C_1''\omega^{1/2}\big(\omega^{1/2}+|\xiebh-\xiebh_0|\big)\\
&\leq -c_0+2(C_1'+C_1'')\Big(\delta+\frac{\delta}{\beta}\Big)\leq
-c_0/4<0.
\end{split}\end{equation*}
provided that $\delta<\frac{2}{16 (C_1'+C_1'')}$,
$\frac{\beta}{\delta}>\frac{16(C_1'+C_1'')}{2}$, i.e.\ that $\delta$
is small, but $\beta/\delta$ is not too small---roughly, $\beta$ can
go to $0$ at most as a multiple of $\delta$ (with an appropriate
constant) as $\delta\to 0.$ Recall also that $\beta<1$, so there is an
upper bound as well for $\beta$, but this is of no significance as we
let $\delta\to 0$.  Thus, we define $C$ to have principal symbol equal
to the product of $(-V_0\phi)^{1/2}$ times a cutoff function
identically $1$ in a neighborhood of $\supp a$, but with sufficiently
small support so that $-V_0\phi>0$ on it. Thus, the $L_0$-term of
\eqref{eq:Box-comm-rad-3} gives rise to the $C^*C$ term of
\eqref{glancingcommutator}, as well as contributing to the $E$ and
$E'$ terms (where $\chi_1$ and $\chi_2$ are differentiated), $W$
(where the weight $|\tau|^s x^{-r}$ is differentiated) and the lower
order term $R''$.

Similarly, the $L_i$, $L'_i$ and $L_{ij}$ terms in which $V_i$
or $V_{ij}$ differentiates $\chi_1$ or $\chi_2$ contribute to the $E$ and
$E'$ terms, while those in which they differentiate the weight contributes
to the $W$ term,
so it remains to consider when $V_i$ and $V_{ij}$ differentiate
$\chi_0$. As we keep $\beta<1$,
$$
|V_i\phi|\leq |V_i\xiebh|+|V_i\omega|\lesssim
1+(\beta^2\delta)^{-1}\omega^{1/2}
\lesssim 1+\beta^{-1}\lesssim\beta^{-1},\ |V_{ij}\phi|\lesssim\beta^{-1},
$$
which thus proves the estimates on the terms arise this way,
namely $R_i$, $R'_i$, $R_{ij}$, above.

We now employ Lemma~\ref{lemma:glancingestimate} to estimate the
$D_{z'_i}$ terms as in the proof of Theorem~\ref{thm:bdy-radial}.
Note that we are using the finer result, Lemma~\ref{lemma:glancingestimate},
rather than its corollary here (unlike in Theorem~\ref{thm:bdy-radial}),
to obtain better control over the constant in front of the $D_{z'_i}$ terms
as we shrink $\delta$ and $\beta$.
The important fact is that $\cG\cap\Sebstar[C]M$
is defined by $\rho_{2n-2k}=0$, $x=0$, $z'=0$, and
$$
\rho_{2n-2k},x,|z'|\lesssim\omega^{1/2}\leq 2\delta\beta
$$
on the wave front set of $C_{ij},R_i,R'_i,R_{ij}$. Thus, we can apply
Lemma~\ref{lemma:glancingestimate} for a $C_1\beta\delta$-neighborhood of
a compact subset of $\gcal$. Noting that $xD_t T_{-1}\in\Psieb{0,0}(M)$,
we conclude that,
with $B_\gamma=B\Lambda_\gamma$, and for Neumann boundary conditions,
\begin{equation}\begin{split}\label{eq:glacing-normal-derivs-1}
&\sum \norm{D_{z_i'} T_{-1}B_\gamma u}^2
\leq C_0C_1\beta\delta\norm{B_\gamma u}^2 +\tilde C \cQ(u,Gu,\tilde Gu)\\
&\cQ(u,Gu,\tilde Gu)=
\|u\|_{\Hes{1,r-(f-1)/2}(M)}^2+\|Gu\|_{\Hes{1,r-(f-1)/2}(M)}^2\\
&\qquad\qquad\qquad\qquad\qquad
+\|\Box u\|_{\dHes{-1,r-(f+3)/2}}^2+\|\tilde G\Box u\|^2_{\dHes{-1,r-(f+3)/2}},
\end{split}\end{equation}
where $G \in \Psieb{s,0}(M)$, $\tilde G\in\Psieb{s+1/2,0}(M)$ (independent
of $\gamma$)
with wave front set in a neighborhood of $\supp a$. For Dirichlet conditions
we simply replace
$$
\Hes{1,r-(f-1)/2}(M)\ \text{by}\ \Hesz{1,r-(f-1)/2}(M)
$$
and
$$
\dHes{-1,r-(f+3)/2}(M)\ \text{by}\ \Hes{-1,r-(f+3)/2}(M).
$$
Note that by \eqref{eq:r-tgt-est} we have for all $w\in L^2$,
\begin{equation*}
\|T_1^* R_i^* w\|\leq C_2 \beta^{-1}\|w\|+\|\tilde R_i w\|,
\ \tilde R_i\in\Psieb{-1,0}(M),
\end{equation*}
with $\tilde R_i$ having the same microsupport as $R_i$.
But
\begin{equation*}\begin{split}
&|\langle R_i D_{z'_i}B_\gamma u,B_\gamma u\rangle|\\
&\leq
|\langle R_i T_1 D_{z'_i} T_{-1}B_\gamma u,B_\gamma u\rangle|
+|\langle R_i [T_1, D_{z'_i}] T_{-1}B_\gamma u,B_\gamma u\rangle|\\
&\leq \|D_{z'_i} T_{-1}B_\gamma u\|\,\|T_1^*R_i^*B_\gamma u\|
+|\langle \hat R_i B_\gamma u, B_\gamma u\rangle|,
\ \hat R_i\in\Diffes{1}\Psieb{-2,0}(M),
\end{split}\end{equation*}
and $|\langle \hat R_i B_\gamma u, B_\gamma u\rangle|$ can be
estimated by the inductive hypothesis, while
\begin{equation*}\begin{split}
&\|D_{z'_i} T_{-1}B_\gamma u\|\,\|T_1^*R_i^*B_\gamma u\|\\
&\leq (C_0C_1\beta\delta)^{1/2}C_2 \beta^{-1}\|B_\gamma u\|^2
+\tilde C^{1/2}\cQ(u,Gu,\tilde Gu)^{1/2}
\|B_\gamma u\|\\
&\leq C_2(C_0C_1\delta/\beta)^{1/2}\|B_\gamma u\|^2
+\parameter^{-1} \tilde C\cQ(u,Gu,\tilde Gu)
+\parameter\|B_\gamma u\|^2.
\end{split}\end{equation*}
As $\parameter>0$ is freely chosen, the main point is that
if $\delta/\beta$ is sufficiently small, the first term can be absorbed into $\|CB_\gamma u\|^2$, for the
principal symbol of $C$ is bounded below by
$(c_0/4)^{1/2}$ on $\supp a$. Since the $R'_i$ term is analogous,
and the $R_{ij}$ term satisfies better estimates (for one
uses \eqref{eq:glacing-normal-derivs-1} directly, rather than its
square root, as for $R_i$), the proof can be finished as
in Theorem~\ref{theorem:edgehyperbolic}.
\end{proof}

Finally, applying arguments that go back to \cite[Section~3 and Proof
of Theorem~5.10]{Melrose-Sjostrand2}, see \cite[Proof of
Proposition~VII.1]{Lebeau5}
and \cite[Proof of Theorem~8.1]{Vasy5} for the
setting of manifolds with corners, we may
put together Theorems~\ref{theorem:edgehyperbolic} and
\ref{theorem:edgeglancing} to obtain propagation of edge-b wavefront
set along \EGBBs over the edge face:

\begin{theorem}\label{thm:prop-sing}
For Neumann boundary conditions, let $\hilbert=\Hes{1,l}(M)$,
$\hilbertp=\dHes{-1,l-2}(M)$; for Dirichlet boundary conditions let
$\hilbert=\Hesz{1,l}(M)$,
$\hilbertp=\Hes{-1,l-2}(M)$.

Let $u \in \hilbert$ solve
$\Box u=f$, $f\in\hilbertp$.
Then for all $s\in\RR\cup\{\infty\},$ $l'\leq 0$,
\begin{equation*}
\big((\WFebX^{s,l'}(u)\setminus\WFebY^{s+1,l'}(f))\cap \Sebstar[\ef]M\big)
\subset\ebSigma
\end{equation*}
is a union of maximally extended \EGBBs in
$\ebSigma\setminus\WFebY^{s+1,l'}(f).$
\end{theorem}

\section{Propagation of fiber-global coisotropic regularity}\label{section:coisotropic}

We now state a microlocal result on the propagation of
coisotropy.  The result says that coisotropic regularity propagates
along \EGBBs provided that we also have infinite order regularity
along all rays arriving at radial points in $\gcal.$

\renewcommand{\theenumi}{\roman{enumi}}
\begin{theorem}[Microlocal propagation of coisotropy]
\label{theorem:microlocal-coiso}
Suppose that $u\in\Hes{1,l}(M)$, $\Box u=0$, with Dirichlet or Neumann
boundary conditions (see
Definition~\ref{def:Dirichlet-Neumann}), $p\in\cH_{W,\bo}.$
Suppose also that
\begin{enumerate}
\item
$q\in (\cH_{\eb\to p,\bo}\cap\cR_{\ebo,O})\setminus\Sebstar[\pa\ef]M$,
\item \label{item:coiso} $u$ has coisotropic regularity of order $k\in \NN$ relative to
  $H^{m}$ on the coisotropic $\fcalW_{I,\reg}$ in an open set
  containing all points in $\fcalW_{I,p,\reg}\cap \{0<x<\delta\}$ that
  are geometrically related to $\fcal_{O,q}.$
\item
$\bWF(u)\cap \fcalW_{p,I,\sing}=\emptyset.$
\end{enumerate}
Then $u$ has coisotropic regularity of order $k$ relative to $H^{m'}$ for all
$$
m'<\min(m,l+f/2)
$$
on $\fcalW_{O,\reg}$, in a neighborhood of $\fcal_{O,q,\reg}$.
\end{theorem}

\begin{proof}
The second numbered assumption and propagation of $\eWF$ through incoming
radial points, Theorem~\ref{theorem:edge} part (1),
implies that along \EGBBs in the backward flow of $q$ which
pass through $\cR_{\ebo,I}\backslash \gcal$ there is no $\ebWF^{m,\tilde l}$,
with
$$
\tilde l=\min(l, m-f/2-0).
$$
In view of Theorem~\ref{thm:bdy-radial},
the third assumption gives the same
along \EGBBs in the backward flow of $q$ which pass through
$\cR_{\ebo,I}\cap \gcal.$ Thus, near $q$, but on the \EGBBs in the backward flow of $q$,
there is no $\ebWF^{m,\tilde l}$ at all. Propagation of singularities
through $q$ (Theorem~\ref{theorem:edge} part (2))
then gives no $\eWF^{\tilde m,\tilde l}$,
$\tilde m=\min(m,\tilde l+f/2-0)$, on the flow-out. Substituting in
$\tilde l$, we see that $\tilde m=\min(m,l+f/2-0)$, giving no
$\ebWF^{m',\tilde l}.$ Thus in $x>0$, near the flow out, there is no $\WF^{m'}$, which
gives the case $k=0$.

We now turn to the general case, $k\neq 0.$ To begin, note that
assumption~\ref{item:coiso} and Theorem~\ref{theorem:edge} imply that
in fact we have coisotropic regularity of order $k$ relative to
$H_{\ebo}^{m,\tilde l}$ at all $q' \in
\rcal_{\ebo,p,I}$ that are connected to $q$ by an \EGBB.  This in turn
yields absence of $\ebWF^{m+k,\tilde l}$ in a \emph{neighborhood} of each
such $q$ in $\Sebstar[\ef] M,$ as the operators in $\coiso$ are all
characteristic only at the radial points over $\ef.$ By
Theorem~\ref{thm:prop-sing} followed by the second part of
Theorem~\ref{theorem:edge}, we then achieve coisotropic regularity of
order $k$ relative to $\eH{m',\tilde l}$ at $q,$ hence in a
neighborhood as well.
\end{proof}

\begin{corollary}\label{corollary:coisotropic}
Suppose that $u\in\Hes{1,l}(M)$, $\Box u=0$, with Dirichlet or Neumann
boundary conditions, $p\in\cH_{W,\bo}$, $k\in\NN$.
Suppose also that
\begin{enumerate}
\item
$u$ has coisotropic regularity of order $k$ relative to $H^{m}$ on
the co\-isotropic $\fcalW_{I,\reg}$
in a neighborhood of $\fcal_{I,p,\reg}$,
\item
$\bWF(u)\cap \fcalW_{p,I,\sing}=\emptyset.$
\end{enumerate}
Then $u$ has coisotropic regularity of order $k$ relative to $H^{m'}$ for all
$$
m'<\min(m,l+f/2)
$$
on $\fcalW_{O,\reg}$ in a neighborhood of $\fcalW_{O,p,\reg}.$
\end{corollary}

Finally, we prove that the regularity with respect to which coisotropic
regularity is gained in the above results is not, in fact, dependent
on the weight $l:$

\begin{corollary}\label{corollary:coisotropic-2}
Suppose that $u\in H^1(M_0)$, $\Box u=0$, with Dirichlet or Neumann
boundary conditions, $p\in\cH_{W,\bo}$, $k\in\NN$, $\ep>0$.
There exists $k'$ (depending on
$k$ and $\ep$) such that if
\begin{enumerate}
\item
$u$ has coisotropic regularity of order $k'$ relative to $H^{s}$ on
the co\-isotropic $\fcalW_{I,\reg}$ in a neighborhood of
$\fcalW_{I,p,\reg}$, and
\item
$\bWF(u)\cap \fcalW_{p,I,\sing}=\emptyset$,
\end{enumerate}
then $u$ has coisotropic regularity of order $k$ relative to $H^{s-\ep}$
on $\fcalW_{O,\reg}$ in a neighborhood of $\fcalW_{O,p,\reg}.$
\end{corollary}

\begin{proof}
Consider Dirichlet boundary conditions first. Then
$$
u\in\Hesz{1,1-(f+1)/2}(M).
$$
Thus, by Corollary~\ref{corollary:coisotropic}, $u$ has coisotropic
regularity of order $k'$ relative
to\footnote{An improved version of the argument, using powers of $\pa_t$
to shift among Sobolev spaces, gives coisotropy of order $k'$
relative to $H^{s-1/2-\ep}$.} $H^{m-\ep}$, $m<\min(s,1/2)$
on $\fcal_{O,\reg}$ near $\fcal_{O,p,\reg}$, strictly away from $\pa M$.

On the other hand, by the propagation of singularities,
\cite[Corollary~8.4]{Vasy5},
$u$ is in $H^s$ along $\fcalW_{O,p}$.
Hence the theorem follows by the
interpolation result of the following lemma,
Lemma~\ref{lemma:coisotropic-interpolate}.

Consider Neumann boundary conditions next. Then $u\in\Hes{1,-(f+1)/2}(M)$,
so by Corollary~\ref{corollary:coisotropic}, $u$ has coisotropic
regularity of order $k'$ relative to $H^{m-\ep}$, $m<\min(s,-1/2)$
on $\fcalW_{O,\reg}$ near $\fcalW_{O,p,\reg}$, strictly away from $\pa M$.

Proceeding now as in the Dirichlet case, using \cite[Corollary~8.4]{Vasy5},
we complete the proof.
\end{proof}

\begin{lemma}\label{lemma:coisotropic-interpolate}
Suppose that $u$ is in $H^s$ microlocally near some point $q$ away from
$\pa M$, and it is coisotropic of order $N$ relative to $H^m$ near $q$
with $s>m$. Then for $\ep>0$ and $k<(\ep N)/(s-m)$,
$u$ is coisotropic of order $k$ relative to
$H^{s-\ep}$ near $q$.

In particular, if $u$ is in $H^s$ microlocally near some point $q$ away from
$\pa M$ and $u$ is coisotropic (of order $\infty$)
relative to $H^m$ near $q$ with $s>m$, then $u$ is coisotropic
relative to $H^{s-\ep}$ for all $\ep>0$.
\end{lemma}

\begin{proof}
If $Q\in\Psi^0(M)$ and $\WF'(Q)$ lies sufficiently close to $q$, then the
hypotheses are globally satisfied by $u'=Qu$. Moreover, being coisotropic,
locally $\fcal$ can be put in a model form $\zeta=0$ by a symplectomorphism
$\Phi$ in some canonical coordinates $(y,z,\eta,\zeta)$, by
\cite[Theorem~21.2.4]{Hormander3} (for coisotropic submanifolds one has
$k=n-l$, $\dim S=2n$, in the theorem).\footnote{Roughly
speaking, $y$ would correspond to
the coordinates $(x,y,t)$ on $M$, while $z$ correspond to the fiber
variables $z$ on $M$; this is literally
true in a model setting.} Further reducing $\WF'(Q)$ if
needed, and using an elliptic $0$th order Fourier integral operator $F$
with canonical relation given by $\Phi$ to consider the induced problem for
$v=Fu'=FQu$, we may thus assume that $v\in H^s$, and $D_z^\alpha v\in H^m$
for all $\alpha,$ i.e.\ $\langle D_z\rangle^N v\in H^m.$ Considering the
Fourier transform $\hat v$ of $v$, we then have $\langle\eta,\zeta\rangle^s
\hat v\in L^2$, $\langle\eta,\zeta\rangle^m\langle\zeta\rangle^N \hat v\in
L^2.$ But this implies
$\langle\eta,\zeta\rangle^{m\theta+s(1-\theta)}\langle\zeta\rangle^{N\theta}
\hat v\in L^2$ for all $\theta\in[0,1]$ by interpolation (indeed, in this
case by H\"older's inequality).  In particular, taking
$\theta=(\ep)/(s-m)$,
$\langle\eta,\zeta\rangle^{s-\ep}\langle\zeta\rangle^k \hat v\in L^2$ if
$k<(N\ep)/(s-m)$, and the lemma follows.
\end{proof}

\section{Geometric theorem}\label{section:geometric}

The final essential ingredient in the proof of the geometric theorem is
the dualization of the coisotropic propagation result,
Corollary~\ref{corollary:coisotropic-2}.
Before proving such a result,
we first make a definition analogous to
Definition~\ref{def:rel-eb-Sob}, but for the b-wave front set. This
relative b-wave front set was used in \cite{Vasy5} to describe the
propagation of singularities on $M_0$.

\begin{definition}
Let $\hilbert\subset\CmI(M_0)$ denote a Hilbert space on which, for
each $K\subset M_0$ compact,
operators in $\Psib^{0}(M_0)$ with Schwartz kernel supported in
$K\times K$ are bounded, with the operator norm of
$\Op (a)$ depending on $K$ and a fixed seminorm of $a.$

For $m\geq 0,$ let
$$
H_{\bo,\hilbert,\loc}^{m}(M_0)=\big\{u \in \hilbert_\loc: Au \in \hilbert_\loc \text{ for
  all } A \in \Psib^{m}(M_0)\big\}.
$$

Let $q\in\Sbstar M_0$, $u \in \hilbert_\loc .$
For $m\geq 0$, we say that $q\notin\WFbX^{m}(u)$ if there exists
$A\in\Psib^{m}(M)$ elliptic at $q$ such that $Au\in\hilbert_\loc .$  We
also define
$q \notin \WFbX^{\infty}(u)$ if there exists
$A\in\Psib^{0}(M_0)$ elliptic at $q$ such that $Au \in H_{\bo,\hilbert,\loc}^\infty(M_0).$
\end{definition}
\nomenclature[H]{$H_{\bo,\hilbert}^{m}$}{b-Sobolev space relative to a
  Hilbert space $\hilbert$}

\begin{theorem}\label{theorem:non-focusing}
Let $u\in H^{-\infty}_{\bo,H^1_{\loc}}(M_0)$
satisfy the wave equation with Dirichlet
or Neumann boundary conditions.  Let $p \in \hcal_{W,\bo},$ and $w \in
\fcalW_{O,p,\reg}.$

Suppose $k\in\NN$ and $\ep>0$. Then there is $k'\in\NN$ (depending
on $k$ and $\ep$) such that
if $\WFb(u)\cap\fcalW_{I,p,\sing}=\emptyset$ and
$u$ is non-focusing of order $k$ relative to $H^s$ on a neighborhood
of $\fcalW_{I,p,\reg}$ in $\fcalW_{I,\reg}$
then $u$ is non-focusing of order $k'$ relative to $H^{s-\ep}$ at $w$.
\end{theorem}

\begin{remark}
  The essential idea of the proof is as follows.  Our results on
  propagation of coisotropic regularity show that coisotropic
  regularity entering the corner along $\fcalW_{I,p,\reg},$ together
  with smoothness at $\fcalW_{I,\sing},$ imply coisotropic regularity
  along $\fcalW_{O,p,\reg}.$ In other words, regularity under
  application of $A^\alpha$ (in the notation of
  \S\ref{section:coisodef}) along $\fcalW_{I,p,\reg}$ together with
  smoothness along singular incoming rays, yields regularity under
  $A^\alpha$ along $\fcalW_{O,p,\reg}.$ Heuristically speaking, the
  dual condition to our incoming regularity hypothesis is that of
  lying in the sum of the ranges of the operators $A^\alpha, R,$ where
  $R$ is an operator of high order microsupported near
  $\fcalW_{I,p,\sing}.$ By time reversal and duality, we thus find that
  the condition of nonfocusing along $\fcalW_{I,p,\reg},$ i.e.\ lying
  in the sum of the ranges of the $A^\alpha$'s microlocalized there,
  plus arbitrary bad regularity near $\fcalW_{I,p,\sing},$ leads to
  nonfocusing along $\fcalW_{O,p,\reg}.$

The difficulty in implementing this plan is primarily in rigorously
making the duality arguments on spaces of coisotropic wave equation
solutions.  The reader familiar with \cite{mvw1} will note that the
arguments used here are considerably more intricate than those in
Section~13 of \cite{mvw1}.  The reason for this is two-fold.  First,
the identification of the dual spaces, denoted $\mathsf{H_i}^*$ in
\cite{mvw1}, was not fully explained and indeed somewhat flawed.  (In
particular, we note that the identification of the dual space of
coisotropic distributions that are also wave equation solutions
requires some effort.)  These defects are remedied in the current
treatment, in which we use the results on duality from
Appendix~\ref{appendix:functional} to identify the duals of
coisotropic distributions, together with results on the inhomogeneous
wave equation.  Second, difficulties are present in the corners
setting that did not arise in \cite{mvw1}; in particular, we must
identify adjoints with respect to the elliptic Dirichlet form of
operators microsupported on $\fcal_{I,p,\sing}.$  This requires
some functional analytic care.
\end{remark}

\begin{proof}
We assume $s\leq 0$ to simplify notation; we return to the general case at the
end of the argument.

Let $T=t(p)$, and choose $T_0<T<T_1$ sufficiently close to $T$.
Let $\chi$ be smooth step function such that $\chi\equiv 1$ on a
neighborhood of $[T,\infty]$ and $\chi\equiv 0$ on a neighborhood of
$(-\infty,T_0].$  We find that $$v\equiv \chi u$$ satisfies
$\Box v=f$ with $f=[\Box,\chi]u,$ and $v$ vanishes on a neighborhood
of $(-\infty,T_0]\times X.$ Thus, we write
$$
\Box^{-1}_+f=v.
$$
By propagation of singularities,
\cite{Vasy5}, only singularities of $f$
on $\fcalW_{I,p}$ affect regularity at $w$, i.e.\ if
$\fcalW_{I,p}\cap\WF^{s+1}_{\bo,\dot H^{-1}(M_0)}(f)=\emptyset$ then
$w\notin\WF^s_{\bo,H^1(M_0)}(u)$, hence in particular $w$ is non-focusing
of order $0$ relative to $H^{s+1}$.
 Thus,
\begin{equation*}\begin{split}
Q_0\in\Psib^0(M_0),
&\ \WF'(\Id-Q_0)\cap\fcalW_{I,p}\cap\Sbstar_{\supp d\chi}M_0=\emptyset\\
&\Longrightarrow
w\notin\WF_{\bo,H^1(M_0)}(\Box^{-1}_+ ((\Id-Q_0)f)),
\end{split}\end{equation*}
so it suffices
to analyze $\Box^{-1}_+(Q_0 f)$. We choose $\WF'(Q_0)$ sufficiently
small such that
\begin{equation*}\begin{split}
q\in \WF'(Q_0)\Rightarrow &\text{ either } q\notin\WFb(f) \text{ or }\\
&f \text{ is non-focusing of order}\ k\ \text{ at } q \text{ relative to } H^{s-1};
\end{split}\end{equation*}
this is possible by  our hypotheses. We may thus replace $f$
by $f_0=Q_0 f$, assume that $f_0$ is the sum of a distribution that is
non-focusing of order $k$ relative
to $H^{s-1}$ and is supported in $M_0^\circ$ plus
an element of $H^\infty_{\bo,H^1(M_0)}$,
and show that $\Box^{-1}_+ f_0$ is non-focusing at $w$ of order $k'$
(for some $k'$ to be determined) relative
to $H^{s-\ep}$.

Let
$$
T_0'<T_0<T_1<T_1'.
$$
We regard $[T_0,T_1]$ as the time interval for analysis, but we enlarge
it to $[T_0',T_1']$ in order to be able to apply some b-ps.d.o's with symbol
elliptic for $t\in[T_0,T_1]$ to elements of our function spaces. (The
ends of the interval would be slightly troublesome.)
We define a Hilbert space $\hilbertx$ to be $$\hilbertx=H^1([T_0', T_1']\times
X_0)$$ in the case of Neumann conditions,
or $$\hilbertx=H^1_0([T_0',T_1']\times X_0)$$ in the case of Dirichlet conditions,
where $0$ indicates vanishing
enforced at $[T_0',T_1']\times\pa X_0$ (but not at the endpoints of the
time interval). Let $\hilbertx^*$ be the $L^2_g$-dual of $\hilbertx$.

We further let
\begin{enumerate}
\item
\begin{equation*}
T_0<\tilde T_0<t_0'<t_0<T<t_1<t_1'<T_1
\end{equation*}
such that $\supp d\chi\subset(t_0,T)$.
\item
$\chi_0\in\CI(\RR)$ such that $\supp(1-\chi_0)\subset(\tilde T_0,+\infty)$,
$\supp\chi_0\subset(-\infty,t_0')$.
\item
$\cU_0$ be an open set
with $\overline{\cU_0}\subset \{t\in (t_0',T)\}$,
$\overline{\cU_0}\cap\fcalW_{I,\sing}=\emptyset$ and
$\WFbXp(f_0)\subset\cU_0$.
\item
$\cU_1$ be a neighborhood of $w$
with $\overline{\cU_1}\subset \{t\in (T,t_1')\}$ and
$\overline{\cU_1}\cap\fcalW_{O,\sing}=\emptyset$.
\item
$B_0,B_1\in\Psib^0(M)$ with
$$
\WF'(B_j)\subset\cU_j,\ w\notin\WF'(\Id-B_1),\ \WF'(\Id-B_0)\cap
\WFbXp(f_0)=\emptyset,
$$
and with Schwartz kernel supported in $(t'_0,t'_1)^2
\times X^2$.
\item
$A_i,$ $i=1,\dots,N,$ denote first-order
pseudodifferential operators, generating $\module$ as defined in
\S\ref{section:coisodef}, but now locally over a neighborhood of
$\overline{\cU_0}\cup\overline{\cU_1}$ in
$M^\circ,$ and with kernels compactly supported in $M^\circ.$
\item
$T_\nu\in\Psib^\nu(M)$ with elliptic principal symbol on
$[T_0,T_1]\times X_0$ with Schwartz kernel supported in $(T_0',T_1')^2\times
(X_0)^2$. Thus, $T_\nu$ can be applied to elements of $\hilbertx$ and $\hilbertx^*$.
\end{enumerate}

Now suppose that we are given $r$ and $\ep>0$. Then, with $k$ as in the
statement of the Theorem,
Corollary~\ref{corollary:coisotropic-2} gives a $k'=k'(r,\ep,k)$
using $s=r$ in the notation of that corollary.
We let $\hilbertp$ be a space of microlocally coisotropic functions
on $\liptic(B_1)$ relative to $\hilbertx^*$ which are in addition extremely
well-behaved elsewhere (they will be finite-order conormal to the boundary)
in $\{t\geq t_0'\}$, but are merely in $H^1_{\bo,\hilbertx^*}$
for $t$ near $T_0$:
Let $N>r>1+\ep$ and set
\begin{equation*}\begin{split}
\hilbertp=\{\psi&\in \hilbertx^*:\\
&\|T_N (\Id-B_1-\chi_0)\psi\|^2_{\hilbertx^*}+\|T_1\chi_0\psi\|^2_{\hilbertx^*}
+\sum_{|\alpha|\leq k'}\|T_r A^\alpha B_1 \psi\|^2_{\hilbertx^*}
<\infty\}.
\end{split}\end{equation*}
Thus,
$$
\psi\in\hilbertp\Rightarrow
\WFbXp^N(\psi)\subset\WF'(B_1)\cup\Sbstar_{\supp\chi_0}M_0,
$$
and
\begin{equation}\label{eq:hilbertp-H1b}
\|\psi\|_{H^1_{\bo,\hilbertx^*}}\lesssim\|\psi\|_{\hilbertp}
\end{equation}
for $\psi$ supported in $[T_0,T_1]\times X_0$ (where the $T_s$ are elliptic).

Also, let
$\hilbertz$ be the space of microlocally coisotropic functions
on $\liptic(B_0)$ relative to $\hilbertx$
(and just in $\hilbertx$ elsewhere):
\begin{equation*}
\hilbertz=\{\phi\in \hilbertx:
\ \sum_{|\alpha|\leq k}\|T_{r-1-\ep}A^\alpha B_0 \phi\|^2_{\hilbertx}<\infty\}.
\end{equation*}
Note that as discussed in Section~\ref{section:coisodef} (in
particular, Lemma~\ref{lemma:duals})
\begin{equation*}
\hilbertz^*=\hilbertx^*+\sum_{|\alpha|\leq k}T_{r-1-\ep}A^\alpha B_0 \hilbertx^*,
\end{equation*}
so by our assumption on  $f_0$, $f_0\in\hilbertz^*$, provided
$-(r-1-\ep)-1\leq s-1$, i.e.\ provided\footnote{Note
that for such $r$, $r\geq 1+\ep$ as required above, since $s\leq 0$.}
$r\geq -s+1+\ep$.
Moreover, if $v_0\in\hilbertp^*$, then
\begin{equation}\label{eq:v_0-form}
v_0\in \hilbertx+T_N (\Id-B_1-\chi_0)\hilbertx+T_1\chi_0\hilbertx
+\sum_{|\alpha|\leq k'}T_r A^\alpha B_1\hilbertx.
\end{equation}
In particular, as $w\notin\WF'(\Id-B_1-\chi_0)\cup\Sbstar_{\supp\chi_0}M_0$,
$v_0$ is non-focusing at $w$ of order
$k'$ relative to $H^{-r+1}$, hence relative to $H^{s-\ep}$, if we actually
choose $r=-s+1+\ep$.

For $I\subset[T_0,T_1]$, let $\dot\cD_I$ denote the subspace
of $H^\infty_{\bo,\hilbertx}$ consisting of functions supported in $I\times X$,
$\dot\cE_I$ denote the subspace $H^\infty_{\bo,\hilbertx^*}$ consisting of
functions supported in $I\times X$, so $\Box:\dot\cD_I\to\dot\cE_I$ is
continuous, $\dot\cD_I\subset\hilbertz$, $\dot\cE_I\subset\hilbertp$
are dense with continuous inclusions. Also let $\dot\cD=\dot\cD_I$,
$\dot\cE=\dot\cE_I$ for $I=(T_0,T_1)$. Finally we also let
$\hilbertz_I$ be space of restrictions of elements of $\hilbertz$ to
$I$, and analogously for $\hilbertp_I$.  Then, as we will prove in Lemma~\ref{lemma:geom1}\footnote{The only reason for
  Corollary~\ref{corollary:coisotropic-2} combined with
  Corollary~\ref{cor:inhomog-coiso} not yielding the result
  immediately is that Corollary~\ref{corollary:coisotropic-2} is
  stated for the {\em homogeneous} wave equation.  This suffices for
  our purposes as we only require inhomogeneities that are very
  regular near the boundary, hence the propagation result of
  \cite{Vasy5} is adequate.}, Corollary~\ref{corollary:coisotropic-2}
implies that
\begin{equation}\label{eq:energy-est}
\phi\in \dot\cD\Rightarrow \|\phi\|_{\hilbertz}\leq C\|\Box\phi\|_\hilbertp,
\end{equation}
where vanishing for $t$ near $T_1$ is used.
In fact, we prove a somewhat more precise statement:\footnote{Notice that $\phi$ is
merely supported in $(T_0,T_1)$ here; not in
$(\tau,T_1)$, which would be \eqref{eq:energy-est} on
$[\tau,T_1]$, except for the loss of going from $\tau'$ to $\tau$.}
\begin{lemma}\label{lemma:geom1}
For $\tau\in[T_0,T_1)$, $\tau'>\tau$,
\begin{equation}\label{eq:energy-est-3}
\phi\in \dot\cD\Longrightarrow \|\phi\|_{\hilbertz_{[\tau',T_1]}}
\leq C\|\Box\phi\|_{\hilbertp_{[\tau,T_1]}}.
\end{equation}
\end{lemma}

\noindent\textsc{Proof of Lemma:}
Recall first that by standard
energy estimates (taking into account the vanishing of $\phi$ near
$T_1$)
\begin{equation}\label{eq:basic-energy-est}
\phi\in \dot\cD\Longrightarrow \|\phi\|_{\hilbertx_{[\tau,T_1]}}
\lesssim \|\Box\phi\|_{\hilbertx^*_{[\tau,T_1]}}
+\|D_t\Box\phi\|_{\hilbertx^*_{[\tau,T_1]}}\lesssim
\|\Box\phi\|_{\hilbertp_{[\tau,T_1]}}.
\end{equation}
Thus, we only need to prove that for $|\alpha|\leq k'$,
$$
\|T_{r-1-\ep}A^\alpha B_0 \phi\|_{\hilbertx_{[\tau',T_1]}}
\lesssim \|\Box\phi\|_{\hilbertp_{[\tau,T_1]}}.
$$
If $\Box\phi$ is supported away from $\pa M$, then this follows from
Corollary~\ref{cor:inhomog-coiso} and Corollary~\ref{corollary:coisotropic-2}.
In general, let $Q\in\Psi^0(M)$ be such that $\WF'(B_1)\cap\WF'(\Id-Q)
=\emptyset$, and $Q$ has compactly supported Schwartz kernel in $(M^\circ)^2$.
Then $Q\Box\phi$ has support away from $\pa M$, so
\begin{equation}\label{eq:off-bdy-inhom}
\|T_{r-1-\ep}A^\alpha B_0 \Box^{-1}_-(Q\Box\phi)\|_{\hilbertx_{[\tau',T_1]}}
\lesssim \|Q\Box \phi\|_{\hilbertp_{[\tau,T_1]}}\lesssim
\|\Box \phi\|_{\hilbertp_{[\tau,T_1]}},
\end{equation}
where $\Box^{-1}_-$ denotes the backward solution of the inhomogeneous
wave equation.
On the other hand,
\begin{equation*}
\|(\Id-Q)\Box\phi\|_{H^N_{b,\hilbertx^*}(M_0)}\lesssim
\|\Box\phi\|_{\hilbertp_{[\tau,T_1]}},
\end{equation*}
so by propagation of b-regularity \cite{Vasy5},
$$
\|\Box^{-1}_-((\Id-Q)\Box\phi)\|_{H^{N-1}_{b,\hilbertx}(M_0)}\lesssim
\|\Box\phi\|_{\hilbertp_{[\tau,T_1]}},
$$
hence the much weaker statement
\begin{equation}\label{eq:near-bdy-inhom}
\|T_{r-1-\ep}A^\alpha B_0 \Box^{-1}_-((\Id-Q)\Box\phi)\|_{\hilbertx_{[\tau',T_1]}}
\lesssim \|\Box \phi\|_{\hilbertp_{[\tau,T_1]}},
\end{equation}
also holds.
Combining \eqref{eq:off-bdy-inhom} and \eqref{eq:near-bdy-inhom}
proves \eqref{eq:energy-est-3}.  \emph{This concludes the proof of
  Lemma~\ref{lemma:geom1},} and hence of \eqref{eq:energy-est} as well.

In particular, recalling that $\dot\cD=\dot\cD_{(T_0,T_1)}$,
\eqref{eq:energy-est} shows that
for $\psi\in\Ran_{\dot\cD}\Box$ there is a unique $\phi\in\dot\cD$
such that $\psi=\Box\phi$; we denote this by $\phi=\Box^{-1}\psi$. Thus,
$$
\|\Box^{-1}\psi\|_{\hilbertz}\leq C\|\psi\|_{\hilbertp},\ \psi\in\Ran_{\dot\cD}\Box.
$$

Now consider the linear functional on $\Ran_{\dot\cD}\Box$ given by
$$
\psi\mapsto \langle f_0,\Box^{-1}\psi\rangle,
\ \psi\in\Ran_{\dot\cD}\Box,
$$
which satisfies
$$
|\langle f_0,\Box^{-1}\psi\rangle|
\leq \|f_0\|_{\hilbertz^*}\|\Box^{-1}\psi\|_{\hilbertz}
\leq C \|f_0\|_{\hilbertz^*}\|\psi\|_{\hilbertp},\ \psi\in\Ran_{\dot\cD}\Box.
$$
This has a unique extension to a continuous linear functional $\ell$ on
$\overline{\Ran_{\dot\cD}\Box}$, the closure of $\Ran_{\dot\cD}\Box$ in
$\hilbertp$.

If we used the Hahn-Banach theorem at this point to extend the linear
functional $\ell$ further to a linear functional $v_0$ on all of $\hilbertp$,
we would obtain a solution of the wave equation $\Box v_0=f_0$ on
$(T_0,T_1)$, as $\langle \Box v_0,\phi\rangle=\langle v_0,\Box \phi\rangle
=\langle f_0,\phi\rangle$ for $\phi\in\dot\cD$, which is indeed non-focusing
at $w$, but we need not just {\em any} solution, but the forward
solution, $\Box_+^{-1}f_0$. So we proceed by extending the linear functional
$\ell$ to a continuous linear functional $L$ on
\begin{equation}\label{eq:sum-of-spaces}
\overline{\Ran_{\dot\cD_{(T_0,t_1')}} \Box}+\overline{\dot\cE_{(T_0,t'_0)}},
\end{equation}
first, in such a manner that the extension is $\ell$ on the first summand
and vanishes
on the second summand. If we actually have such an extension, then
we can further extend it to all of $\hilbertp$, then vanishing
on the first summand shows that it solves the wave equation on $(T_0,t_1')$,
while vanishing on the second summand shows that it
vanishes on $(T_0,t_0')$, so its restriction as a distribution
on $(T_0,t_0')$ is indeed $\Box_+^{-1}f_0$.
In order to obtain such an extension we show:
\begin{lemma}\label{lemma:geom2}\ 
\begin{enumerate}
\item
$\ell$ vanishes on the intersection of the two summands, so $L$ is
well-defined as a (not necessarily continuous) linear map,
\item
The subspace \eqref{eq:sum-of-spaces} of $\hilbertp$ is closed, and
given an element $\psi+\rho$ in the sum, there is a
representation\footnote{Since the intersection of the summands is non-trivial,
this can only be true for some representation, not all representations!}
$\tilde\psi+\tilde\rho$ of $\psi+\rho$ as a sum of elements of the two summands
such that
one can estimate the $\hilbertp$-norm of $\tilde\psi$ and $\tilde\rho$ in terms
of $\psi+\rho$.
\end{enumerate}
\end{lemma}

\noindent\textsc{Proof of Lemma:}
We start with the statement regarding intersection of the summands
in \eqref{eq:sum-of-spaces}.
Thus, we claim that if
\begin{equation*}
\supp f_0\subset[t_0,t_1]\times X,
\end{equation*}
then
\begin{equation}\label{eq:intersection-est}
\psi\in\overline{\Ran_{\dot\cD_{[T_0,t_1']}}\Box}\ \text{and}
\ \supp\psi\subset(T_0,t_0')\Longrightarrow \ell(\psi)=0.
\end{equation}
To see this let $\psi_j\to\psi$ in
$\hilbertp$, $\psi_j=\Box\phi_j$, $\phi_j\in\dot\cD_{[T_0,t_1']}$.
Then $\{\psi_j\}$ is
Cauchy in $\hilbertp$, hence $\{\phi_j\}$ is Cauchy in $\hilbertz$ by
\eqref{eq:energy-est}, hence converges to some $\phi\in\hilbertz$.
By the support condition on $\phi_j$, $\supp\phi\subset[T_0,t_1']$.
As $\Box\phi_j\to\Box\phi$ in $\hilbertx^*$ (for $\Box:\hilbertx\to\hilbertx^*$ is continuous),
and $\Box\phi_j=\psi_j\to\psi$
in $\hilbertp$, hence in $\hilbertx^*$, we deduce that $\psi=\Box\phi$,
i.e.\ $\Box\phi$ is supported in $(T_0,t'_0)$.
Thus, $\psi_j|_{(t'_0,T_1)}
\to 0$ in the $\hilbertp$ topology hence $\phi_j|_{[t_0,T_1]}\to 0$
in the $\hilbertz$ topology using
\eqref{eq:energy-est-3} with $\tau=t'_0$, $\tau'=t_0$,
so, by the support condition on $f$,
$$
|\langle f_0,\phi_j\rangle|\leq \|f_0\|_{\hilbertz_{[t_0,t_1]}^*}
\|\phi_j\|_{\hilbertz_{[t_0,t_1]}}\to 0,
$$
so we deduce that $\langle f_0,\phi\rangle=0$ as claimed.

Next we turn to the closedness of the sum in \eqref{eq:sum-of-spaces}.
First, we claim that if $\psi\in\Ran_{\dot\cD_{(T_0,t_1')}}\Box$,
$\rho\in\dot\cE_{(T_0,t_0')}$ then there
exist $\tilde\psi\in\Ran_{\dot\cD_{(T_0,t_1')}}\Box$,
$\tilde\rho\in\dot\cE_{(T_0,\tilde T_0)}$
such that
$$
\psi+\rho=\tilde\psi+\tilde\rho\ \text{and}
\ \|\tilde\psi\|_{\hilbertp}\lesssim\|\psi+\rho\|_{\hilbertp}.
$$
Indeed, let $\chi_+\in\CI(\RR)$ such that
$$
\supp\chi_+\subset (T_0,+\infty)
\ \Mand\ \supp(1-\chi_+)\subset(-\infty,\tilde T_0).
$$
Let $\Box^{-1}_-(\psi+\rho)$
denote the backward solution of the inhomogeneous wave equation; i.e.\ the
unique solution of $\Box \tilde u=\psi+\rho$ which vanishes on $(t_1,T_1)$.
Then let
$$
\tilde\psi=\Box\big(\chi_+\Box^{-1}_-(\psi+\rho)\big)\in
\Ran_{\dot\cD_{(T_0,t_1')}}\Box,
$$
so
$$
\tilde\rho\equiv\psi+\rho-\tilde\psi
=(1-\chi_+)(\psi+\rho)-[\Box,\chi_+]\Box^{-1}_-(\psi+\rho)
\in\dot\cE_{(T_0,\tilde T_0)}.
$$
Moreover,
\begin{equation}\label{eq:tilde-psi-split}
\tilde\psi=\chi_+(\psi+\rho)+[\Box,\chi_+]\Box^{-1}_-(\psi+\rho)
\end{equation}
satisfies
$$
\|\tilde\psi\|_{\hilbertp}\lesssim\|\psi+\rho\|_{\hilbertp},
$$
as follows by inspecting the two terms on the right hand side of
\eqref{eq:tilde-psi-split}: for the first this is clear,
for the second this follows from
$\|\psi+\rho\|_{H^1_{\bo,\hilbertx^*}}
\lesssim \|\psi+\rho\|_{\hilbertp}$, see \eqref{eq:hilbertp-H1b}, hence one has
a bound in $\hilbertx$ for $\Box^{-1}_-(\psi+\rho)$ by
\eqref{eq:basic-energy-est}, and then $\supp[\Box,\chi_+]\subset\supp d\chi_+
\subset(T_0,\tilde T_0)$ gives the desired bound in $\hilbertp$.
\emph{This concludes the proof of Lemma~\ref{lemma:geom2}.}

Thus, if
$\psi_j\in\Ran_{\dot\cD_{(T_0,t_1')}}\Box$,
$\rho_j\in\dot\cE_{(T_0,\tilde T_0)}$ and $\psi_j+\rho_j$ converges to
some $\nu\in\hilbertp$ then defining $\tilde\psi_j$ and $\tilde\rho_j$
as above, we deduce that due to the Cauchy property of $\{\tilde\psi_j
+\tilde\rho_j\}$, $\{\tilde\psi_j\}$ is Cauchy
in $\hilbertp$, hence so is $\{\tilde\rho_j\}$, thus by the completeness
of $\hilbertp$ they converge to elements in
$\psi\in\overline{\Ran_{\dot\cD_{(T_0,t_1')}}\Box}$,
resp.\ $\rho\in\overline{\dot\cE_{(T_0,\tilde T_0)}}$ with $\psi+\rho=\nu$.
This shows that
$\overline{\Ran_{\dot\cD_{(T_0,t_1')}}\Box}
+\overline{\dot\cE_{(T_0,\tilde T_0)}}$ is closed, and indeed gives an
estimate\footnote{This estimate follows from the open mapping theorem,
given that the sum is closed, but the direct argument yields it anyway.}
that if $\nu\in\overline{\Ran_{\dot\cD_{(T_0,t_1')}}\Box}
+\overline{\dot\cE_{(T_0,\tilde T_0)}}$ then there exists
$\psi\in \overline{\Ran_{\dot\cD_{(T_0,t_1')}}\Box}$ and
$\rho\in \overline{\dot\cE_{(T_0,\tilde T_0)}}$ such that
$\psi+\rho=\nu$ and
\begin{equation}\label{eq:closed-sum-est}
\|\psi\|_{\hilbertp}+\|\rho\|_{\hilbertp}\lesssim\|\nu\|_{\hilbertp}.
\end{equation}

As mentioned earlier, this construction allows us to define a unique
continuous linear functional $L$ on
$$
\overline{\Ran_{\dot\cD_{(T_0,t_1')}} \Box}+\overline{\dot\cE_{(T_0,t'_0)}},
$$
in such a way that it is $\ell$ on the first summand and
it vanishes on the second summand: uniqueness is automatic, existence
(without continuity) follows from \eqref{eq:intersection-est}, as
the two functionals agree on the intersection of the two spaces,
while continuity follows from \eqref{eq:closed-sum-est}.
Then we extend $L$ by the
Hahn-Banach theorem
to a linear functional $v_0$ on $\hilbertp$.

Then $v_0\in\hilbertp^*$ solves $\Box v_0=f_0$ on $(T_0,t_1')$,
since for $\phi\in\dot\cD_{(T_0,t_1')}$
$$
\langle \Box v_0,\phi\rangle=\langle v_0,\Box\phi\rangle=\langle f_0,\phi\rangle,
$$
and $v$ vanishes on $(T_0,t_0)$, for it vanishes on
$\dot\cD_{(T_0,t_0)}$, i.e.\ on test functions supported there, so it
is the restriction of the forward solution of the wave equation to
$(T_0,T_1)$.  We have thus shown that if $f_0\in\hilbertz^*$ is
supported in $[t_0,t_1]$, which holds if $f_0$ satisfies the support
condition, is microlocally non-focusing on $\cU_0$, and is conormal to
the boundary elsewhere, then the forward solution of $\Box v_0=f_0$ is
in $\hilbertp^*$ (cf.\ Lemma~\ref{lemma:nfinclusions}), hence by
\eqref{eq:v_0-form} it is in particular microlocally non-focusing of
order $k'$ relative to $H^{s-\ep}$ at $w.$ This completes the proof of
the theorem if $s\leq 0$.

If $s>0$, one could use a similar argument relative to slightly different
spaces: the only reason for the restriction is that elements of $\hilbertp$
lie in $\hilbertx^*$ and a larger space (which would thus have a smaller
dual relative to $L^2$) would be required to adapt
the argument. However, it is easy to reduce the general case to $s\leq 0$:
replacing $u$ by $\tilde u=(1+D_t^2)^N u$,
$N>s/2$, $\tilde u$ is non-focusing of order $k$ relative
to $H^{s-2N}$ on a neighborhood of $\fcalW_{I,p,\reg}$ and solves
the wave equation, hence it is non-focusing of order $k'$ relative to
$H^{s-2N-\ep}$ at $w$ by the already established
$s\leq 0$ case of this theorem,
and then the microlocal ellipticity of
$(1+D_t^2)^N$ near the characteristic set (recall that $w$ is over
the interior of $M_0$) shows that $u$ itself is
non-focusing of order $k'$ relative to
$H^{s-\ep}$ at $w$, as claimed.
\end{proof}

As a consequence of the proposition of nonfocusing, we are now able to
prove our main theorem:

\begin{theorem}\label{theorem:geometric} Let $u\in H^1_{\loc}(M_0)$ satisfy
  the wave equation with Dirichlet
or Neumann boundary conditions.  Let $p \in \hcal_{W,\bo},$ and $w \in
\fcalW_{O,p,\reg}.$

Assume
\begin{enumerate}
\item $u$ satisfies the nonfocusing condition\footnote{Recall from
Definition~\ref{def:coiso-nonfocus} that
    this means nonfocusing of {\em some} order $k.$  The nonfocusing
    order is irrelevant here: only the space relative to which the
    nonfocusing condition holds matters.} relative to $H^s$
  on an
  open neighborhood of $\fcalW_{I,p,\reg}$ in $\fcalW_{I,\reg},$
\item $\displaystyle \WF^s u \cap \{w' \in \fcalW_{I,p,\reg}: w',w\text{ are
    geometrically related}\}=\emptyset,$
\item $\bWF^{s}(u)\cap \fcalW_{p,I,\sing}=\emptyset.$
\end{enumerate}
Then
$$
w \notin \WF^{s-0} (u).
$$
\end{theorem}

\begin{proof}
By using a microlocal partition of unity (cf.\ the argument at the beginning
of Theorem~\ref{theorem:non-focusing}), we may arrange that (ii) is
strengthened to
\begin{equation}\label{eq:ii-strengthened}
\displaystyle \WF^\infty u \cap \{w' \in \fcalW_{I,p,\reg}: w',w\text{ are
    geometrically related}\}=\emptyset,
\end{equation}
and (iii) to
\begin{equation}\label{eq:iii-strengthened}
\bWF^{\infty}(u)\cap \fcalW_{p,I,\sing}=\emptyset,
\end{equation}
for if a microlocal piece $\tilde u$ of the solution is in $\bH{s}$
then it remains in $\bH{s}$ under forward evolution, by the results of
\cite{Vasy5}.

Let $r<s$.
On the one hand, by conditions (i) and (iii), $u$ satisfies the
non-focusing condition (of some, possibly large, order $k'$)
relative to $H^{r}$ at $w$ due to Theorem~\ref{theorem:non-focusing}.
On the other hand, by Theorem~\ref{theorem:microlocal-coiso},
\eqref{eq:ii-strengthened} and condition (iii),
$u$ is microlocally
coisotropic at $w$, i.e.\ there exists $S \in \RR$ such
that\footnote{The particular choice of $S$ is dependent on the
background regularity of the solution, which in turn can be low, depending
on the order of nonfocusing relative to $H^s$.}
microlocally near $w$
\begin{equation}\label{coiso-reg}
A^{\alpha} u \in H^S\quad \forall \alpha.
\end{equation}
Lemma~\ref{lemma:coisotropic-interpolate} now allows us to interpolate
between nonfocusing and \eqref{coiso-reg} to conclude that
microlocally near $w$, $u \in H^{r-0}.$  Since $r<s$ is
arbitrary, this proves the result.
\end{proof}

\begin{corollary}\label{corollary:lagrangian}
Let $u$ be a solution to $\Box u=0$ with Dirichlet or Neumann boundary
conditions, and let $p \in \hcal_{W,\bo}.$ 
Suppose that for some $\ep_0>0$, in a neighborhood of $\fcalW_{I,p}(\ep_0)$
in $\Sbstar_{M_0\setminus W} M_0,$
$u$ is a Lagrangian distribution of
order $s$ with respect to $\lag\subset T^* M_0^\circ,$ a conic
Lagrangian such that $\lag \cap \fcalW_{I,p,\sing}=\emptyset$
and the intersection
of $\lag$ and $\fcalW_{I,\reg}$ is transverse at $\fcalW_{I,p,\reg}$.

Then if $w \in
\fcalW_{O,\reg}$ is not geometrically related to any point in
$\lag,$
$$
w \notin \WF^{-s-(n+1)/4+(k-1)/2-0} u,
$$
where $k$ is the codimension
of $W.$
\end{corollary}

The a priori regularity of such a solution is $H^{-s-(n+1)/4-0}$ so this
represents a gain in regularity along the diffracted wave of $(k-1)/2-0$ derivatives.

\begin{proof}
Corollary~\ref{corollary:lagrangian} follows from
Theorem~\ref{theorem:geometric} together with the results of
Section~14 of \cite{mvw1}.  We therefore give only a brief sketch of
the proof.

Microlocally near any point in the transverse intersection of $\lag$
and $\fcalW_{I,\reg},$ we may apply a microlocally unitary FIO $T$
quantizing a conic symplectomorphism that brings $\lag$ and
$\fcalW_{I,\reg}$ to the respective normal forms
$$
N^*\{0\}\Mand\{\zeta_1=\dots =\zeta_{k-1}=0\}
$$
inside $T^*(\RR^{n+1})$ with coordinates
$(y_1,\dots, y_{n+2-k}, z_1,\dots z_{k-1})$ and dual coordinates
$\eta,\zeta.$ (One should think of the $y$ coordinates as analogous to
the collection of the coordinates $t,x,y$ used previously, while the
$z$ coordinates are associated the fiber variables, also called $z$
above.)
Thus our test module
$\module$ is generated by $D_{z_1},\dots D_{z_{k-1}}.$ Writing $Tu$ as the
inverse Fourier transform of a symbol $a$ of order $s-(n+1)/4,$ we
find that
$$
(\Id +D_{z_1}^2+\dots +D_{z_k}^2)^{-N} Tu  = \F_{(\eta,\zeta)\to (y,z)}^{-1} \big(
(1+\zeta_1^2+\dots+\zeta_{k-1}^2)^{-N} a \big).
$$
For $N \gg 0,$ the integral in $\zeta$ converges absolutely, and the result
is a continuous (indeed, as smooth as desired) family in $z$ of conormal distributions with
respect to the origin in $y;$ the order of growth of the amplitude is
still $s-(n+1)/4$  but as the dimension is now $(n+1)-k,$ the order of
the Lagrangian distribution is
now $s-(k-1)/4,$ while the resulting Sobolev regularity is
$-s-(n+1)/4+(k-1)/2-0.$  Thus, the nonfocusing condition is satisfied
relative to this Sobolev space, and Theorem~\ref{theorem:geometric}
yields the desired regularity of the diffracted wave.
\end{proof}

\begin{corollary}\label{corollary:fundsoln}
Let $\gamma:(-\ep_0,0] \to \bSigma$ be a \GBB normally
incident at $W$, $\gamma(0)=\alpha\in\hcal_{W,\bo}$,
and let $\overline\gamma$ be its
projection to $M_0^\circ.$  Given $o \in \overline\gamma((-\ep_0,0)),$
let $u_o$ be the forward fundamental solution of $\Box$,
i.e.\ $u_o=\Box_+^{-1}\delta_o$.

There exists $\ep>0$
such that if $o \in \overline\gamma((-\ep,0))$ then
for all $w \in \fcalW_{O,\alpha,\reg},$ such that $w$ is not geometrically related to
point in $\Sbstar_o M_0\cap\bSigma$,
$$
w \notin \WF^{(-n+k+1)/2-0} u_o
$$
where $k$ is the codimension of $W.$
\end{corollary}
Note that this represents a gain of $(k-1)/2-0$ derivatives relative
to the overall regularity of the fundamental solution, which lies in
$H^{-n/2+1-0}.$

\begin{proof}
The hypotheses on the location of $o$ ensure that, with $\lag$ denoting
the flow-out of $\Sbstar_o M_0\cap\bSigma$, $\lag$ is disjoint from
$\fcalW_{I,\alpha,\sing}$ in view of Corollary~\ref{corollary:o-no-glancing}.
Thus, the microlocal setting is the same as that of \cite{mvw1}, hence
the hypotheses of Corollary~\ref{corollary:lagrangian} are satisfied.
\end{proof}

\appendix
\section{Some functional analysis}\label{appendix:functional}

We often encounter the following setup. Suppose
that $\hilberth$, $\hilbertp$ are Banach, resp.\ locally convex,
spaces, and
$$
\iota:\hilbertp\to\hilberth
$$
is a continuous injection with dense range
(so one can think of $\hilbertp$ as a subspace of $\hilberth$ with a
stronger topology). Let $\hilberth',$ $\hilbertp'$ denote the spaces of
linear functionals on $\hilberth,$ $\hilbertp$ endowed with their
respective weak topologies (i.e., the
weak-* topology in the Banach space setting).
Then the adjoint of $\iota$ is the map
$$
\iota^\dagger:\hilberth'\to\hilbertp',
\qquad\iota^\dagger \ell(v)
=\ell(\iota v),
\ \ell\in\hilberth',
$$
and $\iota^\dagger$ is continuous in the respective topologies. The
injectivity of $\iota$ implies that $\iota^\dagger$ has dense range, while
the fact that $\iota$ has dense range implies that $\iota^\dagger$ is
injective. Thus, one can think of $\hilberth'$ as a subspace of
$\hilbertp'$, with a stronger topology.

If $\hilberth$ is a Hilbert space with inner product
$\langle.,.\rangle_{\hilberth}$ $\CC$-linear in the first argument,
there is a canonical (conjugate-linear)
isomorphism $j_\hilberth:\hilberth\to\hilberth'$ given by
$j_\hilberth(u)(v)=\langle v,u\rangle_\hilberth$. Suppose also that there
is a canonical conjugate linear isomorphism
$$
c_\hilberth:\hilberth\to\hilberth,\ c_\hilberth^2=\Id,
\ \langle u,c_\hilberth v\rangle=\langle v,c_\hilberth u\rangle;
$$
if
$\hilberth$ is a function space, this is usually given by pointwise
complex conjugation. Thus,
$$
T_\hilberth=j_\hilberth\circ c_\hilberth:\hilberth\to\hilberth'
$$
is a linear
isomorphism.
Thus, if $A:\hilbertp\to\hilberth$ is continuous linear, then
$A^\dagger:\hilberth'\to\hilbertp'$ continuous linear,
and
$$
A^\flat=A^\dagger\circ j_\hilberth\circ c_\hilberth:
\hilberth\to\hilbertp'
$$
is continuous and linear. In particular, letting $A$ be our continuous injection,
$$
\iota^\flat
=\iota^\dagger\circ j_{\hilberth}\circ c_{\hilberth}:
\hilberth\to\hilbertp'
$$
is linear, injective with dense range, so $\hilberth$
can be considered a subspace of $\hilbertp'$ (with a stronger topology).
In particular,
$$
\iota^\flat\circ\iota:
\hilbertp\to\hilbertp'
$$
is also injective
with dense range. One considers the triple
$(\hilbertp',\hilberth,\iota)$
the $\hilberth$-dual of $\hilbertp;$ we will denote this either simply by
$\hilbertp'$,or by $\hilbertp^*$ if we want to emphasize the inclusion
of $\hilbertp$ into $\hilbertp'$ via $\hilberth$,
in what follows.  Note that if $\hilbertp$ is also a
Hilbert space with a canonical conjugate linear isomorphism\footnote{Again,
pointwise complex conjugation on function spaces
is a good example.} $c_{\hilbertp}$,
$$
\iota\circ c_{\hilbertp}
=c_{\hilberth}\circ\iota,\ c_{\hilbertp}^2=\Id,
$$
then we have the canonical linear isomorphism
$T_{\hilbertp}=j_\hilbertp\circ c_{\hilbertp}:\hilbertp\to\hilbertp'$,
and it is important to keep in mind
that $T_\hilbertp$ is (usually) different from
$\iota^\flat\iota
=\iota^\dagger
\circ T_{\hilberth}\circ\iota$:
\begin{equation*}\begin{split}
&T_{\hilbertp}(u)(v)=\langle v,c_{\hilbertp}u\rangle_{\hilbertp},\\
&\iota^\flat\iota(u)(v)
=\langle  \iota v,(c_{\hilberth}\circ\iota) u\rangle_{\hilberth}
=\langle  \iota v,(\iota \circ c_{\hilbertp})u\rangle_{\hilberth},
\end{split}\end{equation*}
for $u,v\in\hilbertp$.
A simple example, when $X$ a compact manifold with a smooth non-vanishing
density $\nu$
is obtained by $\hilbertp=\dCI(X)$ (a Fr\'echet space)
and $\hilberth=L^2_\nu(X)$ with respect to the density $\nu$, with
$\iota:\hilbertp\to\hilbert$ the inclusion.
Then $\iota^\flat\iota:\dCI(X)\to
\CmI(X)$ is the standard inclusion of Schwartz functions in tempered
distributions: $\iota^\flat\iota f(\phi)=\int f\phi \,\nu$.

In fact, we shall always consider a setting with $\cD$ a dense
subspace of $\hilberth$, with a locally convex topology, with respect to
which the inclusion map is continuous (i.e.\ which is stronger than
the subspace topology), so using
the linear isomorphism $j_{\hilberth}\circ c_{\hilberth}:
\hilberth\to\hilberth'$, we have
continuous inclusions, with dense ranges,
$$
\cD\subset\hilberth\equiv\hilberth'\to\cD'.
$$

Suppose now that $A:\cD\to\cD$, hence $A^\dagger:\cD'\to\cD'$, and suppose
that $A^\dagger$ maps $\cD$, i.e.\ more precisely the range of
$\iota^\flat\iota$ (with $\iota:\cD\to\hilberth$ the inclusion),
to itself, and let
$$
\tilde A=(\iota^\flat\iota)^{-1}A^\dagger(\iota^\flat\iota):\cD\to\cD.
$$
Then for $f,\phi\in\cD$
\begin{equation}\begin{split}\label{eq:A-formal-adjoint}
\langle\iota\phi,\iota\tilde Af\rangle_{\hilberth}
&=(\iota^\flat\iota c_{\cD} \tilde Af)(\phi)
=(c_{\cD'}\iota^\flat\iota \tilde Af)(\phi)=(c_{\cD'} A^\dagger\iota^\flat
\iota f)(\phi)\\
&=\overline{ (A^\dagger\iota^\flat\iota f)(c_{\cD}\phi)}
=\overline{ (\iota^\flat\iota f)(A c_{\cD}\phi)}
=\overline{ \langle \iota c_{\cD}f, \iota A c_{\cD}\phi\rangle}\\
&=\langle  \iota c_{\cD} A c_{\cD}\phi, \iota f\rangle,
\end{split}\end{equation}
so $\tilde A$ is the formal adjoint of $c_{\cD}A c_{\cD}$ with respect
to the $\hilberth$ inner product.

Given a Hilbert space $\hilberth$ as above, hence an inclusion of
$\cD$ into $\cD'$,
we shall also have to consider subspaces $\hilbert$
of $\cD'$ with a locally convex
topology, which contain the image of $\cD$ in $\cD'$ (under the
$\hilberth$-induced inclusion map), and such that the inclusion
maps
$$
\cD\hookrightarrow\hilbert\hookrightarrow\cD'
$$
are continuous, with dense range, hence one has the corresponding sequence
of adjoint maps, which are continuous, with dense range, when all the
duals are equipped with the weak topologies. As $(\cD')'=\cD$,
one obtains
$$
\cD\hookrightarrow\hilbert'\hookrightarrow\cD'.
$$
If further
$$
\cD\hookrightarrow\hilbert\hookrightarrow\hilberth\hookrightarrow\cD'
$$
continuous, with dense ranges, then
$$
\cD\hookrightarrow\hilberth'\equiv\hilberth
\hookrightarrow\hilbert'\hookrightarrow\cD',
$$
and similarly if one had the reverse inclusion between $\hilbert$
and $\hilberth$.

One way that subspaces such as $\hilbertp$ arise is by considering a
a finite number of continuous
linear maps $A_j:\cD\to\cD$, such that there exist continuous
extensions $A_j:\cD'\to\cD'$ (which are then unique by the density of
$\cD$ in $\cD'$), hence
$A_j:\hilbert\to\cD'$, $j=1,\ldots,k$. Then, in what
essentially amounts to constructing a ``joint maximal domain'' for the $A_j$,
and writing $\iota_{\hilbert\cD'}:\hilbert\to\cD'$ for the inclusion,
let
\begin{equation}\label{eq:ops-define-space}
\hilbertp=\{u\in\hilbert:\ \forall j, A_j u\in\Ran\iota_{\hilbert\cD'}\}
\end{equation}
with
\begin{equation}\label{eq:hilbertp-norm}
\|u\|_{\hilbertp}^2=\|u\|^2_{\hilbert}
+\sum \|\iota_{\hilbert\cD'}^{-1}A_j u\|^2_{\hilbert},
\end{equation}
where the injectivity of $\iota_{\hilbert\cD'}$ was used. If $\{u_n\}$ is
Cauchy in $\hilbertp$, then it is such in $\hilbert$,
so converges to some $u\in\hilbert$, and thus
$A_j u_n\to A_j u\in\cD'$. Moreover, if $\{u_n\}$ is Cauchy in
$\hilbertp$ then $\iota_{\hilbert\cD'}^{-1}A_j u_n$ is Cauchy in $\hilbert$
so converges to some $v_j\in\hilbert$, hence $A_j u_n\to
\iota_{\hilbert\cD'} v_j$ in
$\hilbert$. Thus,
$A_j u=\iota_{\hilbert\cD'}v_j$, so $A_j u\in\Ran\iota_{\hilbert\cD'}$, and
$A_j u_n\to A_j u$ in $\hilbert$, proving that $\hilbertp$ is a Hilbert
space. We will simply write $A_j$ for $A_j|_{\hilbertp}:\hilbertp\to\hilbert$.
Note that $\cD\subset \hilbertp$, so $\hilbertp$ is dense in $\hilbert$.
However, $\cD$ is {\em not} necessarily dense in $\hilbertp$.

An example is given by $\hilbert=L^2_g(X)$ on a compact manifold with
or without boundary, $g$ a Riemannian metric, $A_j$ be a finite set of
$\CI$ vector fields which span all vector fields over $\CI(X)$, $\cD$
either $\CI(X)$ or $\dCI(X)$; then $\hilbertp=H^1_g(X)$. If
$\cD=\CI(X)$, then $\cD$ is dense in $\hilbertp$, but if
$\cD=\dCI(X)$, then this is not in general the case: it fails if the
boundary of $X$ is non-empty.  Other examples are given coisotropic
distributions, where the $A_j$ are products of first order ps.d.o's
characteristic on a coisotropic manifold; see a general discussion below for
spaces given by such ps.d.o's.

Another way a subspace like $\hilbertp$ might arise from
continuous linear maps $A_j:\cD\to\cD$ is the following.
In a ``joint minimal
domain'' construction, one can define
\begin{equation}\label{eq:hilbertpp-norm}
\|u\|_{\hilbertpp}^2=\|u\|^2_{\hilbert}
+\sum \|A_j u\|^2_{\hilbert},
\end{equation}
as above, and let $\hilbertpp$ be the completion of $\cD$ with respect
to this norm, so $\hilbertpp$ is a Hilbert space. Moreover, the inclusion
map $\iota_{\cD\hilbert}:\cD\to\hilbert$ as well as $A_j$
extend to continuous linear maps
$$
\widetilde{\iota_{\cD\hilbert}}=\iota_{\hilbertpp\hilbert},
\widetilde {A_j}:\hilbertpp\to\hilbert,
$$
and $\iota_{\hilbertpp\hilbert}$ has
dense range (for $\cD$ canonically
injects into the completion).
In addition, with $\hilbertp$ as above,
the inclusion map $\iota_{\cD\hilbertp}:\cD\to\hilbertp$ extends
continuously to a map
$$
\widetilde{\iota_{\hilbertpp\hilbertp}}:\hilbertpp\to\hilbertp
$$
which is an isometry, and
is in particular injective. This in particular shows that the inclusion
map from $\hilbertpp$ to $\hilbert$ is also injective, with a dense range.
For $X$ a manifold with boundary and $\cD=\dCI(X)$, $A_j$ vector fields as above, one has
$\hilbertp=H^1_0(X)$; with $\cD=\CI(X)$, one has $\hilbertp=H^1(X)$.

Note that the closure of $\cD$ in $\hilbertp$ is $\hilbertpp$, so $\cD$ is
dense in $\hilbertp$ if and only if $\hilbertp=\hilbertpp$
(i.e.\ $\widetilde{\iota_{\hilbertpp\hilbertp}}$ is surjective).
\emph{From this point
on we {\em assume} that $\hilbertp=\hilbertpp$.} This is true, for instance,
if one is given $B_1,\ldots,B_r\in\Psieb{1,0}(M)$, and $A_1,\ldots,A_k$ are
up to $s$-fold products of these, as shown below in Lemma~\ref{lemma:density}.
Thus, $\hilbertp'\subset
\cD'$ (i.e.\ the inclusion map is injective).

Using the inclusion map $\iota_{\hilbertp\hilbert}$ we can now identify
the dual of $\hilbertp$ with respect to $\hilberth$.
We start with the case $\hilberth=\hilbert$.
By the Riesz lemma, $\hilbertp'=T_\hilbertp \hilbertp$ and
$T_\hilbertp$ is unitary, where
$$
T_\hilbertp v(u)=\langle u,c_{\hilbertp}v\rangle_{\hilbertp}.
$$
But
\begin{equation*}\begin{split}
T_\hilbertp (v)(u)&=\langle u, c_{\hilbertp}v\rangle_{\hilbertp}
=\langle \iota_{\hilbertp\hilberth} u, \iota_{\hilbertp\hilberth} c_{\hilbertp} v\rangle_{\hilberth}
+\sum_j\langle A_j u, A_j c_{\hilbertp}v\rangle_{\hilberth}\\
&=j_{\hilberth}\iota_{\hilbertp\hilberth} c_{\hilbertp}
v(\iota_{\hilbertp\hilberth} u)
+\sum_j j_{\hilberth}A_j c_{\hilbertp}v(A_ju).
\end{split}\end{equation*}
Thus,
$$
T_\hilbertp (v)(u)
=\big((\iota_{\hilbertp\hilberth}^\flat\iota_{\hilbertp\hilberth}
+\sum A_j^\flat c_{\hilberth}A_j c_{\hilbertp})v\big)(u).
$$
We conclude that
$$
T_\hilbertp v
=(\iota^\flat_{\hilbertp\hilberth}\iota_{\hilbertp\hilberth}
+\sum A_j^\flat c_{\hilberth}A_j c_{\hilbertp})v,
$$
and
\begin{equation}\label{eq:hilbertp-to-dual}
\hilbertp^*=(\iota^\flat_{\hilbertp\hilberth}\iota_{\hilbertp\hilberth}+\sum A_j^\flat c_{\hilberth}A_j c_{\hilbertp})\hilbertp.
\end{equation}
This also shows that
$$
\hilbertp^*=\Ran\iota_{\hilbertp\hilberth}^\dagger+\sum\Ran A_j^\dagger,
$$
for $\supset$ follows from the definition of $\iota^\dagger$, etc., while
$\subset$ follows from \eqref{eq:hilbertp-to-dual}.
We recall here that by \eqref{eq:A-formal-adjoint}, $A_j^\dagger$ is
the formal adjoint of $c_{\cD} A_j c_{\cD}$.

More generally, we do not need to assume $\hilbert=\hilberth$; rather
assume that
\begin{equation}\label{eq:hilbert-hilberth-ip}
\langle u,v\rangle_{\hilbert}=\sum_k\langle B_k u,B_k v\rangle_{\hilberth},
\end{equation}
where $B_k:\hilbert\to\hilberth$ are continuous linear maps (and there
is no assumption on the relationship in the sense of inclusions
between $\hilbert$ and $\hilberth$). Then
\begin{equation*}\begin{split}
T_\hilbertp (v)(u)&=\langle u, c_{\hilbertp}v\rangle_{\hilbertp}
=\langle \iota_{\hilbertp\hilbert} u, \iota_{\hilbertp\hilbert}
c_{\hilbertp} v\rangle_{\hilbert}
+\sum_j\langle A_j u, A_j c_{\hilbertp}v\rangle_{\hilbert}\\
&=\sum_k\left(\langle B_k\iota_{\hilbertp\hilbert} u, B_k\iota_{\hilbertp\hilbert}  c_{\hilbertp} v\rangle_{\hilberth}
+\sum_j\langle B_kA_j u, B_kA_j c_{\hilbertp}v\rangle_{\hilberth}\right)\\
&=\sum_k\left(j_{\hilberth}B_k\iota_{\hilbertp\hilbert}  c_{\hilbertp} v
(B_k\iota_{\hilbertp\hilbert} u)
+\sum_j j_{\hilberth}B_kA_j c_{\hilbertp}v(B_k A_j u)\right)\\
&=\left(\sum_k\big((B_k\iota_{\hilbertp\hilbert})^\flat c_{\hilberth}B_k\iota_{\hilbertp\hilbert}+\sum_j(B_kA_j)^\flat c_{\hilberth} B_kA_j\big)
c_{\hilbertp}v\right)(u),
\end{split}\end{equation*}
so we conclude as above, using $(B_k A_j)^\flat=A^\dagger B_k^\flat$, etc.,
that
\begin{equation}\label{eq:hilbertp-dual}
\hilbertp^*=\Ran \iota_{\hilbertp\hilbert}^\dagger
+\sum_{j}A_j^\dagger;
\end{equation}
note that the same computation as in \eqref{eq:A-formal-adjoint} with
factors of $\iota$ omitted shows that $A_j^\dagger$ is the formal
adjoint of $c_{\cD} A_j c_{\cD}$.

We now prove the density lemma mentioned above. We start by commuting
bounded families of operators through products of first order ps.d.o's.

\begin{lemma}\label{lemma:Psi-rewrite}
Let $r\geq 1$.
For $s\in\NN$, let
$J_{s}$ be the set of maps $\{j:\{1,\ldots,s\}\to\{1,\ldots,r\}$.

Suppose that $A_1,\ldots,A_r\in\Psieb{1,0}(M)$, and for $s\in\NN$ and
$j\in J_{s}$, let
$$
A_j=A_{j_1}\ldots A_{j_s}.
$$
Then for $k\in\NN$,
$j\in J_{k}$ and $\{Q_n\}$ a uniformly bounded family in
$\Psieb{0,0}(M)$,
$$
A_j Q_n=Q_n A_j+\sum_{s\leq k-1} \sum_{i\in J_s} C_{in} A_i,
$$
with $\{C_{in}:\ n\in\NN\}$ uniformly bounded in $\Psieb{0,0}(M)$, and
the uniform bounds are microlocal (so in particular
$\WF'(\{C_{in}:\ n\in\NN\})\subset\WF'(\{Q_n:\ n\in\NN\})$).

Moreover, for $\ep>0$,
if $Q_n\to \Id$ in $\Psieb{\ep,0}(M)$ then
$C_{in}\to 0$ in $\Psieb{\ep,0}(M)$ as $n\to\infty$.
\end{lemma}

\begin{remark}
We do not need the microlocality of the uniform bounds below, but it is
useful elsewhere.
\end{remark}

\begin{proof}
We proceed by induction, with $k=0$ being clear.

Suppose $k\geq 1$, and the statement has been proved with $k$ replaced by
$k-1$. Then for $j\in J_k$,
$$
A_j Q_n=Q_n A_j+[A_{j_1},Q_n]A_{j_2}\ldots A_{j_k}+\ldots
+A_{j_1}\ldots A_{j_{k-1}}[A_{j_k},Q_n].
$$
Note that $[A_{j_m},Q_n]\in\Psieb{0,0}(M)$ uniformly, and in a microlocal
sense (and $[A_{j_m},Q_n]\to 0$
in $\Psieb{\ep,0}(M)$ if $Q_n\to \Id$ in $\Psieb{0,0}(M)$).
Thus, the first two terms are
of the stated form. For the others, there are $l\leq k-1$ factors
in front of the commutator, which is bounded in $\Psieb{0,0}(M)$ (and
converges to $0$ in $\Psieb{\ep,0}(M)$ if $Q_n\to \Id$ in $\Psieb{\ep,0}(M)$),
so by the inductive hypothesis
$$
A_{j_1}\ldots A_{j_l}[A_{j_{l+1}},Q_n]
$$
can
be rewritten as $\sum_{s\leq l} \sum_{i\in J_s} C_{s,in} A_i$, hence
$$
A_{j_1}\ldots A_{j_l}[A_{j_{l+1}},Q_n] A_{j_{l+2}}\ldots A_{j_k}
$$
is rewritten
as $C_{s,in} A_{i_1}\ldots A_{i_s}A_{j_{l+2}}\ldots A_{j_k}$ with
$s+(k-(l+1))\leq l+k-(l+1)=k-1$ factors of the $A$'s, hence is of
the stated form.
\end{proof}

\begin{lemma}\label{lemma:density}
Suppose that $B_1,\ldots,B_r\in\Psieb{1,0}(M)$, and let $\hilbert$
be a Hilbert space on which $\Psieb{0,0}(M)$ acts, with operator norm
on $\hilbert$ bounded by a fixed $\Psieb{0,0}(M)$-seminorm.
Let $\cD\subset\hilbert$ be a dense subspace with a locally convex topology, and
with all $Q\in\Psieb{-\infty,0}(M)$, $Q:\hilbert\to\cD$
continuous, while for all $Q\in\Psieb{m,0}(M)$, $Q:\cD\to\hilbert$ is
continuous, with bound given by a fixed $\Psieb{m,0}(M)$ seminorm
and a fixed seminorm on $\cD$.

For $k\in\NN$, let
\begin{equation*}
\hilbertp=\{u\in\hilbert:\ \forall s\leq k,\ \forall j\in J_s,
A_j u\in\hilbert\}
\end{equation*}
with
$$
\|u\|_{\hilbertp}^2=\|u\|^2_{\hilbert}
+\sum_{s\leq k}\sum_{j\in J_s} \|A_j u\|^2_{\hilbert},
$$
Then $\cD$ is dense in $\hilbertp$.
\end{lemma}

\begin{proof}
We start by observing that if $Q_n$ is a uniformly bounded family
in $\Psieb{0,0}(M)$, $Q\in\Psieb{0,0}(M)$ and $Q_n\to Q$
in $\Psieb{\ep,0}(M)$ for $\ep>0$,
then $Q_n\to Q$ strongly on $\hilbert$. Indeed, $Q_n$ is uniformly
bounded on $\hilbert$ by the assumptions of the lemma, so it suffices
to prove that for a dense subset of $\hilbert$, which we take to
be $\cD$, $u\in\cD$ implies $Q_n u\to Qu$ in $\hilbert$. But this
is immediate, for $Q_n\to Q$ in $\Psieb{\ep,0}(M)$, hence as a map
$\cD\to\hilbert$, by the assumptions of the lemma.

Now
let $\Lambda_n\in\Psieb{-\infty,0}(M)$ uniformly bounded in $\Psieb{0,0}(M)$
and $\Lambda_n\to\Id$ in $\Psieb{\ep,0}(M)$ for $\ep>0$, so $\Lambda_n\to\Id$
strongly on $\hilbert$.
We claim that
for $s\leq k$,
\begin{equation}\label{eq:A_j-approx}
u\in\hilbertp,\ j\in J_s\Longrightarrow A_j \Lambda_n u\to
A_j u\ \text{as}\ n\to\infty\ \text{in}\ \hilbert.
\end{equation}
Since $\Lambda_n u\in\cD$, this
will prove the lemma. Note that $\Lambda_n u\to u$ in $\hilbert$.

By Lemma~\ref{lemma:Psi-rewrite}, for $j\in J_s$,
$$
A_j \Lambda_n=\Lambda_n A_j+\sum_{l\leq s-1} \sum_{i\in J_l} C_{in} A_i,
$$
with $\{C_{in}:\ n\in\NN\}$ uniformly bounded in $\Psieb{0,0}(M)$,
$C_{in}\to 0$ in $\Psieb{\ep,0}(M)$ for $\ep>0$.
Correspondingly, $\{C_{in}:\ n\in\NN\}$ is uniformly bounded as operators
on $\hilbert$, and $C_{in}\to 0$ strongly on $\hilbert$. Since
$A_j u\in\hilbert$ and $A_i u\in\hilbert$ for $i\in J_l$,
$l\leq s$, we deduce that $A_j\Lambda_n u\to A_j u$ in $\hilbert$,
completing the proof.
\end{proof}

\section{The edge-b calculus}\label{BEPseuda}

Let $M$ be, for this section, a general compact manifold with corners and
let $W$ be one of its boundary hypersurfaces. At the end of the
section we comment on non-compact $M$, which is setting of the main
body of the paper; this is essentially a notational issue as our
problem is indeed local in a relevant sense.  In the body of the paper
above, $M$ is obtained by the blow up of a boundary face $Y$ of a manifold
with corners $M_0$ and $W$ is the front face of the blow up, i\@.e\@.~the
preimage of $Y$ under the blow-down map. In fact our discussion is mostly
local in the interior of $Y$ and hence \emph{we could assume that $Y$ has locally
maximal codimension, so that it has no boundary}. We shall
not, however, make this assumption here, and we include the setting obtained by blow-up
without actually restricting the discussion to it. Instead, we shall suppose that
$W$ is equipped with a fibration
\begin{equation}
\xymatrix{
Z\ar@{-}[r]&W\ar[d]^{\phi}\\
&Y.
}
\label{16.2.2009.1a}\end{equation}
Since the manifolds here may have corners, this is to be a fibration in that
sense, so the typical fiber, $Z,$ is required to be a compact manifold with
corners, the base, $Y,$ is a manifold with corners and $\phi$ is supposed to
be locally trivial in the sense that each $p\in Y$ has a neighborhood $U$
over which there is a diffeomorphism giving a commutative diagram
\begin{equation}
\xymatrix{
\phi^{-1}(U)\ar[rr]^{\simeq}\ar[dr]_{\phi}&& U\times Z\ar[dl]^{\pi_U}\\
&U.
}
\label{16.2.2009.2a}\end{equation}
Mainly for notational reasons we will also assume that $Y$ is connected.

Let $\mathcal{V}_{\eb}(M)\subset\bV(M)$ be the Lie subalgebra of all those
smooth vector fields on $M$ which are tangent to all boundary
(hypersur-)faces and in addition are tangent to the fibers of $\phi$ on
$W.$ The calculus of edge-b pseudodifferential operators will be
constructed in this setting, it is determined by $M$ and $\phi$ and
microlocalizes $\mathcal{V}_{\eb}(M).$

In case the $\phi$ has a single fiber, i\@.e\@.~$Z=W,$ corresponding to the
case that $W$ is not blown up at all, the Lie algebra
$\mathcal{V}_{\eb}(M)$ reduces to $\bV(M)$ and the desired
microlocalization is just the algebra of b-pseudodifferential operators on
$M$ as a manifold with corners. The construction of this algebra is
discussed in \cite{MR83f:58073}, \cite{Melrose-Piazza:Analytic} and of
course $W$ is in no way singled out amongst the boundary hypersurfaces. The
pseudodifferential operators are described in terms of their Schwartz
kernels, which are the conormal distributions with respect to the resolved
diagonal in a blown-up version of $M^2,$ with the additional constraint of
vanishing rapidly at boundary faces which do not meet the lifted
diagonal. The resolved double space in this case is
\begin{equation}
M^2_{\bo}=[M^2;\cB],\ \cB=\{B\times B;B\in\cM_1(M)\}.
\label{30.8.2006.40}\end{equation}
Here, in generality, $\cM_p(M)$ is the collection of all \emph{connected}
boundary faces of codimension $p$ of the manifold with corners $M.$ It is
of crucial importance that the lift to $M^2_{\bo}$ of the diagonal, is a
p-submanifold---the lift in this case is the closure (in $M^2_{\bo})$ of
the inverse image of the interior of the diagonal:
\begin{equation}
\Diag_{\bo}=\clos(\Diag(M)\cap \inside(M^2)).
\label{30.8.2006.41}\end{equation}
Then the operators on functions correspond to the kernels
\begin{equation}
\Psi^m(M)=\left\{A\in I^m(M^2_{\bo};\beta^*\Omega_R);
A\equiv0\Mat \pa M^2_{\bo}\setminus\ff(\beta)\right\}
\label{30.8.2006.47}\end{equation}
where $\ff(\beta)$ is the collection of boundary faces produced by the
blow-ups defining the combined blow-down map $\beta
:M^2_{\bo}\longrightarrow M^2.$

The composition properties of these operators, including the fact that the
``small calculus'' is an algebra, can be obtained geometrically from the
corresponding triple space 
\begin{multline}
M^3_{\bo}=[M^3;\cB^3,\cB^2],\ \cB^3=\{B^3;B\in\cM_1(M)\},\\
\cB^2=\{M\times
B\times B,\ B\times M\times B,\ B\times B\times M;B\in\cM_1(M)\}.
\label{30.8.2006.42}\end{multline}
There is considerable freedom in the order of blow-ups here and this is
sufficient to show that the three projections, $\pi_O,$ from $M^3$ to $M^2$
lift to ``stretched projections'' $\pi_{O,\bo}:M^3_{\bo}\longrightarrow M^2_{\bo},$
$O=S,C,F,$ corresponding to the left, outer and right two factors
respectively; these maps are b-fibrations and factor through a product of
$M^2_{\bo}$ and $M$ in each case.

As already noted, it is crucial for the definition \eqref{30.8.2006.47}
that the lifted diagonal $\Diag_{\bo}$ be a p-submanifold, meaning that it
meets the boundary locally as a product. This also turns out to be
essential in the \emph{construction} of $M^2_{\eb}$ below.

There is another extreme case in which the microlocalization of the Lie
algebra $\mathcal{V}_{\eb}$ is well-established, namely when $W$ is the
only boundary hypersurface, so $M$ is a manifold with boundary; this is the
case of an ``edge'' alone, with no other boundaries. The construction of a
geometric resolution in this case can be found in \cite{MR93d:58152} and
\cite{MR2000m:58046}. It is quite parallel to, and of course includes as a
special case, the b-algebra on a manifold with boundary. In the general
edge case when the fibration $\phi$ is non-trivial (but $W$ itself has no
boundary) the center blown up in \eqref{30.8.2006.40}, which would be
$W^2,$ is replaced by the fiber diagonal
\begin{equation}
\begin{gathered}
\Diag_{\phi}=\{(p,p')\in W^2;\phi(p)=\phi(p')\}=(\phi\times\phi)^{-1}(\Diag(F))\\
M^2_{\phi}=[M^2;\Diag_{\phi}].
\end{gathered}
\label{30.8.2006.43}\end{equation}
Similarly, the triple space is obtained by blow-up of the triple fiber product 
\begin{equation}
\Diag^3_{\phi}=\{(p,p',p'')\in W^3;\phi(p)=\phi(p')=\phi(p'')\}=
(\phi\times\phi\times\phi)^{-1}(\Diag^3(F))
\label{30.8.2006.44}\end{equation}
and then the three partial fiber diagonals, the inverse images,
$\Diag^O_{\phi},$ $O=S,C,F$ of $\Diag_{\phi}$ under the three projections
$\pi_O:M^3\longrightarrow M^2:$ 
\begin{equation}
M^3_{\phi}=[M^3;\Diag^3_{\phi};\Diag^S_{\phi},\Diag^C_{\phi},\Diag^F_{\phi}].
\label{30.8.2006.45}\end{equation}
Again the three projections lift to b-fibrations 
\begin{equation}
\xymatrix@1{
M^3_{\phi}\ar@<2.5ex>[r]^{\pi_{F,\phi}}\ar[r]^{\pi_{S,\phi}}
\ar@<-2.5ex>[r]^{\pi_{C,\phi}}&M^2_{\phi}.
}\label{30.8.2006.46}\end{equation}

The microlocalization of $\mathcal{V}_{\eb}$ is accomplished here by the
combination of these two constructions. The diagonal, even for a manifold
with boundary, is not a p-submanifold---does not meet the boundary faces
in a product manner---and as already noted this is remedied by the
b-resolution. Since the fiber diagonal in $W^2$ is the inverse image of the
diagonal in $Y$ it too is not a p-submanifold in case $Y$ has boundary, but
then the partial b-resolution of $M$ resolves it to a p-submanifold which
can then be blown up.

More explicitly, the boundary hypersurfaces of $M,$ other than $W$ itself,
fall into three classes according to their behavior relative to $\phi.$
Namely there may be some disjoint from $W;$ these are relatively
unimportant in the discussion below. Otherwise the intersection of $W$ and
such a boundary hypersurface, $B,$ is a boundary hypersurface $B\cap W$ of
$W.$ The remaining two cases correspond to this being the preimage under
$\phi$ of a boundary hypersurface of $Y$ or, if this is not the case, then
$B\cap W$ is the union of boundary hypersurfaces of the fibers of $\phi,$
corresponding to a fixed boundary hypersurface of $Z.$ In brief, the
boundary hypersurfaces $B\in\cM_1(M)\setminus\{W\}$ which meet $W$
correspond either to the boundary hypersurfaces of $Y$ or of $Z.$ Let
$\cB',$ $\cB(Y)$ and $\cB(Z)\subset\cM_1(M)$ denote the three disjoint
subsets into which $\cM_1(M)\setminus\{W\}$ is so divided.

To define the double space on which the kernels are conormal distributions
with respect to the lifted diagonal, just as in both special cases
discussed above, we make one blow up for each of the boundary
hypersurfaces. For those other than $W,$ this is the same as for the
b-double space for $M,$ which is to say the corners, $B\times B,$ are to be
blown up for all $B\in\cM_1(M)\setminus\{W\}.$ Since these submanifolds are
mutually transversal boundary faces within $M^2$ they may be blown up in
any order with the same final result. For $W$ we wish to blow up the fiber
diagonal, $\Diag_\phi,$ in \eqref{30.8.2006.43}. This is certainly a
manifold with corners, since it is the fiber product of $W$ with itself as
a bundle over $Y,$ given by $\phi.$ However, as noted, it is not embedded
as a p-submanifold if $Y$ has non-trivial boundary. If $x_i$ and $y_j$ are
respectively boundary defining functions and interior coordinates near some
boundary point of $Y,$ and $x',$ $y',$ $x'',$ $y''$ are their local lifts
to $W^2$ under the two copies of $\phi,$ then $\Diag_\phi\subset W^2$ is
the ``diagonal'' $x'=x'',$ $y'=y''.$ Near a boundary point of $Y$ this is not
a p-submanifold.

Note that in the simplest case, when $\cB(Y)=\emptyset,$ the following
lemma merely says that $\Diag_\phi$ is a p-submanifold of $M^2.$
\begin{lemma}\label{16.2.2009.7a} The fiber diagonal $\Diag_\phi$ lifts to
  a p-submanifold of $$[M^2;\cB(Y)].$$
\end{lemma}
We will still denote the lifted submanifold as $\Diag_\phi.$

\begin{proof} Since $\Diag_\phi\subset W^2$ and this is the smallest
  boundary face of $M^2$ with this property, under the blow up of other
  boundary faces of $M^2,$ $\Diag_\phi$ lifts to the subset (always a
  submanifold in fact) of the lift of $W^2$ under the blow up of the
  intersection of $W^2$ with the boundary face which is the center of the
  blow up. That is, to track the behavior of $\Diag_\phi$ we need simply
  blow up the intersections of the elements of $\cB(Y)$ with $W^2,$ inside
  $W^2.$ This corresponds to exactly the ``boundary resolution'' of $Y^2$ to
  $Y^2_{\bo}$ as discussed briefly above. So the diagonal in $Y$ lifts to
  be a p-submanifold. Since $\phi$ is a fibration over $Y,$ it follows
  easily from the local description that $\Diag_\phi$ lifts to a
  p-submanifold of the blow up, $[W^2;\cB(Y)\cap W^2]$ and hence to a
  p-submanifold of $[M^2;\cB(Y)]$ as claimed.
\end{proof}

Thus blowing up the elements of $\cB(Y)$ in $M^2$ resolves $\Diag_\phi$ to
a p-submanifold, by resolving the diagonal in $Y^2.$ The defining functions
of the elements of $\cB(Z)$ restrict to defining functions of boundary
faces of the lift of $\Diag_\phi$ so all the remaining boundary faces, in
$\cB'\cup\cB(Z)$ are transversal to this lift. Such transversality is
preserved under blow up of boundary faces, so we may define the
$\eb$-double space in several equivalent ways as regards the order of the
blow-ups and in particular:
\begin{equation}
\begin{aligned}
M^2_{\eb}=&[M^2;\cB^2,\Diag_\phi]\\
\equiv&[M^2;\cB(Y)^2,\Diag_\phi,\cB(Z)^2,(\cB')^2],
\end{aligned}
\label{1.5.2008.1a}\end{equation}
where the ``squares'' mean the set of self-products of the elements and the
ordering within the boundary faces is immaterial.

The fibration $\phi$ restricts to a fibration, $\phi_B,$ of $B\cap W$ for
each $B\in\cB(Z),$ over the same base $Y.$ For each $B\in\cB(Y)$ instead
$\phi$ restricts to $B\cap W$ to a fibration, again denoted $\phi_B,$ over
$Y(B),$ the corresponding boundary hypersurface of $Y.$ Thus considering
$B\in\cB(Z)$ or $B\in\cB(Y)$ as manifolds with corners on their own, each
inherits a fibration structure as initially given on $W\subset M$ on the
intersection $B\cap W\in\cM_1(B).$ For the elements of $\cB'$ there is a
corresponding trivial structure with no $W.$

\begin{lemma}\label{16.2.2009.8a} The diagonal in $M^2$ lifts to a
  p-submanifold of $M^2_{\eb}.$  The ``front faces'' of $M^2_{\eb},$ those
boundary hypersurfaces  produced by blow up, are of the form
$B^2_{\eb}\times[0,1],$ corresponding to each $B\in\tcB=\cM_1(M)\setminus\{W\}$
with its induced fibration structure.  That corresponding to
$\Diag_\phi$ is the pull-back of the bundle $[W^2;(\cB(Z)\cap(W))^2],$
defined by blowing up the diagonal corners of the fibers, to a (closed)
quarter ball bundle over $Y.$
\end{lemma}

\begin{proof} These statements are all local and follow by elementary
  computations in local coordinates.
\end{proof}

Thus, the definition of the ``small'' calculus of edge-b pseudodifferential
operators is directly analogous to (and of course extends in generality)
\eqref{30.8.2006.47}: 
\begin{equation}
\Psi^m_{\eb}(M)=\left\{A\in I^m(M^2_{\eb};\beta^*\Omega_R);
A\equiv0\Mat \pa M^2_{\bo}\setminus\ff(\beta)\right\}
\label{30.8.2006.48}\end{equation}
where the particular fibration $\phi$ is not made explicit in the
notation. The fact that these kernels define operators on $\dCI(M)$ and
$\CI(M)$ reduces to the fact that push-forward off the right factor of $M,$
which is to say under the left projection, gives a continuous map 
\begin{equation}
(\pi_{L,\phi})_*:\Psi^m_{\eb}(M)\longrightarrow \CI(M).
\label{30.8.2006.49}\end{equation}
The principal symbol map is well-defined at the level of conormal
distributions, taking values in the smooth homogeneous fiber-densities of
the non-zero part of the conormal bundle to the submanifold in question. In
this case $N^*\Diag_{\eb}={}^{\eb}T^*M$ is a natural identification and the
density factors cancel as in the standard case so 
\begin{equation}
\sigma_m:\Psi^m_{\eb}(M)\longrightarrow \CI({}^{\eb}S^*M;N_m)
\label{30.8.2006.50}\end{equation}
where $N_m$ is the bundle of functions which are homogeneous of degree $-m.$ 

The structure of the front faces leads directly to the ``symbolic''
structure of the (small) algebra of pseudodifferential
operators. Namely, there are homomorphisms to model operator algebras
corresponding to each boundary face of $M,$ known as normal
operators.  For faces other than $W$ the model is a parametrized
(``suspended'') family of edge-b (or for those boundary faces not
meeting $W$ simply b-) operators corresponding to the fibrations of
boundary hypersurfaces of $W.$ Note that if $z_j$ is a defining
function for such a face, the operator $z_j D_{z_j}$ maps in this
correspondence to the operation of multiplication by the corresponding
suspension parameter.  For $W$ the model is a family of b-operators on
the fiber times a half-line, parametrized by the cosphere bundle of
the base of the fibration.  (We do not employ the normal operator
homomorphism for the face $W$ in this paper.)

The corresponding triple space can be defined by essentially the same
modifications to the construction of $M^3_{\bo}$ as correspond to obtaining
$M^2_{\eb}$ in place of $M^2_{\bo}.$

\begin{lemma}\label{30.8.2006.51} Under the blow-down map for the partial
  triple b-product 
\begin{equation}
\beta ^3_{\tcB}:M^3_{\tcB}=
[M^3;\cB(Y)^3;\cB(Y)^2;\cB(Z)^3;(\cB')^3;\cB(Z)^2,(\cB')^2]
\longrightarrow M^3
\label{30.8.2006.52}\end{equation}
the triple fiber diagonal and the three partial fiber diagonals 
\begin{equation}
\begin{gathered}
\Diag_{\phi}^3=\{(p,p',p'')\in W^3;\phi(p)=\phi(p')=\phi(p'')\},\\
\Diag^O_{\phi}=(\pi_O)^{-1}(\Diag_{\phi}),\ O=S,C,F,
\end{gathered}
\label{30.8.2006.53}\end{equation}
all lift to p-submanifolds.
\end{lemma}

\begin{proof} This reduces to the same argument as above, namely that the
  triple and three partial diagonals in $Y^3$ are resolved to
  p-submanifolds in $Y^3_{\bo}$ and the effect of the first two sets of
  blow-ups in \eqref{30.8.2006.52} on $Y^3$ is to replace it by
  $Y^3_{\bo}$ and hence to resolve the submanifolds in
  \eqref{30.8.2006.53}. Under the subsequent blow-ups of boundary faces any
  p-submanifold lifts to a p-submanifold.
\end{proof}

Thus we may define the edge-b triple space to be 
\begin{equation}
M^3_{\eb}=[M^3_{\tcB};\Diag_{\phi}^3;\Diag^S_{\phi};\Diag^C_{\phi};\Diag^F_{\phi}].
\label{30.8.2006.54}\end{equation}

\begin{proposition}\label{30.8.2006.55} The three partial diagonals lift to
  b-submanifolds intersecting in the lifted triple diagonal and the three
  projections lift to b-fibrations  
\begin{equation}
\xymatrix{
M^3_{\eb}\ar@<2.5ex>[r]^{\pi_{F,\eb}}\ar[r]^{\pi_{S,\eb}}
\ar@<-2.5ex>[r]^{\pi_{C,\eb}}&M^2_{\phi}.
}
\label{30.8.2006.56}\end{equation}
where $\pi_{O,\eb}$ is transversal to the other two lifted diagonals.
\end{proposition}

\begin{proof} The existence of the stretched projections as smooth maps
  follows from the possibility of commutation of blow-ups. For the sake of
  definiteness, concentrate on $\pi_F,$ the projection onto the right two factors.

After the blow up of the triple fiber diagonal in \eqref{30.8.2006.54}, the
three partial fiber diagonals are disjoint so the other two can be blown up
last. When it is to be blown up, the triple fiber diagonal is a submanifold
of $\Diag^F_{\phi}$ so the order can be exchanged, showing that there is a
composite blown-down map 
\begin{equation}
M^3_{\eb}\longrightarrow [M^3_{\tcB};\Diag^F_{\phi}].
\label{30.8.2006.57}\end{equation}

The manifold with corners $M^3_{\tcB}$ is the b-resolved triple product
where the boundary hypersurface $W$ is ignored. The commutation arguments
showing the existence of a composite blow-down map
$M^3_{\bo}\longrightarrow M\times M^2_{\bo}$ carry over directly to give an
alternative construction
\begin{equation}
M^3_{\tcB}=[M\times M^2_{\tcB};\cF]
\label{30.8.2006.58}\end{equation}
where $\cF$ consists of those boundary faces in \eqref{30.8.2006.52} which
involve a defining function on the first factor of $M^3$---so all the
triple products and the double products with boundary hypersurface in the
first factor. These are all transversal to $\Diag^F_{\phi},$ realized as a
p-submanifold of $M\times M^2_{\tcB}$ so can be commuted past it in the
blow up, giving the map $\pi_{F,\phi}$ in \eqref{30.8.2006.56}. That it is
a b-submersion follows from its definition as a composite of blow-downs of
boundary faces, together with the corresponding fact for the edge
case. That it is a b-fibration follows from the fact that the image of a
boundary hypsersurface is either a boundary hypersurface or the whole
manifold since this is true locally in the interior of boundary
hypersurfaces.
\end{proof}

These facts together show that the small calculus of edge-b
pseudodifferential operators, as defined in \eqref{30.8.2006.48}, is a
filtered algebra. It also follows directly that the symbol
\eqref{30.8.2006.50} is multiplicative as in the standard case. The
extension to operators on sections of bundles is essentially
notational.

If $M$ is non-compact but the fibers of $\phi$ are compact,
the same construction goes through, but we require
proper supports, i.e.\ that the projections $\pi_{L,\phi}$ and
$\pi_{R,\phi}$ are proper when restricted to the support of $A$:
\begin{equation}\begin{split}
\Psi^m_{\eb}(M)=\big\{&A\in I^m(M^2_{\eb};\beta^*\Omega_R);\\
&A\equiv0\Mat \pa M^2_{\bo}\setminus\ff(\beta)\Mand A\ \text{has proper
support}\big\}.
\label{eq:proper-eb}\end{split}\end{equation}
Then $\Psi^m_{\eb}(M)$ acts on $\CI(M)$, $\dCI(M)$ and $\CmI(M)$, as
well as on their compactly supported versions.

\newpage

\printnomenclature

\newpage

\def\cprime{$'$} \def\cprime{$'$} \def\cdprime{$''$} \def\cprime{$'$}
  \def\cprime{$'$} \def\ocirc#1{\ifmmode\setbox0=\hbox{$#1$}\dimen0=\ht0
  \advance\dimen0 by1pt\rlap{\hbox to\wd0{\hss\raise\dimen0
  \hbox{\hskip.2em$\scriptscriptstyle\circ$}\hss}}#1\else {\accent"17 #1}\fi}
  \def\cprime{$'$} \def\ocirc#1{\ifmmode\setbox0=\hbox{$#1$}\dimen0=\ht0
  \advance\dimen0 by1pt\rlap{\hbox to\wd0{\hss\raise\dimen0
  \hbox{\hskip.2em$\scriptscriptstyle\circ$}\hss}}#1\else {\accent"17 #1}\fi}
  \def\cprime{$'$} \def\bud{$''$} \def\cprime{$'$} \def\cprime{$'$}
  \def\cprime{$'$} \def\cprime{$'$} \def\cprime{$'$} \def\cprime{$'$}
  \def\cprime{$'$} \def\cprime{$'$} \def\cprime{$'$} \def\cprime{$'$}
  \def\polhk#1{\setbox0=\hbox{#1}{\ooalign{\hidewidth
  \lower1.5ex\hbox{`}\hidewidth\crcr\unhbox0}}} \def\cprime{$'$}
  \def\cprime{$'$} \def\cprime{$'$} \def\cprime{$'$} \def\cprime{$'$}
  \def\cprime{$'$} \def\cprime{$'$} \def\cprime{$'$} \def\cprime{$'$}
  \def\cprime{$'$} \def\cprime{$'$} \def\cprime{$'$} \def\cprime{$'$}
  \def\cprime{$'$} \def\cprime{$'$} \def\cprime{$'$} \def\cprime{$'$}
  \def\cprime{$'$} \def\cprime{$'$} \def\cprime{$'$} \def\cprime{$'$}
  \def\cprime{$'$} \def\cprime{$'$} \def\cprime{$'$} \def\cprime{$'$}
  \def\cprime{$'$} \def\cprime{$'$}
\providecommand{\bysame}{\leavevmode\hbox to3em{\hrulefill}\thinspace}
\providecommand{\MR}{\relax\ifhmode\unskip\space\fi MR }
\providecommand{\MRhref}[2]{%
  \href{http://www.ams.org/mathscinet-getitem?mr=#1}{#2}
}
\providecommand{\href}[2]{#2}


\end{document}